\documentclass[toc=bib,parskip=half,openany]{scrbook}

\usepackage[utf8]{inputenc}
\usepackage[UKenglish]{babel}

\usepackage{amsmath,amsthm,amssymb} 
\usepackage{mathabx}

\usepackage{enumerate}
\usepackage{enumitem}
\usepackage{mathdots}
\usepackage[margin=1in]{geometry}
\usepackage{latexsym}
\usepackage[dvipdfm]{graphicx} 
\usepackage{bmpsize}
\usepackage{mathtools}
\usepackage{subfig}
\usepackage{xfrac}
\usepackage{fix-cm}
\usepackage{color,transparent}
\usepackage{verbatim}
\usepackage[edge-length=7mm]{dynkin-diagrams}
\usepackage{amsmath, amssymb, etoolbox, expl3, mathtools, pgfkeys, pgfopts, xparse, xstring}

\usepackage{hyperref}
\hypersetup{
  colorlinks,
 linkcolor={red!70!black},
 citecolor={blue!50!black},
 urlcolor={blue!80!black}
}
\usepackage{doi}
\usepackage{stmaryrd}
\usepackage{multirow}
\usepackage{float}
\usepackage{array}
\usepackage{pdfsync}
\usepackage{tikz}
\usepackage{tikz-cd}
\usepackage{float}
\usepackage[T1]{fontenc}
\usepackage[utf8]{inputenc}
\usepackage[textsize=small]{todonotes}

\usepackage{hyphenat}
\usepackage{booktabs}
\usepackage{adforn}
\usepackage{import}

\graphicspath{{mpdraws/}{draws/}}

\setenumerate[1]{label=(\arabic*)}

\numberwithin{equation}{chapter}
\numberwithin{figure}{chapter}

\newtheorem{theorem}{Theorem}[chapter]
\newtheorem{lemma}[theorem]{Lemma}

\newtheorem{corollary}[theorem]{Corollary}
\newtheorem{proposition}[theorem]{Proposition}
\newtheorem{exercise}{Exercise}[chapter]
\newtheorem{hint}{Hint}[chapter]
\newtheorem{ex*}[theorem]{Exercise}
\newtheorem{remark}[theorem]{Remark}

\theoremstyle{definition}
\newtheorem{remarks}[theorem]{Remarks}
\newtheorem{example}[theorem]{Example}
\newtheorem{examples}[theorem]{Examples}
\newtheorem{definition}[theorem]{Definition}

\newtheorem{conjecture}[theorem]{Conjecture}
\newtheorem{notation}[theorem]{Notation}
\newtheorem{fact}[theorem]{Fact}
\newtheorem{warning}[theorem]{Warning}

\newtheorem{upshot}[theorem]{Upshot}

\theoremstyle{plain}

\newcommand*\from{\colon}
\newcommand{\Z}{\mathbb{Z}}

\newcommand{\C}{\mathbb{C}}

\newcommand{\N}{\mathbb{N}}
\newcommand{\R}{\mathbb{R}}

\newcommand{\fraka}{\mathfrak{a}}
\newcommand{\frakg}{\mathfrak{g}}

\newcommand{\rk}{\mathrm{rk}}
\newcommand{\Hit}{\mathrm{Hit}}
\newcommand{\Flag}{\mathrm{Flag}}
\newcommand{\Mat}{\mathrm{Mat}}

\def\H{\ensuremath{\mathbb{H}}} % Redefining "long Hungarian umlaut"

\DeclareMathOperator{\Sp}{\mathrm{Sp}}

\renewcommand{\Re}{\operatorname{Re}}
\renewcommand{\Im}{\operatorname{Im}}

\DeclareMathOperator{\tr}{tr}

\DeclareMathOperator{\Hom}{Hom}
\DeclareMathOperator{\End}{End}

\DeclareMathOperator{\id}{id}

\DeclareMathOperator{\Sym}{Sym}

\newcommand{\del}{\partial}

\newcommand{\tran}[1]{\prescript{t}{}{\!#1}} %transpose 

\renewcommand{\epsilon}{\varepsilon}

\DeclareMathOperator{\OO}{\mathrm{O}}
\DeclareMathOperator{\UU}{\mathrm{U}}
\DeclareMathOperator{\SO}{\mathrm{SO}}
\DeclareMathOperator{\GL}{\mathrm{GL}}

\DeclareMathOperator{\SL}{\mathrm{SL}}
\DeclareMathOperator{\PSL}{\mathrm{PSL}}
\DeclareMathOperator{\SU}{\mathrm{SU}}

\let\sl\relax
\DeclareMathOperator{\sl}{\mathfrak{sl}}
\DeclareMathOperator{\gl}{\mathfrak{gl}}

\DeclareMathOperator{\mfg}{\mathfrak{g}}
\DeclareMathOperator{\mfk}{\mathfrak{k}}
\DeclareMathOperator{\mfp}{\mathfrak{p}}
\DeclareMathOperator{\mfa}{\mathfrak{a}}
\DeclareMathOperator{\mfu}{\mathfrak{u}}
\DeclareMathOperator{\mfl}{\mathfrak{l}}
\DeclareMathOperator{\mfz}{\mathfrak{z}}
\DeclareMathOperator{\stab}{Stab}
\DeclareMathOperator{\Ad}{Ad}
\DeclareMathOperator{\ad}{ad}
\DeclareMathOperator{\visb}{\partial_{vis}X}
\DeclareMathOperator{\cham}{\overline{\mathfrak{a}^+}}
\DeclareMathOperator{\Span}{Span}

\DeclareMathOperator{\Tr}{Tr}

\DeclareMathOperator{\rank}{rank}

\let\P\relax % I know this is not best practice, but do you really use the ¶ character ???
\DeclareMathOperator{\P}{\mathbb{P}}
\DeclareMathOperator{\im}{im}
\DeclareMathOperator{\diag}{diag}
\newcommand{\diff}{\mathrm{d}}
\DeclarePairedDelimiterX\IP[2]\langle\rangle{#1,#2}

\newcommand{\hgline}[3]{
\pgfmathsetmacro{\thetaone}{#1}
\pgfmathsetmacro{\thetatwo}{#2}
\pgfmathsetmacro{\theta}{(\thetaone+\thetatwo)/2}
\pgfmathsetmacro{\phi}{abs(\thetaone-\thetatwo)/2}
\pgfmathsetmacro{\close}{less(abs(\phi-90),0.0001)}
\ifdim \close pt = 1pt
    \draw[#3] (\theta+180:1) -- (\theta:1);
\else
    \pgfmathsetmacro{\R}{tan(\phi)}
    \pgfmathsetmacro{\distance}{sqrt(1+\R^2)}
    \draw[#3] (\theta:\distance) circle (\R);
\fi
}

\DeclarePairedDelimiter\abs{\lvert}{\rvert}
\DeclarePairedDelimiter\norm{\lVert}{\rVert}

\newcommand{\Utheta}{U_{\Theta}^{>0}}

\usetikzlibrary{arrows}
\tikzset{
    labl/.style={anchor=south, rotate=90, inner
    sep=.5mm}
}
\tikzstyle{every picture}=[> = to]
\tikzset{cdlabel/.style={execute at begin node=$\scriptstyle,execute at end node=$}}
\tikzset{implication/.style={double equal sign distance, -implies}}
\tikzset{biimplication/.style={double equal sign distance, implies-implies}}

\makeindex

\begin{document}

\frontmatter

\thispagestyle{empty}

\begin{center}

\Huge{\textit{Proceedings of the Young Researchers Workshop on}}
    
\vspace{6cm}
    
\Huge{\textbf{Positivity in Lie Groups}}

\vfill

\large{\textit{Edited by}}

\vspace{.5cm}

\LARGE  \textbf{Xenia Flamm}

\vspace{.1cm}

\large ETH Zürich, Switzerland

\vspace{.5cm}

\LARGE \textbf{Arnaud Maret}
\vspace{.1cm}

\large Ruprecht-Karls-Universit\"at Heidelberg, Germany

\vspace{2cm}
    
\Large November 2022

\end{center}

\thispagestyle{empty}

\chapter*{Foreword}
\addcontentsline{toc}{chapter}{Foreword}

The workshop on Theta-positivity (hereafter abbreviated $\Theta$-positivity) and Higher Teichm\"uller Theory took place in January 2022 in Heidelberg. The original initiative to host a workshop on $\Theta$-positivity was put forward by Xenia Flamm and Mareike Pfeil in the fall of 2019. Soon after, the COVID-19 pandemic started and the event had to be postponed a couple of times. By the time it was possible to host the workshop in person, Mareike already defended her PhD thesis and had left Heidelberg. I then offered to help Xenia with the local organization in Heidelberg in replacement of Mareike. 

There were twenty-five participants, mostly PhD students, who all took an active part by either giving a talk or preparing an exercise class. The event took place at the Internationales Wissenschaftsforum Heidelberg (IWH) which provided both lecture rooms and bedrooms for the participants. The location is ideal. It is quiet, with bright rooms and a large garden. On the last day of the workshop, we had the pleasure of welcoming Anna Wienhard for a Q\&A session. Anna took all our questions for two hours and provided some expertise on many aspects of $\Theta$-positivity.

Xenia and I later had the idea of writing up a set of notes about the workshop. We are extremely pleased that all the speakers accepted to type in the content of their presentations. After some editing work, we were able to compile all the contributions in a single document. We tried our best to homogenize notations. Shall there remain mistakes or typos, we would greatly appreciate these to be communicated to us.

I want to express my gratitude to Xenia and Mareike for the initial work they put in when they first attempted to organise the workshop during the pandemic. The workshop could have never taken place without the financial contributions from STRUCTURES - Cluster of Excellence Young Researcher’s Convent to which we are extremely grateful.

\bigskip

November 2022 \hfill Arnaud Maret

\begin{center}
\adforn{11}
\end{center}
\bigskip

The idea for this workshop was born in the fall 2019, when Mareike Pfeil and I were visiting MSRI for the semester program \textit{Holomorphic Differentials in Mathematics and Physics}.
Together with Ivo Slegers we started to learn about Lusztig's positivity.
Studying together inspired us to organize a workshop to create an environment in which we can share our knowledge to learn something new.
Unfortunately, both Mareike and Ivo could not attend the final workshop, but I am confident they would have enjoyed it as much as I did.

At this point I must thank Arnaud Maret who so willingly offered to help organize the workshop when Mareike could not do it anymore.
Without his bold decisions the workshop would probably not have taken place at the time it did.
Thank you so much, Arnaud!

I follow Arnaud's words when thanking Anna Wienhard, the IWH and its team, as well as the STRUCTURES - Cluster of Excellence Young Researcher’s Convent.

Lastly I need to express my deep gratitude to all the participants of the workshop, without who it would not have been possible.
Thank you for the time, work and effort you put into your preparations, presentations and now in these final notes.
I hope they will provide a great reference for us and others in the future.

\bigskip

November 2022 \hfill Xenia Flamm
\tableofcontents

\mainmatter

\thispagestyle{empty}
\chapter*{Introduction}
\addcontentsline{toc}{chapter}{Introduction}

These notes transcribe a workshop about the notion of total positivity and $\Theta$-positivity and its relation to Higher Teichm\"uller Theory. $\Theta$-positivity is a notion of positivity in semisimple Lie groups and was recently introduced by Guichard and Wienhard in \cite{GuichardWienhard18} as a generalization of Lusztig's total positivity. It is believed to be the cathartic notion to classify higher Teichmüller spaces. Without doubt, substantial progress will be achieved in the near future on the study of $\Theta$-positive structures. These notes provide an account of the state of the art as of 2021.
For the latest developments and further references (some of which were not treated in the workshop), we refer to \cite{GuichardWienhard_GeneralizingLusztigTotalPositivity}.

Working in representation theory, one often comes across the concept of ``positivity''. It appears in different contexts, for example the positive reals, the order on the circle, total positivity for matrices, positivity of triples in flag varieties and the Maslov index. An especially important role is played by Lusztig's total positivity in the context of split real Lie groups. The new notion of $\Theta$-positivity generalizes this to other types of semisimple Lie groups and, in particular, includes the notion of positivity for Lie groups of Hermitian type.

The notes are organized into four chapters. The first two chapters are reminders of classical material. Chapter~\ref{chap1} is an introduction to \emph{Higher Teichmüller Theory}: the study of connected components of discrete and faithful representations of surface groups into semisimple Lie groups. It focuses on some classical examples of higher Teichmüller spaces consisting of maximal and Hitchin representations. Chapter~\ref{chap2} recalls some important definitions from Lie theory, including root systems, Dynkin diagrams, and representation theory of Lie algebras. The notions of split real Lie groups and Lie groups of Hermitian type are introduced along with an in-depth study of the particular case of the Lie group $\SO(p,q)$.

The last two chapters deal with the various notions of positivity in Lie groups. Chapter~\ref{chap3} introduces the already known concepts of positivity with a special focus on Lusztig's positivity. This is a preparation for the definition of $\Theta$-positivity for semisimple Lie groups which is treated in Chapter~\ref{chap4}. The main motivation for studying positivity lies in its close relation to representation theory and the question of existence of higher Teichmüller spaces. The two well-known family of examples of higher Teichm\"uller spaces - given by maximal representations and by Hitchin representations - can in fact be characterized by their positive structures. The guiding conjecture of Guichard and Wienhard is that so-called $\Theta$-positive representations provide a new class of higher Teichm\"uller spaces. The conjecture has, by now, been proven for the most part in \cite{GLW21} and \cite{BeyrerPozzetti21}, building up on results of \cite{BradlowCollierGarciaPradaGothen} on magical $\mathfrak{sl}_2$-triples.

%%%%%%%%%%%%%%%%%%%%%%%%%%%%%%%%%%%%%%%%%%%%%%%%%%%%%%%%%%%%%%%%%%%%%%%%%%%%%%%%%%%%%%%%%%%%%%%%%%%%%%%%

\part[Higher Teichm{\"u}ller Theory]{Higher Teichm{\"u}ller Theory
\textnormal{
\begin{minipage}[c]{15cm}
\begin{center}
    \vspace{2cm}
    {\Large Sofia Amontova}\\
    \vspace{-4mm}
    {\large \textit{Université de Genève}}\\
    \vspace{.5cm}
    {\Large Lisa Ricci}\\
    \vspace{-4mm}
    {\large \textit{ETH Zürich}}\\
    \vspace{.5cm}
    {\Large Thomas Le Fils}\\
    \vspace{-4mm}
    {\large \textit{Sorbonne Université}}\\
    \vspace{.5cm}
    {\Large Marta Magnani}\\
    \vspace{-4mm}
    {\large \textit{Ruprecht-Karls-Universit\"at Heidelberg}}\\
    \vspace{.5cm}
     {\Large Arnaud Maret}\\
    \vspace{-4mm}
    {\large \textit{Ruprecht-Karls-Universit\"at Heidelberg}}\\
    \vspace{.5cm}
     {\Large Enrico Trebeschi}\\
    \vspace{-4mm}
    {\large \textit{Università di Pavia}}
\end{center}
\end{minipage}
}}\label{chap1}

\thispagestyle{empty}

\chapter[Introduction]{Introduction \\ {\Large\textnormal{\textit{by Sofia Amontova}}}}
\addtocontents{toc}{\quad\quad\quad \textit{Sofia Amontova}\par}

\emph{Higher Teichm{\"u}ller theory} can be roughly speaking understood as the study of representations 
\[
\fbox{\text{surface groups}} \quad \longrightarrow \quad \fbox{\text{interesting Lie groups}}.
\]
The purpose of this introductory section is to present the \emph{classical Teichm{\"u}ller theory} as a prototype: we first characterize it on the one hand via its standard definition as a geometric object (see Section \ref{section: geometric realization}) and on the other hand as an algebraic object (see Section \ref{section: algebraic realization}). The latter perspective allows then to detect the classical Teichm{\"u}ller space as a special object in the representation-theoretical setting (see Section \ref{section: special object}) which is precisely the viewpoint that motivates to define the more general notion of a higher Teichm{\"u}ller space (see Section \ref{section: generalisation to higher}) and serves as a transition to the next section on maximal representations by Lisa Ricci.

\textbf{Setting:} For the entirety of this section, let $\Sigma$ be a closed connected oriented surface of genus $g \geq 2$.

\section{Geometric realization}\label{section: geometric realization}
We first remind of the geometric picture of the Teichm{\"u}ller space and start with its folklore characterization:

\begin{definition}
The \emph{(classical) Teichm{\"u}ller space} $\mathcal{T}(\Sigma)$ is given by
\[
\mathcal{T}(\Sigma) = 
	\{\text{marked conformal structure} \; (X,f) \; \text{on} \; \Sigma\}_{\big/ \thicksim}
\]
where 
\begin{itemize}
	\item $X$ is a Riemann surface, 
	\item with the homotopy equivalence $f: \Sigma \to X$ as its marking,
	\item $(X,f) \thicksim (X,f')$ is an equivalence of two marked conformal structures if and only if  there exists a biholomorphism $\phi: X \to X'$ so that
\[
	\begin{tikzcd}[row sep=tiny]
& X \arrow[dd, "\phi"] \\
\Sigma \arrow[ur, "f"] \arrow[dr, "f'"] & \\
& X'
\end{tikzcd}
\]
homotopy-commutes.
\end{itemize}
\end{definition}

Notably, the diversity of Teichm\"uller theory is for one due to different perspectives of the Teichm\"uller space $\mathcal{T}(\Sigma)$ allowing to endow it
with various structures, such as complex, hyperbolic, symplectic, algebraic structures, including 
several Riemannian metrics, e.g.
\begin{itemize}
	\item \emph{Weil-Petersson metric} which is a K\"ahler metric,
	\item unique K\"ahler-Einstein metric,
	\item Bergmann metric
\end{itemize}
and non-Riemannian metrics, e.g.
\begin{itemize}
	\item \emph{Teichm\"uller metric} (which coincides with the Kobayashi metric) that roughly measures the distortion between conformal structures,
	\item \emph{Thurston metric} that roughly measures the distortion between hyperbolic structures,
	\item Carath\'eodory metric.
\end{itemize}
Moreover, there is a natural discrete action by the mapping class group on the Teichm\"uller space, interesting geodesic and horocyclic flows on its quotient Riemann moduli space, a quantization theory of its Poisson structure, to mention a few. As such Classical Teichm\"uller theory is a rich theory where techniques from geometry, analysis and dynamics confluence; it is an active field in both pure mathematics and theoretical physics.

Now, for our purposes, regarding $\mathcal{T}(\Sigma)$ as the parameter space of conformal structures is not the viewpoint we would like to consider but rather the analogue parameter space of hyperbolic structures known as the Fricke space:

\begin{definition}\label{fricke}
The \emph{Fricke space} $\mathcal{F}(\Sigma)$ is given by 
\[
\mathcal{F}(\Sigma) = 
	\{\text{marked hyperbolic structure} \; (M,h) \; \text{on} \; \Sigma\}_{\big/ \thicksim}
\]
where 
\begin{itemize}
	\item $M$ is a complete hyperbolic surface, 
	\item with the orientation preserving homeomorphism $h\colon \Sigma \to M$ as its marking,
	\item $(M,h) \thicksim (M',h')$ is an equivalence of two marked conformal structures if and only if  there exists a isometry $i\colon M \to M'$ so that
\[
	\begin{tikzcd}[row sep=tiny]
& M \arrow[dd, "i"] \\
\Sigma \arrow[ur, "h"] \arrow[dr, "h'"] & \\
& M'
\end{tikzcd}
\]
isotopy-commutes.
\end{itemize}
\end{definition}

More precisely, we can make the change of viewpoints thanks to the uniformization theorem which implies that $\mathcal{T}(\Sigma)$ can be identified with $\mathcal{F}(\Sigma)$. 

\fbox{
\textbf{Abuse of notation:} From now and on the Teichm\"uller space $\mathcal{T}(\Sigma)$ refers to the Definition \ref{fricke}. 
}

We shall take a closer look at complete hyperbolic structures in the next subsection.

\begin{remark}
To get a better geometric feel for the object $\mathcal{T}(\Sigma)$ we outline the idea of the following fact
\begin{center}
The Teichm\"uller space $\mathcal{T}(\Sigma)$ is homeomorphic to $\R^{6g-6}$. 
\end{center}
\begin{proof}[Idea of proof.]
Using hyperbolic geometry one can see this for instance by parametrising the Teichm\"uller space by \emph{Fenchel-Nielsen coordinates}:
Let $(M,h)$ be a marked hyperbolic structure. 
The closed hyperbolic surface $M$ has a collection of pairwise disjoint simple closed geodesics that decompose the surface into a union of $3g-3$ disjoint pairs of pants\footnote{A pair of pants is a compact hyperbolic surface homeomorphic to a sphere with three boundary components.} (see Figure \ref{figure pants}). 
\begin{figure}[H]
	\centering
    \includegraphics[width=\textwidth]{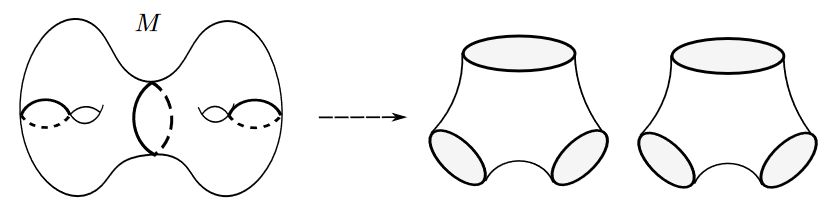}
	\caption{Pants decomposition.}
	\label{figure pants}
	\end{figure}

We call the corresponding lengths by $l_i$  \emph{length parameters} and the twists that determine the gluings pairwise $\tau_i$ \emph{twist parameters} for $1 \leq i \leq 3g-3$. In fact these parameters determine the hyperbolic structure of $M$.
In particular, it holds that the following map is a homeomorphism
\begin{align*}
\mathcal{T}(\Sigma) &\to \R_{>0}^{3g-3} \times \R^{3g-3}, \\
[(M,h)] &\mapsto (l_1,\dots,l_{3g-3},\tau_1,\dots,\tau_{3g-3}),
\end{align*}
where the tuple $(l_1,\dots,l_{3g-3},\tau_1,\dots,\tau_{3g-3})$ is called the \emph{Fenchel-Nielsen coordinates}
 associated to $(M,h)$ following the procedure above.
\end{proof}
\end{remark}

\section{Algebraic realization}\label{section: algebraic realization}

The goal of this subsection is essentially to explain the following 

\fbox{
\textbf{Idea:}
Any hyperbolic structure on $\Sigma$ gives rise to a representation $\rho \from \pi_1(\Sigma) \to  \PSL(2,\R)$.
}

This then permits to view the Teichm\"uller space $\mathcal{T}(\Sigma)$ as an algebraic object by identifying it with a connected component of the \emph{character variety} 
\[
\chi(\Sigma,\PSL(2,\R)):=\mathrm{Hom}(\pi_1(\Sigma), \PSL(2,\R))/{\PSL(2,\R)}.
\]

\paragraph{Ingredients to realise the idea above.}
We first fix a marked hyperbolic structure $(M,h)$ on $\Sigma$; we refer to it as merely $h$ for short. Further, 
\begin{itemize}
	\item
let $s$ be a basepoint of $M$ and 
	\item
let $p\colon \widetilde{M}\to M$ be a universal cover. 
\end{itemize}
Also recall that the group of orientation-preserving isometries of the hyperbolic plane $\mathrm{Isom}^+(\mathbb{H}^2)$ can be identified with $\PSL(2,\R)$.

\paragraph{Motto: Globalization.}
Recall that $M$ comes equipped with the coordinate atlas $\{(U_i, \psi_i)\}_{i \in I}$.  In particular, for each connected $U_i \cap U_j \neq \emptyset$ there exists a unique transition map $g_{i,j} \in \mathrm{Isom}^+(\mathbb{H}^2)$ such that for the coordinate charts one has that
	\begin{equation}\label{transition functions}
		\psi_{i|_{U_i \cap U_j}} = g_{i,j} \circ \psi_{j|_{U_i \cap U_j}}.
	\end{equation}
The strategy now is to globalize the coordinate charts in terms of the universal cover $p\colon \widetilde{M} \to M$ in order to produce an orientation preserving isometry \[f_h\colon \widetilde{M} \to \mathbb{H}^2,\] the so-called \emph{developing map}. This in turn will then induce a representation \[\rho_h\colon\pi_1(M) \to \mathrm{Isom}^+(\mathbb{H}^2).\] 
The construction of the pair $(f_h, \rho_h)$ can be outlined in following two steps:
\begin{enumerate}
	\item \textbf{Construct $f_h$ as a local isometry.} Recall that $\widetilde{M}$ is the set of homotopy classes of paths starting at the basepoint $s$ in $M$. 
	Let $[\gamma]\in \widetilde{M}$ and the path $\gamma \in M$ with endpoints $s$ and $e$ be a choice of representative of $[\gamma]$. 

	We now formalize the idea of globalising the local hyperbolic structure by extending a coordinate chart containing the basepoint $s$ along the path $\gamma$ with help of transition maps $g_{i,j}$ (see Figure \ref{figure developing}):

	We cover $\gamma$ with a finite collection of coordinate charts $\{(U_i,\psi_i)\}_{1 \leq i \leq n}$ and suppose w.l.o.g. $s \in U_1$ and $U_i\cap U_j \neq \emptyset$ is connected for any consecutive $i$ and $j$.
	Using \eqref{transition functions} we can extend $\psi_1\colon U_1 \cap \gamma \to \mathbb{H}^2$ to an isometry $\psi\colon U_1 \cup U_2 \cap \gamma \to \mathbb{H}^2$
	such that 
	\[\psi(x) = \begin{cases}
 	\psi_1(x),  &  \text{if } x \in U_1 \cap \gamma \\
  	g_{1,2} \circ \psi_2(x), & \text{if } x \in U_2\cap \gamma.
	\end{cases}
	\]
	We continue this procedure inductively for the remaining coordinate charts $\{(U_i,\psi_i)\}_{3 \leq i \leq n}$ until we reach the chart $U_n$ containing the endpoint $e$ and we set
	\[
	\psi(e)=g_{1,2} \circ \dots \circ g_{n-1,n} \circ \psi_{n|_{U_n \cap \gamma}}(e) \in \mathbb{H}^2.
	\]

	Finally, we define the map $f_h\colon \widetilde{M} \to \mathbb{H}^2$ by setting
	\[
	f_h([\gamma]):= \psi(e).
	\]
 	
	\begin{figure}[H]
	\centering
    \includegraphics[width=\textwidth]{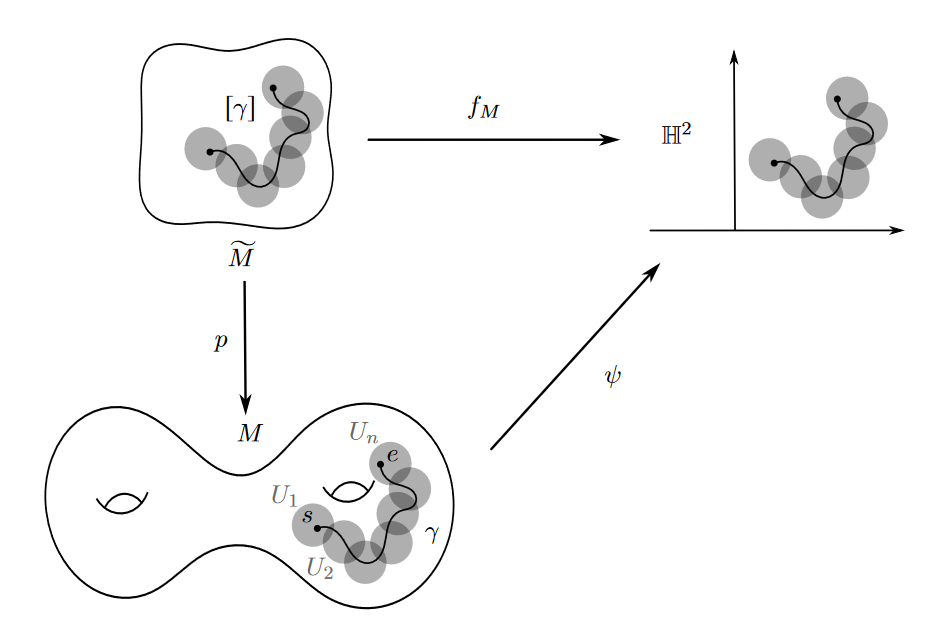}
	\caption{Construction of a developing map $f_h$ as a local isometry.}
	\label{figure developing}
	\end{figure}

	\begin{fact}
	The map $f_h$ depends only on the initial chart (composition by isometry) and the homotopy class of paths (look at succession of small homotopies).
	\end{fact}
	We conclude that $f_h\colon\widetilde{M} \to M$ is well-defined (in particular independent of chart-refinement) and a local orientation-preserving isometry with respect to the hyperbolic structure on $\widetilde{M}$ inherited from $M$.
	\item \textbf {Obtain holonomy $\rho_h$ via $f_h$.}
	Recall that $\pi_1(M)$ can be identified with the group of deck transformations of the universal cover $p\colon \widetilde{M} \to M$ that acts on $\mathbb{H}^2$ by isometries (see exercise 1 for more details). 

	Now for any deck transformation $\gamma \in \pi_1(M)$ there exists a unique $g_\gamma \in \mathrm{Isom}^+(\mathbb{H}^2)$ so that the following diagram commutes
	\[
	\begin{tikzcd}
	\widetilde{M} \arrow[r, "f_h"] \arrow[d, "\gamma"]
	& \mathbb{H}^2 \arrow[d, "g_\gamma"] \\
	\widetilde{M} \arrow[r,"f_h"]
	& \mathbb{H}^2
	\end{tikzcd}.
	\]
	That is $f_h$ is a $\rho_h$-equivariant map, where $\rho_h$ is given by
	\begin{align*}
		\rho_h\colon \pi_1(M) &\to \mathrm{Isom}^+(\mathbb{H}^2), \\
		\gamma &\mapsto g_\gamma.
	\end{align*}
	In fact $\rho_h$ is a homomorphism called the \emph{holonomy representation} of the hyperbolic structure $(M,h)$.
	\begin{remark}\label{remark: global}\;
	\begin{enumerate}
		\item The pair $(f_h,\rho_h)$ is unique up to a $\PSL(2,\R)$-action defined by
		\[
			(f_h, \rho_h) \xrightarrow{g} (g \circ f_h, g  \rho_h g^{-1}).
		\]
		\item In accordance with our motto stated above notice that indeed
		\begin{itemize}
			\item $f_h$ globalizes the coordinate charts,
			\item $\rho_h$ globalizes the coordinate changes.
		\end{itemize}
	\end{enumerate}
	\end{remark}
\end{enumerate}
	Finally, observe that completeness of $M$ implies that $f_h \from \widetilde{M} \to \mathbb{H}^2$ is in fact a \emph{global} orientation preserving isometry.
	\begin{proof}[Idea of the proof.]
	Completeness of $M$ implies completeness of $\widetilde{M}$ so that $f_h\colon\widetilde{M} \to \mathbb{H}^2$ is a surjective covering via the path-lifting property. Since $\widetilde{M}$ is simply connected and $f_h$ is a homeomorphism and a local isometry, then indeed $f_h$ is a global isometry.
	\end{proof}
	Conversely, it is an exercise to prove that $M \cong \mathbb{H}^2/\rho_h(\pi_1(M))$.

\begin{upshot}
We have constructed the injective map
\begin{align*}
	\mathrm{hol} \from \mathcal{T}(\Sigma) &\to \chi (\Sigma,\PSL(2,\R))\\
	[(M,h)] &\mapsto [\rho:=\rho_h \circ h_* \from \pi_1(\Sigma)\cong\pi_1(M) \to \PSL(2,\R)].
\end{align*}
\end{upshot}

The following important fact is to be checked
\begin{fact}
$\mathcal{T}(\Sigma)$ can be identified with a connected component consisting entirely of 
\begin{equation}\label{special}
\fbox{
\text{faithful and discrete} 
}
\end{equation}
representations in the representation variety $\chi (\Sigma,\PSL(2,\R))$, i.e.\ every holonomy is a discrete embedding into $\PSL(2,\R)$.
\end{fact}

\begin{remark}
$\mathcal{T}(\Sigma)$ and $\mathcal{T}(\overline{\Sigma})$ (Teichm\"uller space associated to a closed surface with reversed orientation $\overline{\Sigma}$) are the only two connected components with the special property \eqref{special}.
\end{remark}

\section{Tool to detect Teichm\"uller space as a special connected component in the character variety}
\label{section: special object}
 
We shall introduce a tool that will later generalize to a broader setting to single out special analogues of $\mathcal{T}(\Sigma)$ in the sense of \eqref{special} when passing from $\PSL(2,\R)$ to some more general Lie group $G$. 

Let $\rho \from \pi_1(\Sigma) \to \PSL(2,\R)$ be a 
representation and $p \from \widetilde{\Sigma} \to \Sigma$ a universal cover. 

\begin{fact}\label{fact: flat bundle construction}
There exists a smooth $\rho$-equivariant map $f \from \widetilde{\Sigma} \to \mathbb{H}^2$. 
\end{fact}

\begin{proof}[Idea of the proof.]
We construct the associated flat $(\PSL(2,\R),\mathbb{H}^2)$-bundle 
\[
	E_\rho:=\widetilde{\Sigma} \times \mathbb{H}^2/\pi_1(\Sigma) \to \Sigma.
\]
Since the fibres $\mathbb{H}^2$ are contractible, there exists a smooth 
section $\Sigma \to E_\rho$ which lifts to a smooth
$\rho$-equivariant
map $f \from \widetilde{\Sigma} \to \widetilde{\Sigma} \times \mathbb{H}^2 \to \mathbb{H}^2$. 
\end{proof}

Since $f$ is $\rho$-equivariant, the volume form $\omega_{\mathbb{H}^2} \in \Omega^2(\mathbb{H}^2,\R)^{\PSL(2,\R)}$ pullbacks to $f^*\omega_{\mathbb{H}^2} \in \Omega^2(\widetilde{\Sigma},\R)^{\pi_1(\Sigma)}$.
We then may take the pushforward $\overline{f^*\omega_{\mathbb{H}^2}}:=p_*f^*\omega_{\mathbb{H}^2}$  summarizing in the following diagram:
\begin{center}
\begin{tikzcd}
(\widetilde{\Sigma},f^*\omega_{\mathbb{H}^2})\arrow{d}{p} \arrow{r}{f}
& (\mathbb{H}^2,\omega_{\mathbb{H}^2})  \\
(\Sigma,\overline{f^*\omega_{\mathbb{H}^2}}).
\end{tikzcd}
\end{center}

We now may define our detection tool:
\begin{definition}
The \emph{Toledo number} associated to the representation $\rho$ is given by
\[
	\tau(\rho) = \frac{1}{2\pi}\int_\Sigma \overline{f^*\omega_{\mathbb{H}^2}} \in \R.
\]
\begin{remark}
Notice that $\tau(\rho)$ is well-defined as any two $\rho$-equivariant maps are homotopic.
\end{remark}
\end{definition}
Now in case $\rho=\rho_h$ is a \emph{holonomy} representation associated to a hyperbolic structure $(M,h)$ and we rerun the procedure above, then one can check that in the Fact \ref{fact: flat bundle construction} we obtain a $\rho_h$-equivariant map $f_h$ as the unique \emph{developing} map (and thus an isometry) resulting from a lift of a \emph{developing} section $\Sigma \to E_{\rho_h}$.

\begin{remark}
As a continuation of the philosophy in Remark \ref{remark: global}, notice that 
the \emph{developing} section can be obtained as a graph by globalizing the coordinate atlas. 
\end{remark}

In particular, we can take $f_M$ as the lift of the orientation-preserving homeomorphism (marking) $h \from \Sigma \to M$ and the diagram above extends to the following:

\begin{center}
\begin{tikzcd}
(\widetilde{\Sigma},f^*\omega_{\mathbb{H}^2})\arrow{d}{p} \arrow{r}{f_M}
& (\mathbb{H}^2,\omega_{\mathbb{H}^2}) \arrow{d}  \\
(\Sigma,\overline{f^*\omega_{\mathbb{H}^2}}) \arrow{r}{h}&(M,\omega_{\mathbb{H}^2}),
\end{tikzcd}
\end{center}
which together with Gau{\ss}-Bonnet gives that
\[
\tau(\rho_M) = \frac{1}{2\pi}\int_{f(\Sigma)} \omega_{\mathbb{H}^2} = \frac{1}{2\pi}\int_M \omega_{\mathbb{H}^2} =-\frac{2\pi}{2\pi}|\chi(\Sigma)|=2g-2.
\]
This is in fact the maximal absolute number the Toledo number can take: 
\begin{theorem}[Milnor-Wood inequality]
For any representation $\rho \from \pi_1(\Sigma) \to \PSL(2,\R)$ we have that
\[
|\tau(\rho)| \leq 2g-2.
\]
\end{theorem}
Crucially, the following theorem by Goldman characterizes this numerical invariant as a detection tool for the special property \eqref{special}:
\begin{theorem}[\cite{Gol88}]
\begin{itemize}
	\item $\tau(\rho)$ distinguishes connected components in $\chi(\Sigma,\PSL(2,\R))$ and has values in $\Z \cap [\chi(\Sigma),-\chi(\Sigma)]$.
	That is, there are $4g-3$ components.
	\item $\rho$ is a holonomy representation of a hyperbolic structure if and only if $\tau(\rho)=2g-2$.
\end{itemize}
\end{theorem}

\begin{upshot}
$\mathcal{T}(\Sigma)$ is the connected component singled out by the maximal Toledo number (in analogy $\mathcal{T}(\overline{\Sigma})$ by the minimal Toledo number). 
\end{upshot}

\begin{remark}
The theorems above were actually proven for the so-called Euler number $e(\rho)$ of a representation, but in fact $\tau(\rho)=e(\rho)$.
\end{remark}

The maximality property of the Toledo invariant holds in a more general setting leading to the study of maximal representations (see next Chapter~\ref{lisa}).

\section{Generalization}\label{section: generalisation to higher}

Let $G$ be a Lie group.
This section culminates in generalizing the special property~\eqref{special} in this more general setting:
\begin{definition}
A \emph{higher Teichm{\"u}ller space} is a subset of $\chi(\Sigma,G):=\mathrm{Hom}(\pi_1(\Sigma), G)/G$ which is a union of connected components that consists entirely of discrete and faithful representations.
\end{definition}

\begin{remark}
The existence of higher Teichm{\"u}ller spaces is a prior not clear.
In fact unless $G$ is locally isometric to $\PSL(2,\R)$, the set of discrete and faithful representations is only a closed  set in $\mathrm{Hom}(\pi_1(\Sigma), G)/{G}$. 
A non-example is when $G$ is a simply-connected complex Lie group.  
\end{remark}

\newpage

\thispagestyle{empty}

\chapter[Maximal Representations]{Maximal Representations \\ {\Large\textnormal{\textit{by Lisa Ricci}}}}
\addtocontents{toc}{\quad\quad\quad \textit{Lisa Ricci}\par}\label{lisa}

\section{Hermitian Lie groups: two examples}

In the following $G$ is a Lie group of Hermitian type. In particular, it has an associated symmetric space $X=G/K$, where $K$ is a maximal compact subgroup, such that
\begin{enumerate}
    \item $X$ admits a $G$-invariant metric $\langle\cdot ,\cdot\rangle$,
    \item $X$ has a complex manifold structure with an almost complex structure $J$ such that
    \begin{itemize}
    \item the metric is Hermitian: $\langle v,w\rangle =\langle J_xv,J_xw \rangle$ for all $v,w\in T_xX$,
    \item $J$ is $G$-invariant: $d_x L_g \circ J_x = J_{gx}\circ d_xL_g$.
    \end{itemize}
\end{enumerate}
Recall that an almost complex structure $J$ assigns to each $x\in X$ an endomorphism $J_x\in\End(T_x)$ such that $J_x^2=-\id$.

It follows that if one sets $\omega_X(\cdot,\cdot)\coloneqq \langle J\cdot,\cdot\rangle $ one obtains the so called \textit{K\"ahler form} $\omega_X\in\Omega^2(X)$, which is
\begin{itemize}
    \item $G$-invariant: $L_g^\ast \omega = \omega$ for all $g\in G$,
    \item non-degenerate,
    \item closed.
\end{itemize}

\begin{example}
\begin{enumerate}
    \item $G=\SL(2,\R)$, $K=\OO(2)$, $X=G/K\simeq\H^2$. The Riemannian metric is given by $ds^2=\frac{dx^2+dy^2}{y^2}$ and identifying $T_x\H^2\simeq \R^2$ the almost complex structure is $J_x\begin{pmatrix}v\\w
    \end{pmatrix}=\begin{pmatrix}-w\\v
    \end{pmatrix}$.
    Therefore in $i\in\H^2$
    \begin{align*}
        (\omega_{\H^2})_i\left(\begin{pmatrix}v_1\\w_1
    \end{pmatrix},\begin{pmatrix}v_2\\w_2
    \end{pmatrix}\right) 
        &=\left\langle J_i\begin{pmatrix}v_1\\w_1
    \end{pmatrix},\begin{pmatrix}v_2\\w_2
    \end{pmatrix}\right\rangle_i
        = \begin{pmatrix}-w_1\\v_1
    \end{pmatrix}\cdot \begin{pmatrix}v_2\\w_2
    \end{pmatrix}\\
        &=v_1w_2-w_1v_2
        = (dx\wedge dy) \left(\begin{pmatrix}v\\w
    \end{pmatrix},\begin{pmatrix}v\\w
    \end{pmatrix}\right)        
    \end{align*}
    
    \item $G=\Sp(2n,\R)$, $K=\Sp(2n,R)\cap \OO(2n,\R) \cong \UU(n)$ and the associated symmetric space is the \textit{Siegel upper half-space}
    \[
    \mathcal{X}_n=\lbrace A+iB : A,B\in\Sym(n,\R), B\gg 0\rbrace.
    \]
    
    The $G$-action on $\mathcal{X}_n$ is by M\"obius transformations: given $g=\begin{pmatrix}A&B\\C&D
    \end{pmatrix}\in\Sp(2n,\R)$ and $Z\in \mathcal{X}_n$: 
    \[g\cdot Z= (AZ+B)(CZ+D)^{-1}.
    \]
    Under the identification $T_{A+iB}\mathcal{X}_n\simeq\lbrace V+iW : V,W\in\Sym(n,\R)\rbrace$ the Riemannian metric is given by
    \[\langle U_1,U_2\rangle _{iI_n}=\tfrac{1}{2}\Tr(U_1\overline{U_2}+\overline{U_2}U_1).
    \]
    The almost complex structure is multiplication by $i$. Then the K\"ahler form is
    \[(\omega_X)_{iI_n}(V_1+iW_1,V_2+iW_2) = \Tr(V_1W_2 - W_1V_2).
    \]
\end{enumerate}
\end{example}

\section{Toledo number: definition and examples}

Let $\Sigma$ be a closed\footnote{That is compact and without boundary.} surface of genus $g\geq 2$, $G$ be a Lie group of Hermitian type with associated symmetric space $X$ and let $\rho\colon\pi_1(\Sigma)\rightarrow G$ be a homomorphism. Let $p\colon\widetilde{\Sigma}\rightarrow \Sigma$ be the universal cover of $\Sigma$.

\begin{fact}
There exists a smooth $\rho$-equivariant map $f_\rho\colon\widetilde{\Sigma}\rightarrow X$.
\end{fact}

The pullback $f_\rho^\ast\omega_X\in\Omega^2(\widetilde{\Sigma})$ of the K\"ahler form is $\pi_1(\Sigma)$-invariant and therefore defines a $2$-form $\overline{f_\rho^\ast\omega_X}\in\Omega^2(\Sigma)$, which satisfies
\[p^\ast\overline{f_\rho^\ast\omega_X} = f_\rho^\ast\omega_X.
\]

\begin{definition} Retain the above notation. The \textit{Toledo number} of $\rho$ is 
\[
\tau(\rho)\coloneqq \frac{1}{2\pi}\int_\Sigma \overline{f_\rho^\ast\omega_X}.
\]
\end{definition}

\begin{fact}\label{fact - toledo is well-def}
\begin{enumerate}
    \item The existence of a $\rho$-equivariant map is equivalent to the existence of a section of the bundle
    \[\pi_1(\Sigma)\setminus(\widetilde{\Sigma}\times X)\rightarrow \pi_1(\Sigma)\setminus\widetilde{\Sigma}\cong \Sigma
    \]
    with contractible fiber $X$.
    \item $\tau(\rho)$ is well-defined since any two $\rho$-equivariant maps $\widetilde{\Sigma}\rightarrow X$ are homotopic via a $\rho$-equivariant homotopy (this follows from the fact that $X$ is contractible).
    \item For all $g\in G$ it holds $\tau(g\rho g^{-1})=\tau(\rho)$.
\end{enumerate}
\end{fact}

\begin{example}[The holonomy representation $\pi_1(\Sigma)\rightarrow \PSL(2,\R)$]
Let $\Sigma$ be a closed surface of genus $g\geq 2$, $p\from \widetilde{\Sigma}\rightarrow\Sigma$ the universal cover, $h$ an hyperbolic structure on $\Sigma$ with associated developing map $f_h \from \widetilde{\Sigma}\rightarrow \H^2$.
This is an orientation-preserving isometry, in particular $f_h\in\PSL(2,\R)$ and the holonomy representation is
\begin{align*}
    \rho_h \from \pi_1(\Sigma)&\rightarrow \PSL(2,\R)\\
    \gamma &\mapsto f_h\circ \gamma \circ f_h^{-1},
\end{align*}
where $\gamma\in\PSL(2,\R)$ denotes the action of $\pi_1(\Sigma)$ on $\widetilde{\Sigma}$ by deck-transformations.

We see immediately that $f_h$ is $\rho_h$-equivariant: for all $\gamma\in \pi_1(\Sigma)$, $x\in\widetilde{\Sigma}$ it holds $f_h(\gamma\cdot x)=\rho_h(\gamma)\cdot f(x)$.

Let $\omega_\Sigma,\omega_{\widetilde{\Sigma}}$ and $\omega_{\H^2}$ be the volume forms on $\Sigma$, $\widetilde{\Sigma}$ and $\H^2$, respectively. Then since local isometries preserve the volume form it holds $p^\ast\omega_\Sigma=\omega_{\widetilde{\Sigma}}$ and $f_h^\ast\omega_{\H^2}=\omega_{\widetilde{\Sigma}}$. Therefore
\[p^\ast\omega_\Sigma=\omega_{\widetilde{\Sigma}}=f_h^\ast\omega_{\H^2}.
\]
Therefore
\[\tau(\rho_h) = \frac{1}{2\pi}\int_\Sigma \overline{f_h^\ast\omega_{\H^2}} = \frac{1}{2\pi}\int_\Sigma\omega_\Sigma = \frac{2\pi}{2\pi}\vert \chi(\Sigma)\vert = 2g-2.
\]
In the second-to-last equality we used the Gau{\ss}-Bonnet theorem.
\end{example}

\begin{example}[The diagonal embedding] Recall that 
\[
\Sp(2n,\R)=\lbrace x\in\SL(2n,\R) : \tran{x}J_{n,n}x = J_{n,n}\rbrace,
\]
where $J_{n,n}=\begin{pmatrix} 0 & I_n \\ -I_n & 0\end{pmatrix}$. Consider the diagonal embedding
\begin{align*}
    d:\SL(2,\R) &\rightarrow \Sp(2n,\R)\\
    A=\begin{pmatrix}a&b\\c&d
    \end{pmatrix}   &\mapsto \begin{pmatrix} aI_n&bI_n\\cI_n&dI_n
    \end{pmatrix}.
\end{align*}
Let $X$ be the Siegel upper half space and define $\varphi\from \H^2\rightarrow X$, $z\mapsto zI_n$. Then $\varphi$ is $d$-equivariant:
\begin{align*}
    d(A)\varphi(z) &=(aI_n\varphi(z)+bI_n)(cI_n\varphi(z)+dI_n)^{-1} = (az+b)I_n((cz+d)I_n)^{-1}=\frac{az+b}{cz+d}I_n
    =\varphi(Az).
\end{align*}
Let $\rho\colon\pi_1(\Sigma)\rightarrow\SL(2,\R)$ be any holonomy with developing map $f\colon\widetilde{\Sigma}\rightarrow \H^2$. We compute the Toledo number $T(d\circ \rho)$. It is easy to check that $\varphi\circ f\colon\widetilde{\Sigma}\rightarrow X$ is ($d\circ \rho$)-equivariant. Moreover, $\varphi^\ast\omega_X=n\omega_{\H^2}$.
Indeed,
\begin{align*}
    (\varphi^\ast\omega_X)_i\left(\begin{pmatrix}v_1\\w_1\end{pmatrix},\begin{pmatrix}v_2\\w_2\end{pmatrix}\right)
    &=(\omega_X)_{iI_n}\left(d_i\varphi\begin{pmatrix}v_1\\w_1\end{pmatrix},d_i\varphi\begin{pmatrix}v_2\\w_2\end{pmatrix}\right)\\
    &=(\omega_X)_{iI_n}\left((v_1+iw_1)I_n, (v_2+iw_2)I_n\right)\\
    &=\Tr(v_1w_2I_n - w_1v_2I_n) = n(v_1w_2-w_1v_2) \\
    &= n(\omega_{\H^2})_i\left(\begin{pmatrix}v_1\\w_1\end{pmatrix},\begin{pmatrix}v_2\\w_2\end{pmatrix}\right).
\end{align*}
Thus $(\varphi\circ f)^\ast\omega_X=f^\ast\varphi^\ast\omega_X=nf^\ast\omega_{\H^2}$ and
\[\tau(d\circ \rho) = \frac{1}{2\pi}\int_\Sigma \overline{(\varphi\circ f)^\ast\omega_X} = \frac{n}{2\pi}\int_\Sigma\overline{f^\ast\omega_{\H^2}} = n\tau(\rho) = n\vert \chi(\Sigma)\vert.
\]
\end{example}

\section{Maximal representations}

The Toledo number has the following properties.

\begin{proposition}[{\cite[Corollary 5.7, Corollary 5.9]{BIW14}}] Let $G$ be a Lie group of Hermitian type and $\Sigma$ be a closed surface of genus $g\geq 2$.
\begin{enumerate}
    \item The map \begin{align*}
        \tau\colon \Hom(\pi_1(\Sigma),G)&\rightarrow \R\\
        \rho&\mapsto \tau(\rho)
    \end{align*}
    is continuous.
    \item $\tau$ takes discrete values. More precisely, it takes values in $\frac{1}{n_X}\Z$, where $n_X\in\N$ depends only on $G$.
    \item $\vert \tau(\rho) \vert \leq \rank(X) \vert\chi(\Sigma)\vert$.
\end{enumerate}
\end{proposition}

\begin{definition}
A representation $\rho\colon\pi_1(\Sigma)\rightarrow G$ is \textit{maximal} if $\tau(\rho)=\rank(X)\vert \chi(\Sigma)\vert$.
\end{definition}

Since $\tau$ is continuous and takes discrete values, it is constant on connected components of $\Hom(\pi_1(\Sigma),G)$ and the set of maximal representations is a union of connected components.

By Fact \ref{fact - toledo is well-def} (3) the Toledo number gives a well-defined continuous map on $\chi(\Sigma,G)$ and the set of maximal representations in $\chi(\Sigma,G)$ is also a union of connected components.

\begin{definition}
A \textit{higher Teichm\"uller space} is a union of connected components of $\chi(\Sigma,G)$ which consists entirely of discrete and faithful representations.
\end{definition}

For $G\neq \PSL(2,\R)$ the set of discrete and faithful representations is only a closed subset of $\chi(\Sigma,G)$, so it is not even clear that higher Teichm\"uller spaces exist. However, we have the following.

\begin{proposition}[{\cite[\S 4]{BIW10}}]\label{prop - max rep are injective and discrete image} Maximal representations are injective and have discrete image.
\end{proposition}

The proof relies on the following theorem.

\begin{theorem}[{\cite[Theorem 8]{BIW10}}]\label{theorem - max reps are positive} Let $G$ be a Lie group of Hermitian type with associate symmetric space $X$ and let $\rho\colon\pi_1(\Sigma)\rightarrow G$ be a representation. Then $\rho$ is maximal if and only if there exists a continuous $\rho$-equivariant map $\varphi\colon \partial \H^2\rightarrow \widecheck{\Sigma}$ which sends positively oriented triples in $\partial\H^2\cong S^1$ to maximal triples in the \emph{Shilov boundary} $\widecheck{\Sigma}$ of $X$.
\end{theorem}

\begin{remark}
\begin{enumerate}
    \item Here the action of $\pi_1(\Sigma)$ on $\partial\H^2$ is given by choosing an hyperbolization of $\pi_1(\Sigma)$ as a lattice in $\PSL(2,\R)$.
    \item In the setting of $G$ being an Hermitian Lie group, a triple of elements $(L_1,L_2,L_3)\in\widecheck{\Sigma}^3$ is maximal if its Maslov index $M_{\widecheck{\Sigma}}(L_1,L_2,L_3)$ takes the maximal value possible, which is $\rk(X)$.
    
    For example, when $G=\Sp(2n,\R)$ the Shilov boundary $\widecheck{\Sigma}$ consists of the Lagrangian subspaces of $\R^{2n}$ and a triple of Lagrangian subspaces $(L_1,L_2,L_3)$ is maximal if and only if its Maslov index is equal to $n$.
\end{enumerate}
\end{remark}

We now sketch the proof of Proposition \ref{prop - max rep are injective and discrete image}.

\begin{proposition}[{\cite[\S 4.4]{BIW10}}] Let $G$ be a Lie group of Hermitian type and $\Sigma$ a closed surface of genus $g\geq 2$. Then maximal representations $\rho \from \pi_1(\Sigma)\rightarrow G$ are injective.
\end{proposition}
\begin{proof}
Suppose by contradiction that $\ker(\rho)\neq \{1\}$. Then there exists $\gamma\in\ker(\rho)$ of infinite order and we can find an open interval $I\subset \partial\H^2\cong S^1$ such that $I,\gamma I,\gamma^2 I$ are pairwise disjoint and positively oriented. Choose $I_1,I_2,I_3\subset I$ pairwise disjoint and positively oriented. 

Let $\varphi \from \partial \H^2\rightarrow \widecheck{\Sigma}$ be the continuous, $\rho$-equivariant map which preserves positivity given by Theorem \ref{theorem - max reps are positive}.

Let $(x,y,z)\in I_3\times \gamma I_2\times \gamma^2 I_1$. Since this triple is positively oriented in $\partial\H^2$, we have
\[M_{\widecheck{\Sigma}}(\varphi(x),\varphi(y),\varphi(z))=\rk(X).
\]
On the other hand, the triple $(x,\gamma^{-1}y,\gamma^{-2}z)$ is negatively oriented and therefore
\[M_{\widecheck{\Sigma}}(\varphi(x),\varphi(\gamma^{-1}y),\varphi(\gamma^{-2}z))=-\rk(X).
\]
By $\rho$-equivariance and since $\rho(\gamma)=e$:
\[\varphi(\gamma^{-1}y)=\rho(\gamma)^{-1}\varphi(y)=\varphi(y),
\]
and 
\[\varphi(\gamma^{-2}z)=\rho(\gamma^{-2})\varphi(z)=\varphi(z).
\]
But this implies $\rk(X)=-\rk(X)$, which is the desired contradiction.
\end{proof}

The proof of the fact that maximal representations have discrete image can be found in \cite{BIW10}, Section 4.3.

\section{Relation to Theta-positivity}
Lie groups of Hermitian type are an example of a broader class of Lie groups, namely those which \textit{admit a $\Theta$-positive structure}. If $G$ is such a Lie group, there is a notion of $\Theta$-positive representations $\pi_1(\Sigma)\rightarrow G$ and Theorem \ref{theorem - max reps are positive} tells us that maximal representations into Hermitian Lie groups are $\Theta$-positive.

It turns out that for any simple Lie group $G$ admitting a $\Theta$-positive structure there exist higher Teichm\"uller spaces in $\chi(\Sigma,G)$.

\begin{theorem}[{\cite[Theorem A]{GLW21}}] Let $G$ be a simple Lie group admitting a $\Theta$-positive structure. Then there exists a higher Teichm\"uller space in $\chi(\Sigma,G)$.
\end{theorem}

The theorem is a corollary of the following two facts.

\begin{theorem}[{\cite[Theorem B and Theorem E]{GLW21}}] Let $G$ be a simple Lie group admitting a $\Theta$-positive structure. 
\begin{enumerate}
    \item There is the following inclusion of sets of representations 
    \[\lbrace\Theta\text{-positive}\rbrace \subseteq\lbrace\Theta\text{-Anosov}\rbrace\subseteq\lbrace \text{discrete and faithful representations}\rbrace.
    \]
    \item There exists a union of connected components of $\chi(\Sigma,G)$ consisting entirely of $\Theta$-positive representations.
\end{enumerate}
\end{theorem}

It is still open whether or not the set of $\Theta$-positive representations forms higher Teichm\"uller spaces.

\begin{conjecture}[\cite{Wie18}]
The set of $\Theta$-positive representations is open and closed in $\chi(\Sigma,G)$.
\end{conjecture}

\newpage

\thispagestyle{empty}

\chapter[Hitchin Representations]{Hitchin Representations \\ {\Large\textnormal{\textit{by Thomas Le Fils}}}}
\addtocontents{toc}{\quad\quad\quad \textit{Thomas Le Fils}\par}

\section{Definition}
Let $\Sigma$ be a connected closed oriented surface of genus $g \geqslant 2$. Recall that the character variety $\chi(\Sigma,\PSL (2,\R))$ has two connected components $\mathcal{T}(\Sigma)$ and $\mathcal{T}(\overline \Sigma)$ consisting of discrete and faithful representations.

For every $n\geqslant 2$, there exists a unique irreducible representation 
\[\iota_n \from \SL (2,\R)\to \SL (n,\R)\] up to conjugacy.
Namely we can take $\iota_n$ as the map that sends $A = \begin{pmatrix}
a & b\\
c & d
\end{pmatrix}$ to the matrix of the linear map $f_A$ in the basis $\mathcal B$ where $f_A$ is the following:
    \begin{align*}
        f_A\colon\R_{n-1}[X,Y]&\longrightarrow \R_{n-1}[X,Y]\\
        P&\longmapsto P(aX + cY, bX + dY)
    \end{align*}

and $\mathcal B = (X^{n-1}, X^{n-2}Y, \ldots, Y^{n-1})$.

\begin{remark}
The matrix $\begin{pmatrix}
\lambda & 0\\
0 & \lambda^{-1}
\end{pmatrix}$ is sent to the matrix $\begin{pmatrix}
\lambda^{n-1} & 0 & 0 & 0\\
0 & \lambda^{n-3} & 0 & 0\\
0 & 0 & \ddots   & 0\\
0 & 0 & 0 & \lambda^{1-n}
\end{pmatrix}$. Therefore hyperbolic matrices are sent to diagonalizable ones with pairwise distinct eigenvalues.
\end{remark}

Therefore the map $\iota_n$ induces $\iota_n\colon \PSL (2,\R)\to \mathrm{PSL}(n,\R)$.
\begin{example}
If $n=2$, then $\mathcal B = (X,Y)$.
Let $A = \begin{pmatrix}
a & b\\
c & d
\end{pmatrix}$.
The map $f_A$ sends $X$ to $aX + cY$ and $Y$ to $bX + dY$ hence $\iota_2$ is the identity. 
\end{example}

\begin{remark}
The \emph{Veronese embedding}
    \begin{align*}
        V\colon\mathbb{RP}^1&\longrightarrow \mathbb{RP}^{n-1}\\
        [x:y]&\longmapsto [x^{n-1}: x^{n-2}y:\ldots;y^{n-1}]
    \end{align*}
satisfies for all $A\in \SL (n,\R)$, for all $p\in  \mathbb{RP}^1$
\[V(A\cdot p) = \iota_n(A)\cdot V(p).\]
\end{remark}

Let us now define Fuchsian representations.
\begin{definition}
A \emph{Fuchsian representation} is a representation $\rho\colon\pi_1(\Sigma)\to \PSL(n,\R)$ of the form $\iota_n\circ j$ where the conjugacy class of $j$ belongs to $\mathcal{T}(\Sigma)$.
\end{definition}

We can now define the Hitchin component.

\begin{definition}
A \emph{Hitchin representation} is a representation $\rho\colon \pi_1(\Sigma)\to \PSL(n,\R)$ that can be deformed into a Fuchsian representation. In other words, its conjugacy class lies in the same connected component of $\chi(\Sigma,\PSL(n,\R))$ as that of a Fuchsian representation. 
\end{definition}

Let us call this component the \emph{Hitchin component} and denote it by $\Hit_n$.
Hitchin classified the connected components of $\chi(\Sigma, \PSL(n,\R))$ in \cite{Hitchin92}.
\begin{theorem}[Hitchin]
Let $n\geq 3$.
The number of connected components of $\chi(\Sigma,\PSL(n,\R))$ is 3 if $n$ is odd, 6 if $n$ is even.
\end{theorem}

In the same article, Hitchin showed that this component has trivial topology.
\begin{theorem}
The Hitchin component $\Hit_n$ is homeomorphic to $\R^{(2g-2)(n^2-1)}$.
\end{theorem}

\begin{remark}
If $n$ is even, there exists another copy $\overline {\Hit}_n$ of the Hitchin component.
The sets $\Hit_n$ and  $\overline {\Hit}_n$ are obtained from each other by conjugation by an element of $\GL(n,\R)$ with negative determinant.
\end{remark}

\begin{ex*} Let $j\colon \pi_1(\Sigma)\to \PSL (2,\R)$ be such that  $[j]\in{\mathcal{T}(\overline \Sigma)}$.
The conjugacy class of the representation $\iota_n\circ j$ is in $\overline{\Hit}_n$ if $n \equiv 2 \pmod 4$ and in $\Hit_n$ otherwise.
\end{ex*}

Hitchin components have simple geometric interpretations for small $n$.
For example let us remark that for $n=2$ it coincides with Teichm\"uller space.
\begin{remark}
If $n = 2$ then $\iota_2 = I_2$ thus $\Hit_2 = \mathcal T(\Sigma)$.
\end{remark}

For $n=3$, a theorem of Choi-Goldman \cite{ChoiGoldman93, ChoiGoldman97} gives a geometric interpretation of the Hitchin component.
\begin{theorem}[Choi-Goldman]
The space $\Hit_3$ parametrizes the marked convex real projective structures on $\Sigma$.
Such a structure is an identification 
\[\Sigma\cong \mathcal{O}/\rho(\pi_1(\Sigma))\]
where $\mathcal{O}\subset \mathbb{RP}^2$ is a properly convex open set and $\rho\colon \pi_1(\Sigma)\to \PSL(3,\R)$ is a Hitchin representation.
\end{theorem}

\subsection{Properties}

Let us denote by $\Flag(\R^n)$ the variety of full flags of $\R^n$: 
\[\Flag(\R^n) = \{(F_0, F_1, \ldots, F_n): F_i\subset F_{i+1},\quad \dim (F_i) = i\}.\]

We say that two flags $E$ and $F$ in $\Flag(\R^n)$ are transverse if we have $E_i\oplus F_{n-i}=\R^n$ for all $0\leqslant i\leqslant n$. The group $\mathrm{GL}(n,\R)$ acts naturally on flags and this action is transitive on pairs of transverse flags.
Therefore for any pair of transverse flags $(E,F)$, there exists a matrix $A\in \mathrm{GL}(n,\mathbb R)$ such that $(A\cdot E, A\cdot F) = (E^0, F^0)$ where $E^0_i$ is the vector space spanned by the $i$-th first vectors of the canonical basis of $\R^n$ and $F^0_i$ is the one spanned by the $i$-th last ones, for all $0\leqslant i \leqslant n$.

Let us recall the following consequence of the \v{S}varc-Milnor Lemma. Let $j\colon \pi_1(\Sigma)\to \PSL (2,\R)$ be such that $[j]\in\mathcal{T}(\Sigma)$. The map $\pi_1(\Sigma)\to\mathbb{H}^2$ defined by $\gamma\mapsto j(\gamma)\cdot x_0$ is a quasi-isometry and allows us to identify $\partial\pi_1(\Sigma)$ with $\partial\mathbb{H}^2=\mathbb{RP}^1$. The action of $\gamma\in \pi_1(\Sigma)$ on $\partial \pi_1(\Sigma)$ in this identification translates into the action of $j(\gamma)$ on $\mathbb{RP}^1$. Therefore the action of a non trivial $\gamma\in \pi_1(\Sigma)$ on the boundary $\partial \pi_1(\Sigma)$ has two fixed points that we will denote by $\gamma^+$ and $\gamma^-$.

Labourie showed in his seminal work \cite{Labourie:06} that Hitchin components form higher Teichm\" uller spaces, i.e.\ they contain only discrete and faithful representations. 
Indeed Labourie showed for each Hitchin representation the existence of special equivariant limit map, a \emph{Frenet} map.

\begin{theorem}[Labourie]
If $\rho$ is a Hitchin representation, then there exists a continuous map $\xi\colon \partial\pi_1(\Sigma)\to\Flag(\R^d)$ that is $\rho$-equivariant, $\xi=(\xi^1,\ldots,\xi^n)$ and, such that for all integers $n_1,\ldots,n_k\geqslant 1$ that add up to $\sum_i n_i=m\leqslant d$, \[\bigoplus_{i=1}^k \xi^{n_i}(x_i) \xrightarrow[x_i\neq x_j]{x_i\to x} \xi^m(x) \]
\end{theorem}

The map $\xi^1 : \partial \pi_1(\Sigma)\to \mathbb{P}(\mathbb R^d)$ is called a \emph{hyperconvex curve}: it satisfies for any $x_1, \ldots, x_d\in \partial \pi_1(\Sigma)$ that are pairwise distinct, the following sum is direct:
\[\bigoplus_{i=1}^d \xi^1(x_i) = \mathbb R^d.\]
Observe that the map $\xi$ completely determines $\xi^1$ by the continuity property.
Guichard \cite{Guichard08} showed that actually the existence of such a continuous curve that is $\rho$-equivariant implies that $\rho$ is in the Hitchin component.

\begin{theorem}[Guichard]
There exists a $\rho$-equivariant continuous hyperconvex map $\partial\pi_1(\Sigma)\to\mathbb{P}(\R^d)$ if and only if $\rho$ is Hitchin.
\end{theorem}

Let us now show that the existence of a $\rho$-equivariant Frenet map as in Labourie's theorem implies that $\rho$ is discrete and faithful. Let us observe before that if $x\neq y$ are in $\partial \pi_1(\Sigma)$, then $\xi(x)$ is transverse to $\xi(y)$ and in particular $\xi$ is injective.
\begin{itemize}

\item{Let us show that $\rho$ is faithful.}
Let $\gamma\in \pi_1(\Sigma)\setminus \{1\}$.
Pick any point $x\in \partial \pi_1(\Sigma)\setminus \{\gamma^+, \gamma^-\}$. 
We have $\gamma\cdot x\neq x$ thus $\xi(\gamma\cdot x) \neq \xi (x)$. But $\xi(\gamma\cdot x) = \rho(\gamma)\cdot \xi(x)$ hence $\rho(\gamma)$ is not trivial.

\item{Let us show that $\rho(\gamma)$ is diagonalizable for any $\gamma\in \pi_1(\Sigma)$.}
Let $V_i$ be the vector space $\xi^i(\gamma^+)\cap \xi^{n-i+1}(\gamma^-)$ for all $1\leqslant i \leqslant n$. Observe that for all $i$, $\dim(V_i) = 1$ and that $\mathbb R^n = \bigoplus_i V_i$.
Indeed since $\xi(\gamma^+)$ is transverse to $\xi(\gamma^-)$, there exists a basis $f_1, \ldots, f_d$ of $\mathbb R^n$ such that $\xi^i(\gamma^+) = \langle f_1, \ldots, f_i\rangle $ and $\xi^{n-i+1}(\gamma^-) = \langle f_n, \ldots, f_i\rangle $ thus $V_i = \langle f_i\rangle$.
Moreover $\rho(\gamma)$ preserves $V_i$. Indeed $\rho(\gamma)\cdot \xi^i(\gamma^+) = \xi_i(\gamma\cdot \gamma^+) = \xi_i (\gamma^+)$ and $\rho(\gamma)\cdot \xi^{n-i+1}(\gamma\cdot \gamma^-) = \xi^{n-i+1}(\gamma\cdot \gamma^-) = \xi^{n-i+1}(\gamma^-)$.

\item{Let us show that $\rho$ is discrete.}
Let $(\gamma_k)_{k \in \N}$ be such that $\rho(\gamma_k)\to \pm I_n$.
For all $x\in \partial \pi_1(\Sigma)$ we have $\gamma_k\cdot x\to x$.
Indeed suppose $\gamma_{k_\ell} \to y$.
We have  and $\xi(\gamma_{k_\ell}\cdot x)\to \xi(y)$ and $\xi(\gamma_{k_\ell} \cdot x) = \rho(\gamma_{k_\ell})\cdot \xi(x) \to \xi(x)$.
Therefore $\xi(x) = \xi(y)$ and $x=y$.
The sequence $\gamma_k\cdot x$ has only $x$ as an accumulation point thus it converges to $x$.
Let us now pick three distinct points $x_1, x_2, x_3\in \partial \pi_1(\Sigma)$.
We have $\gamma_k\cdot x_i\to x_i$ for all $1\leqslant i\leqslant 3$.
The action of $\PSL(2,\R)$ on $\mathbb{RP}^1$ is fully determined by its action on three points hence $\gamma_k\to 1$.
By discreteness we have $\gamma_k = 1$ for $k$ large enough.
\end{itemize}

\begin{corollary}
The Hitchin component $\Hit_n$ is a higher Teichm\"uller space.
\end{corollary}

We can moreover show that the eigenvalues of an element in the image of a Hitchin representation are all distinct.

\begin{remark} Each Hitchin representation $\rho$ is purely hyperbolic.
Namely for every $\gamma\in \pi_1(\Sigma)\setminus \{1\}$, the eigenvalues $\lambda_i$ associated with $V_i$ satisfy $|\lambda_i|>|\lambda_{i+1}|$.
\end{remark}

We can define concretely $\xi$ as follows.
For $\gamma\in \pi_1(\Sigma)\setminus \{1\}$ then we can define $\xi(\gamma^+)$ to be the flag associated with the ordered eigenvalues of $\rho(\gamma)$.
This defines $\xi$ on a dense set of $\partial \pi_1(\Sigma)$ and we can then extend $\xi$ uniquely by continuity.

Let us now give examples of maps $\xi^1$ in particular settings.
\begin{example}
Suppose that $\rho\colon \pi_1(\Sigma) \to \PSL(n,\R)$ is in the Fuchsian locus: $\rho = \iota_n\circ j$, with $j\colon \pi_1(\Sigma)\to \PSL (2,\R)$ such that $[j]\in \mathcal T(\Sigma)$.
The map $j$ allows us to identify $\partial \pi_1(\Sigma)$ with $\mathbb P(\R^2) = \mathbb{RP}^1$
Then the map $\xi^1(p) = V(p)$ is a hyperconvex curve and is $\rho$-equivariant:\[\xi^1(\gamma\cdot p) = V(j(\gamma)\cdot p) = \iota_n(j(\gamma))\cdot V(p)= \rho(\gamma)\cdot \xi^1(p).\]
\end{example}

\begin{example}
In the case $n=3$, the theorem of Choi-Goldman allows us to identify $\partial \pi_1(\Sigma)$ with $\partial \mathcal O$ and thus gives a map $\xi^1\colon\partial \pi_1(\Sigma)\to \mathbb P(\R^3)$ that is a $\rho$-equivariant hyperconvex curve.
\end{example}

\section{Positivity}

\subsection{Positivity of flags}

Let us define positivity for triples of flag.
\begin{definition}
A triple of flags of the form $(E^0,T,F^0)$ with $T$ transverse to $E^0$ is said to be \emph{positive} if $T=u_T\cdot E^0$, with $u_T=\begin{pmatrix}
1 & 0 & \cdots & 0 \\ * & \ddots & \ddots & \vdots \\ \vdots & \ddots & \ddots & 0 \\ * & \cdots & * & 1
\end{pmatrix}$ that has all its minors positive, unless they are zero because of its shape.

A triple $(E, T, F)$ is positive if there exists $A\in \mathrm{SL}(n,\R)$ such that $A\cdot (E, T, F) = (E^0, A\cdot T, F^0)$ is positive.
\end{definition}

Hitchin representations are characterized by the existence of an equivariant limit map that sends positive triples to positive triples.
\begin{theorem}[Labourie, Guichard, Fock-Goncharov]
A representation $\rho\colon\pi_1(\Sigma)\to \PSL(n,\R)$ is Hitchin if and only if there exists a $\rho$-equivariant limit map $\xi\colon\partial\pi_1(\Sigma)\to\Flag(\R^n)$ that sends positive triples to positive triples.
\end{theorem}

\subsection{More general setting}
One can define Hitchin representation in a more general context. Namely when $G$ is an adjoint real split semi-simple Lie group.

\begin{example}
These conditions are met for $G = \SL(n,\R)$, $\Sp(2n,\R)$, $\SO(n, n+1)$.
\end{example}

For this class of groups, there exists an embedding $\iota\colon\SL(2,\R)\to G$ unique up to conjugation.

\begin{example}
For the groups $G = \SL(n,\R)$, $\Sp(2n,\R)$, $\SO(n,n+1)$, the map $\iota$ is just the embedding $\iota_n$.
\end{example}

We can define use this map to define Fuchsian representations as representation of the form $\iota \circ j$ where $j\colon \pi_1(\Sigma)\to \PSL (2,\R)$ has its conjugacy class in $\mathcal T(\Sigma)$. Then we can also define Hitchin representations as the deformation of those. 
The theorems we saw still hold.
\begin{theorem}
The Hitchin components are homeomorphic to $\R^{(2g-2)\dim(G)}
$. Hitchin representations are discrete and faithful.
\end{theorem}
We can also generalize the notion of positivity to $G/B$ where $B$ is a Borel subgroup of $G$.
\begin{theorem}[Labourie, Guichard, Fock-Goncharov]
A representation $\rho\colon \pi_1(\Sigma)\to G$ is Hitchin if and only if there exists a $\rho$-equivariant $\xi\colon\partial\pi_1(\Sigma)\to G/B$ that sends positive triples to positive triples.
\end{theorem}

For more information on these generalizations we refer to \cite{Wie18} and references therein.

\newpage

\thispagestyle{empty}

\chapter[Exercises]{Exercises \\ {\Large\textnormal{\textit{by Marta Magnani, Arnaud Maret, Enrico Trebeschi}}}}
\addtocontents{toc}{\quad\quad\quad \textit{Marta Magnani, Arnaud Maret, Enrico Trebeschi}\par}

We want to explore the following facts throughout this list of exercises.

\begin{enumerate}
    \item The holonomies of hyperbolic structures are discrete and faithful. Any discrete and faithful representation can be realized as a hyperbolic structure.
    \item The boundary of $\pi_1(\Sigma)$ is homeomorphic to the boundary of the hyperbolic plane.
    \item The limit map for Teichm\"uller space.
    \item The subspace of discrete and faithful representations is closed in the space of representations.
    \end{enumerate}
    
\bigskip
    
\begin{exercise}

	Let $(M,h)$ be a marked hyperbolic structure over a closed surface $\Sigma$ of genus $g\geq 2$. We recall that if $p\colon(\tilde{M},\tilde{x}_{0})\to(M,x_0)$ is a universal cover, $\pi_{1}(M,x_{0})$ acts discretely and faithfully on $(\widetilde{M},\tilde{x}_{0})$ by deck transformations. There exists an isometry $f_h\colon\widetilde M\to\mathbb{H}^2$, which induces an isomorphism $f_h^{*}\colon\mathrm{Isom}(\widetilde M)\to\mathrm{Isom}(\mathbb{H}^2)$.
	
	The holonomy of a point $[(M,h)]\in\mathcal{T}(\Sigma)$ is the data of the representation $$[\rho_h]\in\chi(\Sigma,\PSL(2,\R)),$$
	where $[\rho_h]$ is the class of conjugacy of the following homomorphism:
	$$\begin{tikzcd}[row sep=small]
		\pi_{1}(\Sigma,s_0)\arrow[r,"h_{*}"]&\pi_{1}(M,h(s_0))\arrow[r,"\textnormal{Deck}"]&\mathrm{Isom}^{+}(\widetilde M,\tilde{x}_{0})\arrow[r,"f_h^{*}"]&\mathrm{Isom}^{+}(\mathbb{H}^{2})=\mathbb{P}\mathrm{SL}(2,\mathbb{R}).
	\end{tikzcd}$$
	
	To prove that $\rho_h$ is well defined one have to check that the above construction is well defined up to conjugation.
	 
	\begin{enumerate}[label=(\alph*)]
		\item It is a fact that the isomorphism $$\mathrm{Deck}\colon\pi_{1}(M,x_0)\to\mathrm{Isom}(\widetilde M,\tilde{x}_{0})<\mathrm{Isom}(\widetilde M)$$ only depends on the basepoint $x_0$ and that changing basepoint change the conjugacy class of $\mathrm{Isom}(\widetilde M,\tilde{x}_{0})$ in $\mathrm{Isom}(\widetilde M)$. Deduce that the morphism $\pi_{1}(\Sigma)\to\mathrm{Isom}(\widetilde M)$ is well defined, up to conjugation.
		
		\item Describe explicitly the behaviour of $f_h^{*}$ to prove that the choice of the isometry $\widetilde M\to\mathbb{H}^2$ only changes the conjugacy class of the representation.
		
		\item Let $(M,h)\sim(M',h')$ be two hyperbolic structures, \emph{i.e.} $(h')^{-1}h\colon M'\to M$ is homotopic to an isometry. Use the homotopy lifting property and the previous point to deduce that the two deck transformation are conjugate by the lifted isometry.
	\end{enumerate}

Let $\rho\colon\pi_{1}(\Sigma)\to\PSL(2,\mathbb{R})$ be a discrete and faithful representation. We want to show that this is a covering action on $\mathbb{H}^{2}$, so that $\mathbb{H}^{2}/\rho(\pi_1(\Sigma))$ is a surface with fundamental group $\rho(\pi_1(\Sigma))\cong\pi_1(\Sigma)$, hence diffeomorphic to $\Sigma$ because of the classification theorem for closed surfaces, and it inherits a hyperbolic structure by the projection $p\colon\mathbb{H}^{2}\to\mathbb{H}^{2}/\rho(\pi_1(\Sigma))$.

It is a fact that an action is a covering one if and only if it is free and properly discontinuous.
\begin{enumerate}[label=(\alph*)]
	\item Prove that a discrete subgroup $G$ of $\mathrm{Isom}(X)$ acts properly discountinuously if $(X,d)$ is a complete metric space.
	
	For this point use Arzel\`a-Ascoli Theorem, that is
	\begin{theorem}[Arzel\`a-Ascoli]
 		Let $K$ be a compact Hausdorff space and $X$ a metric space. Then $F\subset C(K,X)$ is compact in the compact-open topology if and only if it is equicontinuous, pointwise relatively compact and closed. 
	\end{theorem}
	We recall that $F$ is \emph{pointwise relatively compact} if $\forall x\in K$, the set $Fx=\{f(x),\ f\in F\}$ is relatively compact in $X$.
	
	$F$ is \emph{equicontinuous} if $\forall\varepsilon>0$, $\exists\delta>0$ such that $$d_{X}(f(x),f(y))<\varepsilon,\qquad \forall f\in F,\ \forall x,y\in K\ \text{such that}\ d_{K}(x,y)<\delta.$$
	
	\textit{Hint: in a complete metric space compact is equivalent to closed and bounded.}
	
	\item Deduce from the previous point that a discrete and faithful action is properly discountinuous.
	
	\item Describe the stabilizer of a point in $\mathrm{Isom}^{+}(\mathbb{H}^{2})=\PSL(2,\mathbb{R})$. Without loss of generality, consider the stabilizer of $i$ in the half-plane model.
	
	\item It is a fact that a group having a one-relator presentation $\langle s_{i},\,i\in I\,:\,r\rangle$, with $r$ cyclically reduced is with torsion if and only if the $r$ is a proper power.
	
	The fundamental group of a closed orientable surface of genus $g$ can be presented as $$\pi_{1}(\Sigma)=\langle a_1,b_1,\dots,a_g,b_g\,:\,[a_1,b_1]\dots[a_g,b_g]\rangle,$$
	hence is without torsion. Deduce from the previous point that a discrete and faithful action has to be free.
\end{enumerate}
\end{exercise}

\begin{exercise} 

\emph{Step 1}: The (Gromov) boundary of $\H^2$ is $$\partial \mathbb{H}^{2}:= \{ r\from [0,\infty) \to \mathbb{H}^{2} \text{ geodesic ray} \}/ _{finite \  distance}$$ where $$r_{1} \sim r_{2} \iff \sup_{t}d_{\mathbb{H}^{2}}(r_{1}(t),r_{2}(t)))<\infty$$ 

Convince yourself that the Gromov boundary of $\mathbb{H}^{2}$ is homeomorphic to $S^{1}$. 
A basis for the topology on $\partial \mathbb{H}^{2}$ is 

$$
\mathcal{U}_{[\alpha],c,t}=\{ \beta\from [0,\infty) \to \mathbb{H}^{2} \text{ geod.\ rays } : \ \beta(0)=\alpha(0), B_{c}(\alpha(t)) \text{ intersects } \beta \}
$$
    
\emph{Step 2}: It is well known that $\pi_{1}(\Sigma)= \langle a_{1},b_{1},\ldots,a_{g},b_{g}: \prod_{i=1}^{g}[a_{i},b_{i}]=1 \rangle$. We can define the boundary $\partial \pi_{1}(\Sigma) $ using the same definition of Step 1, where instead of $\mathbb{H}^{2}$ we consider the Cayley graph of $\pi_{1}(\Sigma)$ equipped with the word metric $d_{S}$ for a finite generating set $S$.

\emph{Step 3}: We want to show that $\partial \mathbb{H}^{2} $ and $\partial \pi_{1}(\Sigma) $ are homeomorphic. The key ingredient is the following

\begin{lemma}[Milnor-\v{S}varc] Let $(X,d)$ be a geodesic space and let $\Gamma$ be a group acting properly discontinuously, cocompactly and by isometries on $X$. Then $\Gamma$ is finitely generated and for every finite generating set $S$ and every point $x \in X$ the map  
\[
\begin{aligned}
f\from (\Gamma,d_{S}) &\to (X,d)\\
\gamma& \mapsto \gamma \cdot x
\end{aligned}
\]
is a quasi-isometry.

\end{lemma}

In our case $X=\H^{2}$. How does the quasi-isometry extend to the boundary? 

\begin{remark}
The intuitive thing to do would be looking at the map
\[
\begin{aligned}
\widehat{f}:\partial \Gamma &\to \partial\mathbb{H}^{2}\\
[\gamma]& \mapsto  [f \circ \gamma]
\end{aligned}
\]

but in general quasi-isometries send geodesics to quasi-geodesics, so $\widehat{f}$ is not well defined.
\end{remark}

The result will follow from the following
 
\begin{theorem}[Stability of quasi-geodesics in hyperbolic spaces]
  Let $X$ be a hyperbolic metric space, $\gamma:[0,L] \to X$ a quasi-geodesic and $\gamma':[0,L'] \to X$ a geodesic with same starting and ending point as $\gamma$. Then $\exists \Delta \geq 0$ such that
 
$$
\Im(\gamma') \subset B_{\Delta}(\Im \gamma) \text{  and  } \Im(\gamma) \subset B_{\Delta}(\Im \gamma')
$$
(that is, one is contained in the $\Delta$-neighbourhood of the other) 
\end{theorem}
\end{exercise}

\begin{exercise}
When $G=\PSL(2,\mathbb{R})$ the Hitchin component and the space of maximal representations both agree with Teichm\"uller space.
The existence of a nice limit map tells us something about the representation.
More precisely:  
\begin{quote}
   Let $\rho \from \pi_1(\Sigma) \to \PSL(2,\mathbb{R})$ be a representation and let $\xi \from \partial \pi_1(\Sigma) \to \partial \mathbb{H}^{2}$ be an injective continuous $\rho$-equivariant map, i.e.\ $\xi(\gamma \cdot r)= \rho(\gamma)\xi(r)$.
   Then $\rho$ is injective and discrete. 
\end{quote}

The group $\pi_{1}(\Sigma)$ acts on its Cayley graph by isometries.
This action extends to the boundary (definition in Exercise 2): for $[r(t)] \text{ geodesic ray } \in \partial \pi_{1}(\Sigma)$ and $\gamma \in \pi_{1}(\Sigma)$, $\gamma \cdot [r(t)]=[\gamma \cdot r(t)]$.

\emph{Injectivity}: Since $\rho \in \Hom(\pi_{1}(\Sigma),\PSL(2,\mathbb{R}))$ it suffices to prove 
\[\gamma \in \pi_{1}(\Sigma), \gamma \neq 1 \Rightarrow \rho(\gamma) \neq I_2.\]

Use injectivity and $\rho$-equivariance of $\xi$ to prove injectivity of $\rho$.

\emph{Discreteness}: $\rho$ discrete $\iff \nexists$ sequence $\rho(\alpha_{n})$ accumulating on $I_2$.
Use properties of $\xi$ to show the non-existence of such a sequence.

\end{exercise}
\begin{exercise}

The goal of this exercise is to understand the following statement.

\begin{quote}
    \emph{The subspace of discrete and faithful representations is closed in the space of representations.}
\end{quote}

Formally, let $G$ be a Lie group and $\Gamma$ be a non-cyclic, torsion-free, hyperbolic group. The fundamental groups of hyperbolic surfaces are examples of such groups $\Gamma$.

The key ingredient is the so called Margulis Lemma (also known as Margulis-Zassenhaus Lemma).

\begin{theorem}[Margulis Lemma]
Let $G$ be a Lie group. There exists a neighbourhood $U$ of the identity such that, given a discrete subgroup $J\subset G$ such that $J\cap U$ generates $J$, then $J$ is nilpotent.
\end{theorem}

Recall that $J$ being \emph{nilpotent} means that the sequence $J_1:=[J,J]$, $J_{i+1}:=[J_i,J]$ is eventually the trivial group.\\

We will apply Margulis Lemma to prove that if $\rho_n\colon \Gamma\to G$ is a sequence of discrete and faithful representations that converge to a representation $\rho\colon \Gamma\to G$, then $\rho$ is discrete and faithful. First, assume that $\rho$ is not faithful, i.e.\ there is $g\in \Gamma\setminus \{1\}$ such that $\rho(g)=e$.

\begin{itemize}
    \item[(a)] Let $g^+:=\lim_{n\to +\infty}g^n$ in the boundary of $\Gamma$. What kind of subgroup is $\stab(g^+) < \Gamma$ ? Conclude that there exists $h\in \Gamma\setminus$ $\stab(g^+)$.
    \item[(b)] What is $\stab(g^+)\cap\stab(h^+)$ ?
\end{itemize}

We consider the subgroup of $\Gamma$ defined by $J:=\langle g, hgh^{-1}\rangle$.

\begin{itemize}
    \item[(c)] Prove that $\rho_n(J)$ is nilpotent for some $n$ large enough and deduce that $J$ is nilpotent.
\end{itemize}

It is a fact that a nilpotent subgroup of a torsion-free hyperbolic group is either trivial or isomorphic to $\mathbb{Z}$. So, we conclude that $J\cong \mathbb Z$.

\begin{itemize}
    \item[(d)] (harder) Prove that $\langle g,h^2\rangle\cong \mathbb{Z}^2$.
\end{itemize}

This is a contradiction because $\langle g,h^2\rangle$ is a nilpotent subgroup of $\Gamma$ that is not trivial or a copy of $\mathbb{Z}$. Hence, $\rho$ is faithful.

Assume now that $\rho$ is not discrete. So, there exists a sequence $g_n\in \Gamma\setminus \{1\}$ such that $\rho(g_n)$ converges to the identity.

\begin{itemize}
    \item[(e)] Convince yourself that there exists $h$ not contained in $\stab(g_n^+)$ for all $n$ (but maybe finitely many).
    \item[(f)] (somewhat tricky)  Deduce that $J_n:=\langle g_n,hg_nh^{-1}\rangle$ is nilpotent for some $n$. Conclude.
\end{itemize}
\end{exercise}
\newpage

%%%%%%%%%%%%%%%%%%%%%%%%%%%%%%%%%%%%%%%%%%%%%%%%%%%%%%%%%%%%%%%%%%%%%%%%%%%%%%%%%%%%%%%%%%%%%%%%%%%%%%%%

\part[Background in Lie Theory]{Background in Lie Theory
\textnormal{
\begin{minipage}[c]{15cm}
\begin{center}
    \vspace{2cm}
    {\Large Luca De Rosa}\\
    \vspace{-4mm}
    {\large \textit{ETH Zürich}}\\
    \vspace{.5cm}
     {\Large Victor Jaeck}\\
    \vspace{-4mm}
    {\large \textit{ETH Zürich}}\\
    \vspace{.5cm}
    {\Large Max Riestenberg}\\
    \vspace{-4mm}
    {\large \textit{Ruprecht-Karls-Universit\"at Heidelberg}}\\
    \vspace{.5cm}
    {\Large Daniel Soskin}\\
    \vspace{-4mm}
    {\large \textit{The University of Notre Dame}}\\
    \vspace{.5cm}
    {\Large Jacques Audibert}\\
    \vspace{-4mm}
    {\large \textit{Sorbonne Université}}\\
    \vspace{.5cm}
     {\Large Alex Moriani}\\
    \vspace{-4mm}
    {\large \textit{Université Côte d'Azur}}\\
    \vspace{.5cm}
     {\Large Colin Davalo}\\
    \vspace{-4mm}
    {\large \textit{Ruprecht-Karls-Universit\"at Heidelberg}}
\end{center}
\end{minipage}
}}\label{chap2}

\thispagestyle{empty}

\chapter[Lie Groups, Lie Algebras and Their Symmetric Spaces]{Lie Groups, Lie Algebras and Their Symmetric Spaces\\ {\Large\textnormal{\textit{by Luca De Rosa, Victor Jaeck}}}}
\addtocontents{toc}{\quad\quad\quad \textit{Luca De Rosa, Victor Jaeck}\par}

\section{Lie groups and Lie algebras}
\begin{definition}
A \emph{Lie group} $G$ is a group endowed with the structure of smooth manifold, such that the operations of multiplication $G \times G \to G$ and inverse $G \to G$ are smooth.
\end{definition}

\begin{remark} 
In particular, a Lie group is a locally compact Hausdorff second countable topological group.
\end{remark}

In the remaining of the workshop we will probably be mostly interested in \emph{linear groups}, i.e. Lie groups arising as subgroups of $\mathrm{GL}(n, \R)$, for some $n$.

\begin{ex*} 
Let $G$ be a connected Lie group and $U \subseteq G$ a neighborhood of the identity $e \in G$. Show that $U$ generates $G$.
\end{ex*}

\begin{examples}
\begin{enumerate}
    \item $(\mathbb{R}, +)$ and $(\R \setminus \{0\} , \cdot)$.
    \item $G = \GL(n,\R)$. The linear group $G$ is an open subset of $\R^{n^2}$, and hence it inherits a smooth structure from it. The usual matrix multiplication makes it a group. Notice moreover that matrix multiplication is a polynomial function of the entries, and the inverse of a matrix is a rational function. Hence both are smooth. It follows that $G$ is a Lie group, called the \emph{linear group}.
    \item Non-example: Consider the space of homeomorphisms $\mathrm{Homeo} (X)$ with $X$, for instance, the regular tree of degree $d \geq 3$, endowed with the compact-open topology. One can show that $\mathrm{Homeo}(X)$ is not even a locally compact topological group. In fact, a neighbourhood of the identity contains maps that fix balls of $X$ of larger and larger radius.
    \item For more interesting examples (e.g. $\mathrm{SL}(n, \R)$) we will use the inverse function theorem, as we will see in the first exercise session.
\end{enumerate}
\end{examples}

\begin{definition}
A \emph{Lie algebra} $\mathfrak{g}$ is a vector space together with a skew-symmetric bilinear map
\[
[\cdot,\cdot] \from \mathfrak{g} \times \mathfrak{g} \to \mathfrak{g}
\]
satisfying the Jacobi identity: for all $X,Y,Z \in \mathfrak{g}$
\begin{equation}
\label{jacobi_identity}
[X, [Y,Z]] + [Z, [X,Y]] + [Y, [Z,X]] = 0
\end{equation}
\end{definition}

\begin{definition}
Let $\mathfrak{g, h}$ be Lie algebras. A \emph{Lie algebra homomorphism} is a linear map
$\varphi\colon \mathfrak{g} \to \mathfrak{h}$ respecting the bracket operation, i.e.
\[
\varphi([X,Y]) = [\varphi(X), \varphi(Y) ]
\]
for all $X,Y \in \mathfrak{g}$.

A \emph{Lie subalgebra} $\mathfrak{m} \subseteq \mathfrak{g}$ is a subspace of $\mathfrak{g}$ such that for all $X,Y \in \mathfrak{m}$, $[X,Y] \in \mathfrak{m}$.
\end{definition}

\begin{definition}
Let $A$ be a commutative algebra over a field $k$. A \emph{derivation} on $A$ is an endomorphism $\delta\colon A \to A $ of the underlying $k$-vector space $A$ which satisfies the \emph{Leibniz rule} $\delta(ab) = \delta (a)\, b + a \, \delta (b)$ for all $a, b \in A$.
We write $\mathrm{Der}(A)$ for the set of derivations on $A$.
\end{definition}

Note that smooth vector fields $\mathrm{Vect}^\infty (M)$ on a manifold $M$ can be identified with derivations on $\mathcal{C}^\infty (M)$. Precisely, given a point $p \in M$ and a tangent vector $X_p \in T_p M$, we can associate
\[
\delta_{X_p} \colon \mathcal{C}^\infty (p) \to \R , \quad [f] \mapsto (d_p f)(X_p),
\]
where $\mathcal{C}^\infty (p)$ denotes the space of germs of functions around $p$.
It follows that vector fields correspond to endomorphisms of $\mathcal{C}^\infty (M)$.

\begin{proposition}
Let $M$ be a smooth manifold. The map
\[
\alpha \from \mathrm{Vect}^\infty (M) \to \mathrm{End}(\mathcal{C}^\infty (M)),\quad X \mapsto \left( f \mapsto (Xf \from p \mapsto (d_p f)(X) )) \right)
\]
is an isomorphism onto its image which consists of derivations $\mathrm{Der} (\mathcal{C}^\infty (M))$ of the $\R$-algebra $\mathcal{C}^\infty (M)$
\end{proposition}

\begin{remark}
Notice that $\mathrm{Der} (\mathcal{C}^\infty (M))$ is not a subalgebra with respect to the natural operation of composition.
Consider for instance $M= \R$ and a derivation $X$ such that $X f = f'$.
Then $X \circ X (f) = X f' = f''$.
Hence the element $X^2$ sends $f$ to its second derivative $f''$. This endomorphism of $\mathcal{C}^\infty (M)$ clearly does not satisfy the Leibniz rule $\delta(ab) = \delta (a) b + a \delta (b)$.
Nevertheless there is another operation on $\mathrm{Der} (\mathcal{C}^\infty (M))$ that, differently from the composition, gives back a derivation.
\end{remark}
\begin{definition}
For $X, Y \in \mathrm{Vect}^\infty (M)$, define
\[
[X,Y] := [\delta_X, \delta_Y] := \delta_X \circ \delta_Y - \delta_Y \circ \delta_X
\]
where $\delta_X : \mathcal{C}^\infty (M) \to \mathcal{C}^\infty (M)$ is the derivation corresponding to $X$.
\end{definition}

\begin{ex*}
Let $\delta_1$ and $\delta_2$ be derivations. Show that $[\delta_1, \delta_2]$ is a derivation as well, i.e.\ it satisfies the Leibniz rule.
\end{ex*}

We have shown that $\mathrm{Vect}^\infty (M)$ endowed with the brackets operation above is a Lie algebra.
It is in fact a very important and somehow prototypical example of a Lie algebra, as we will see in the following subsection.

Our next goal is to define a Lie algebra $\mathfrak{g}$ canonically attached to a given Lie group $G$.

\section{Left invariant vector fields and the Lie algebra of a Lie group}
Let $G$ be a Lie group, acting smoothly on a smooth manifold $M$. Given $g \in G$ we denote with $L_g$ the \emph{left action map}
\[
L_g \from M \to M ,\quad m \mapsto g \cdot m.
\]
The induced map at the level of vector fields is given by
\[
(L_g)_* \colon \mathrm{Vect}^\infty (M) \to \mathrm{Vect}^\infty (M),\quad X \mapsto \left( m \mapsto (d_m L_g ) (X_m) \right).
\]

\begin{definition}
A vector field $X \in \mathrm{Vect}^\infty (M)$ is said to be \emph{$G$-invariant}, or more precisely $G$-\emph{left-invariant} if for any $g \in G$, we have $(L_g)_* (X) = X$. More explicitly, $X \in \mathrm{Vect}^\infty (M)$ is $G$-invariant if for any $g \in G$ and $m \in M$
\[
(d_m L_g)(X_m) = X_{L_g (m)} = X_{gm}
\]
From now on we will denote with $\mathrm{Vect}^\infty (M)^G \subseteq \mathrm{Vect}^\infty (M)$ the Lie subalgebra of $G$-invariant vector fields on $M$.
\end{definition}

Recall that a Lie group $G$ is also a smooth manifold, and its group structure provides a natural action of $G$ on itself. Precisely, we now put $G = M$, and take as smooth action $G \times M \to M$ the multiplication action
\[
G \times G \to G,\quad (g,h) \mapsto gh
\]
 The following lemma is key to the definition of a Lie algebra of a Lie group.
 
 \begin{lemma}
 Let $G$ be a Lie group. The map
 \[
 \mathrm{Vect}^\infty (G)^ G \to T_e G,\quad X \mapsto X_e
 \]
 is an isomorphism of vector spaces.
 \begin{proof}
 Injectivity: Let $X_e=0$. Then for all $g \in G$ we have $
 (d_e L_g)(X_e) = X_g = 0$ and hence the vector field $X$ is identically zero.
 
 Surjectivity: Given $v \in T_e G$, define the vector field $X$ as follows: for any $g \in G$
 \[
 X_g := (d_e L_g) (v)
 \]
 Then $X$ is left invariant by construction, and $X_e = v$.
 \end{proof}
 \end{lemma}
 
 From now on, given $v \in T_e G$, let $v^L$ be the corresponding left-invariant vector field constructed above.
 
 \begin{definition}
 Let $G$ be a Lie group. The \emph{Lie algebra} $\mathfrak{g}$ \emph{of} $G$ is the vector space $T_e G \cong \mathrm{Vect}^\infty (M) ^G$ consisting of \emph{left}-invariant vector fields, endowed with the brackets operation
 \[
 [v,w] := [v^L, w^L]_e
 \]
 for $v,w \in T_e G$.
 \end{definition}
 
 Lie algebras are probably the central object in Lie theory. The Lie algebra $\mathfrak{g}$ encodes in its $G$-invariance most of the structure of the Lie group $G$. Moreover, Lie algebras are linear objects, hence much easier to study.
 
 \begin{ex*}
 The Lie algebra of $G=\GL(n, \R)$ is given by $\Mat(n,\R)$ with bracket operation given by the commutator
 \[
 [A,B] = A B - B A.
 \]
 \end{ex*}

\begin{definition}
A \emph{simple} Lie group is a connected non-abelian Lie group $G$ which does not have nontrivial closed connected normal subgroups.
\end{definition}

\begin{definition}
A Lie algebra $\mathfrak{g}$ is said \emph{semisimple} if it is direct sum of simple Lie algebras.
\end{definition}

\section{The exponential map and the adjoint representation}

\begin{definition}
Let $M$ be a manifold and $X$ a vector field on $M$. An \emph{integral curve} of $X$ through $p \in M$ is a smooth curve $\gamma \colon (-\delta, \delta) \to M $ such that $\gamma(0)=p$ and $\gamma ' (t) = X_{\gamma(t)}$.

A vector field $X$ on $M$ is \emph{complete} if for every $q \in M$ the integral curve of $X$ through $q$ is defined on all $\R$.
\end{definition}

The following result will motivate the definition of the exponential map.

\begin{theorem}
Let $G$ be a Lie group. The following hold
\begin{enumerate}
    \item Left invariant vector fields on $G$ are complete;
    \item For every $v \in T_e G$, let $\phi_v \colon \R \to G$ be the integral curve of $v^L$ through $e \in G$. Then $\phi_v$ is a smooth homomorphism $\R \to G$.
    In particular for all $t_1,t_2 \in \R$ we have
    \[
    \phi_v (t_1 + t_2) = \phi_v (t_1) \phi_v (t_2).
    \]
    \item The \emph{flow} $\Phi \colon \R \times G \to G $
    of $v^L$ is given by 
    \[
    \Phi (t, g) = g \phi_v (t).
    \]
\end{enumerate}
\end{theorem}

We are now ready to define the exponential map.
\begin{definition}
The $G$ be a Lie group with Lie algebra $\mathfrak{g}$.
The \emph{exponential map} is defined by
\[
\exp_G \colon \mathfrak{g} \to G , \quad v \mapsto \phi_v (1).
\]
\end{definition}

\begin{examples}
If $G = \GL(n, \R)$, then $\mathfrak{g} = \mathfrak{gl}(n, \mathbb{R}) = \Mat(n,\mathbb{R})$ and the Lie exponential map turns out to be
\[
\mathrm{exp}_{\GL(n,\R)}\from \Mat(n,\R) \to \GL(n,\R) , \quad A \mapsto \sum_{n = 0} ^\infty {\frac{A^n}{n !}}
\]
\end{examples}

For the remaining of this subsection, let $G$ be a Lie group with Lie algebra $\mathfrak{g}$.
The map 
$\mathrm{int}(g) \from G \rightarrow{G}$ that sends $x$ to $gxg^{-1}$ is
a smooth automorphism of $G$ and the associated map $ G  \rightarrow \mathrm{Aut}(G)$, $g \mapsto \mathrm{int}(g)$ is a homomorphism.

\begin{definition}
The \emph{adjoint representation} of $G$ is $\mathrm{Ad}:=d_e \mathrm{int}$.
For every $g \in G$, $\mathrm{Ad}(g) := d_e\mathrm{int}(g)$ is an element of $\mathrm{GL}(\mathfrak{g})$ and the map $\mathrm{Ad} \from G \rightarrow{\GL(\mathfrak{g})}$ is a
homomorphism. 
\end{definition}

\begin{ex*}
For every $t \in \mathbb{R}$ and $X \in \mathfrak{g}$ it holds 
\[
g \exp_G(tX)g^{-1} = \exp_G\left( t \mathrm{Ad}(g) X \right).
\]
Moreover, if $G = \GL(n,\R)$ then  $\mathrm{Ad}(g)X = gXg^{-1}$.
\end{ex*}

\begin{definition}
Let $\mathfrak{g}$ be a Lie algebra.
A \emph{representation} of $\mathfrak{g}$ into a finite-dimensional vector space $V$ is a Lie algebra homomorphism from $\mathfrak{g}$ to $\End(V)$.
\end{definition}  

\begin{definition}
We also call \emph{adjoint representation} the map $\mathrm{ad} \from \mathfrak{g} \rightarrow \mathfrak{gl}(\mathfrak{g})$ that sends $X \mapsto [X, -]$. 
\end{definition}

\begin{ex*}
The adjoint representation $\mathrm{ad}$ is a representation of the Lie algebra $\mathfrak{g}$. 
Moreover, the derivative of $\mathrm{Ad}$ at $e \in G$ is the Lie algebra representation $\mathrm{ad}$.
\end{ex*}

\begin{definition}
The \emph{Killing form} of a Lie algebra $\mathfrak{g}$ is the symmetric bilinear form 
\[
B_\mathfrak{g}(x,y):=\Tr(\mathrm{ad}(x)\mathrm{ad}(y))
\]

\end{definition}

The following criterion motivates alone the relevance of the Killing form:

\begin{remark}[Cartan's Criterion]
A real Lie algebra is semisimple if and only if its Killing form is non-degenerate.
\end{remark} 

If $\mathfrak{g}$ is a semisimple Lie algebra, then any non-degenerate symmetric invariant bilinear form on $\mathfrak{g}$ is a scalar multiple of the Killing form.

\section{Symmetric spaces and the Cartan decomposition}

\begin{definition}[Riemannian viewpoint]
A Riemannian manifold $M$ is \emph{locally symmetric} if for every $p \in M$, there is a normal neighborhood $p \in U$ and an isometry $S_p : U \rightarrow{U}$ such that $S_p^2=e$ and $p$ is the only fixed point of $S_p$ in $U$.
Moreover, $M$ is \emph{globally symmetric} if each $S_p$ can be extended to an isometry of $M$.
\end{definition}

We consider now a \emph{Lie group viewpoint} of the notion of symmetric space. We are mostly interested in the case when $G$ is a semisimple Lie group with finite center. The following theorem states that one can construct symmetric spaces from a Lie group and its maximal compact subgroups.

\begin{theorem}
\label{bigtheorem}
Let $G$ be a connected semisimple Lie group with finite center, and let $K$ be a maximal compact subgroup. Then $X=G/K$ admits an essentially unique $G$-invariant metric (e.g. the one induced by the Killing form). The space $X$ endowed with this metric is a non-positively curved symmetric space, hence complete and contractible (by Cartan-Hadamard).
\end{theorem}

We can finally introduce the Cartan decomposition of a Lie algebra.

\begin{definition}
 An \emph{involution} on $\mathfrak {g}$ is a non-trivial Lie algebra automorphism $\theta$ of $\mathfrak {g}$ whose square is equal to the identity.
 Such an involution is called a Cartan involution on $\mathfrak {g}$ if $- B_{\mathfrak{g}} \left(X, \theta (Y) \right)$ is a positive definite bilinear form.
\end{definition}

\begin{definition}
Let $G$ and $K$ be as in theorem \ref{bigtheorem}. 
Since $\theta^2 = \mathrm{id}$, the Cartan involution $\theta$ is diagonalizable and has eigenvalues $+1$ and $-1$.
We denote with $\mathfrak{k}$ and $\mathfrak{p}$ the corresponding eigenspace, and hence obtain the decomposition
\[
\mathfrak{g} = \mathfrak{k} \oplus \mathfrak{p}
\]
named the \emph{Cartan decomposition} with respect to $\theta$.
\end{definition}

\begin{theorem}
Every real semisimple Lie algebra has a Cartan involution that is unique up to inner automorphism.
\end{theorem}

\begin{ex*}
Show that, with the above notations, $[\mathfrak{k},\mathfrak{k}] \subseteq \mathfrak{k}, [\mathfrak{p}, \mathfrak{p}] \subseteq \mathfrak{k}$ and $[\mathfrak{k}, \mathfrak{p}] \subseteq \mathfrak{p}$.
\end{ex*}

\newpage

\thispagestyle{empty}

\chapter[Restricted Root Systems and Parabolic Subgroups]{Restricted Root Systems and Parabolic Subgroups\\ {\Large\textnormal{\textit{by Max Riestenberg}}}}
\addtocontents{toc}{\quad\quad\quad \textit{Max Riestenberg}\par}

Assumptions: We assume $G$ is a connected, semisimple Lie group with finite center. $K$ denotes a maximal compact subgroup. The associated symmetric space $X=G/K$ is of noncompact type. 

Review: I am assuming we know about Lie groups $G$, Lie algebras $\mathfrak{g}$, symmetric spaces $X=G/K$, and the Cartan decomposition $\mathfrak{g}=\mathfrak{k} \oplus \mathfrak{p}$.

Goal: We need to cover the restricted root space decomposition, flag manifolds, and parabolic subgroups. In particular we need to give the description of parabolic subgroups in terms of restricted roots. 

\section{Constructing the associated symmetric space}

I added this subsection after my talk, since there was some discussion about how to associate a symmetric space to a semisimple Lie group. I will only summarize the construction.

Suppose $G$ is a real semisimple connected Lie group with finite center. Let $K$ be a maximal compact subgroup. Form the quotient $X=G/K$. The orbit map $G \to X$ given by $g \mapsto gK$ has differential at the identity $\mfg \to T_{[K]}G/K.$ The kernel of this linear map is $\mfk$. Define $\mfp$ to be the Killing form perpendicular of $\mfk$. Then the differential of the orbit map restricts to an isomorphism $\mfp \to T_{[K]}G/K.$ The Killing form restricts to a positive definite inner product on $\mfp$, and $K$ acts on $\mfp$ via the adjoint action by isometries for this inner product. The inner product may be extended to all of $X=G/K$ by left translation by the action of $G$.

The resulting Riemannian manifold turns out to be a symmetric space of non-compact type. We define the Cartan involution $\theta$ on $\mfg = \mfk \oplus \mfp$ by setting $\mfk$ to be the $+1$-eigenspace and $\mfp$ to be the $-1$-eigenspace. This involution integrates to a unique involution $\sigma \colon G \to G$. The geodesic symmetry at $[K]$ is then given by $gK \mapsto \sigma(g)K$. 

Let me emphasize that a few points in my summary are not obvious (to me). In particular, it is not obvious that the restriction of $B$ to $\mfk$ is negative definite and that the definition of $\theta$ yields a Cartan involution. (Recall that a \textit{Cartan involution} is a Lie algebra involution $\theta$ such that the modified Killing form $-B(X,\theta Y)$ is an inner product.) Let me not attempt to prove this here.

\begin{example}
    Our running example will be $\SL(n,\R)$. Its associated symmetric space $X_n$ can be modelled as the space of $n \times n$ real symmetric positive definite matrices of determinant $1$. The action $G \times X \to X$ may be given by $g\cdot x = g x \tran{g}$. Then $X$ has a basepoint $p = I_n$, the identity matrix, with stabilizer $K=\SO(n)$, with involution $\sigma(g) = \tran{g^{-1}}$ and Cartan involution $\theta(X)=-\tran{X}$. The Cartan decomposition is given by $\mfk = \mathfrak{so}(n)$ and $\mfp$ is the traceless symmetric matrices. 
\end{example}

\section{Maximal flats}

A \textit{flat} in $X$ is a totally geodesic submanifold isometric to Euclidean space. A flat is \textit{maximal} if it is maximal with respect to inclusion. 

\begin{example}
    The set of diagonal matrices with positive entries on the diagonal, multiplying to $1$ is a maximal flat in $X_n$. 
\end{example}

Recall that a point $p$ in $X$ determines the Cartan decomposition $\mfg = \mfp \oplus \mfk$ and a canonical identification $\mfp = T_pX$. Via this identification, the Riemannian exponential map corresponds to the Lie theoretic exponential map:
$$ \exp_p(X) = e^Xp .$$
Moreover, for a nonpositively curved symmetric space, the Riemannian exponential map $\exp_p \colon \mfp \to X$ is a diffeomorphism. 

\begin{proposition}
    $$ \{ \text{maximal flats through } p \} \xleftrightarrow{} \{ \text{maximal abelian subspaces of } \mfp \} $$
\end{proposition}

\begin{example}
    The set of traceless diagonal matrices form a maximal abelian subspace of $\mfp$ in $\mathfrak{sl}(n,\R)$.
\end{example}

\begin{proof}
    Let $\mfa$ be an abelian subspace of $\mfp$. Via the identification $\mfp = T_pX$, the Riemann curvature tensor at $p$ is given by $R \colon \mfp \times \mfp \times \mfp \to \mfp$,
    $$ R(X,Y)Z = -[[X,Y]Z] ,$$
    which clearly restricts to zero on $\mfa$. Hence $\exp_p(\mfa)$ is a flat.
    
    On the other hand, the sectional curvature of the plane spanned by $X,Y \in \mfp$ is given by 
    $$ \kappa(X,Y) = \frac{ \langle R(X,Y)Y,X \rangle}{\langle X,X \rangle \langle Y,Y \rangle - \langle X,Y \rangle^2 } $$
    which reduces to $B([X,Y],[X,Y])$ if $X,Y$ are orthonormal. If $X,Y$ correspond to vectors tangent to a flat, this must be zero. On the other hand, the restriction of $B$ to $\mfk$ is negative definite, so $[X,Y]=0$. It follows that the preimage of a flat under $\exp_p$ is an abelian subspace of $\mfp$.
\end{proof}

\section{Restricted root space decomposition}

Let $B \colon \mfg \times \mfg \to \R$ be the Killing form of $\mfg$. Recall that it is a symmetric bilinear form, nondegenerate if and only if $\mfg$ is semisimple (We assume that $G$ is semisimple in this section). The point $p \in X$ defines the Cartan involution $\theta \colon \mfg \to \mfg$ so that $\mfk$ is the $(+1)$-eigenspace of $\theta$ and $\mfp$ is the $(-1)$-eigenspace of $\theta$. We may define $B_p \colon \mfg \times \mfg \to \R$ by $B_p(X,Y)=-B(X,\theta (Y))$. 

\begin{lemma}\label{B_p}
    $B_p$ is an inner product. For $X \in \mfk$, $\ad X$ is skew-symmetric with respect to $B_p$. For $X \in \mfp$, $\ad (X)$ is symmetric with respect to $B_p$.
\end{lemma}

\begin{example}
    For $\SL(n,\R)$ and $p=I_n$, $B_p(X,Y) = 2n\tr(X \tran{Y})$. Note that $\tr(X \tran{Y})$ is the entrywise dot product for matrices, also known as the Frobenius inner product. 
\end{example}

By the spectral theorem of linear algebra, each $\ad (X) \colon \mfg \to \mfg$ is real-diagonalizable over $\R$. 

\begin{center}
\adforn{11}
\end{center}

Here's a brief aside to prepare us for root theory. The takeaway is that roots are natural generalizations of eigenvalues and root spaces are natural generalizations of eigenvectors to the setting of a commuting family of linear transformations.

\begin{ex*}
    Let $f_1,\dots,f_n$ be a finite set of commuting, diagonalizable linear transformations of a vector space $V$. Show that they admit a common diagonalization, i.e.\ there exists a basis $\{e_i\}$ of $V$ such that each $f_i$ is diagonal in the basis $\{e_i\}$. 
\end{ex*}

Recall that an eigenvalue of a linear transformation is a number $\lambda$ so that $f v = \lambda v$. When you play with the previous exercise, you will discover that the eigenvalues of linear combinations of the $f_i$ is the linear combination of the eigenvalues of the $f_i$. In other words, rather than think of an eigenvalue as a number, it is better to think of it as a gadget that eats linear combinations of $f_i$ and spits out a number. And this operation should be linear, i.e.\ it should commute with taking linear combinations. In other words, an eigenvalue of a commuting family of linear transformations ${f_i}$ is a linear functional on the span of ${f_i}$. 
\begin{center}
\adforn{11}
\end{center}

We apply these observations to a maximal abelian subspace $\mfa$ of $\mfp$. $\mfa$ acts on $\mfg$ via the adjoint action, and this is a commuting family of diagonalizable linear transformations. Therefore we may consider its simultaneous eigenspace decomposition, which in this case is called the \textit{restricted root space decomposition}: 
$$ \mfg = \mfg_0 \oplus \bigoplus_{\alpha \in \Sigma} \mfg_\alpha .$$
  where $\mfg_\alpha = \{ X \in \mfg : \forall A \in \mfa, [A,X]=\alpha(A)X \}$ and $\Sigma = \Sigma(\mfg,\mfa)= \{ \alpha \in \mfa^\ast \setminus \{0\} : \mfg_\alpha \ne 0 \}$ is the set of \textit{restricted roots}.

\begin{example}
    For $\mfa$ as above in $\mathfrak{sl}(n,\R)$, $\mfg_0 = \mfa$ and each root $\alpha \colon \mfa \to \R$ is given by the difference of two diagonal entries. Each root space is one-dimensional, spanned by an off-diagonal elementary matrix.  
\end{example}

\begin{ex*}
    Give an example where $\mfa \ne \mfg_0$.
\end{ex*}

\section{Weyl group and Weyl chambers}

The restricted root space decomposition depends only on $p$ and $\mfa$. Therefore 
$$ N_K(\mfa) = \{ k \in K : \Ad(k)(\mfa)=\mfa \} $$
acts on $\mfa$ and also the set of roots (recall that $K$ denotes the stabilizer of $p$ in $G$). The kernel of the action on $\mfa$ is 
$$ Z_K(\mfa) = \{ k \in K : \forall A \in \mfa, \Ad(k)(A)=A \} $$
and the quotient $W = N_K(\mfa)/Z_K(\mfa)$ is called the \textit{(restricted) Weyl group}. 

Each root $\alpha \in \Sigma$ has a \textit{wall}
$$ w_\alpha = \ker \alpha = \{ A \in \mfa : \alpha(A) = 0 \} .$$
The components of $\mfa \setminus \cup_{\alpha \in \Sigma} w_\alpha$ are called \textit{(open) Euclidean Weyl chambers}. Elements of $\mfa$ in open Euclidean Weyl chambers are called \textit{regular}. We often choose an open Euclidean Weyl chamber and denote it $\mfa^+$ (sometimes called the \textit{positive Weyl chamber}). 

\begin{example}
    For $\SL(n,\R)$, the restricted Weyl group $W$ is isomorphic to the permutation group $S_d$ on $d$ symbols. It acts on $\mfa$ via simultaneously permuting the rows and columns.
\end{example}

Via the exponential map, we can talk about walls and Euclidean Weyl chambers in a maximal flat as well. It turns out that each wall of a maximal flat is an intersection of maximal flats.  

Choose a regular element $X \in \mfa^+$. There is an associated set of \textit{positive roots} $\Sigma^+ = \{\alpha \in \Sigma : \alpha(X) > 0 \}$ and a similarly defined set of {negative roots}. Every root is positive or negative. There is also a set of \textit{simple roots} $\Delta$, uniquely defined by the property that each positive root can be written as a linear combination of simple roots by nonnegative integers. Each wall bounding $\cham$ is the kernel of a unique simple root in $\Delta$. In particular, $\abs{\Delta} = \dim \mfa = \rank X$.

\section{The KAK decomposition}

\begin{theorem}\label{almost KAK}
Let $\cham$ be a closed Euclidean Weyl chamber in a maximal abelian subspace of $\mfp$. Then
    $$ \mfp = \bigcup_{k \in K} \Ad(k)(\cham) .$$
Moreover, if $\Ad(k)(\cham)=\cham$ then $\Ad(k)$ fixes $\cham$ pointwise.
\end{theorem}

\begin{example}
    For $\SL(n,\R)$, the stabilizer of the point-chamber pair $(p,\cham)$, i.e.\
    $$ M = \{ g \in G : gp =p, \Ad(g)(\cham) = \cham \} $$
    is given by diagonal matrices with entries $\pm 1$, with an even number of $-1$'s. 
\end{example}

Theorem \ref{almost KAK} implies that $G$ acts transitively on point-chamber pairs, but in general this action is not simply transitive.

\begin{ex*}
    Give an example of a symmetric space $G/K$ where the stabilizer of a point-chamber pair in $G$ has positive dimension.
\end{ex*}

Unfortunately, the following decomposition is also sometimes referred to as the ``Cartan decomposition." Fortunately, it can be unambiguously referred to as the ``$KAK$ decomposition."

\begin{corollary}\label{KAK}
    If $G$ is a connected semisimple Lie group with finite center, $K$ is a maximal compact and $\overline{\mfa}^+$ is a closed Euclidean Weyl chamber, then $G = K\exp(\overline{\mfa}^+)K$. Moreover, for any two factorizations $g=kak', g=lbl'$ we have $a=b$. 
\end{corollary}

\begin{example}
    Applied to $\SL(n,\R)$, this is just the \textit{singular value decomposition}.
\end{example}

\begin{proof}
    One first shows that any two points of $X$ are related by a transvection. Moreover, each transvection lies in a 1-parameter subgroup corresponding to an element of $\mfp$. It follows that $G = K \exp(\mfp)$. Now apply Theorem \ref{almost KAK}.
\end{proof}

The previous corollary implies that the \textit{Cartan projection} $\mu \colon G \to \overline{\mfa}^+$ is well-defined. (It depends on the basepoint $p$.)

\section{The visual boundary}

In this section we introduce the visual boundary $\visb$ of $X$. There is a natural action of $G$ on $\visb$. The orbits of this action are called \textit{flag manifolds} and the point stabilizers are proper parabolic subgroups. We describe this picture in more detail, then give a description of parabolic subgroups in terms of the restricted root space decomposition.

We say that two unit-speed geodesic rays $c,c' \colon [0,\infty) \to X$ are \textit{asymptotic} if $\sup_t d(c(t),c'(t))$ is finite. The set of equivalence classes of asymptotic rays is called the \textit{visual boundary} of $X$, denoted $\visb$. There is a natural topology on $X \sqcup \visb$ homeomorphic to a closed ball called the \textit{visual compactification} of $X$. $G$ acts on $\visb$ via $g[c] = [g\circ c]$.

\begin{ex*}
    Prove this is indeed an equivalence relation and that this action is well-defined.
\end{ex*}

Since the action of $G$ on $X$ is transitive, we can move the basepoint of any geodesic ray to a fixed point $p$, and by Theorem \ref{almost KAK} in the last section, we may use $K$ to move the geodesic ray into (the exponential of) a preferred Weyl chamber $\cham$. Moreover, there is an essentially unique way of doing this. 

\begin{definition}
    Let $\cham$ be a closed Euclidean Weyl chamber. Its ideal boundary $\sigma$ is called an \textit{ideal/spherical Weyl chamber.}
\end{definition}

\begin{theorem}\label{model Weyl chamber}
    Let $\sigma$ be an ideal Weyl chamber, let $K$ be a maximal compact, and let $\xi$ be an ideal point. Then $G \cdot \xi = K \cdot \xi$ and 
    $$ \abs{(G \cdot \xi) \cap \sigma} = \abs{(K \cdot \xi) \cap \sigma} = 1  .$$
\end{theorem}

The previous Theorem is a consequence of Theorem \ref{almost KAK}. An ideal Weyl chamber is a spherical simplex, and its faces are also simplices. If $g \in G$ takes a simplex $\tau$ to itself, it must fix $\tau$ pointwise. (Warning: this uses our assumption that $G$ is connected!)

\section{Flag manifolds and parabolic subgroups}

\begin{definition}
    The stabilizer $P=G_\xi = \{ g \in G : g\xi = \xi \} $ of an ideal point $\xi \in \visb$ is called a \textit{parabolic subgroup} of $G$.
    The homogeneous space $G/P \cong G \cdot \xi$ is called a \textit{flag manifold}.
\end{definition}

\begin{example}
    We describe the visual boundary, flag manifolds, and parabolic subgroups for $\SL(n,\R)$. In this case, each ideal Weyl chamber $\sigma$ corresponds to a full flag in $\R^n$, i.e.\ a chain of subspaces $0=V_0 \subset V_1 \subset \dots \subset V_{n-1} \subset V_n=\R^n$ with $\dim V_i=i$. Each simplex $\tau$ corresponds to a partial flag. Each parabolic subgroup stabilizes a partial flag. The maximal parabolics correspond to Grassmannians, and the minimal parabolic subgroups (a.k.a.\ Borel subgroup) stabilize a full flag, hence are conjugate to the subgroup of upper triangular matrices. 
    
    Each ideal point $\xi \in \visb$ lies in the interior of some simplex $\tau$. It corresponds to an \textit{eigenvalue-flag} pair, i.e.\ a partial flag $(V_i)$ together with a vector $(\lambda_i) \in \R^n$, with the constraints that $\sum \lambda_i = 0 $ and $\sum \lambda_i^2 =1$ and $\lambda_j = \lambda_{j+1}$ when a subspace of dimension $j$ is missing from the partial flag. Given a point $p \in X$, we can construct a geodesic ray from $p$ and an eigenvalue-flag pair by using the inner product on $\R^n$ corresponding to $p$ to construct an orthonormal basis compatible with the partial flag, then using the eigenvalue data to construct a matrix $X \in \mfp$ with the appropriate eigenvalues. See Eberlein's book \cite{eberlein1996geometry} for more details.
\end{example}

We have seen that a pair of ideal points have the same stabilizers if and only if they span the same simplex of $\visb$. In other words, parabolic subgroups correspond to simplices of $\visb$. 

By Theorem \ref{model Weyl chamber}, we may assume (up to translating by an element of $K$) that our ideal point $\xi \in \visb$ is represented by the geodesic ray 
$$ c(t) = e^{tX}p $$
for $X \in \cham \subset \mfa \subset \mfp$. We can describe $G_\xi$ explicitly in terms of the restricted roots $\Sigma$. Let $\Delta$ be the simple roots corresponding to $
\cham$ and let $\Theta \subset \Delta$ be the simple roots that are positive on $X$. Let $\Sigma^+$ be the positive roots and $\Sigma^-$ be the negative roots. 

The roots which are zero on $X$ lie in $\Span(\Delta \setminus \Theta)$, so the set of roots positive on $X$ is $\Sigma^+_\Theta = \Sigma^+ \setminus \Span(\Delta \setminus \Theta)$. We define 
$$ \mfu_\Theta \coloneqq \sum_{\alpha \in \Sigma^+_\Theta} \mfg_\alpha \quad \text{ and } \quad \mfu_\Theta^{opp} \coloneqq \sum_{\alpha \in \Sigma^+_\Theta} \mfg_{-\alpha}. $$

\begin{ex*}
    $\mfu_\Theta$ and $\mfu_\Theta^{opp}$ are nilpotent subalgebras of $\mfg$. 
\end{ex*}

\begin{example}
    For $\SL(n,\R)$, the standard subalgebras $\mfu_\Theta$ are strictly upper triangular.
\end{example}

\begin{theorem}
    Let $P_\Theta$ denote the normalizer of $u_\Theta$ in $G$. Then $P_\Theta = G_\xi$. 
\end{theorem}

We skip the proof, but mention the main idea: each group is equal to the set of elements $g \in G$ such that the limit $\lim_{t \to \infty} e^{-tX}ge^{tX}$ exists. See Eberlein's book \cite{eberlein1996geometry} for more details.

\section{Levi subgroups}

The subgroup $P_\Theta$ is called the \textit{standard parabolic} associated to $\Theta$. (`Standard' means standard with respect to the choices of fixed basepoint, flat, and chamber.) Let $P_\Theta^{opp}$ denote the normalizer of $u_\Theta^{opp}$ in $G$. The intersection $L_\Theta \coloneqq P_\Theta \cap P_\Theta^{opp}$ is called the \textit{Levi subgroup}, and its Lie algebra is given by
$$ \mfl_{\Theta} = \mfg_0 \oplus \sum_{\alpha \in \Span(\Delta \setminus \Theta) \cap \Sigma^+} \left(\mfg_\alpha \oplus \mfg_{-\alpha} \right) .$$
The Lie algebra $\mfp_\Theta$ of $P_\Theta$ decomposes as 
$$ \mfp_\Theta  = l_\Theta \oplus u_\Theta .$$

\begin{example}
    For $\SL(d,\R)$, a Levi subgroup is a block diagonal subgroup. It acts on the parallel set of block diagonal symmetric matrices. 
\end{example}

Guichard-Wienhard allow the subset $\Theta$ to be empty, in which case $P_\emptyset = G$, so $G$ itself is considered a parabolic subgroup. At the other extreme, $P_\Delta$ is a minimal parabolic subgroup, a.k.a.\ a \textit{Borel subgroup.}

\begin{remark}
    Warning! The subspace $\mfp$ and subalgebra $\mfp_\Theta$ have essentially nothing to do with each other. Guichard-Wienhard avoid the notation $\mfp$ in the Cartan decomposition by replacing it with $\mfk^\perp$.
\end{remark}

\section{Parallel sets}

Recall that we considered an ideal point $\xi$ in the previous section. Let $c \colon \R \to X$ be the unit-speed geodesic with $c(0)=p$ in the asymptote class of $\xi$. Define $\Theta$ and $P_\Theta$ as before. The Levi subgroup $L_\Theta = P_\Theta \cap P_\Theta^{opp}$ stabilizes a \textit{parallel set}, which is the union of biinfinite geodesics parallel to $c$, equivalently, it is the union of maximal flats containing $c$. If $c$ is regular, it lies in a unique maximal flat, $\Theta = \Delta$, $P_\Theta = P_\Delta$ is a Borel subgroup and the parallel set is a maximal flat. In general, a parallel set is a totally geodesic subspace containing $p$. It is a nonpositively curved symmetric space, but not necessarily of non-compact type i.e.\ it has a Euclidean de Rham factor.

The presence of a Euclidean factor means that $\mfl_\Theta$ is no longer semisimple, but it is still reductive, i.e.\ the direct sum of its center with a semisimple subalgebra. Observe that $L_\Theta$ acts on $\mfu_\Theta$ via the adjoint representation. This representation is crucial for the definition of $\Theta$-positivity!

\begin{center}
\adforn{14}
\end{center}

\begin{ex*}
    Give an example of a symmetric space with an isometry $g$ which fixes $p$ and preserves $\cham$ but does not fix $\cham$ pointwise. (Hint: such an isometry cannot be in the identity component.)
\end{ex*}

\begin{ex*}
    The action of $L_\Theta$ on $G/P_\Theta^{opp}$ fixes the identity coset, so $L_\Theta$ also acts on the tangent space $T_{[P_\Theta^{opp}]}G/P_\Theta^{opp}$. Show that this representation is equivalent to the adjoint action of $L_\Theta$ on $\mfu_\Theta$.
\end{ex*}

\newpage

\thispagestyle{empty}

\chapter[Representations and Dynkin Diagrams]{Representations and Dynkin Diagrams\\ {\Large\textnormal{\textit{by Daniel Soskin}}}}
\addtocontents{toc}{\quad\quad\quad \textit{Daniel Soskin}\par}

\section{Representations}
Let $\mathfrak{g}$ be a complex Lie algebra and $V$ a vector space over the field of complex numbers, unless another field is introduced explicitly.
\begin{definition}
A \emph{representation of a Lie algebra} $\mathfrak{g}$ is a Lie algebras homomorphism 
$\phi \from \mathfrak{g} \to \mathfrak{gl}(V)$, where $\gl(V)$ is the general linear algebra, i.e.\ the algebra of endomorphisms of a vector space $V$ with product given by $[X,Y]=XY-YX$. 
\end{definition}

\begin{definition} A vector space $V$, endowed with an operation $\mathfrak{g} \times V\to V$, is called a \emph{$\mathfrak{g}$-module} if the following conditions are satisfied: for any $X,Y \in \mathfrak{g}$, $v,w \in V$, and $a,b \in F$ we have
\begin{enumerate}
    \item $(aX+bY).v=a(X.v)+b(Y.v)$,
    \item $X.(av+bw)= a(X.v)+b(X.w)$,
    \item $[Y,X].v=XY.v-YX.v$.
\end{enumerate}
\end{definition}

Given a representation $\phi\colon\mathfrak{g} \to \gl(V)$, we can view $V$ as a $\mathfrak{g}$-module via the action $X.v=\phi(X)(v)$.
Conversely, given a $\mathfrak{g}$-module $V$, this equation defines a representation $\phi\colon\mathfrak{g} \to \gl(V)$.

\begin{definition}
A \emph{homomorphism of $\mathfrak{g}$-modules} is a linear map $\phi\colon V \to W$ such that $\phi(X.v)=X.\phi(v)$.
\end{definition}

\begin{definition}
A $\mathfrak{g}$-module is called \emph{irreducible} if it has precisely two $\mathfrak{g}$-submodules (itself and $0$).
\end{definition}

\begin{definition} A $\mathfrak{g}$-module is called \emph{completely reducible} if it is a direct sum of irreducible $\mathfrak{g}$-submodules.
\end{definition}

\begin{lemma}[Schur's Lemma] Let $\phi\colon\mathfrak{g} \to \gl(V)$ be an irreducible representation of $\mathfrak{g}$.
Then the only endomorphisms of $V$ commuting with $\phi(\mathfrak{g})$ are the scalars. 
\end{lemma}

\begin{theorem}[Weyl] Let $\phi\colon\mathfrak{g} \to \gl(V)$ be a (finite-dimensional) representation of a semisimple Lie algebra. Then $\phi$ is completely reducible.
\end{theorem}
Let $\mathfrak{g}$ be a semisimple Lie algebra, $\mathfrak{h}$ a fixed Cartan subalgebra and $\Sigma$ a root system relative to $\mathfrak{h}$. If $V$ is a finite-dimensional $\mathfrak{g}$-module, $\mathfrak{h}$ acts diagonally on $V$, and $V$ is a direct sum of its weight spaces:  $V= \bigoplus_{\lambda \in \mathfrak{h}^{*}} V_{\lambda}$, where $V_{\lambda}=\{v \in V|h.v=\lambda(h)v$ $\forall h \in \mathfrak{h} \}$. Whenever $V_{\lambda}\neq 0$, it is called \emph{weight space} and $\lambda$ is called \emph{weight}.
\begin{example}
The Lie algebra $\mathfrak{g}$ itself is an $\mathfrak{g}$-module via the adjoint representation.
Then, weights are the roots $\alpha \in \Sigma$ with the weight spaces $\mathfrak{g}_{\alpha}$ from the root space decomposition along with the weight space $\mathfrak{h}$ of $0$.
\end{example}
\begin{ex*}
Verify that $\mathfrak{g}_{\alpha}$ maps $V_{\lambda}$ into $V_{\lambda+\alpha}$ , $\lambda \in \mathfrak{h}^{*}, \alpha \in \Sigma$.
\end{ex*}

\begin{example}
Let us consider $\mathfrak{g}=\sl(2,\C)$ with the standard basis 
$X=\begin{pmatrix} 0&1 \\ 0&0  \end{pmatrix}$,
$Y=\begin{pmatrix} 0&0 \\ 1&0  \end{pmatrix}$,
$H=\begin{pmatrix} 1&0 \\ 0&-1  \end{pmatrix}$. One can verify that $[H,X]=2X$, $[H,Y]=-2Y$, $[X,Y]=H$, and thus for $\sl(2,\C)$-module $V$ it follows that if $v \in V_{\lambda}$, then $X.v \in V_{\lambda+2}$ and $Y.v \in V_{\lambda-2}$. Assume now that $V$ is an irreducible finite-dimensional $\mathfrak{g}$-module.
There exists a nonzero $v_{0}\in V_{\lambda}$ such that $V_{\lambda+2}=0$, and $X.v_{0}=0$. Define $v_{i}=\frac{1}{i!}Y^{i}.v_{0}$, ($i>-1$). Then,

\begin{enumerate}
    \item $H.v_{i}=(\lambda-2i)v_{i}$;
    \item $Y.v_{i}=(i+1)v_{i+1}$;
    \item $X.v_{i}=(\lambda-i-1)v_{i-1}$.
\end{enumerate}

All nonzero $v_{i}$ are linearly independent.
Let $m$ be an integer such that $v_{m}\neq0$ and $v_{m'}=0$ for $m<m'$.
The subspace of $V$ with basis $v_{0}, v_{1}, ..., v_{m}$ is a $\mathfrak{g}$-submodule, so it must be equal to $V$, since the latter is irreducible.
Note, that from (3) it follows that $m=\lambda$, and $m=\dim(V)-1$, because each weight space is of dimension one. Thus, there exist only one irreducible $\sl(2,\C)$-module of each possible dimension.

It turns out, that for any complex semisimple Lie algebra $\mathfrak{g}$ is built from copies of $\sl(2,\C)$. For any root $\alpha \in \Sigma$, and any $X_{\alpha} \in \mathfrak{g}_{\alpha}$, there exist $Y_{-\alpha} \in \mathfrak{g}_{-\alpha}$ such that $X_{\alpha}$, $Y_{\alpha}$ and $[X_{\alpha},Y_{\alpha}]$ span a subalgebra of $\mathfrak{g}$ isomorphic to $\sl(2,\C)$.
\end{example}

\begin{definition}
Let $\mathfrak{h}_{0}$ be the real form of $\mathfrak{h}$.
An element $\lambda \in \mathfrak{h}_{0}^\ast$ is \emph{algebraically integral} if $$\frac{2(\lambda,\alpha)}{(\alpha,\alpha)}$$ are integers for all roots $\alpha$.
\end{definition}
\begin{ex*}
The weight of any finite-dimensional representation is algebraically integral. 
\end{ex*}
The \emph{fundamental weights} $w_{1},\ldots,w_{l}$ are defined in a way that they form basis of $\mathfrak{h}_{0}^\ast$ such that $\frac{2(w_{i},\alpha_{j})}{(\alpha_{j},\alpha_{j})}=\delta_{ij}$, where $\alpha_{i}$ are simple roots. An element $\lambda$ is algebraically integral if and only if it is an integral combination of the fundamental weights. Thus, the set of all algebraically integral weights in $\mathfrak{h}_{0}^\ast$ form the \emph{weight lattice} for $\mathfrak{g}$. 

Suppose that the Lie algebra $\mathfrak{g}$ is the Lie algebra of a Lie group $G$. Then, $\lambda \in \mathfrak{h}_{0}^\ast$ is \emph{analytically integral} if for each $t$ in $\mathfrak{h}$ such that $\exp(t)=e$ in $G$ we have $(\lambda,t) \in 2\pi i\Z$. If a representation of $\mathfrak{g}$ arises from representation of $G$, then the weights of the representation are analytically integral. For semisimple $G$ the set of analytically integral weights form a sub-lattice of the lattice of algebraically integral weights. If $G$ is simply connected then both lattices coincide. 

There is a partial order on the space of weights. We say that $\mu$ is \emph{higher} that $\lambda$ ($\mu \succeq \lambda$), if $\mu-\lambda$ is expressible as a linear combination of positive roots with non-negative coefficients. 

An integral element $\lambda$ is called \emph{dominant} if it is a non-negative integer combination of the fundamental weights. Note, that it is not the same as being higher than $0$.
\begin{definition}
A weight $\lambda$ of a representation $V$ of $\mathfrak{g}$ is called \emph{highest weight} if it is higher than every other weight of $V$. 
\begin{theorem}
\begin{enumerate}
    \item  Every irreducible finite-dimensional representation has a highest weight.
\item  The highest weight is dominant and algebraically integral element.
\item Two irreducible representations with the same highest weight are isomorphic.
\item Every dominant, algebraically integral element is the highest weight of an irreducible representation. 
\end{enumerate}
\end{theorem}

\end{definition}

\section{Dynkin diagrams and related classifications}

Let $\Sigma$ be a root system of dimension $l$, with ordered basis of simple roots $\{\alpha_{1},\ldots,\alpha_{l}\}$. For two roots $\alpha,\beta \in \Sigma$, let $\langle \beta,\alpha \rangle:=\frac{2(\beta,\alpha)}{(\alpha,\alpha)}$.
\begin{definition}
The root system $\Sigma$ is called \emph{irreducible} if it cannot be partitioned into the union of two orthogonal proper subsets.
\end{definition}
\begin{theorem} The root system of a simple Lie algebra is irreducible. 
\end{theorem}
The following statement reduces the problem of characterizing semisimple Lie algebras by their root systems to the problem of characterizing simple ones by their irreducible root systems.
\begin{theorem}
Let $\mathfrak{g}$ be a semisimple Lie algebra with $\mathfrak{h}$ and $\Sigma$ as defined before. If $\mathfrak{g}=\mathfrak{g}_{1}\oplus...\oplus \mathfrak{g}_{t}$ is the decomposition of $\mathfrak{g}$ into simple ideals, then $\mathfrak{h}_{i}=\mathfrak{h}\cap \mathfrak{g}_{i}$ is a Cartan subalgebra of $\mathfrak{g}_{i}$ with relative irreducible root subsystem $\Sigma_{i}$ of $\Sigma$ in such a way that $\Sigma=\Sigma_{1}\cup\ldots\cup\Sigma_{t}$ is the decomposition of $\Sigma$ into its irreducible components.
\end{theorem}
Irreducible root systems can be classified by so called \emph{Dynkin diagrams}.

Since $\langle \beta,\alpha \rangle$ must be an integer (called \emph{Cartan integer}), there are only a few possible values: $0$, $\pm1$, $\pm2$, $\pm3$.
Moreover, at most two distinct root lengths may occur in an irreducible system $\Sigma$, which are referred to as \emph{long} and \emph{short} roots.
The matrix $(\langle \alpha_{i},\alpha_{j} \rangle)_{i,j=1,\ldots,l}$ is called the \emph{Cartan matrix} of $\Sigma$. Define the \emph{Coxeter graph} of $\Sigma$ to be a graph on $l$ vertices $\alpha_1,\ldots,\alpha_l$, where the $\alpha_i$ is joined to $\alpha_j$ ($i\neq j$) by $\langle \alpha_{i},\alpha_{j} \rangle\,\langle \alpha_{j},\alpha_{i} \rangle$ edges. Whenever a double or triple edge occurs in the Coxeter graph of $\Sigma$, we add an arrow pointing to the shorter of the two roots.
The resulting object is called a \emph{Dynkin diagram}.
Since for simple roots $(\alpha,\beta)\leq0$, all Cartan integers can be recovered from the Dynkin diagram.
Note that the Coxeter graph is connected if and only if its corresponding root system $\Sigma$ is irreducible. 
\begin{theorem}
If $\Sigma$ is an irreducible root system of rank $l$, its Dynkin diagram is one of the following ($l$ vertices in each case):
$A_{l}$ $(l \geq 1)$, $B_{l}$ $(l \geq 2)$,
$C_{l}$ $(l \geq 3)$,
$D_{l}$ $(l \geq 4)$, $E_{6}$, $E_{7}$, $E_{8}$, $F_{4}$, $G_{2}$, see Table \ref{fig:general}.
\end{theorem}

\def\do{}
\begin{table}[ht]
    \centering
    \begin{tabular}{l l}
         $A_\ell \, (\ell \geq 1)$: & 
         \dynkin[labels={\alpha_1,\alpha_2,\alpha_{\ell-1},\alpha_\ell}, scale=1.5]A{} \\
         $B_\ell \, (\ell \geq 2)$: & {\dynkin[labels={\alpha_1,\alpha_2,\alpha_{\ell-2},\alpha_{\ell-1},\alpha_{\ell}}, scale=1.5]B{}}\\
         $C_\ell \, (\ell \geq 3)$: &{{\dynkin[labels={\alpha_1,\alpha_2,\alpha_{\ell-2},\alpha_{\ell-1},\alpha_{\ell}}, scale=1.5]C{}} }\\
        $D_\ell \, (\ell \geq 4)$: &{{\dynkin[labels={\alpha_1,\alpha_2,\alpha_{\ell-3},\alpha_{\ell-2},\alpha_{\ell-1},\alpha_{\ell}}, label directions={,,,right,,}, scale=1.5]D{}} }\\
        $E_6$: & {{\dynkin[labels={\alpha_1,\alpha_2,\alpha_3,\alpha_4,\alpha_5,\alpha_6}, scale=1.5] E6}}\\
        $E_7$: & {{\dynkin[labels={\alpha_1,\alpha_2,\alpha_3,\alpha_4,\alpha_5,\alpha_6,\alpha_7}, scale=1.5] E7}}\\
        $E_8$: &{{\dynkin[labels={\alpha_1,\alpha_2,\alpha_3,\alpha_4,\alpha_5,\alpha_6, \alpha_7, \alpha_8}, scale=1.5] E8}}\\
        $F_4$: &{{\dynkin[labels={\alpha_1,\alpha_2,\alpha_3,\alpha_4}, scale=1.5] F4}}\\
        $G_2$: & {{\dynkin[labels={\alpha_1,\alpha_2}, scale=1.5] G2}}
    \end{tabular}
    \caption{List of Dynkin diagrams}
    \label{fig:general}
\end{table}
\begin{theorem}For each Dynkin diagram of type A-G, there exists an irreducible root system having the given diagram.
\end{theorem}

Classification of real Lie algebras differs from the case of Lie algebras over algebraically closed fields. \begin{definition}
The \emph{complexification of a real Lie algebra} $\mathfrak{g}$ is obtained from $\mathfrak{g}$ by extending the field of scalars from real to complex. Elements of $\mathfrak{g}^{\C}$ can be considered as pairs $(u,v)$, $u,v \in \mathfrak{g}$ (or $u+iv$) with the following operations:

\begin{enumerate}
    \item $(u_{1},v_{1})+(u_{2},v_{2})=(u_{1}+u_{2},v_{1}+v_{2})$,

    \item $(\alpha+i\beta)(u,v)=(\alpha u-\beta v, \alpha v + \beta u)$ for any real $\alpha, \beta$,

\item $[(u_{1},v_{1}),(u_{2},v_{2})]=([u_{1},u_{2}]-[v_{1},v_{2}],[v_{1},u_{2}]+[u_{1},v_{2}])$.
\end{enumerate}
\end{definition}

If $\mathfrak{g}$ is a finite-dimensional real Lie algebra, its complexification is either simple or a product of a simple complex Lie algebra and its conjugate.
Thus, real simple Lie algebras can be classified by the classification of complex Lie algebras and some additional information, which can be done by \emph{Satake diagrams}, which generalize Dynkin diagrams.

\newpage

\thispagestyle{empty}

\chapter[The Indefinite Special Orthogonal Group]{The Indefinite Special Orthogonal Group\\ {\Large\textnormal{\textit{by Jacques Audibert}}}}
\addtocontents{toc}{\quad\quad\quad \textit{Jacques Audibert}\par}

\section{The Lie group and its Lie algebra}

Denote $I_{p,q}=
\begin{pmatrix}
I_p&0\\
0&-I_q
\end{pmatrix}
$.

\begin{theorem}
For every $B\in\GL(n,\mathbb{R})$ symmetric there exists $p,q\geq0$ and $P \in\GL(n,\mathbb{R})$ such that $\tran{P}BP=I_{p,q}$.
\end{theorem}

For every $B\in\GL(n,\mathbb{R})$ symmetric define $\SO(B)=\{M\in\SL(n,\mathbb{R}): \tran{M}BM=B\}$.

\begin{remark}
If $\tran{P}BP=I_{p,q}$ then $\SO(B)=P\SO(I_{p,q})P^{-1}$.
\end{remark}

\begin{example}
Let $0\leq p\leq q$.
Denote by
\begin{equation*}
    Q_{p,q}=\begin{pmatrix}
    0&0&W_p\\
    0&-I_{q-p}&0\\
    \tran{W_p}&0&0
    \end{pmatrix}\in\GL(p+q,\mathbb{R})
\end{equation*}
where
\begin{equation*}
    W_p=\begin{pmatrix}
    &&&(-1)^{p-1}\\
    &&\iddots&\\
    &-1&&\\
    1&&&
    \end{pmatrix}\in\GL(p,\mathbb{R})
\end{equation*}
Denote
\begin{equation}
    P=
    \begin{pmatrix}
    \frac{(-1)^{p-1}}{\sqrt{2}}&&&&&&&&\frac{(-1)^p}{\sqrt{2}}\\
    &\ddots&&&&&&\iddots&\\
    &&-\frac{1}{\sqrt{2}}&&&&\frac{1}{\sqrt{2}}&&\\
    &&&\frac{1}{\sqrt{2}}&&-\frac{1}{\sqrt{2}}&&&\\
    &&&&-I_{q-p}&&&&\\
    &&&\frac{1}{\sqrt{2}}&&\frac{1}{\sqrt{2}}&&&\\
    &&\frac{1}{\sqrt{2}}&&&&\frac{1}{\sqrt{2}}&&\\
    &\iddots&&&&&&\ddots&\\
    \frac{1}{\sqrt{2}}&&&&&&&&\frac{1}{\sqrt{2}}
    \end{pmatrix}
\end{equation}
Then $\tran{P}Q_{p,q}P=I_{p,q}$ thus $\SO(Q_{p,q})=P\SO(I_{p,q})P^{-1}$.
\end{example}

\begin{proposition}
\begin{itemize}
    \item $\SO(Q_{p,q})$ is a Lie group.
    \item If $p=0$, then $\SO(Q_{0,q})$ is connected and compact.
    \item If $p>0$, then $\SO(Q_{p,q})$ has two connected components and is not compact.
\end{itemize}
\end{proposition}

Denote $\SO_0(Q_{p,q})$ the connected component of $\SO(Q_{p,q})$ containing the identity.

\begin{theorem}
$\SO_0(Q_{p,q})$ is homeomorphic to $\SO(I_p)\times\SO(I_q)\times\mathbb{R}^{pq}$.
\end{theorem}

\begin{ex*}
Determine the fundamental group of $\SO_0(Q_{p,q})$.
\end{ex*}

\begin{proposition}
The Lie algebra of $\SO(Q_{p,q})$ is 
\[\mathfrak{so}(Q_{p,q})=\{M\in\Mat(p+q,\mathbb{R}): \tran{M}Q_{p,q}+Q_{p,q}M=0,\ \Tr(M)=0\}=\]
\begin{equation*}\left\{
    \begin{array}{ll}
        \begin{pmatrix}
            M_{11}&M_{12}&M_{13}\\
            M_{21}&M_{22}&\tran{M}_{12}W_p\\
            M_{31}&\tran{W}_p \tran{M}_{21}&-\tran{W}_p \tran{M}_{11} W_p
        \end{pmatrix}
        :
        \tran{M}_{22}=-M_{22},\ M_{13}=-W_p \tran{M}_{13} W_p,\ M_{31}=-W_p \tran{M}_{31} W_p
    \end{array}\right\}
\end{equation*}
with 
\begin{equation*}
M_{ij}=-\tran{W}_p \tran{M}_{ij} W_p \Longleftrightarrow M_{ij}=
\begin{pmatrix}
    a_{11}&a_{12}&\cdots&a_{1p-1}&0\\
    a_{21}&a_{22}&\cdots&0&a_{1p-1}\\
    \vdots&\vdots&\ddots&\vdots&\vdots\\
    a_{p-11}&0&\cdots&(-1)^{p-2}a_{22}&(-1)^{p-1}a_{12}\\
    0&a_{p-11}&\cdots&(-1)^{p-1}a_{21}&(-1)^pa_11
\end{pmatrix}.
\end{equation*}
\end{proposition}

The Lie bracket is $[X,Y]=XY-YX$.

We thus see that $\dim\mathfrak{so}(Q_{p,q})=\frac{(p+q)(p+q-1)}{2}$.

\begin{proposition}
Its Killing form is $B(X,Y)=(p+q-2)\Tr(XY)$ $(p+q\geq3)$.
\end{proposition}

Since, for $p+q\geq3$, $B(X,Y)$ is non-degenerate, $\mathfrak{so}(Q_{p,q})$ is semisimple.

\begin{remark}
Let $p+q\geq3$.
Then $\mathfrak{so}(Q_{p,q})$ is simple unless
\begin{itemize}
    \item $p=q=2$ because $\mathfrak{so}(Q_{2,2})=\mathfrak{sl}(2,\mathbb{R})\oplus\mathfrak{sl}(2,\mathbb{R})$,
    \item $p=0$ and $q=4$ because $\mathfrak{so}(Q_{0,4})=\mathfrak{su}(I_2)\oplus\mathfrak{su}(I_2)$.
\end{itemize}
\end{remark}

\begin{ex*}
Prove that $\SO_0(I_{2,2})\cong \SL(2,\mathbb{R})\times\SL(2,\mathbb{R})/\{\pm(I_2,I_2)\}$ and $\SO_0(I_4)\cong \SU(I_2)\times\SU(I_2)/\{\pm(I_2,I_2)\}$.
\end{ex*}

\begin{example}
$\SO(Q_{2,3})$ has Lie algebra
\begin{equation*}
    \mathfrak{so}(Q_{2,3})=\left\{
    \begin{array}{ll}
    \begin{pmatrix}
    a_{12}&a_{12}&a_{13}&a_{14}&0\\
    a_{21}&a_{22}&a_{23}&0&a_{14}\\
    a_{31}&a_{32}&0&a_{23}&-a_{13}\\
    a_{41}&0&a_{32}&-a_{22}&a_{12}\\
    0&a_{41}&-a_{31}&a_{21}&-a_{11}
    \end{pmatrix}
    : \ a_{ij}\in\mathbb{R}
    \end{array}
    \right\}.
\end{equation*}
\end{example}

\section{The restricted root system}

Let $0\leq p\leq q$ with $p+q\geq 3$.

\begin{proposition}
The map $\theta\colon \mathfrak{so}(Q_{p,q})\rightarrow\mathfrak{so}(Q_{p,q})$, $M\mapsto -\tran{M}$ is a Cartan involution.
\end{proposition}

\begin{proof}
Since $\mathfrak{so}(Q_{p,q})$ is close under transpose, $\theta$ is a Lie involution of $\mathfrak{so}(Q_{p,q})$. For any $X\in\mathfrak{so}(Q_{p,q})$ non-zero $-B(X,\theta(X))=(p+q-2)\Tr(X\tran{X})>0$
\end{proof}

Thus $\mathfrak{so}(Q_{p,q})=\mathfrak{k}\oplus\mathfrak{p}$ where $\mathfrak{k}=\{M\in\mathfrak{so}(Q_{p,q})| -\tran{M}=M\}$ and $\mathfrak{p}=\{M\in\mathfrak{so}(Q_{p,q})| \tran{M}=M\}$. This is the Cartan decomposition.

Denote
\begin{equation*}
    \mathfrak{a}=\left\{\begin{array}{ll}
    D_{\lambda}=\begin{pmatrix}
    \lambda_1&&&&&&\\
    &\ddots&&&&&\\
    &&\lambda_p&&&&\\
    &&&0&&&\\
    &&&&-\lambda_p&&\\
    &&&&&\ddots&\\
    &&&&&&-\lambda_1
    \end{pmatrix}
    : \lambda_i\in\mathbb{R}
    \end{array}\right\}
\end{equation*}
This is a maximal abelian subspace of $\mathfrak{p}$ as the following computations show.

Denote
\begin{equation*}
    (E_{kl})_{ij}=\left\{\begin{array}{ll}
    1 & \mbox{if $i=k$ and $j=l$},\\
    0 & \mbox{otherwise}.
\end{array}\right.
\end{equation*}

Then we compute the following brackets:

    \begin{equation}
    \begin{bmatrix}D_{\lambda},\begin{pmatrix}
    E_{ij}&&\\
    &0&\\
    &&(-1)^{i+j+1}E_{p-j+1p-i+1}
    \end{pmatrix}\end{bmatrix}=(\lambda_i-\lambda_j)
    \begin{pmatrix}
    E_{ij}&&\\
    &0&\\
    &&(-1)^{i+j+1}E_{p-j+1p-i+1}
    \end{pmatrix}
    \end{equation}
    \begin{equation}
    \begin{bmatrix}D_{\lambda},\begin{pmatrix}
    0&&\\
    &E_{ij}&\\
    &&0
    \end{pmatrix}\end{bmatrix}=0
    \end{equation}
    \begin{equation}
    \begin{bmatrix}D_{\lambda},\begin{pmatrix}
    0&E_{ij}&\\
    &0&(-1)^{p-i}E_{jp-i+1}\\
    &&0
    \end{pmatrix}\end{bmatrix}=\lambda_i
    \begin{pmatrix}
    0&E_{ij}&\\
    &0&(-1)^{p-i}E_{jp-i+1}\\
    &&0
    \end{pmatrix}
    \end{equation}
    \begin{equation}
    \begin{bmatrix}D_{\lambda},\begin{pmatrix}
    0&&\\
    E_{ij}&0&\\
    &(-1)^{p-j}E_{p-j+1i}&0
    \end{pmatrix}\end{bmatrix}=-\lambda_j
    \begin{pmatrix}
    0&&\\
    E_{ij}&0&\\
    &(-1)^{p-j}E_{p-j+1i}&0
    \end{pmatrix}
    \end{equation}
    \begin{equation}
    \begin{bmatrix}D_{\lambda},\begin{pmatrix}
    &&E_{ij}+(-1)^{p+i+j}E_{p-j+1p-i+1}\\
    &0&\\
    0&&
    \end{pmatrix}\end{bmatrix}=(\lambda_i+\lambda_{p-j+1})
    \begin{pmatrix}
    &&E_{ij}+(-1)^{p+i+j}E_{p-j+1p-i+1}\\
    &0&\\
    0&&
    \end{pmatrix}
    \end{equation}
    \begin{equation}
    \begin{bmatrix}D_{\lambda},\begin{pmatrix}
    &&0\\
    &0&\\
    E_{ij}+(-1)^{p+i+j}E_{p-j+1p-i+1}&&
    \end{pmatrix}\end{bmatrix}=(-\lambda_{p-i+1}-\lambda_j)
    \begin{pmatrix}
    &&0\\
    &0&\\
    E_{ij}+(-1)^{p+i+j}E_{p-j+1p-i+1}&&
    \end{pmatrix}
    \end{equation}

\begin{remark}
The rank of $\mathfrak{so}(Q_{p,q})$ is the dimension of $\mathfrak{a}$ which is $p$.
\end{remark}

We see that 
\begin{equation*}
    Z_{\mathfrak{k}}(\mathfrak{a})=\{X\in\mathfrak{k}\ :\ [A,X]=0\ \forall A\in\mathfrak{a}\}=\left\{
    \begin{array}{ll}
        \begin{pmatrix}
            0&&\\
            &M_{22}&\\
            &&0
        \end{pmatrix}:\ M_{22}^\top=-M_{22}
    \end{array}\right\}.
\end{equation*}

\begin{definition}
A real semisimple Lie algebra is said to be \emph{split} if $Z_{\mathfrak{k}}(\mathfrak{a})=\{0\}$.
\end{definition}

Thus $\mathfrak{so}(Q_{p,q})$ is split if and only if $p=q$ or $p=q-1$.

Denote $\epsilon_i:\mathfrak{a}\rightarrow\mathbb{R},D_{\lambda}\mapsto\lambda_i$ for every $i$.

\begin{proposition}
The set of restricted roots of $\mathfrak{so}(Q_{p,q})$ relative to $\mathfrak{a}$ is
\begin{itemize}
    \item if $p\neq q$, then $\Sigma=\{\epsilon_i\ (1\leq i\leq p),\ -\epsilon_i\ (1\leq i\leq p),\ \epsilon_i+\epsilon_j\ (1\leq i<j\leq p),\ -\epsilon_i-\epsilon_j\ (1\leq i<j\leq p),\ \epsilon_i-\epsilon_j\ (1\leq i\neq j\leq p)\}$,
    \item if $p=q$, then $\Sigma=\{\epsilon_i+\epsilon_j\ (1\leq i<j\leq p),\ -\epsilon_i-\epsilon_j\ (1\leq i<j\leq p),\ \epsilon_i-\epsilon_j\ (1\leq i\neq j\leq p)\}$,
\end{itemize}
and the root space decomposition is
\begin{equation*}
    \mathfrak{so}(Q_{p,q})=\mathfrak{a}\oplus Z_{\mathfrak{k}}(\mathfrak{a})\oplus\bigoplus_{1\leq i\leq p}\mathfrak{g}_{\epsilon_i}\oplus\bigoplus_{1\leq i\leq p}\mathfrak{g}_{-\epsilon_i}\oplus\bigoplus_{1\leq i\neq j\leq p}\mathfrak{g}_{\epsilon_i-\epsilon_j}\oplus\bigoplus_{1\leq i<j\leq p}\mathfrak{g}_{\epsilon_i+\epsilon_j}\oplus\bigoplus_{1\leq i<j\leq p}\mathfrak{g}_{-\epsilon_i-\epsilon_j}
\end{equation*}
where $\mathfrak{g}_{\lambda}=\{X\in\mathfrak{so}(Q_{p,q}) : [A,X]=\lambda(A)X\ \forall A\in\mathfrak{a}\}$
with $\dim(\mathfrak{g}_{\epsilon_i})=q-p$, $\dim(\mathfrak{g}_{-\epsilon_i})=q-p$, $\dim(\mathfrak{g}_{\epsilon_i-\epsilon_j})=1$, $\dim(\mathfrak{g}_{\epsilon_i+\epsilon_j})=1$ and $\dim(\mathfrak{g}_{-\epsilon_i-\epsilon_j})=1$.
\end{proposition}

\begin{proof}
This is a consequence of the previous bracket computations.
\end{proof}

\begin{proposition}
If $p\neq q$, then $\Delta=\{\epsilon_1-\epsilon_2,\ \epsilon_2-\epsilon_3, \ldots,\ \epsilon_{p-1}-\epsilon_p,\ \epsilon_p\}$ is a set of simple roots. If $p=q$ then $\Delta=\{\epsilon_1-\epsilon_2, \epsilon_2-\epsilon_3, \ldots,\ \epsilon_{p-1}-\epsilon_p,\ \epsilon_{p-1}+\epsilon_p\}$ is a set of simple roots.
\end{proposition}

\begin{proof}
We only prove the case $p\neq q$. We have to express each root as a combination of the simple ones involving only non-negative or non-positive coefficients. We see that $\epsilon_i=\epsilon_i-\epsilon_{i+1}+\epsilon_{i+1}-\epsilon_{i+2}+\ldots+\epsilon_{p-1}-\epsilon_p+\epsilon_p$. Also for $i<j$ $\epsilon_i+\epsilon_j=\epsilon_i-\epsilon_{i+1}+\epsilon_{i+1}-\epsilon_{i+2}+\ldots+\epsilon_{j-1}-\epsilon_j+2\epsilon_j-2\epsilon_{j+1}+2\epsilon_{j+1}-2\epsilon_{j+2}+\ldots+2\epsilon_{p-1}-2\epsilon_p+2\epsilon_p$ and $\epsilon_i-\epsilon_j=\epsilon_i-\epsilon_{i+1}+\epsilon_{i+1}-\epsilon_{i+2}+\ldots+\epsilon_{j-1}-\epsilon_j$. We deduce the other equations by multiplying those by $-1$.
\end{proof}

\begin{ex*}
Check the case $p=q$.
\end{ex*}

The Dynkin diagram of $\mathfrak{so}(Q_{p,q})$ is $B_p$ if $p\neq q$ and $D_p$ if $p=q$.

\begin{example}
Case of $\mathfrak{so}(Q_{2,3})$:
\begin{equation*}
    \mathfrak{a}=\left\{
    \begin{array}{ll}
        \begin{pmatrix}
            \lambda_1&&&&\\
            &\lambda_2&&&\\
            &&0&&\\
            &&&-\lambda_2&\\
            &&&&-\lambda_1
        \end{pmatrix}:\ \lambda_i\in\mathbb{R}
    \end{array}
    \right\}
\end{equation*}
so $\mathfrak{so}(Q_{2,3})$ has rank 2. We have $Z_{\mathfrak{k}}(\mathfrak{a})=\{0\}$ so $\mathfrak{so}(Q_{2,3})$ is split.
Then $\Sigma=\{\epsilon_1,\ \epsilon_2,\ -\epsilon_1,\ -\epsilon_2,\ \epsilon_1+\epsilon_2,\ -\epsilon_1-\epsilon_2,\ \epsilon_1-\epsilon_2,\ \epsilon_2-\epsilon_1\}$ and all root spaces have dimension 1. As a set of simple roots we can pick $\Delta=\{\epsilon_1-\epsilon_2,\ \epsilon_2\}$ and the Dynkin diagram is thus $B_2$.
\end{example}

\newpage

\thispagestyle{empty}

\chapter[Exercises - Part 1]{Examples and Exercises - Part 1\\ {\Large\textnormal{\textit{by Alex Moriani}}}}
\addtocontents{toc}{\quad\quad\quad \textit{Alex Moriani}\par}

\section{Closed linear groups}
In this session we will talk about linear Lie groups, that is, closed Lie subgroups of $\GL(n,\R)$. But first, let us just recall how to see the Lie algebra of a linear Lie group as a group of matrices.\\
The group $\GL(n,\R)$ is open in $\Mat(n,\R)$ (which is a real vector space), and so it is a manifold of dimension $n^2$.

\begin{exercise}
Prove that the general linear group $\GL(n,\R)$ is a Lie group. Prove that its Lie algebra is isomorphic to $\Mat(n,\R)$ endowed with the usual bracket on matrices $[A,B]=AB-BA$.
\end{exercise}

\textit{Solution.}
$\GL(n,\R)$ is an open subset of $\Mat(n,\R)$ (a real vector space), so a smooth manifold. Multiplication and inverse are polynomial in the coefficients so are smooth and $\GL(n,\R)$ is a Lie group.
The Lie algebra of a Lie group is defined as being the set of left invariant vector fields, which is in bijection with the tangent space at the identity. For $\GL(n,\R)$, the Lie algebra may so be identified with the vector space $\Mat(n,\R)$.
It will be the same for every closed subgroup of $\GL(n,\R)$, its Lie algebra will be a subspace of $\Mat(n,\R)$.

\begin{definition}
The linear Lie algebra of a closed Lie subgroup $G$ of $\GL(n,\R)$ is its tangent space at $I_n$ seen as a subspace of $\Mat(n,\R)$ :\\
\[\mathfrak{g}=\{\gamma'(t)\ |\ \gamma \from \R\to G\ \mathrm{smooth},\ \gamma(0)=I_n\},\]
endowed with the usual bracket on matrices $[X,Y]=XY-YX$.
\end{definition}

\begin{exercise}
Show that it is a Lie algebra (vector space, skew-symmetry and Jacobi identity).
\end{exercise}
\textit{Solution.}
If $X,Y\in \Mat(n,\R)$ corresponding to $x,y \from \R\to \GL(n,\R)$, $xy$ is smooth and its derivative at $t+0$ gives $(xy)'(0)=x'(0)+y'(0)=X+Y$, so $X+Y\in\mathfrak{g}$. Looking at $kx$, for $k\in\R$, we see that $kX\in\mathfrak{g}$. So $\mathfrak{g}$ is a vector space.
Computations show that $[\cdot,\cdot]$ is bilinear skew-symmetric and verify the Jacobi identity. So $\mathfrak{g}$ is a Lie algebra.

\begin{exercise}
Show for the case of $\GL(n,\R)$, its Lie algebra $\mathfrak{gl}(n,\R)$ and its linear Lie algebra $\mathfrak{g}$ are isomorphic.
\end{exercise}

\textit{Hint.} They are isomorphic as vector spaces. Recall that the bracket of the Lie algebra $\mathfrak{gl}(n,\R)$ is given by the bracket of vector fields ($\mathfrak{gl}(n,\R)=\{\mbox{left invariant vector fields}\}$).
\begin{align*}
\Phi : \mathfrak{g} & \to \mathfrak{gl}(n,\R)\\
X & \mapsto X^L,
\end{align*}
where $X^L$ is the left invariant vector field with $(X^L)_{I_n}=X$.
We want to show that this map is a Lie algebra isomorphism.
It is well-defined, one-one, onto.
We need to check that it respects the Lie bracket, i.e.\ we have to verify that for every $X,Y\in\mathfrak{g}$, $[X^L,Y^L]=[X,Y]^L=(XY-YX)^L$, and to do that we check that they act in the same way as derivations on functions.
It is differential calculus.
If this is true, the two vector fields are the same, and so is their value at the identity, i.e. $[X^L,Y^L]_{I_n}=[X,Y]=XY-YX$.\\
Remark: recall that the exponential map for Lie groups gives the exponential of matrices in a linear group.
To find a particular curve with derivative $X$ one can choose $t\mapsto \exp(tX)$.
Here is another way of thinking of the linear Lie algebra of a closed linear subgroup $G$:
\[ \mathfrak{g}=\{X\in \Mat(n,\R)\ |\ \exp(tX)\in G\ \forall t\in\R\}. \]

\section{Exercises}

\begin{exercise} Show that the classical groups \begin{itemize}
    \item $\SL(n,\R)=\{A\in \GL(n,\R): \det(A)=1\}$,
    \item $\Sp(2n,\R)=\{A\in \GL(n,\R): \tran{A} J_{n,n} A=J_{n,n}\}\cap \SL(2n,\R)$,
    \item $\SO(p,q)=\{A\in \GL(n,\R): \tran{A} I_{p,q}A=I_{p,q}\}\cap \SL(p+q,\R)$,
\end{itemize}
are closed linear groups, and compute their linear Lie algebras.
\end{exercise}

\textit{Hint.}  Use the inverse function theorem.

\textit{Solution.} \begin{itemize}
    \item $\SL(n,\R)=\det^{-1}(1)$, $(d_{I_n}\det) (X)=\tr(X)$, so
    $\mathfrak{sl}(n,\R)=((d_{I_n}\det))^{-1}(0)=\{X\in\mathfrak{gl}(n,\R): \tr(X)=0\}$,
    \item $\Sp(2n,\R)=f^{-1}(J_{n,n})$, where
\[
    f \from \SL(2n,\R)\to \Mat(2n,\R), \quad A\mapsto \tran{A}J_{n,n}A.
\]
Then $(d_{I_{2n}}f)(X)= \tran{X} J_{n,n}+J_{n,n}X$, so
    $\mathfrak{sp}(2n,\R)=\{X\in \mathfrak{gl}(2n,\R): \tran{X} J_{n,n}+J_{n,n}X=0\}$,
    \item $\SO(p,q)= g^{-1}(I_{p,q})$, where
\[
    g \from \SL(p+q,\R)\to \Mat(p+q,\R), \quad A\mapsto \tran{A}I_{p,q}A.
\]
Thus
    $\mathfrak{so}(p,q)=\{X\in\mathfrak{gl}(p+q,\R): \tran{X} I_{p,q} + I_{p,q}X=0\}$.
\end{itemize}

\begin{exercise}
Compute the Killing form on $\mathfrak{gl}(n,\R)$.
Compute the Killing form on $\mathfrak{sl}(3,\R)$ and $\mathfrak{sp}(4,\R)$.
\end{exercise}

\textit{Solution.} Taking the basis $(E_{ij})$ (matrices with zeros, and only one $1$ on the $i$-th row and $j$-th column), and a lot of patience, one can compute the killing form on $\mathfrak{gl}(n,\R)$ to obtain that $B(X,Y)=2n \Tr(XY)-2\Tr(X)\Tr(Y)$.
Since $\mathfrak{sl}(n,\R)$ is an ideal of $\mathfrak{gl}(n,\R)$, its Killing form is the restriction of that of $\mathfrak{gl}(n,\R)$, i.e.\ $B(X,Y)=2n\Tr(XY)$.
For $\mathfrak{sl}(3,\R)$, decompose it into skew and symmetric part, so it becomes $\mathfrak{sl}(3,\R)=\mathfrak{so}(3,\R)\oplus \mathfrak{p}$.
Take the basis \[ \begin{pmatrix}0 & 1 & 0 \\ -1 & 0 & 0 \\ 0 & 0 & 0 \end{pmatrix},\begin{pmatrix}0 & 0 & 1 \\ 0 & 0 & 0 \\ -1 & 0 & 0 \end{pmatrix},\begin{pmatrix} 0 & 0 & 0 \\ 0 & 0 & 1\\ 0 & -1 & 0\end{pmatrix},\begin{pmatrix}0 & 1 & 0 \\ 1 & 0 & 0 \\ 0 & 0 & 0 \end{pmatrix},\begin{pmatrix}0 & 0 & 1 \\ 0 & 0 & 0 \\ 1 & 0 & 0 \end{pmatrix},\begin{pmatrix} 0 & 0 & 0 \\ 0 & 0 & 1\\ 0 & 1 & 0\end{pmatrix},\]

\[\begin{pmatrix} 1 & 0 & 0 \\ 0 & 0 & 0 \\ 0 & 0 & -1 \end{pmatrix},\begin{pmatrix} 0& 0 & 0\\0 & 1 & 0 \\ 0 & 0 & -1\end{pmatrix}, \]
and find the matrix of the Killing form being equal to
\[\begin{pmatrix} -12 & 0 & 0 & 0 & 0 & 0 & 0 & 0\\
0 & -12 & 0 & 0 & 0 & 0 & 0 & 0 \\
0 & 0 & -12 & 0 & 0 & 0 & 0 & 0 \\
0 & 0 & 0 & 12 & 0 & 0 & 0 & 0 \\
0 & 0 & 0 & 0 & 12 & 0 & 0 & 0 \\
0 & 0 & 0 & 0 & 0 & 12 & 0 & 0 \\
0 & 0 & 0 & 0 & 0 & 0 & 12 & 6 \\
0 & 0 & 0 & 0 & 0 & 0 & 6 & 12 
\end{pmatrix}, \]
which is non-degenerate of signature $(5,3)$.

For $\mathfrak{sp}(2n,\R)$, a computation as that of $\mathfrak{gl}(n,\R)$ gives us $B(X,Y)=(2n+2)\Tr(XY)$, which is non-degenerate of signature $(n^2,n^2+n)$ in $\mathfrak{sp}(2n,\R)$.

\begin{exercise}
Show that the Lie algebras above are semisimple Lie algebras, except for $\mathfrak{gl}(n,\R)$, which is only reductive (i.e. direct sum of an abelian Lie algebra and a semisimple one).
\end{exercise}

\begin{remark}
We see that the Killing form is non-degenerate for $\mathfrak{sl}(n,\R), \mathfrak{sp}(2n,\R)$ and $\mathfrak{so}(p,q)$.
\end{remark} 

Recall the Cartan criterion: a Lie algebra is semisimple if and only if its Killing form is non-degenerate.
Write $\mathfrak{gl}(n,\R)$ as $\mathfrak{sl}(n,\R) \oplus \R I_n$.

\begin{exercise}
Find a Cartan decomposition for $\mathfrak{sl}(n,\R)$, $ \mathfrak{sp}(2n,\R)$, $\mathfrak{so}(p,q)$.
Find the associated symmetric space of the Lie group $\SL(n,\R)$ and its maximal flats.
\end{exercise}

\textit{Solution.}
We can take the skew and symmetric part in each case (that are eigenspaces of the involution $X\mapsto -\tran{X}$).
This gives the following Cartan decompositions:
\begin{itemize}
    \item $\mathfrak{sl}(n,\R)=\mathfrak{so}(n,\R)\oplus\mathfrak{p}$ with $\mathfrak{p}$ being the symmetric matrices with trace $0$,
    \item $\mathfrak{sp}(2n,\R)=\mathfrak{k}\oplus \mathfrak{q}$,
    \item $\mathfrak{so}(p,q)=\mathfrak{so}(p,\R)\oplus\mathfrak{so}(q,\R)\oplus\mathfrak{r}$.
\end{itemize}

Let us detail the cases $\mathfrak{sl}(n,\R)$ and $\mathfrak{sp}(2n,\R)$. In $\mathfrak{sl}(n,\R)$, the skew part gives all the skew-symmetric matrices, i.e.\ the Lie algebra of $\SO(n,\R)$, so a maximal compact of $\SL(n,\R)$ is $\SO(n,\R)$.
The symmetric part of $\mathfrak{sl}(n,\R)$ are the symmetric matrices with trace zero.
In $\mathfrak{p}$, define $\mathfrak{a}$ as the set of diagonal matrices of $\mathfrak{p}$, i.e.\ diagonal matrices of trace zero. The space $\mathfrak{a}$ is abelian, and is maximal for this property.
Let us say that if you take a symmetric matrix with zero trace, then its exponential will give a positive definite matrix with determinant one (symmetric implies diagonalizable, and then apply exponential to give another symmetric matrix with eigenvalues $\exp(l_i)>0$, and recall that $\exp\circ\Tr=\det\circ\exp$).
So the exponential of $\mathfrak{p}$ will give us $ \mathrm{Sym}_+^1(n,\R) $ the set of positive definite symmetric matrices of determinant one.
The Cartan decomposition can be written as \[ \SL(n,\R)\cong \SO(n,\R)\times\mathrm{Sym}_+^1(n,\R), \]
which is nothing else than the polar decomposition.
The symmetric space is then $X=\SL(n,\R)/\SO(n,\R)\cong\mathrm{Sym}_+^1(n,\R)$.
A maximal flat of $X$ is then given by integrating the subspace $\mathfrak{a}\subset\mathfrak{p}$, and this gives diagonal matrices with positive entries and determinant one.

In $\mathfrak{sp}(2n,\R)$, we can decompose the matrices as \[X=\begin{pmatrix} M_1 & M_2 \\ M_3 & -\tran{M_1} \end{pmatrix},\] with $M_1\in\mathfrak{sl}(n,\R)$ and $M_2$,$M_3$ symmetric.
Then we see that the Lie subalgebra $\mathfrak{k}$ is made of matrices $\begin{pmatrix} M_1 & -M_2 \\ M_2 & M_1 \end{pmatrix}$, with $M_1$ skew- and $M_2$ symmetric. At the Lie group level, the maximal compact $K$ associated to the Lie sub-algebra $\mathfrak{k}$ can be thought of in a convenient way. We recall the definition \[ \UU(n)=\{A\in \GL(n,\C): A^*A=I_n\},\] where $A^*$ stands for $\overline{\tran{A}}$ (conjugate transpose).
Recall also that complex matrices can be seen as real matrices of twice the size by writing \[ A=\Re(A)+i\Im(A)\mapsto \begin{pmatrix} \Re(A) & -\Im(A) \\ \Im(A) & \Re(A) \end{pmatrix}. \]
Looking at $\UU(n)$ as a group of real matrices, and doing the computation $\tran{A} J_{n,n}A$ for $A\in \UU(n)$, we find $J_{n,n}$, i.e\ $\UU(n)\subset \Sp(2n,\R)$. Moreover, computing the Lie algebra of $\UU(n)$ seen in $\GL(2n,\R)$ we find exactly $\mathfrak{k}$. Then just remark that $\UU(n)$ is connected so we can say that a maximal compact of $\Sp(2n,\R)$ is $\UU(n)$. The symmetric space of $\Sp(2n,\R)$ will be studied in the following worksheet.

\begin{exercise}
Compute the restricted root spaces, the Weyl chambers (and Weyl group), the $KAK$ decomposition and the parabolic subgroups of $\SL(3,\R)$.
\end{exercise}

For $\SL(3,\R)$, recall that \[\mathfrak{a}=\left\{\begin{pmatrix}\lambda & & \\ & \mu & \\ & & \nu \end{pmatrix} : \lambda,\mu,\nu\in\R,\ \lambda+\mu+\nu=0\right\}.\]
The space $\mathfrak{a}^*$ is generated by $f_i$, for $1\leq i\leq 3$, where $f_i\left(\begin{pmatrix}\lambda_1 & & \\ & \lambda_2 & \\ & & \lambda_3 \end{pmatrix}\right)=\lambda_i$.
Let us take $A=\begin{pmatrix}\lambda_1 & & \\ & \lambda_2 & \\ & & \lambda_3 \end{pmatrix}\in\mathfrak{a}$, and $X=(x_{ij})\in\mathfrak{sl}(3,\R)=\mathfrak{g}$.
We have 
\[ [A,X]=\begin{pmatrix} 0 & (\lambda_1-\lambda_2)x_{12} & (\lambda_1-\lambda_3)x_{13} \\ (\lambda_2-\lambda_1)x_{21} & 0 & (\lambda_2-\lambda_3)x_{23} \\ (\lambda_3-\lambda_1)x_{31} &  (\lambda_3-\lambda_2)x_{32} & 0 \end{pmatrix} ,\]
so the roots are $\Sigma=\{\lambda_i-\lambda_j : i\neq j\}$ and the root spaces associated are $\mathfrak{g}_{(ij)}=\{X\in\mathfrak{g}\ |\ x_{kl}=0\ \forall (kl)\neq(ij)\}$ of dimension one. Here $\mathfrak{g}_0=\mathfrak{a}$.
An example of a Lie algebra with $\mathfrak{a}\neq \mathfrak{g}_0$ is $\mathfrak{sl}(n,\C)$, where all works like the real case (with Hermitian instead of symmetric) but $\mathfrak{g}_0=\mathfrak{a}\oplus i\mathfrak{a}$.

Let us compute the Weyl chambers. Let us take the root $\alpha=\lambda_1-\lambda_2$ and compute $\ker(\alpha)=\left\{\begin{pmatrix} d & & \\ & d & \\ & & -2d\end{pmatrix}: d\in \R\right\}$. We can represent the space $\mathfrak{a}$ as a 2-plane, namely the hyperplane $x+y+z=0$ in $\R^3$, and draw the walls we obtain.

\newpage

\thispagestyle{empty}
\chapter[Exercises - Part 2]{Examples and Exercises - Part 2\\ {\Large\textnormal{\textit{by Colin Davalo}}}}
\addtocontents{toc}{\quad\quad\quad \textit{Colin Davalo}\par}

\section{Hermitian symmetric spaces, in general}

Several equivalent definitions of Hermitian symmetric spaces exist. Here is a first definition.

\begin{definition}
A symmetric space $(X,g)$ is called of \emph{Hermitian type} if it admits a complex structure $J\colon TX\to TX$ that is invariant by the identity component of $G$ and such that $g(J\cdot,J\cdot)=g(\cdot,\cdot)$.
\end{definition}

In particular $g$ becomes a Hermitian metric, and $\omega(\cdot,\cdot)=g(\cdot,J\cdot)$ defines a non-degenerate $2$-form on $X$. 

\begin{exercise}
Any differential form on a symmetric space $X$ that is invariant under the identity component of the isometry group is closed. In particular the previous $\omega$ is a symplectic form.
\end{exercise}

\textit{Hint.}
Consider $x\in X$ and the symmetry $\sigma_x$ at $x$. Then if $\omega$ is an invariant $k$-form, then $\sigma_x^* \omega$ is also an invariant $k$-form, equal to $(-1)^k\omega$.

A characterisation of symmetric spaces of Hermitian type is the following: let $X$ whose identity component of the isometry group is $G$. Let $\mathfrak{g}=\mathfrak{k}+\mathfrak{p}$ be the Cartan decomposition associated with $X$. Then $X$ is of Hermitian type if and only if $\mathfrak{k}$ has a non-zero center.

\medskip

There exist Hermitian symmetric spaces of compact and non-compact type, but we are more interested in the Hermitian symmetric spaces of non-compact type.

\section{An interesting example}

We study one example of an Hermitian symmetric space. Through this example, we try to understand the main features of Hermitian symmetric spaces.

\medskip

Let $n\geq 1$ be an integer, and let $\Sp(2n,\R)$ be the subgroup of $\GL(2n,\R)$ of elements preserving the symplectic form $\omega$, which is the bilinear form defined by $\omega(X,Y)=\tran{X} J_{n,n} Y$ for $X,Y\in \R^{2n}$ with :
$$J_{n,n}=\begin{pmatrix} 0 & I_n \\ -I_n & 0\end{pmatrix}.$$

The goal of this exercise session is to study the properties of the symmetric space associated with $\Sp(2n,\R)$. 

\begin{exercise}
Find a Cartan involution, and a maximal abelian subalgebra $\mathfrak{a}\subset \mathfrak{p}$ for $\Sp(2n,\R)$. What is the associated restricted root system? And the Weyl group? 
\end{exercise}

\textit{Hint.} 
The Killing form is proportional to $A,B\mapsto \Tr(AB)$. The Weyl groups acts by isometry on $\mathfrak{a}$

A model for this symmetric space associated with $\Sp(2n,\R)$ is the \emph{Siegel upper half-space} $\mathcal{X}_n$, defined as the space of complex $n\times n$ matrices which can be written $S+iT$ with $S,T$ real symmetric $n\times n$ matrices, and $T$ positive definite.

\begin{remark}
When $n=1$, the Siegel space $\mathcal{X}_1$ becomes the upper half-plane:
$$\{z\in \C : \Im(z)>0\}.$$
It is a model for the symmetric space $\mathbb{H}^2$ associated with $\Sp(2,\R)=\SL(2,\R)$. 
\end{remark}

The symplectic groups acts on the Siegel space in a way that is similar to the action of $\SL(2,\R)$ on the upper half-plane. Consider for
$$g=\begin{pmatrix} A & B \\ C & D\end{pmatrix}\in \Sp(2n,\R),$$
for some $n\times n$ blocks $A,B,C,D$. Let $Z=S+iT\in \mathcal{X}_n$ for some symmetric real matrices $S,T$. We define $g\cdot Z=(AZ+B)(CZ+D)^{-1}$.

\begin{exercise} 
Assuming that this is well defined, check that it is an action. Find an equivariant diffeomorphism between $\mathcal{X}_n$ and the symmetric space $\Sp(2n,\R)/\UU(n)$.
\end{exercise}

\textit{Hint.} 
Show that the stabilizer of $iI_n\in \mathcal{X}_n$ is isomorphic to $\UU(n)$.

The action of any element  $\Sp(2n,\R)$ on $\mathcal{X}_n$ is holomorphic. Hence the symmetric space associated to $\Sp(2n,\R)$ inherits a complex structure.

\begin{exercise}
Check that the Riemannian metric actually defines an Hermitian metric for this complex structure.
\end{exercise}

A Hermitian symmetric space is of \emph{tube type} if it is biholomorphic to a domain of the form $V+i\Omega$ with $V$ a real vector space and $\Omega$ a convex cone.

\medskip

It is handy to realize the symmetric space as a bounded domain of a complex vector space. For every symmetric space of Hermitian type, a general construction called the \emph{Harish-Chandra embedding}.
For the group $\Sp(2n,\R)$, this domain is $\mathcal{D}_n$, the space of complex symmetric $n\times n$ matrices $Z$ such that $I_n- {^tZ}\overline{Z}$ is positive definite.

\begin{exercise}
Find a biholomorphism from $\mathcal{X}_n$ to $\mathcal{D}_n$, inspired from a biholomorphism from the upper half-plane to the unit disk, that can be extended continuously to the boundary of $\mathcal{X}_n$.
\end{exercise}

Recall the maximal principle: let $f$ be a continuous function from the closure $\overline{\Omega}$ of a bounded domain $\Omega$ of a complex vector space into $\C$ that is holomorphic on $\Omega$. Then:
$$\max_{x\in \overline{\Omega}}|f(x)|=\max_{x\in \partial \Omega}|f(x)|.$$

However this inequality can be improved sometimes. We say that a closed subset $F\subset \partial \Omega$ is a \emph{boundary} if for all function $f$ continuous on $\overline{\Omega}$ and holomorphic on $\Omega$ :
$$\max_{x\in \overline{\Omega}}|f(x)|=\max_{x\in F}|f(x)|.$$

The intersection of all boundaries of $\Omega$ is called the \emph{Shilov boundary} of $\Omega$. 

\begin{exercise}
Show that the Shilov boundary of $D\times D\subset \C^2$ where $D\subset \C$ is the unit disk is the set of pairs $(z_1,z_2)\in \C^2$ such that $|z_1|=|z_2|=1$.

 Show that $\Sp(2n,\R)$ acts on $\partial\mathcal{D}_n$, and that any $Z\in \partial\mathcal{D}_n$ such that $Z \overline{Z}=I_n$ is in the same orbit as $I_n$.

Show that the Shilov boundary of $\mathcal{D}_n$ is the set of symmetric complex matrices $Z$ such that $Z\overline{Z}=I_n$. (harder)
\end{exercise}

The Shilov boundary of $\mathcal{D}_n$ is invariant by the action of $\Sp(2n, \R)$. Let $Q_n$ be the stabilizer of a Lagrangian in $\R^{2n}$, i.e. of a $n$-dimensional subspace on which $\omega$ is degenerate. The space $\mathcal{L}_n$ of Lagrangians can be seen as the flag manifold $\Sp(2n,\R)/Q_n$.

\begin{exercise}
Check that $Q_n$ is a parabolic subgroup. To what set of root is it associated ? Show that the Shilov boundary of $\mathcal{D}_n$ is diffeomorphic in a $\Sp(2n,\R)$-equivariant way to $\mathcal{L}_n$.
\end{exercise}

\newpage

%%%%%%%%%%%%%%%%%%%%%%%%%%%%%%%%%%%%%%%%%%%%%%%%%%%%%%%%%%%%%%%%%%%%%%%%%%%%%%%%%%%%%%%%%%%%%%%%%%%%%%%%

\part[Known Notions of Positivity]{Known Notions of Positivity
\textnormal{
\begin{minipage}[c]{15cm}
\begin{center}
    \vspace{2cm}
    {\Large Romeo Troubat}\\
    \vspace{-4mm}
    {\large \textit{Université de Strasbourg}}\\
    \vspace{.5cm}
    {\Large Xian Dai}\\
    \vspace{-4mm}
    {\large \textit{Ruprecht-Karls-Universit\"at Heidelberg}}\\
    \vspace{.5cm}
    {\Large Jingyi Xue}\\
    \vspace{-4mm}
    {\large \textit{National University of Singapore}}\\
    \vspace{.5cm}
     {\Large Raphael Appenzeller}\\
    \vspace{-4mm}
    {\large \textit{ETH Zürich}}\\
    \vspace{.5cm}
     {\Large Francesco Fournier-Facio}\\
    \vspace{-4mm}
    {\large \textit{ETH Zürich}}\\
    \vspace{.5cm}
     {\Large Samuel Bronstein}\\
    \vspace{-4mm}
    {\large \textit{ENS Paris}}\\
    \vspace{.5cm}
     {\Large Ilia Smilga}\\
    \vspace{-4mm}
    {\large \textit{Institut des Hautes Études Scientifiques}}
\end{center}
\end{minipage}
}}\label{chap3}

\thispagestyle{empty}

\chapter[Totally Positive Matrices]{Totally Positive Matrices\\ {\Large\textnormal{\textit{by Romeo Troubat}}}}
\addtocontents{toc}{\quad\quad\quad \textit{Romeo Troubat}\par}

In the section, our aim will be to introduce a notion of positivity in $\gl(n, \R)$. First, we will discuss the link between triples of points in $S^1$ and matrices in $\SL(2, \R)$ which will give us a notion of positivity for $\SL(2, \R)$. We will then generalize this notion for matrices in $\gl(n, \R)$ using \emph{totally positive matrices}.

\section{Positivity for the special linear group}

The tangent bundle of $S^1$ is equivalent to the trivial bundle $S^1 \times \R$, thus the choice of a half line in $\R$ gives a causal structure to the space $S^1$. If we choose $\R^{>0} \subset \R$, we can for instance say that clockwise rotations in $S^1$ are positive whereas counter clockwise rotations are negative.

\begin{definition}
A triple of points $(x,y,z)$ in $S^1$ is said to be \emph{positive} if the points are pairwise distinct and if one has to meet the point $y$ when going from $x$ to $z$ following the positive rotation.
\end{definition}

The group $\SL(2, \R)$ acts on $S^1 \simeq \mathbb{RP}^1$. Up to an action by $\SL(2, \R)$, we can assume that $x = \mathbb{R} e_2$ and $z = \mathbb{R} e_1$ where $e_1 = (0,1)$ ans $e_2 = (1, 0)$. Let us consider the subgroup

\[ U := \left\{ g \in \SL(2, \R), g = \begin{pmatrix}
1 & t\\
0 & 1
\end{pmatrix} \right\}. \]

One could check that $U$ fixes $z$ and acts transitively on $S^1 \setminus \{z\}$. It can be identified with $\R$ using the application

\[ t \in \R \longmapsto \begin{pmatrix} 1 & t \\ 0 & 1 \end{pmatrix} \in U, \]

and the choice of a half-line $\R^{>0} \subset \R$ gives us a subsemigroup $U^{>0} \subset U$ defined by

\[ U^{>0} = \left\{ \begin{pmatrix} 1 & t \\ 0 & 1 \end{pmatrix}, t > 0 \right\}. \]

Any point on $S^1$ different from $z$ can be written in a unique way as $y = u_y x$. If $y = t_y e_1 + e_2$, then $u_y$ is equal to $\begin{pmatrix} 1 & t_y \\ 0 & 1 \end{pmatrix}$. It is then easy to verify that the triple $(x,y,z)$ is positive if and only if $t_y > 0$, i.e if and only if $u_y \in U^{>0}$.

\indent In the same way, let us define the group

\[ O := \left\{ g \in \SL(2, \R), g = \begin{pmatrix}
1 & 0\\
t & 1
\end{pmatrix} \right\} \]

and its subsemigroup

\[ O^{>0} = \left\{ \begin{pmatrix} 1 & 0 \\ t & 1 \end{pmatrix}, t > 0 \right\}. \]

Finally, we call $A = \left\{ \begin{pmatrix} \lambda & 0 \\ 0 & \lambda^{-1} \end{pmatrix}, \lambda \neq 0 \right\}$ the set of diagonal matrices in $\SL(2, \R)$ and $A^{\circ} = \left\{ \begin{pmatrix} \lambda & 0 \\ 0 & \lambda^{-1} \end{pmatrix}, \lambda > 0 \right\}$ the connected component of the identity in $A$.

\begin{definition}
We call $\SL(2, \R)^{>0} = O^{>0} A^{\circ} U^{>0}$ the set of \emph{totally positive matrices} in $\SL(2, \R)$.
\end{definition}

\begin{proposition}
The set of totally positive matrices is the set of matrices in $\SL(2, \R)$ whose coefficients are all positive. In particular, it is a subsemigroup of $\SL(2, \R)$.
\end{proposition}

\begin{proof}

It is easy to compute that

\[ \begin{pmatrix} 1 & 0 \\ s & 1 \end{pmatrix} \begin{pmatrix} \lambda & 0 \\ 0 & \lambda^{-1} \end{pmatrix} \begin{pmatrix} 1 & t \\ 0 & 1 \end{pmatrix} = \begin{pmatrix} \lambda & \lambda t \\ \lambda s & \lambda s t + \lambda^{-1} \end{pmatrix}. \]

By choosing the parameters $\lambda,s,t$ correctly, we can put any positive number we want in the top left, top right and bottom left spots in the matrix. The bottom right coefficient is then determined by the fact that the determinant has to be $1$. Thus, $\SL(2, \R)^{>0} = O^{>0} A^{\circ} U^{>0}$.

\end{proof}

\section{Totally positive matrices}

We will now try to generalize the previous notion of positivity to matrices in $\gl(n,\R)$.

\begin{definition}
A matrix in $\gl(n, \R)$ is said to be \emph{totally positive} if all of its minors are positive. We call $\gl(n, \R)^{>0}$ the set of totally positive matrices in $\gl(n, \R)$. We call $U$ (resp. $O$) the set of upper triangular (resp. lower triangular) matrices with $1$ on all diagonal entries and $A$ the set of diagonal matrices in $\gl(n, \R)$.
\end{definition}

\begin{proposition}
\label{prop:cauchybinet}

The set of totally positive matrices forms a subsemigroup of $\gl(n, \R)$.
\end{proposition}

\begin{proof}
Let $A$ and $B$ be two totally positive matrices and let $I$ and $J$ be two subsets of $\{1, \ldots, n\}$ of cardinal $k$. We write $A|I$ the matrix we obtain by extracting the lines in $I$ from $A$ and $A|_{I, J}$ the one we obtain by extracting the lines from $I$ and the columns from $J$. We have

\[ (AB)|_{I,J} (e_1 \wedge ... \wedge e_k) = A|_I B|_J (e_1 \wedge ... \wedge e_k).\]

The matrix $B|_J$ is an application from $\R^k$ to $\R^n$. By projecting $B|_J(e_1 \wedge \ldots \wedge e_k)$ onto $e_{i_1} \wedge \ldots \wedge e_{i_k}$, we get the determinant of $p \circ B\vert _J$, where $p$ is the projection on $\langle e_{i_1}, \ldots, e_{i_k}\rangle$. Thus, we have

\[ B|_J (e_1 \wedge \ldots \wedge e_k) = \sum_{|K| = k} \det(B|_{K, J}) e_{i_1} \wedge \ldots \wedge e_{i_k}, \]

which gives us

\[ \begin{split} (AB)|_{I,J} (e_1 \wedge \ldots \wedge e_k) &= \sum_{|K| = k} \det(B|_{K, J}) A_{I} (e_{i_1} \wedge \ldots \wedge e_{i_k}) \\
&= \left(\sum_{|K| = k} \det(A|_{I,K})\det(B|_{K, J}) \right) e_1 \wedge \ldots \wedge e_k,  
\end{split} \]

thus, $\det((AB)|_{I,J}) = \sum_{|K| = k} \det(A|_{I,K})\det(B|_{K, J})$. This is known as the \emph{Cauchy--Binet formula}. Since all of $A$ and $B$'s minors are positive, the product $AB$ is a totally positive matrix.
\end{proof}

\begin{definition}
We call $U^{>0}$ (resp. $O^{>0}$) the set of matrices in $U$ (resp. $O$) which have all positive minors, except for those which have to be equal to $0$ by virtue of being in $U$ (resp. $O$). Finally, we call $A^{\circ}$ the connected component of the identity in $A$.
\end{definition}

We give an explicit parametrisation of the sets $U^{>0}$ and $O^{>0}$. The group $U$ is generated by the elementary matrices

\[ u_i(t) = I_n + t E_{i, i+1},\quad i = 1, \ldots, n-1.\]

We call $U^{\geqslant 0}$ the group generated by the elements $u_i(t)$ for $i = 1, \ldots, n-1$ and $t > 0$. For the case $n=2$, we have shown that $U^{\geqslant 0} = U^{>0}$, but that is no longer the case for $n \geqslant 3$ since the matrices $u_i(t)$, $t >0$, are not contained in $U^{>0}$ as too many of their minors are equal to zero.

\indent To get the appropriate parametrisation, let us study the group $S_n$ of permutations on the set $\{1, \ldots, n\}$.
Each of its elements can be decomposed as a product of adjacent transpositions $(i, i+1)$ and the length of a permutation can be defined as the smallest number of adjacent transposition in one of its decomposition. For each group $S_n$, there exists a permutation $\omega_0$ of maximum length, 

\[ \omega_0 = (n, n-1)(n-1, n-2)\ldots(2, 1)(n, n-1)(n-1, n-2)\ldots(3 ,2)\ldots(n, n-1)(n-1, n-2)(n, n-1). \]

For $n=4$, we have $\omega_0 = (4, 3)(3, 2)(2, 1) (4, 3) (3, 2) (4, 3)$. Let $\omega_0 = \sigma_{i_1} \ldots \sigma_{i_k}$ be a decomposition of $\omega_0 \in S_n$ and let us define the map

\[ F_{\sigma_{i_1} \ldots \sigma_{i_k}} \from (t_1, \ldots, t_k) \in \R^k \longmapsto u_{i_1}(t_1) \ldots u_{i_k}(t_k). \]

\begin{proposition}
The map $F_{\sigma_{i_1} \ldots \sigma_{i_k}} |_{(\R^{+})^k}$ is a bijection unto $U^{>0}$. It provides a parametrisation of $U^{>0}$. There is a symmetric result for $O^{>0}$.
\end{proposition}

For $n=3$, this gives us the parametrisation

\[ u_1(a)u_2(b)u_1(c) = \begin{pmatrix} 1 & a & 0 \\ 0 & 1 & 0 \\ 0 & 0 & 1 \end{pmatrix} \begin{pmatrix} 1 & 0 & 0 \\ 0 & 1 & b \\ 0 & 0 & 1 \end{pmatrix} \begin{pmatrix} 1 & c & 0 \\ 0 & 1 & 0 \\ 0 & 0 & 1 \end{pmatrix} = \begin{pmatrix} 1 & a+c & ab \\ 0 & 1 & b \\ 0 & 0 & 1 \end{pmatrix} \]

In the case $n=2$, we defined the set of totally positive matrices using the decomposition $\SL(2, \R)^{>0} = O^{>0} A^{\circ} U^{>0}$. Let us show a similar result for $n \geqslant 3$.

\begin{proposition}
For $n \geqslant 2$, the set of totally positive matrices can be decomposed as $\gl(n, \R)^{>0} = O^{>0} A^{\circ} U^{>0}$.
\end{proposition}

\begin{proof}

Let $A = (a_{i,j})$ be a totally positive matrix. In particular, all of its coefficients are positive. We are going to reduce our matrix using operations on lines and columns. First, let us subtract $\frac{a_{i, 1}}{a_{i-1, 1}} L_{i-1}$ to $L_i$ for $i = 2, ..., n$. This gives us the matrix

\[ A = \left(I_n + \frac{a_{n, 1}}{a_{n-1, 1}} E_{n, n-1}\right)\ldots\left(I_n + \frac{a_{2, 1}}{a_{1, 1}} E_{2, 1}\right) \begin{pmatrix} 1 & a_{1,j} \\ 0 & a_{2,j} - \frac{a_{2,1}}{a_{1,1}} \\
\vdots & \vdots \\
0 & a_{n,j} - \frac{a_{n,1}}{a_{n-1,1}}\end{pmatrix} \]

Let us write $B = \begin{pmatrix}  a_{1,j} \\ a_{2,j} - \frac{a_{2,1}}{a_{1,1}} \\
\vdots \\
a_{n,j} - \frac{a_{n,1}}{a_{n-1,1}}\end{pmatrix}$. We admit that the matrix $A$ is totally positive if and only if $B$ is totally positive \cite{Whitney}. By repeating the same process, we can reach the decomposition

\[ A = \left(I_n + \frac{a_{n, 1}}{a_{n-1, 1}} E_{n, n-1}\right)\ldots\left(I_n + \frac{a_{n-1, n-1}}{a_{n-2, n-1}}E_{n-1, n-2}\right)\left(I_n + \frac{a_{n, n}}{a_{n-1, n}}E_{n, n-1}\right) C, \]

where $C$ is an upper triangular matrix. Using the parametrisation we have shown, we can see that the product of transvections is equal to a matrix in $O^{>0}$. We now only have the repeat the same operations on the columns to obtain the decomposition $\gl(n, \R)^{>0} = O^{>0} A^{\circ} U^{>0}$.
\end{proof}

\newpage

\thispagestyle{empty}

\chapter[Lusztig's Total Positivity]{Lusztig's Total Positivity\\ {\Large\textnormal{\textit{by Xian Dai}}}}
\label{section:LusztigTotalPositivity}
\addtocontents{toc}{\quad\quad\quad \textit{Xian Dai}\par}

Lusztig's total positivity generalizes total positivity for $n\times n$ matrices in $\mathrm{GL}(n,\mathbb{R})$ to split real Lie groups. Our goal in this section is to explain this general notion of total positivity. 

\section{Some motivation for total positivity}

We recall a $n\times n$ matrix is said to be \emph{(totally) positive} if all of its minors are positive (i.e. in $\mathbb{R}^{>0}$). The set of all totally
positive  $n\times n$ matrices form a subset $\mathrm{GL}(n,\mathbb{R})^{>0}\subset \mathrm{GL}(n,\mathbb{R})$.

Let $V\cong \mathbb{R}^n$. Another characterization of total positivity in $\mathrm{GL}(n,\mathbb{R})$ due to I.J. Schoenberg \cite{Schoenberg} is as follows.
\begin{proposition} \label{equivPositivity}
A matrix $g\in\mathrm{GL}(V) $ is positive if and only if the coefficients of $k$-th exterior power $\Lambda^k g\colon \Lambda^k V \to \Lambda^k V$ are positive for $k=1,2, \ldots, n$.
\end{proposition}

\begin{proof}
Suppose a totally ordered basis of $V$ is $e_1, e_2, \ldots, e_n$, then for any $k\in[1,n]$, the $k$-th exterior power $\Lambda^k V$ has a basis $(e_{r_1}\wedge e_{r_2}\cdots \wedge e_{r_k})$ indexed by the sequences $r_1< r_2< \cdots < r_k$ in $[1,n]$ and $\Lambda^k g\colon \Lambda^k V \to \Lambda^k V$ are positive for $k=1,2, \ldots, n$ is given by

$$\Lambda^k g (e_{r_1}\wedge e_{r_2}\cdots \wedge e_{r_k})=g e_{r_1}\wedge g e_{r_2}\cdots \wedge A e_{r_k}.  $$

The conclusion follows once we notice that coefficients of above matrices are exactly minors of $g$.

\end{proof}

One source of the idea of total positivity  appear in the work of Gantmacher and Krein in 1935. The following Theorem can be viewed as one of the initial motivation of total positivity.

\begin{theorem}[Gantmacher-Krein]
If $g\in\mathrm{GL}(n,\mathbb{R}) $ is totally positive, then $g$ has distinct real positive eigenvalues.
\end{theorem}

\begin{proof}
Suppose $g\in \mathrm{GL}(n,\mathbb{R})^{>0}$. We denote $c_1,\ldots, c_n$ to be eigenvalues of $g$ arranged so that $|c_1|\geq|c_2|\geq\cdots \geq|c_n|$. Then the eigenvalues of $\Lambda^k g$ are $c_{r_1}c_{r_2} \cdots c_{r_k}$ for $r_1< r_2 \cdots < r_k$ and $r_i\in [1,n]$.
For instance, the first two eigenvalues of $\Lambda^k g$ in decreasing order of absolute values are $c_1c_2\cdots c_k$ and $c_1c_2\cdots c_{k-1}c_{k+1}$. Since $\Lambda^k g$ has positive entries by Proposition \ref{equivPositivity}. By Perron-Frobenius Theorem, we must have 
$$c_1c_2\cdots c_{k}>0$$

Taking $k=1$ yields $c_1>0$. By induction on $k$, we obtain $c_1> c_2 >\ldots >c_n>0$.
\end{proof}

\section{Lusztig's total positivity for matrices}

We start from a presentation of Lusztig's total positivity in a simple case which is $\mathrm{GL}(n,\mathbb{R})$. Totally positive matrices in $\mathrm{GL}(n,\mathbb{R})$ satisfy a decomposition theorem and form in fact a semigroup.

Let $U$ be the group of upper triangular matrices with ones on the diagonal, $O$ be the group of lower triangular matrices with ones on the diagonal, and $A$ the group of diagonal matrices. The group $U$ is generated by elementary matrices 
$$u_i(t)= I_n+ t E_{i,i+1},\quad i=1, \ldots , n-1,$$
where $I_n$ denotes the identity matrix and $E_{i,i+1}$ the matrix with the single entry $1$ in the $i$-th row and $(i+1)$-th column.

Let furthermore $U^{>0}\subset U$ and $O^{>0}\subset O$ be the subsets of totally positive unipotent matrices, i.e. those matrices of $U$, where all minors are positive, except those which have to be zero by the condition of being an element of $U$, similarly for $O$.  We can parametrize the set $U^{>0}$ using the symmetric group $S_n$ on $n$ letters. We denote by $\sigma_i, i=1,\ldots, n-1,$ the tranposition $(i,i+1)$ and by $\omega_0$ the longest element of the symmetric group, which send $(1,2,\ldots, n)$ to $(n,\ldots, 2,1)$. Let $k=\frac{n(n-1)}{2}$. For every way to write $\omega_0= \sigma_{i_1}\sigma_{i_2}\cdots \sigma_{i_k}$, we can define a map
\begin{align*}
F_{\sigma_{i_1}\sigma_{i_2}\cdots\sigma_{i_k}}\colon\mathbb{R}^k&\to U\\
(t_1,\ldots,t_k) &\mapsto u_{i_1}(t_1)u_{i_2}(t_2)\cdots u_{i_k}(t_k).
\end{align*}

An element is in $U^{>0}$ if and only if it is of the form $u_{i_1}(t_1)u_{i_2}(t_2)\cdots u_{i_k}(t_k)$ with $t_i\in \mathbb{R}_+$. Therefore the map $F_{\sigma_{i_1}\sigma_{i_2}\cdots\sigma_{i_k}}|_{(\mathbb{R}^{>0})^k}$ is a bijection onto $U^{>0}$ and provide a parametrization of $U^{>0}$ by $\mathbb{R}_+^k$. The same works for $O^{>0}$.  

The following decomposition theorem is due to A.~Whitney \cite{Whitney}: 
\[\mathrm{GL}(n,\mathbb{R})^{>0}=O^{>0} A^0 U^{>0},\]
where $A^0$ is the connected component of the identity in $A$, i.e.\ the diagonal matrices all of whose entries are positive.

\section{Lusztig's total positivity for split real Lie groups}

In this subsection,  we want to explain the generalized notion of total positivity of arbitrary split real reductive Lie group introduced by Lusztig \cite{Lusztig94}. Moreover, the explicit parametrization of $U^{>0}$ (resp. $O^{>0}$) can also be generalized to arbitrary split real reductive Lie group.

Let $G$ be a split semisimple algebraic group over $\mathbb{R}$. Let $\mathfrak{g}$ be the Lie algebra of $G$.
To describe total positivity in $G$, we start from some  discussion of some important ingredients involved in defining positivity in this general setting.

\subsection{Root space decomposition and Chevalley generators}

Fix a Cartan subalgebra $\mathfrak{h}$ of $\mathfrak{g}$. We denote by $\Sigma$ the set of all roots and $\Sigma^{+}$ a choice of positive roots (resp. $\Sigma^{-}$ negative roots) and $\Delta \subset \Sigma^{+}$ the set of simple roots.  The Lie algebra $\mathfrak{g}$ admits the root space decomposition,
$$\mathfrak{g} = \mathfrak{h}  \oplus  \bigoplus_{\alpha\in \Sigma} \mathfrak{g}_{\alpha}.$$

Denote $h_{\alpha}$ as the coroot of $\alpha \in \Sigma$, we say a family $(X_{\alpha})_{\alpha \in \Sigma}$ is a \emph{Chevalley basis} for $(\mathfrak{g}, \mathfrak{h})$ (see \cite[Chapter VIII, Section 2.4, Definition 3]{nicolas1975lie} or  \cite[Chapter VII, Section 25]{HumphreysRepresentation}) if 
\begin{enumerate}
\item
$X_{\alpha}  \in \mathfrak{g}_{\alpha} \text{ for all } \alpha \in \Sigma$.

\item $[X_{\alpha}, X_{-\alpha}]= -h_{\alpha} \text{ for all } \alpha \in \Sigma $.
\item the linear map from $\mathfrak{g}$ to $\mathfrak{g}$ which is equal to $-1$ on the Cartan subalgebra $\mathfrak{a}$ and takes $X_{\alpha}$ to $X_{-\alpha}$ for all $\alpha\in \Sigma$ is an automorphism of $\mathfrak{g}$.
\end{enumerate}

Take $e_{\alpha}=X_{\alpha}$ and $f_{\alpha}=X_{-\alpha} $ for $\alpha \in \Delta$ so that $(e_{\alpha},f_{\alpha}, h_{\alpha})_{\alpha \in \Delta}$ form a \emph{Chevalley generators} of $\mathfrak{g}$. We define 
$$x_{\alpha}(t)= \textnormal{exp}(te_{\alpha}) \textnormal{ , } y_{\alpha}(t)= \textnormal{exp}(tf_{\alpha})  $$

In analogy to the last subsection, the elements $e_{\alpha}$ (resp. $f_{\alpha}$ ) play the roles of $E_{i,i+1}$ (resp. $E_{i+1,i}$). Also $x_{\alpha}(t)$ (resp. $y_{\alpha}(t)$) are similar to $u_{i}(t)$(resp. $v_{i}(t)$) in $\mathrm{GL}(n,\mathbb{R})$ case.

\subsection{Longest element in Weyl group}

We denote by $W$ the Weyl group of $G$ generated by $s_\alpha$ for $\alpha\in \Delta$. $W$ together with $(s_\alpha)_{\alpha\in \Delta}$ is a Coxeter group.

Let $l\colon W\to \mathbb{N}$ be the standard length function. For $\omega\in W$, we let $\Delta_\omega$ be the set of all sequences $(\alpha_{1}, \alpha_{2}, \ldots, \alpha_{p})$ in $\Delta$ so that $p=l(\omega)$ and 
$$s_{\alpha_{1}}s_{\alpha_{2}} \cdots s_{\alpha_{p}}=\omega.$$ 

We let $\omega_0$ be the element such that $p_0=l(\omega_0)$ is maximum. To make a link with the theory in the last subsection, one notices that the symmetric group $S_n$ is the Weyl group for the case $\mathrm{GL}(n,\mathbb{R})$.

\subsection{Parametrization of $U^{\pm}$ and total positivity}

With a fixed $\omega_0$ so that $p_0=l(\omega_0)$ reaches maximum, we can define maps 

\begin{align*}
\Phi^{\pm}_{\omega_0}\colon \mathbb{R}^p_{\geq 0} &\to U^{\pm}\\
\Phi^+_{\omega_0}(a_1,a_2,\ldots, a_p):=& x_{\alpha_1}(a_1)x_{\alpha_2}(a_2)\ldots x_{\alpha_p}(a_p),\\
\Phi^{-}_{\omega_0}(a_1,a_2,\ldots, a_p):=& y_{\alpha_1}(a_1)y_{\alpha_2}(a_2)\ldots y_{\alpha_p}(a_p).
\end{align*}

\noindent Here $U^{\pm}$ are the unipotent radicals of the Borel subgroups (minimal parabolic subgroups) $B^{\pm}$ corresponding to $\Sigma^{\pm}$ respectively. It turns out that the definition of $\Phi^{\pm}_{\omega_0}$ do not depend on a particular choice of $(\alpha_{1}, \alpha_{2}, \ldots, \alpha_{p})$.

We then define

\[U_{{\omega_0},\geq 0}^{\pm}:=\Phi^{\pm}_{\omega_0}(\mathbb{R}^p_{\geq 0}) \subset U^{\pm},\] and
\[U_{{\omega_0},> 0}^{\pm}:=\Phi^{\pm}_{\omega_0}(\mathbb{R}^p_{>0}) \subset U^{\pm}.
\]
 
\noindent The sets $U_{\geq 0}^{\pm}=U_{{\omega_0},\geq 0}^{\pm}$ and $U_{> 0}^{\pm}=U_{{\omega_0},> 0}^{\pm}$ are semigroups of $G$. 

\begin{proposition}
The maps $\Phi^{\pm}_{\omega_0}$ are homeomorphisms between $(\mathbb{R}^{>0})^p$ and $U_{> 0}^{\pm}$. Therefore  $U_{> 0}^{\pm}$ are cells.
\end{proposition}

Now consider 
$$G_{\geq 0}= U_{\geq 0}^{+} T_{>0}U_{\geq 0}^{-}=U_{\geq 0}^{-} T_{>0}U_{\geq 0}^{+},$$
$$G_{> 0}= U_{> 0}^{+} T_{>0}U_{> 0}^{-}=U_{> 0}^{-} T_{>0}U_{> 0}^{+}, $$
 
\noindent where $T$ is the subgroup of $G$ with Lie algebra $\mathfrak{a}$ (sometimes called the maximal torus of $G$). We let $T^0$ be the connected component of $T$ containing the identity. $G_{\geq 0}$ and $G_{> 0}$ are semigroups in $G$. More precisely, $G_{\geq 0}$ is a semigroup with the identity element $e$ of $G$ but $G_{> 0}$ is a semigroup without the identity element $e$.

As one could easily guess, an element $g$ in $G$ is called \emph{(totally) positive} if $g\in G_{> 0}$. We have $ G_{> 0}$ is also a cell because of the following,

\begin{theorem}

The map $U_{> 0}^{+}\times T_{>0}\times U_{> 0}^{-} \to G$, given by multiplication in $G$, is a homeomorphism onto the subset $ G_{> 0} \subset G$. 
\end{theorem}

This can be viewed as a generalisation of the Whitney decomposition mentioned in the last subsection.

\newpage

\thispagestyle{empty}

\chapter[Positivity of Lie Groups of Hermitian Type]{Positivity of Lie Groups of Hermitian Type\\ {\Large\textnormal{\textit{by Jingyi Xue}}}}
\addtocontents{toc}{\quad\quad\quad \textit{Jingyi Xue}\par}

\section{Examples of Hermitian type Lie groups}

\subsection{$\Sp(2n,\R)$}

Again we first look at the symplectic group $\Sp(2n,\R)$, which is the subgroup of $\GL(2n,\R)$ of elements preserving the symplectic form $\omega$, which is the bilinear form defined by $\omega(X,Y) = \tran{X} J_{n,n} Y$ for $X,Y\in \R^{2n}$ with :
$$J_{n,n}=\begin{pmatrix} 0 & I_n \\ -I_n & 0\end{pmatrix}.$$

Recall from the previous sections, $\UU(n)$ embeds as a maximal compact subgroup. Results of this example come from Bump's book \cite[Chap.\ 28]{book:1415659}.

\begin{proposition} If $Z=X+iY \in \mathcal{X}_{n}$ (the Siegel upper half-space) and $g=\left(\begin{array}{cc}A & B \\ C & D\end{array}\right) \in$ $\operatorname{Sp}(2 n, \mathbb{R})$, then $C Z+D$ is invertible. Define
$$
g(Z)=(A Z+B)(C Z+D)^{-1}
$$
Then $g(Z) \in \mathcal{X}_{n}$, and this defines an action of $\operatorname{Sp}(2 n, \mathbb{R})$ on $\mathcal{X}_{n}$. The action is transitive, and the stabilizer of $i I_{n} \in \mathcal{X}_{n}$ is $\UU(n)$. If $W$ is the imaginary part of $g(Z)$ then
$$
W=\left(\tran{\bar{Z}} C+\tran{D}\right)^{-1} Y(C Z+D)^{-1}.
$$
\end{proposition}

\begin{proof}
	Using the conditions for $g$ being in $\operatorname{Sp}(2 n, \mathbb{R})$, one easily checks that
	$$
	\frac{1}{2 i}\left(\left(\tran{\bar{Z}} C+\tran{D}\right)(A Z+B)-\left(\tran{\bar{Z}}A+\tran{B}\right)(C Z+D)\right)=\frac{1}{2 i}(Z-\bar{Z})=Y,
	$$
	From this it follows that $C Z+D$ is invertible since if it had a nonzero nullvector $v$, then we would have $\tran{\bar{v}} Y v=0$, which is impossible since $Y\gg 0$. To check that $g(Z)$ is symmetric, $g(Z)=\tran{g(Z)}$ is equivalent to
	$$
	(A Z+B)\left(\tran{Z} C+\tran{D}\right)=\left(\tran{Z} A+ \tran{B}\right)(C Z+D)
	$$
	which is easily confirmed.
	Next the imaginary part $W$ of $g(Z)$ is positive definite since
	$$
	W=\frac{1}{2 i}(g(Z)-\overline{g(Z)})=\frac{1}{2 i}\left((A Z+B)(C Z+D)^{-1}-\left(\tran{\bar{Z}} C+\tran{D}\right)^{-1}\left(\tran{\bar{Z}} A+\tran{B}\right)\right)
	$$
	Simplifying this gives the desired expression; and $W$ is Hermitian (real) and that $W\gg 0$. It is easy to check that $g\left(g^{\prime}(Z)\right)=\left(g g^{\prime}\right)(Z)$. To show that this action is transitive, note that if $Z=X+i Y \in \mathcal{X}_{n}$, then
	$$
	\left(\begin{array}{cc}
		I_n & -X \\
		0 & I_n
	\end{array}\right) \in \operatorname{Sp}(2 n, \mathbb{R}),
	$$
	and this matrix takes $Z$ to $i Y$. Now if $h \in \operatorname{GL}(n, \mathbb{R})$, then $\begin{pmatrix} h & 0\\ 0 & \tran{h}^{-1}\end{pmatrix}$ takes $i Y$ to $i Y^{\prime}$, where $Y^{\prime}=h Y \tran{h}$. Since $Y\gg 0$, we may choose $h$ so that $Y^{\prime}=I_n$. This shows that any element in $\mathcal{X}_{n}$ may be moved to $i I_{n}$, and the action is transitive. To check that $\mathrm{U}(n)$ is the stabilizer of $i I_{n}$, let
	$$
	A\left(i I_{n}\right)+B=\left(C\left(i I_{n}\right)+D\right)\left(i I_{n}\right)
	$$
	which implies $A=D, B=-C$.
\end{proof}

	Therefore $\operatorname{Sp}(2 n, \mathbb{R}) / \mathrm{U}(n) \cong \mathcal{X}_{n}$ is a tube type Hermitian symmetric space since $\mathrm{U}(n)$ has center $S^1$. 
	
	Just like the classical Cayley transform $c(z)=\frac{z-i}{z+i}$ which maps $\mathbb{H}^2$ to the unit disc $\mathbb{D}^2$, the generalized Cayley transform is applicable to Hermitian symmetric spaces. It was shown by Cartan and Harish-Chandra that any non-compact type Hermitian symmetric space is biholomorphic to a bounded domain in a complex vector space. Piatetski-Shapiro \cite{book:481039} gave unbounded realizations. Kor\'anyi and Wolf \cite{wolf1965generalized,koranyi1965realization} gave a completely general theory relating bounded symmetric domains to unbounded ones by means of the Cayley transform.

	Consider now the Cayley transform $c\in \operatorname{Sp}(2n)=\operatorname{Sp}(2n,\C)\cap \operatorname{U}(2n)$ for $\operatorname{Sp}(2n, \R)/\operatorname{U}(n)$:
	$$
	c=\frac{1}{\sqrt{2 i}}\left(\begin{array}{cc}
		I_{n} & -i I_{n} \\
		I_{n} & i I_{n}
	\end{array}\right), \quad c^{-1}=\frac{1}{\sqrt{2 i}}\left(\begin{array}{cc}
		i I_{n} & i I_{n} \\
		-I_{n} & I_{n}
	\end{array}\right)
	$$
	
	To embed $\mathcal{X}_n$ in its compact dual $G_c/K = \operatorname{Sp}(2n)/\operatorname{U}(n)$, the first step is to interpret $G_c/K$ as an analog of the Riemann sphere, a space on which the actions of both groups
	$G$ and $G_c$ may be realized as linear fractional transformations.
	
	Define the \emph{Siegel parabolic subgroup} $H^{\C}W^{\C}$ (see below) of $G^{\C}=\operatorname{Sp}(2n,\C)$:
	$$
	P=\left\{\left(\begin{array}{ll}
		h & 0\\
		0 & \tran{h}^{-1}
	\end{array}\right)\left(\begin{array}{cc}
		I & X \\
		0 & I
	\end{array}\right) : h \in \operatorname{GL}(n, \mathbb{C}), X \in \Mat(n,\mathbb{C}), X=\tran{X}\right\}. 
	$$
	
	Followed from Iwasawa decomposition, it is easy to verify the following:
	
\begin{proposition}\label{incl.}
	$$
	P\Sp(2 n)=\operatorname{Sp}(2 n, \mathbb{C}),\ \  P \cap \operatorname{Sp}(2 n)=cK c^{-1}=\left\{\left(\begin{array}{ll}
		g & 0\\
		0 & \tran{g}^{-1}
	\end{array}\right) : g \in \mathrm{U}(n)\right\}.
	$$
\end{proposition}
	
	Define $\mathfrak{R}_{n}=G_{\mathbb{C}} / P$ with dense open subset
	$$
	\mathfrak{R}_{n}^{\circ}= \left\{gP: g=\left(\begin{array}{cc}A & B \\ C & D\end{array}\right) \in \operatorname{Sp}(2n,\C), \operatorname{det}C\neq 0\right\}\subset \mathfrak{R}_{n}.
	$$
	
	Here notice that if $g$ is in $\operatorname{Sp}(2n,\C)$ of block form with $C$ invertible, then $AC^{-1}=\tran{C}^{-1}\cdot \tran{C}A\cdot C^{-1}$ is symmetric since $\tran{C}A$ is symmetric by definition. $g, g^\prime$ define same coset if and only if there exists some $h\in \operatorname{GL}(n,
	\C)$ such that $A^\prime=Ah$, $C^\prime=Ch$ if and only if $A^\prime(C^\prime)^{-1}=AC^{-1}$. From this observation, we may define a bijection $\sigma$ from $\operatorname{Sym}(n,\C)$ to $\mathfrak{R}_{n}^{\circ}$ by $\sigma(Z)=\left(\begin{array}{cc}Z & -I \\ I & 0\end{array}\right)P$, which can be rewritten as $\left(\begin{array}{cc}A & B \\ C & D\end{array}\right)P$ if and only if $AC^{-1}=Z$.
	
	With the above notations, we have
	$$
	g(\sigma(Z))=\left(\begin{array}{cc}
		A & B \\
		C & D
	\end{array}\right)\left(\begin{array}{ll}
		Z & -I \\
		I
	\end{array}\right) P=\left(\begin{array}{ll}
		A Z+B & -A \\
		C Z+D & -C
	\end{array}\right) P = \sigma((AZ+B)(CZ+D)^{-1})
	$$
	if $CZ+D$ is invertible.
	
	Therefore we may identify $\operatorname{Sym}(n,\C)$ with $\mathfrak{R}_{n}^{\circ}$, and the action of $\operatorname{Sp}(2n,\C)$ is by linear fractional transformations.
	
	We may also identify $\mathfrak{R}_{n}$ with $G_c/K$ by means of the compositions of bijections
	$$
	G_c / K \xrightarrow{\text{conjugation}} G_c / c K c^{-1} \xrightarrow{\text{Proposition }\ref{incl.}} G^{\mathbb{C}} / P=\mathfrak{R}_{n}
	$$
	
	Now we may apply $c$ to embed $\mathcal{X}_n$ into its compact dual $\mathfrak{R}_n$:
	
	\begin{proposition} \label{gp action}
	$$
	c(\mathcal{X}_{n})=\mathfrak{D}_{n}:=\left\{W \in \mathfrak{R}_{n}^{\circ} : I-\overline{W} W>0\right\}=\left\{W \in \Mat(n,\C) : W=\tran{W}, I-W^*W>0\right\} .
	$$
	The group $c \operatorname{Sp}(2 n, \mathbb{R}) c^{-1}$, acting on $\mathfrak{D}_{n}$ by linear fractional transformations, consists of all symplectic matrices of the form
	$$
	\begin{pmatrix} A & B \\ \overline B & \overline A \end{pmatrix}.
	$$
	\end{proposition}
	
	\begin{proof}
		This is a direct computation:
		$$
		\begin{aligned}
		\operatorname{Im}(c^{-1}(W))&=\operatorname{Im}(-i(W-I)(W+I)^{-1})\\
		&=-\frac{1}{2}\left((W-I)(W+I)^{-1}+(\overline{W}-I)(\overline{W}+I)^{-1}\right)\\
		&=-\frac{1}{2}\left((W-I)(W+I)^{-1}+(\overline{W}+I)^{-1}(\overline{W}-I)\right).
		\end{aligned}
		$$
		This is positive definite if and only if $(\overline{W}+I)\operatorname{Im}(c^{-1}(W))(W+I)>0$;
		where the latter is
		$$
		-\frac{1}{2}((\overline{W}+I)(W-I)+(\overline{W}-I)(W+I))=I-\overline{W} W.
		$$
		For $g\in \operatorname{Sp}(2n,\C)$, $c^{-1} g c=\overline{c^{-1} g c}$ gives us the desired form of elements in $c \operatorname{Sp}(2 n, \mathbb{R}) c^{-1}$.
	\end{proof}
	
	\begin{proposition} \label{boundary}
	\begin{enumerate}
	    \item  The closure of $\mathfrak{D}_{n}$ is contained within $\mathfrak{R}_{n}^{\circ}$. The boundary of $\mathfrak{D}_{n}$ consists of all complex symmetric matrices $W$ such that $I-\overline{W} W$ is positive semidefinite but such that $\operatorname{det}(I-\overline{W} W)=0$.
		\item If $W$ and $W^{\prime}$ are points of the closure of $\mathfrak{D}_{n}$ in $\mathfrak{R}_{n}^{\circ}$ that are congruent modulo $c G c^{-1}$, then the ranks of $I-\overline{W} W$ and $I-\overline{W^{\prime}} W^{\prime}$ are equal.
		\item Let $W$ be in the closure of $\mathfrak{D}_{n}$, and let $r$ be the rank of $I-\overline{W} W$. Then there exists $g \in c G c^{-1}$ such that $g(W)$ has the form
		$$
		\left(\begin{array}{cc}
			W_{r} & 0 \\
			0 & I_{n-r}
		\end{array}\right), \quad W_{r} \in \mathfrak{D}_{r} .
		$$
			\end{enumerate}
	\end{proposition}
	
	\begin{proof}
		The diagonal entries in $\overline{W} W$ are the squares of the lengths of the rows of the symmetric matrix $W$. If $I-\overline{W} W$ is positive definite, these must be less than 1. So $\mathfrak{D}_{n}$ is a bounded domain within the set $\mathfrak{R}_{n}^{\circ}$. The rest of (1) is clear.
		
		For (2), if $g \in c G c^{-1}$, see \text{Proposition }\ref{gp action} of the form, $W^\prime=g(W)$,
		$$
		I-\overline{W}^{\prime} W^\prime=I-\left(\overline{W} {}^{t}B+{ }^{t} A\right)^{-1}\left(\bar{W} {}^{t}\overline{A}+{}^{t}\overline{B}\right)(A W+B)(\overline{B} W+\overline{A})^{-1} .
		$$
		Now the rank of $I-\overline{W}^{\prime} W^{\prime}$ is the same as the rank of
		$$
		\begin{aligned}
			\left(\overline{W} {}^{t}B+{ }^{t} A\right)\left(I-\bar{W}^{\prime} W^{\prime}\right)(\overline{B} W+\overline{A})&= 
			(\overline{W} {}^{t}B+{ }^{t} A)(\overline{B} W+\overline{A})-(\overline{W} {}^{t}\overline{A}+{}^{t} \overline{B})(A W+B)\\
			&=I-\overline{W}W.
		\end{aligned}
		$$
		
		For (3), observe that up to an action by $\left(\begin{array}{cc}
			u & 0\\
			0 & \bar{u}
		\end{array}\right)$ for some $u\in \mathrm{U}(n)$, we may assume that $W$ does not have $1$ as an eigenvalue. Now $c^{-1}(W)= i(I+W)(I-W)^{-1}=X+iY\in \overline{\mathcal{X}_n}$. Since $Y\geq 0$, there exists some $k\in \mathrm{O}(n)$ such that $D=kYk^{-1}$ is diagonal with nonnegative eigenvalues. Consider $\gamma=\left(\begin{array}{cc}
		k & 0\\
		0 & k
	\end{array}\right)\left(\begin{array}{cc}
	I &-X \\
	0 & I
	\end{array}\right) \in \operatorname{Sp}(2n,\R)$, $\gamma(X+iY)=iD$, then $c\gamma c^{-1}(W)=W^\prime= (D-I)(D+I)^{-1}$. The diagonal entries of $I-\overline{W^\prime}W^\prime$ are $0$ if and only if the corresponding entries of $D$ are $0$. So after exchanging order of orthogonal basis of $k$, we will write $W^\prime=\diag(w_1,\dots, w_r, -1,\dots,-1)$, where $-1< w_i<1$. Then $\left(\begin{array}{cc}
	iI_n & 0 \\
	0 & -iI_n
	\end{array}\right)\cdot W^\prime$ will be the desired special form.
	\end{proof}
	
	Now let us fix $r<n$ and consider
	$$
	\mathfrak{B}_{r}=\left\{\left(\begin{array}{cc}
		W_{r} & 0\\
		0 & I_{n-r}
	\end{array}\right) : W_{r} \in \mathfrak{D}_{r}\right\}
	$$
	By Proposition \ref{gp action}, the subgroup of $c G c^{-1}$ of the form
	$$
	\left(\begin{array}{cccc}
		A_{r} & 0 & B_{r} & 0 \\
		0 & I_{n-r} & 0 & 0 \\
		\overline{B_{r}} & 0 & \overline{A_{r}} & 0 \\
		0 & 0 & 0 & I_{n-r}
	\end{array}\right)
	$$
	is isomorphic to $\operatorname{Sp}(2 r, \mathbb{R})$, and $\mathfrak{B}_{r}$ is homogeneous with respect to this subgroup. Thus, $\mathfrak{B}_{r}$ is a copy of the lower-dimensional Siegel space $\mathfrak{D}_{r}$ embedded into the boundary of $\mathfrak{D}_{n}$.

	\begin{theorem}
		The Shilov boundary of $\mathfrak{D}_{n}$ is $\operatorname{Sym}(n,\C)\cap\mathrm{U}(n)$.
	\end{theorem}
	
	\begin{proof}
		Let $f\in \mathcal{O}(\mathfrak{D}_{n})\cap\mathcal{C}(\overline{\mathfrak{D}_{n}})$, we show that $|f|$ must take its maximum on $\operatorname{Sym}(n,\C)\cap\mathrm{U}(n)$, which is the unique closed $cGc^{-1}$-orbit in $\overline{\mathfrak{D}_{n}}$ (union of all orbits of $\mathfrak{B}_0$). This is sufficient because $cGc^{-1}$ acts transitively on $\operatorname{Sym}(n,\C)\cap\mathrm{U}(n)$ (Proposition \ref{boundary}), so the set $\operatorname{Sym}(n,\C)\cap\mathrm{U}(n)$ cannot be replaced by any strictly smaller subspace with the same maximizing property.
		
		Suppose $x \in \overline{\mathfrak{D}_{n}}$ maximizes $|f|$. By Proposition \ref{boundary}, $\exists g\in cGc^{-1}$, s.t. $g\cdot x\in \mathfrak{B}_r$ for some $r<n$. If $r>0$, then noting that $f|_{g^{-1}\cdot \mathfrak{B}_r}\in \mathcal{O}(g^{-1}\cdot \mathfrak{B}_r)$, the maximum modulus principle implies that $f$ is constant on $g^{-1}\cdot \mathfrak{B}_r$ and hence $|f|$ takes the same maximum value on $\partial(g^{-1}\cdot \mathfrak{B}_r)$, which intersects $\operatorname{Sym}(n,\C)\cap\mathrm{U}(n)$.
	\end{proof}
	
	Notice that the action of $\UU(n)$ on Lagrangian Grassmannians $\mathcal{L}$ ($\UU(n)$ embedded in $\operatorname{Sp}(2n,\R)$) is transitive; the stabilizer of a point is readily verified to be $\mathrm{O}(n)$ and we get $\mathcal{L} \simeq \mathrm{U}(n) / \mathrm{O}(n)$. Now consider the action of $\UU(n)$ on $\operatorname{Sym}(n,\C)\cap\mathrm{U}(n)$ by conjugation, the action is still transitive and the stabilizer of $I_n$ is $\mathrm{O}(n)$, so $\operatorname{Sym}(n,\C)\cap\mathrm{U}(n) \cong \mathrm{U}(n) / \mathrm{O}(n)$.
	
	\emph{Therefore we can identify $\mathcal{L}$ with $\operatorname{Sym}(n,\C)\cap\mathrm{U}(n)$.}

\subsection{$\mathrm{SU}(n, n)$} 

This is quite similar to $\Sp(2n,\R)$. Explicitly,
$$
\mathrm{SU}(n, n)=\left\{g \in \mathrm{SL}(2 n, \mathbb{C}) :  g^* I_{n,n} g= I_{n,n} \right\}, \quad \text { where } \quad I_{n,n}=\left(\begin{array}{cc}
	I_{n} &  0\\
	0 & -I_{n}
\end{array}\right).
$$

Consider domain $\mathcal{T}:= \left\{Z \in \Mat(n,\mathbb{C}): Z^* Z<I_{n}\right\}$, then one can check that the linear fractional transformation $g(Z)=(A Z+B)(C Z+D)^{-1}$ for $g=\left(\begin{array}{cc}A & B \\ C & D\end{array}\right) \in \operatorname{SU}(n,n)$ and $Z=X+iY\in \mathcal{T}$ is well-defined, which determines a transitive action.

The stabilizer of $0$ is
$$
\begin{aligned}
	\operatorname{Stab}_{\operatorname{SU}(n,n)}(0)&=\left\{\left(\begin{array}{cc}
		A & 0 \\
		0 & D
	\end{array}\right) \in \mathrm{SU}(n, n)\right\}\\
	&=\left\{\left(\begin{array}{cc}
		A & 0 \\
		0 & D
	\end{array}\right): \begin{array}{l}
		A^* A=D^* D=I_{n} \\
		\operatorname{det}(A) \operatorname{det}(D)=1
	\end{array}\right\}=\mathrm{S}(\mathrm{U}(n) \times \mathrm{U}(n))
\end{aligned}
$$
which is maximal compact.

The center $Z\left(\mathrm{S}(\mathrm{U}(n) \times \mathrm{U}(n))\right)=\left\{\left(\begin{array}{cc}\lambda I_{n} & 0 \\ 0 & \mu I_{n}\end{array}\right): \begin{array}{c}|\lambda|=|\mu|=1 \\ (\lambda \mu)^{n}=1\end{array}\right\}$ is non-trivial.

If $Z\in \mathcal{T}$, then $I-Z$ is invertible, the inverse Cayley transform 
$$
\begin{aligned}
	c^{-1}: \mathcal{T}\  &\longrightarrow\  \operatorname{Herm}(n,\mathbb{C})+i \operatorname{Herm}^{+}(n,\mathbb{C})\\
	Z\  &\longmapsto\  i\frac{I+Z}{I-Z}
\end{aligned}
$$
is biholomorphic (verify: $\frac{1}{2i}\left[c(Z) -c(Z)^{*}\right]=\left(I-Z^{*}\right)^{-1}\left[I-Z^{*} Z\right]\left(I-Z\right)^{-1}\in \operatorname{Herm}^{+}(n,\mathbb{C})$), where $\operatorname{Herm}(n,\mathbb{C})$ is the (real) vector space of Hermitian matrices and $\operatorname{Herm}^{+}(n,\mathbb{C})$ is the (real) cone of positive definite ones. 

Therefore $\operatorname{SU}(n,n)/\mathrm{S}(\mathrm{U}(n) \times \mathrm{U}(n))\cong \mathcal{T}\cong \operatorname{Herm}(n,\mathbb{C})+i \operatorname{Herm}^{+}(n,\mathbb{C})$ is a Hermitian symmetric space of tube type.

The Shilov boundary of $\mathcal{T}$ is 
$$
\widecheck{S}_{n, n}=\left\{Z \in \operatorname{M}_n(\mathbb{C}) : I_n-Z^* Z=0\right\}
$$
which corresponds to the space of maximal isotropic subspaces
$$
\mathrm{Iso}_{n, n}=\left\{L \in \operatorname{Gr}_{n}\left(\mathbb{C}^{2n}\right): |\omega|_{L} = 0\right\} \subset \overline{\mathcal{X}_{n, n}}
$$
under the Borel embedding $\mathcal{X}_{n, n} =\left\{L \in \operatorname{Gr}_{n}\left(\mathbb{C}^{2n}\right): |\omega|_{L}>0\right\}\subset \operatorname{Gr}_{n}\left(\mathbb{C}^{2n}\right)$

\begin{remark}
	A general description of Borel embedding, Harish-Chandra embedding and the Shilov boundaries in each realizations can be found in \cite{https://doi.org/10.48550/arxiv.math/0501258}.
\end{remark}

\subsection{$\mathrm{SO}(n,2)$ ($n > 2$)}

Consider the domain (also known as the \emph{Lie ball}):
$$
	\mathcal{D}_{I V_{n}}:=\left\{Z \in \mathbb{C}^{n}: Z^*Z<\frac{1}{2}\left(1+\left|\tran{Z}Z\right|^{2}\right)<1\right\}
	=\left\{z \in \mathbb{C}^{n} : \sum\left|z_{i}\right|^{2}<\frac{1}{2}\left(1+\left|\sum z_{i}^{2}\right|^{2}\right)<1\right\}
$$
then one can check that linear fractional transformation
$$
\begin{aligned}
	\mathrm{SO}(n, 2) \times \mathcal{D}_{I V_{n}} & \longrightarrow \mathcal{D}_{I V_{n}} \\
	\left(\left(\begin{array}{cc}
		A & B \\
		C & D
	\end{array}\right), Z\right) &\mapsto \frac{A Z+B\left(\begin{array}{c}
			\frac{1}{2}({}^tZZ+1) \\
			\frac{i}{2}({}^tZZ-1)
		\end{array}\right)}{(1, i) \cdot\left(C Z+D\left(\begin{array}{c}
			\frac{1}{2}({}^tZZ+1) \\
			\frac{i}{2}({}^tZZ-1)
		\end{array}\right)\right)}
\end{aligned}
$$
is a transitive action, which can be extended continuously to $\overline{\mathcal{D}_{I V_{n}}}$. The stabilizer of $0$ is easily verified to be $\mathrm{SO}(n)\times \mathrm{SO}(2)$, which has non-trivial center.

To describe a tube domain expression of $\mathcal{D}_{I V_{n}}$, first we introduce the \emph{future tube} in $\C^{n+1}$ ($n\geq 0$):
$$
\tau^{+}(n)=\mathbb{R}^{n+1}+i V^{+}=\left\{z \in \mathbb{C}^{n+1} : y^{2}=y_{0}^{2}-y_{1}^{2}-\ldots-y_{n}^{2}>0, y_{0}>0\right\}.
$$
where the sharp convex subset
$$
V^{+} =V^{+}(n)=\left\{y \in \mathbb{R}^{1,n} : y^{2}>0, y_{0}>0\right\} .
$$
is usually called the \emph{future cone}.

The reason of its name is that for $n=3$, $\R^{1,3}$ with the Lorentz inner product becomes the Minkowski spacetime. The boundary $\partial \tau^{+}(3)$ consists of the smooth part
$$
S=\left\{\zeta=\xi+i \eta \in \mathbb{C}^{4} : \eta^{2}=0, \eta_{0}>0\right\}
$$
and the distinguished boundary
$$
M=\left\{\zeta=\xi+i \eta \in \mathbb{C}^{4} : \eta=0\right\}\cong\mathbb{R}^{4},
$$
Through any point $\zeta \in S$ there passes a generator $l_{\zeta}$ of the cone
$$
\Gamma^{+}=\partial V^{+}=\left\{\eta \in \mathbb{R}^{1,3} : \eta^{2}=0, \eta_{0} \geq 0\right\}
$$
called a \emph{real light ray}. The complexification $\lambda_{\zeta}$ of the ray $l_{\zeta}$ which coincides with a complex halfplane $\lambda_{\zeta}=\{\xi+\alpha \eta : \alpha \in \mathbb{C}, \operatorname{Im} \alpha>0\}$ is called a \emph{complex light ray}.

We will give a biholomorphism between the $n$-dimensional Lie ball $\mathcal{D}_{I V_{n}}$ and the future tube $\tau^{+}(n-1)$. This biholomorphism is a composition of two mappings.

The first mapping is a realization of $\tau^{+}(n-1)$ as a domain on a complex quadric in $\mathbb{C P}^{n+1}$. Let us introduce the new variables
$$
z_{1}=\frac{s_{1}}{s_{0}}, \ldots, z_{n-1}=\frac{s_{n-1}}{s_{0}}, \ z_{0}=\frac{s_{n}}{s_{0}} .
$$
In these variables the domain $\tau(n-1)=\left\{z \in (\mathbb{R}^{1,n-1})^2 : y^{2}>0\right\}$ will transform to the domain
$$
\begin{aligned}
T^{\prime}=&\left\{s \in \mathbb{C}^{n+2} : -\left|s_{0}\right|^{2}-\cdots-\left|s_{n-1}\right|^{2}+\left|s_{n}\right|^{2}+2 \operatorname{Re}\left(\bar{s_{0}} s_{n+1}\right)>0\right.,\\
&\left.-s_{0}^{2}-\cdots-s_{n-1}^{2}+s_{n}^{2}+2 s_{0} s_{n+1}=0\right\}
\end{aligned}
$$
Changing the variables $s_{0}$, $s_{1}, \ldots, s_{n+1}$ to the variables $t_{0}=s_{0}-s_{n+1}, t_{1}=s_{1}, \ldots, t_{n+1}=s_{n+1}$ we can write $T^\prime$ in the form
$$
\begin{aligned}
	T=&\left\{t \in \mathbb{C}^{n+2}: -\left|t_{0}\right|^{2}-\cdots-\left|t_{n-1}\right|^{2}+\left|t_{n}\right|^{2}+\left|t_{n+1}\right|^{2}>0\right.,\\
	&\left.-t_{0}^{2}-\cdots-t_{n-1}^{2}+t_{n}^{2}+t_{n+1}^{2}=0\right\}
\end{aligned}
$$
which is a section of the domain $\widetilde{T}=\left\{t \in \mathbb{C}^{n+2}:\left|t_{0}\right|^{2}+\cdots+\left|t_{n-1}\right|^{2}<\left|t_{n}\right|^{2}+\left|t_{n+1}\right|^{2}\right\}$ by the complex quadric $\left\{t_{0}^{2}+\cdots+t_{n-1}^{2}=t_{n}^{2}+t_{n+1}^{2}\right\} .$ The domains $T$ and $\widetilde{T}$ are given by homogeneous relations so it is more natural to consider them as domains in $\mathbb{C P}^{n+1}$. Note that the Levi form of the domain $\widetilde{T}$ being restricted to the complex tangent space of $\partial \widetilde{T}$ at a point $t$ with $t_{n+1} \neq 0$ has one negative and $n$ positive eigenvalues. The domain $T$ has two components distinguished by the sign of $\operatorname{Im} \frac{t_{n}}{t_{n+1}}$. $\tau^{+}(n-1)$is biholomorphic to the domain $T_{+}$ on the quadric in $\mathbb{C P}^{n+1}$ given in homogeneous coordinates as follows
$$
\begin{gathered}
	T_{+}=\left\{\left[t_{0}: t_{1}: \cdots: t_{n+1}\right]:\left|t_{0}\right|^{2}+\cdots+\left|t_{n-1}\right|^{2}<\left|t_{n}\right|^{2}+\left|t_{n+1}\right|^{2}\right. ,\\
	\left.t_{0}^{2}+\cdots+t_{n-1}^{2}=t_{n}^{2}+t_{n+1}^{2}, \operatorname{Im} \frac{t_{n}}{t_{n+1}}>0\right\}
\end{gathered}
$$

The second mapping given by
$$
w_{0}=\frac{t_{0}}{t_{n}+i t_{n+1}}, \ldots,\  w_{n-1}=\frac{t_{n-1}}{t_{n}+i t_{n+1}} .
$$
transforms the domain $T_{+}$ biholomorphically onto the domain
$\mathcal{D}_{I V_{n}}$.

The composed mapping of $\tau^{+}(n-1)$ onto $\mathcal{D}_{I V_{n}}$ is given by
$$
w_{0}=i \frac{1+z^{2}}{(z+\mathbf{i})^{2}}, w_{1}=i \frac{2 z_{1}}{(z+\mathbf{i})^{2}}, \ldots, w_{n-1}=i \frac{2 z_{n-1}}{(z+\mathbf{i})^{2}}
$$
where $\mathbf{i}=(i, 0, \ldots, 0)$. We conclude that $\operatorname{SO}(n,2)/(\operatorname{SO}(n)\times\operatorname{SO}(2))\cong \mathcal{D}_{I V_{n}}\cong \tau^{+}(n-1)$ is a Hermitian symmetric space of tube type.

The distinguished boundary of $\tau^{+}(n-1)$ transforms into the set
$$
S_{L}=\left\{\left|w_{0}\right|^{2}+\cdots+\left|w_{n-1}\right|^{2}=1,\left|w_{0}^{2}+\cdots+w_{n-1}^{2}\right|=1\right\}.
$$
Set $w=u+i v,(z, \omega)=z_{0} \omega_{0}+z_{1} \omega_{1}+\cdots+z_{n-1} \omega_{n-1}$. Then the intersection of $S_{L}$ with the complex sphere $\Sigma_{1}=\{w\in\C^n\mid (w, w)=1\}$ is given by the equations $|u|^{2}-|v|^{2}=1,(u, v)=0,|u|^{2}+|v|^{2}=1$. It follows that $v=0$; hence $\Sigma_{1}$ intersects $S_{L}$ in the $(n-1)$-dimensional real sphere $\left\{u \in \mathbb{R}^{n}\mid |u|^{2}=1\right\}$. So $S_{L}$ can be written as
$$
\widecheck{S}:=S_L=\left\{e^{i \theta} u\mid u \in S^{n-1} \subset \R^{n}\right\} \simeq S^1\times S^{n-1}/\Z_2,
$$
which is also the (distinguished) Shilov boundary of $\mathcal{D}_{I V_{n}}$, known as the \emph{Lie sphere}.

\begin{remark}
For the convenience of readers, we record at here the classification of all irreducible Hermitian symmetric spaces, according as to whether or not they are of tube type: 
$$
\begin{array}{c|c}
	\text { tube type } & \text { nontube type } \\
	\hline \hline \mathrm{SU}(n, n) & \mathrm{SU}(p, q), p>q \\
	\hline \mathrm{Sp}(2 n, \mathbb{R}) & \\
	\hline \mathrm{SO}^{*}(2 n)  \text {, n even } & \mathrm{SO}^{*}(2 n) \text {, n odd } \\
	\hline \mathrm{SO}(n,2) & \\
	\hline \mathrm{E}_{7}(-25) & \mathrm{E}_{6}(-14) \\
	\hline
\end{array}
$$
where $\mathrm{E}_{7}(-25)$ and $\mathrm{E}_{6}(-14)$ correspond to the exceptional Hermitian symmetric spaces of complex dimension 27 and 16 , respectively.
\end{remark}

\section{Maslov index and a semigroup}
For $\operatorname{Sp}(2 n, \mathbb{R})$, consider
$$
\begin{aligned}
	&V=\left\{g \in \operatorname{Sp}(2 n, \mathbb{R}) : g=\left(\begin{array}{cc}
		I_{n} & 0 \\
		M & I_{n}
	\end{array}\right), M \in \operatorname{Sym}(n, \mathbb{R})\right\} \\
	&W=\left\{g \in \operatorname{Sp}(2 n, \mathbb{R}) : g=\left(\begin{array}{cc}
		I_{n} & N \\
		0 & I_{n}
	\end{array}\right), N \in \operatorname{Sym}(n, \mathbb{R})\right\}
\end{aligned}
$$
and
$$
H=\left\{g \in \operatorname{Sp}(2 n, \mathbb{R}) : g=\left(\begin{array}{cc}
	A & 0 \\
	0 & {}^t A^{-1}
\end{array}\right)\right\} \cong \mathrm{GL}(n, \mathbb{R})
$$
where the matrices are written with respect to a symplectic basis. Define
$$
\operatorname{Sp}(2 n, \mathbb{R})^{> 0}:=V^{> 0} H^{\circ} W^{> 0}
$$
where
$$
V^{> 0}=\left\{\left(\begin{array}{cc}
	I_{n} & 0 \\
	M & I_{n}
\end{array}\right) \in V : M \in \operatorname{Pos}(n, \mathbb{R})\right\},
W^{> 0}=\left\{\left(\begin{array}{cc}
	I_{n} & N \\
	0 & I_{n}
\end{array}\right) \in W : N \in \operatorname{Pos}(n, \mathbb{R})\right\}
$$
and $H^{\circ}$ is the identity component of $H$. To show $\operatorname{Sp}(2 n, \mathbb{R})^{> 0}$ is a subsemigroup, we only need to verify the product of an element in $W^{> 0}$ with an element in $V^{> 0}$ is in $V^{> 0} H^{\circ} W^{> 0}$. Notice if $N, M \in \operatorname{Pos}(n, \mathbb{R})$, then $I+N M$ is invertible since if it had a nonzero nullvector $v$, then $M v=-N^{-1} v$, then $\tran{v} M v=-\tran{v} N^{-1} v$, a contradiction. An explicit calculation (in view of block Gaussian elimination) is
$$
\begin{aligned}
	&\left(\begin{array}{ll}
		I & N \\
		0 & I
	\end{array}\right)\left(\begin{array}{cc}
		I & 0 \\
		M & I
	\end{array}\right)=\left(\begin{array}{cc}
		I+N M & N \\
		M & I
	\end{array}\right) \\
	&=\left(\begin{array}{ccc}
		I & 0 \\
		M(I+N M)^{-1} & I
	\end{array}\right)\left(\begin{array}{cc}
		I+N M & 0 \\
		0 & I-M(I+N M)^{-1} N
	\end{array}\right)\left(\begin{array}{cc}
		I & (I+N M)^{-1} N \\
		0 & I
	\end{array}\right) .
\end{aligned}
$$

\begin{remark}
	The definition for $\operatorname{Sp}(2n,\R)^{> 0}$ is obtained from $\Theta$-positive structure by taking $\Theta=\{\alpha_n\}$, where the restricted roots $\alpha_{i}=e_i-e_{i+1}$, $\alpha_n=2e_n$. 
\end{remark}

We will see that a triple $(L_1, L_2,L_3)\in \mathcal{L}^3$ is positive if and only if $\exists g\in \operatorname{Sp}(2n,\R)$ such that $g(L_1,L_2,L_3)=(L_E,uL_E,L_F)$ for some $u\in W^{> 0}$ if and only if $(L_1, L_2,L_3)$ has maximal Maslov index $n$, where $L_E=\operatorname{span}\{e_1,\cdots,e_n\}$, $L_F=\operatorname{span}\{f_1,\cdots,f_n\}$ are two standard Lagrangians, $\{f_1,\cdots,f_n,e_1,\cdots,e_n\}$ is a symplectic basis with respect to $\omega$.

Now we generalize this construction to any Hermitian Lie group $G$ which is of tube type. 

First we briefly introduce the notion of Jordan algebras and their relationship with symmetric cones. For details, a good reference is \cite{book:834711}.

\begin{definition}
	A \emph{Jordan algebra} is a vector space $V$ with a bilinear (not necessarily associative) product $V \times V \rightarrow V,(x, y) \mapsto x y$ with
	$$
	\begin{aligned}
		&y x=x y \\
		&x\left(x^{2} y\right)=x^{2}(x y).
	\end{aligned}
	$$
	We always assume that $V$ has an identity element $e$.
	
	An \emph{idempotent} in $V$ is an element $c \in V$ satisfying $c^{2}=c$. Two idempotents $c$ and $d$ are said to be \emph{orthogonal} if $cd=0$. Since then
	$$
	\langle c, d\rangle=\langle c^{2}, d\rangle=\langle c, cd\rangle=\langle c, 0\rangle=0,
	$$
	orthogonal idempotents are orthogonal with respect to the inner product.
	
	\begin{remarks}
	\begin{enumerate}
	    \item  $\forall x\in V$, we usually denote by $L(x):V\to V$ the multiplication by $x$.
		\item Let $c$ be an idempotent in a Jordan algebra. It is not hard to verify that the only possible eigenvalues of $L(c)$ are $0$, $\frac{1}{2}$ and $1$. (verify the identity $2 L(c)^{3}-3 L(c)^{2}+L(c)=0$)
			\end{enumerate}
	\end{remarks}
	
	A Jordan algebra over $\R$ is said to be \emph{Euclidean} if there exists a positive definite symmetric bilinear form $\langle\cdot, \cdot\rangle: V \times V \rightarrow \mathbb{R}$ which is associative; i.e. $\langle x u, y\rangle=\langle u, x y\rangle$ for all $x, y, u \in V$.
	
	An idempotent is called \emph{primitive} if it is non-zero and cannot be written as sum of two (necessarily orthogonal) non-zero idempotents. A complete system of orthogonal primitive idempotents or a \emph{Jordan frame}, is a set of primitive idempotents $c_{1}, \ldots, c_{m}$ satisfying
	$$
		c_{i} c_{j}=0,\quad \forall i \neq j, \quad
		\sum_{i=1}^{m} c_{i}=e .
	$$
\end{definition}

Let $V$ be a finite-dimensional Euclidean Jordan algebra
over $\R$, we define the rank of $V$ as
$$
\rk(V):=\max \{\operatorname{deg}(x)\mid x \in V\},
$$
where $\operatorname{deg}(x)$ is the degree of the unique minimal polynomial of $x$. An element $x \in V$ is called \emph{regular} if $\operatorname{deg}(x)=\rk(V)$.

\begin{proposition}[Characteristic polynomial]\label{char poly}
	The set of regular elements is open and dense in $V$. There exist polynomials $a_{1}, \ldots, a_{r}$ on $V$ such that the minimal polynomial of every regular element $x \in V$ in the variable $\lambda$ is given by
	$$
	f(\lambda, x)=\lambda^{r}-a_{1}(x) \lambda^{r-1}+\cdots+(-1)^{r} a_{r}(x) .
	$$
	The polynomials $a_{i}$ are unique and homogeneous of degree $i$. 
\end{proposition}

Another useful result is
\begin{theorem}[Spectral decomposition of type II]
	Suppose $V$ has rank $r$. Then for every $x$ in $V$ there exists a Jordan frame $c_{1}, \ldots, c_{r}$ and real numbers $\lambda_{1}, \ldots, \lambda_{r}$ such that
	$$
	x=\sum_{j=1}^{r} \lambda_{j} c_{j} .
	$$
	The numbers $\lambda_{j}$ (with their multiplicities) are uniquely determined by x. Moreover,
	$$
	a_{k}(x)=\sum_{1 \leq i_{1} \leq \cdots \leq i_{k} \leq r} \lambda_{i_{1}} \ldots \lambda_{i_{k}},
	$$
	where $a_{k}(1 \leq k \leq r)$ is the polynomial defined in the above proposition.
\end{theorem}

The reason why we choose Jordan algebras is the cones of squares give a full characterization of symmetric cones that we now describe.

Let $V$ be a Euclidean Jordan $\mathbb{R}$-algebra and let $\mathcal{K}$ be the set of all squares, i.e,
$$
\mathcal{K}:=\left\{x^{2}: x \in V\right\} .
$$
If $x^{2} \in \mathcal{K}$ and $\alpha \geq 0$ then $\alpha x^{2}=(\sqrt{\alpha} x)^{2} \in \mathcal{K}$. It follows that $\mathcal{K}$ is a cone. We call $\mathcal{K}$ the \emph{cone of squares} of $V$. Furthermore, it can be shown that $\mathcal{K}$ is a closed symmetric cone, cf. \cite{book:834711}. (i.e. homogeneous and self-dual; $\operatorname{Aut}(\mathcal{K})$ transitively on the interior of $\mathcal{K}$, $\mathcal{K}=\mathcal{K}^*$)

\begin{remark}
	Since $\mathcal{K}$ is closed, here $\mathcal{K}^*$ is the dual closed cone $\mathcal{K}^{*}:=\{y \in V\mid \langle x, y\rangle \geq 0, \forall x \in \mathcal{K}\}$. If $\mathcal{C}$ is an open convex cone in $V$, we say that $\mathcal{C}$ is homogeneous if  $\operatorname{Aut}(\mathcal{C})$ transitively on $\mathcal{C}$; is self-dual if $\mathcal{C}=\mathcal{C}^*$, where $\mathcal{C}^*$ is the dual open cone $\mathcal{C}^{*}:=\{y \in V\mid \langle x, y\rangle > 0, \forall x \in \overline{\mathcal{C}}\setminus\{0\}\}$.
\end{remark}

\begin{proposition}
	Let $\mathcal{I}$ be the set of invertible elements in $V$, then $\mathcal{K}^{\circ}=\left\{x^{2}\mid x \in \mathcal{I}\right\}$, where $\mathcal{K}^{\circ}$ is the interior of $\mathcal{K}$.
\end{proposition}

\begin{proof}
	$y\in \mathcal{K}^*$ if and only if $\langle y, x^2\rangle=\langle L(y)x,x\rangle \geq 0$ for all $x\in V$, if and only if $L(y)$ is positive semidefinite. Since $\mathcal{K}=\mathcal{K}^*$, we need to verify that $L(y)$ is positive definite if and only if $y=w^2$ for some invertible $w$. 
	
	Necessity: Suppose $L(y)>0$, and $y=\sum_{i=1}^{r} \lambda_{i} c_{i}$ is the spectral decomposition. Since
	$$
	\left\langle L(y) c_{i}, c_{i}\right\rangle=\lambda_{i}\left\langle c_{i}, c_{i}\right\rangle=\lambda_{i}\left\|c_{i}\right\|^{2}>0,
	$$
	we know $\lambda_{i}>0$. Therefore $y$ is invertible and $y=w^{2}$ with $w=\sum_{i=1}^{r} \sqrt{\lambda_{i}} c_{i}$.
	
	Sufficiency: Suppose $y=w^2$ for invertible $w= \sum_{i=1}^{r} \lambda_{i} c_{i}$, notice that we have  $\left\langle z L\left(w^{2}\right), z\right\rangle=\left\langle\sum_{i=1}^{r} \lambda_{i}^{2} L\left(c_{i}\right) z, z\right\rangle$. Since $w$ is invertible, the eigenvalues of $w^2$ are greater than zero, i.e. $\lambda_{i}>0$ for all $i$. Hence we only need to verify for each $z\in V\setminus\{0\}$, there is some $i$ such that $\langle L(c_i)z,z\rangle>0$. Suppose not, $\exists z_0\in V\setminus\{0\}$ s.t. $\langle L(c_i)z_0,z_0\rangle=0$ for all $i$. (remark that $L(c_i)\geq 0$) Therefore
	$$
	0=\sum_{i=1}^{r}\left\langle L\left(c_{i}\right) z_0, z_0\right\rangle=\sum_{i=1}^{r}\left\langle c_{i} z_0, z_0\right\rangle=\langle e z_0, z_0\rangle=\langle z_0, z_0\rangle,
	$$
	which is a contradiction since $z_0\neq 0$.
\end{proof}

From this relation one can also prove that $\mathcal{K}^\circ$ is an open symmetric cone. We have mentioned that $\mathcal{K}$ is a closed symmetric cone; in fact, the converse also holds. We state the following result which is a Jordan algebraic characterization of symmetric cones.

\begin{theorem}[Theorem III.3.1 in \cite{book:834711}]
	A cone is symmetric if and only if it is the cone of squares of some Euclidean Jordan algebra.
\end{theorem}

\begin{proposition}
	Let $V$ be an Euclidean Jordan $\mathbb{R}$-algebra with rank $r$ and $\mathcal{K}$ its cone of squares. If $x \in V$ is such that $x=\sum_{i=1}^{r} \lambda_{i} c_{i}$, where $\left(c_i\right)_{i=1,\dots,r}$ is a Jordan frame, then $\lambda_{i} \geq 0$, resp. $>0$. for $i=1, \ldots r$ if and only if $x \in \mathcal{K}$, resp. $\mathcal{K}^\circ$.
\end{proposition}

\begin{proof}
	If $\lambda_{i} \geq 0$ then we can write $x=y^2$, with $y=\sum_{i=1}^{r} \sqrt{\lambda_{i}} c_{i}$, which means that $x \in \mathcal{K}$. In case that $\lambda_{i}>0$ for $i=1, \ldots, r$, we know that $x$ is invertible, and it follows that $x \in \mathcal{K}^\circ$. Conversely, if $x \in \mathcal{K}$ we have $x=y^2$, with $y \in V$. If we denote $\alpha_{i}$ the eigenvalues of $y$, we can write $x=\sum_{i=1}^{r} \alpha_{i}^2 c_{i}$ where $\alpha_{i}^2$'s are the eigenvalues of $x$, which are greater or equal than zero. If $x \in \mathcal{K}^\circ$ then $y$ is invertible, therefore $\alpha_{i}^2$'s are greater than zero.
\end{proof}

We denote $\mathcal{K}^\circ$ by $\Omega$, tube domain $T_\Omega:=V \oplus i \Omega=\left\{\sum \mu_{i} c_{i} \mid \operatorname{Im} \mu_{i}>0\right\}$. Via the Cayley transform $c(z)=(z-i e)(z+i e)^{-1}$, $T_\Omega$ is biholomorphic to bounded symmetric domain in $V^\C$:
$$
\mathcal{D}:=\left\{v=\sum \lambda_{i} c_{i} \in V^{\mathbb{C}} : \left(c_{i}\right) \text { Jordan frame, }\left|\lambda_{i}\right|<1\right\}.
$$
The Shilov boundary of $T_\Omega$ is $V$. (Proposition IX.5.5 in \cite{book:834711}) For the Shilov boundary $\check{S}$ of $\mathcal{D}$, we have:

\begin{proposition}[Proposition.X.2.3, Theorem.X.4.6 in \cite{book:834711}] \label{Shilov bd}
	Let $\mathcal{D} \subset V^{\mathbb{C}}$ be a bounded symmetric domain in the complexification of an Euclidean Jordan algebra $V$. The following are equivalent:
	\begin{enumerate}
	\item $z \in \check{S}$,
	\item $z=\sum \lambda_{i} c_{i}$, where $c_{1}, \ldots, c_{r}$ is a Jordan frame and $\left|\lambda_{i}\right|=1$,
	\item $z \in \overline{c(V)}$,
	\item $\bar{z}=z^{-1}$.
		\end{enumerate}
\end{proposition}

\begin{example}
	For $\mathcal{X}_n=\operatorname{Sym}(n,\R)+i\operatorname{Pos}(n,\R)$, define $X \circ Y:=\frac{X Y+Y X}{2}$ for $X,Y\in \operatorname{Sym}(n,\R)$, where $XY$ denotes the usual matrix product. The inner product is defined as $\langle X, Y\rangle=\Tr(X \circ Y)=\Tr(X Y) \text { for all } X, Y \in \operatorname{Sym}(n,\R)$. One can check that $\left(\operatorname{Sym}(n,\R),\circ\right)$ is a Euclidean Jordan algebra. Any $X \in \operatorname{Sym}(n,\R)$ has spectral decomposition:
	$$
	X=\sum_{i=1}^{n} \lambda_{i} q_{i} \tran{q_{i}}
	$$
	where $\lambda_{i}\in\R$, $i=1, \ldots, n$, are the eigenvalues of $X$ and $q_{i}, i=1, \ldots, n$ are unitary eigenvectors of $X$. Each $q_{i} \tran{q_{i}}$ is an idempotent; $\left\{q_{1} \tran{q_{1}}, \ldots, q_{n} \tran{q_{n}}\right\}$ is a Jordan frame.
	
	$\operatorname{Aut}(\operatorname{Pos}(n,\R))=\operatorname{GL}(n,\R)/\{\pm I_n\}$ acts on $\operatorname{Sp}(2n,\R)$ transitively by $g \cdot x:=g x \tran{g}$. For any $Y\in \operatorname{Pos}(n,\R)^*$, $\langle Y,\xi\tran{\xi}\rangle>0$ for all nonzero $\xi\in\R^n$. Namely, $\operatorname{tr}(Y\xi\tran{\xi})=\sum_{i,j} y_{i j} \xi_{i} \xi_{j}>0$, which implies $Y\in \operatorname{Pos}(n,\R)$. On the other hand, take any $X\in\operatorname{Pos}(n,\R)$, $Y\in \overline{\operatorname{Pos}(n,\R)}\setminus\{0\}$, we write $Y=\sum_{i=1}^{n} \mu_{i} q_{i} {}^t q_{i}$ for $\mu_i\geq 0,\forall i$ but not all $0$. Then $\Tr(X Y)=\sum_{i}\mu_{i} {}^tq_iXq_i>0$, hence $X\in \operatorname{Pos}(n,\R)^*$. This shows $\operatorname{Pos}(n,\R)=\operatorname{Pos}(n,\R)^*$, therefore $\operatorname{Pos}(n,\R)$ is an open symmetric cone.
\end{example}

For the Maslov cocycle, we follow Anna Wienhard's definition in \cite{https://doi.org/10.48550/arxiv.math/0501258}, which is a slight variation of Clerc's definition \cite{clerc2004indice}.

Let $T:X=G/K\to T_\Omega=V+i\Omega$ be a biholomorphism, where $\Omega\subset V$ is a symmetric convex cone in the real vector space $V$; $\mathcal{D}$ the bounded domain realization and $\check{S}$ its Shilov boundary; $r_X$ the rank of $X$. Then as mentioned in previous sections, $\check{S}$ is a homogeneous $G$-space of the form $G/Q$, where $Q$ is a specific parabolic subgroup (which is maximal if $X$ is irreducible). Two points $x,y\in\check{S}$ are \emph{transversal} if $(x,y)$ lies in the open $G$-orbit in $\check{S}^2$.

Let
$$
\check{S}^{[3]}:=\left\{\left(x_{1}, x_{2}, x_{3}\right) \in \check{S}^{3} : x_{i} \text { is transverse to } x_{j}, x_{k} \text { for some } i\right\}
$$
be the space of triples, where one point is transverse to the other two. For any triple $\left(x_{1}, x_{2}, x_{3}\right) \in \check{S}^{[3]}$, we may assume that $T\left(x_{3}\right)=\infty$ and $y_{1}=T\left(x_{1}\right), y_{2}=T\left(x_{2}\right) \in V$ by using the transitivity of the $G$-action on $\check{S}$, then the \emph{Maslov cocycle} is defined to be
$$
\tau\left(x_{1}, x_{2}, x_{3}\right):=k_{+}\left(y_{2}-y_{1}\right)-k_{-}\left(y_{2}-y_{1}\right),
$$
where $k_{\pm}$ are the numbers of positive respectively negative eigenvalues in the spectral decomposition of $\left(y_{2}-y_{1}\right)$ with respect to a Jordan frame $\left(c_{j}\right)_{j=1, \ldots, r_X}$.

\begin{remark}
	 Let $\left(x_{1}, x_{2}, x_{3}\right) \in \check{S}^{[3]}$ with $\tau\left(x_{1}, x_{2}, x_{3}\right)=r_X$, then $\left(x_{1}, x_{2}, x_{3}\right)$ are pairwise transverse since if we assume $x_3$ is transverse to $x_1$, $x_2$ w.l.o.g., then $\tau=r_X$ implies eigenvalues of $y_2-y_1$ are all positive; especially $y_2-y_1$ is invertible.
\end{remark}

Given any triple $x_{1}, x_{2}, x_{3} \in \check{S}$, using $G$-action we may assume that $x_{1}, x_{2}, x_{3}$ are transverse to $\infty$, hence $y_{1}, y_{2}, y_{3}$ in $V$. The cocycle identity allows us to define $\tau$ on $\check{S}^{3}$ as
$$
\begin{aligned}
	\tau\left(x_{1}, x_{2}, x_{3}\right) &=\tau\left(x_{1}, x_{2}, \infty\right)+\tau\left(x_{2}, x_{3}, \infty\right)-\tau\left(x_{1}, x_{3}, \infty\right) \\
	&=\left(k_{+}\left(y_{2}-y_{1}\right)-k_{-}\left(y_{2}-y_{1}\right)\right) \\
	&+\left(k_{+}\left(y_{3}-y_{2}\right)-k_{-}\left(y_{3}-y_{2}\right)\right) \\
	&-\left(k_{+}\left(y_{3}-y_{1}\right)-k_{-}\left(y_{3}-y_{1}\right)\right) .
\end{aligned}
$$

This clearly defines a $G$-invariant real function on $\check{S}^{3}$. We call a triple $\left(x_{1}, x_{2}, x_{3}\right)\in \check{S}^3$ \emph{maximal} if $\tau\left(x_{1}, x_{2}, x_{3}\right)$ attains its maximal possible value $\operatorname{rk}(V)=r_X$. 

If we restrict ourselves to pairwise transverse triples (usually denoted by $\check{S}^{(3)}$), we will recover Clerc and {\O}rsted's definition of Maslov index given in \cite{clerc2001maslov}.

Under the Cayley transform $c$, $c(\infty)=e$, $c(0)=-e$, we also let
$$
	\varepsilon_{k}:=\sum_{i=1}^{k} c_{i}-\sum_{i=k+1}^{r_X} c_{i} \in \check{S}
$$
which is a point on $\check{S}$ following from Proposition \ref{Shilov bd}, where $\left(c_{j}\right)_{j=1, \ldots, r_X}$ is a fixed Jordan frame. We have
$$
\begin{aligned}
		c(\varepsilon_k)&=\left(\sum_{j=1}^k(1-i)c_j-\sum_{j=k+1}^{r_X}(1+i)c_j\right)\cdot\left(\sum_{j=1}^k(1+i)c_j-\sum_{j=k+1}^{r_X}(1-i)c_j\right)^{-1}\\
		&=\left(\sum_{j=1}^k(1-i)c_j-\sum_{j=k+1}^{r_X}(1+i)c_j\right)\cdot\frac{1}{2}\left(\sum_{j=1}^k(1-i)c_j-\sum_{j=k+1}^{r_X}(1+i)c_j\right)=(-i)\varepsilon_{k},
\end{aligned}
$$
hence $\tau\left(-e, (-i)\varepsilon_{k}, e\right)=k_+\left(\varepsilon_k\right)-k_-\left(\varepsilon_k\right)=2k-r_X$. 

\begin{theorem}[Theorem 4.3 in \cite{clerc2001maslov}, Theorem 3.5 in \cite{clerc2004indice}]
	There are exactly $r_X+1$ orbits of pairwise transverse triples in $\check{S}^3$ under the action of $G$. Each $\left(-e, (-i)\varepsilon_j, e\right)$, $0\leq j\leq r_X$, represents one orbit. The function $\tau$ takes values in $\{-r_X,-r_X+2, \ldots, r_X-2, r_X\}$ and it classifies all these $G$-orbits.
\end{theorem}

Therefore we can also define for any pairwise transverse triple $\left(x_1,x_2,x_3\right)$, $\tau\left(x_{1}, x_{2}, x_{3}\right)=2k-r_X$, where $k$ is the unique integer, $0\leq k\leq r_X$, such that $\left(x_1,x_2,x_3\right)$ is conjugate under $G$ to the triplet $\left(-e, (-i)\varepsilon_{k}, e\right)$. See \cite{clerc2001maslov}.

\begin{remarks}
	For tube type bounded symmetric domain $\mathcal{D}$, $\pi\tau$ actually coincides with the restriction of Bergmann cocycle $\beta$ on $\check{S}^{(3)}$. For non-tube type $\mathcal{D}$, however, $\beta\left(\check{S}^{(3)}\right)=\left[-\pi r_{X}, \pi r_{X}\right]$. See \cite{https://doi.org/10.48550/arxiv.math/0501258} Chapter 5 for details.
\end{remarks}

Fix $x,z\in \check{S}=G/Q$ such that $\operatorname{Stab}_G(z)=Q$, $\operatorname{Stab}_G(x)=Q^{opp}$. Given any $y\in G/Q$ transverse to $z$, there exists unique $u\in \operatorname{UniRad}(Q)$ such that $y=ux$, where $\operatorname{UniRad}(Q)$ is the unipotent radical of the parabolic $Q$.

Denote by 
$$
\begin{aligned}
	&U^{> 0}:=\left\{u\in \operatorname{UniRad}(Q)\mid \tau\left(x,ux,z\right)=r_X\right\},\\
	&U^{opp,> 0}:=\left\{v\in \operatorname{UniRad}(Q^{opp})\mid \tau\left(z,vz,x\right)=r_X\right\},\\
	&L^\circ:=\text{ identity component of } Q\cap Q^{opp}.
\end{aligned}
$$
then we can define the positive subsemigroup $G^{> 0}\subset G$ to be the subsemigroup generated by $U^{> 0}$, $U^{opp,> 0}$ and $L^\circ$.

\newpage

\thispagestyle{empty}

\chapter[Positivity of Triples of Flags]{Positivity of Triples of Flags\\ {\Large\textnormal{\textit{by Raphael Appenzeller, Francesco Fournier-Facio}}}}
\addtocontents{toc}{\quad\quad\quad \textit{Raphael Appenzeller, Francesco Fournier-Facio}\par}

\section{Flags}

In this section we introduce the objects that we will be working with. We start with the flag variety, then consider the relevant (semi)groups, and finally put both together by letting the (semi)groups act on the flag variety.

\subsection{The flag variety}

\begin{definition}
A \emph{flag} is an $(n-1)$-tuple $(F_1, \ldots, F_{n-1})$, where $F_i$ is a vector subspace of $\mathbb{R}^n$ of dimension $i$, and $F_i \subset F_{i+1}$ for all $1 \leq i \leq (n-2)$.

The set of flags is denoted by $\Flag(\R^n)$ and is called the \emph{(full) flag variety}.
\end{definition}

The two most basic examples of flags also turn out to be the most important ones:

\begin{example}
Let $\{ e_i \}_{i = 1}^n$ be the canonical basis of $\mathbb{R}^n$.

The \emph{standard ascending flag} is $F = (F_1, \ldots, F_{n-1})$, where $F_i = \langle e_1, \ldots, e_i \rangle$.

The \emph{standard descending flag} is $E = (E_1, \ldots, E_{n-1})$, where $E_i = \langle e_n, \ldots, e_{n-i+1} \rangle$.
\end{example}

In some sense, these two flags are \emph{opposite} each other. One way to formalize this is the following:

\begin{definition}
Two flags $F^1, F^2$ are \emph{transverse}, denoted $F^1 \pitchfork F^2$, if $F^1_i \cap F^2_{n-i} = \{0\}$ for every $1 \leq i < n$. Equivalently, $F^1$ and $F^2$ are transverse if $F^1_i \oplus F^2_{n-i} = \mathbb{R}^n$ for every $1 \leq i < n$.

Given a flag $F$, the set of flags transverse to $F$ is denoted by $\Omega_F$.
\end{definition}

\begin{example}
The standard ascending and descending flags are transverse to each other.
\end{example}

The full flag variety $\Flag(\R^n)$ comes equipped with a natural topology, induced by the inclusion in a product of Grassmannians $\Flag(\R^n) \subset \operatorname{Gr}(n, 1) \times \cdots \times \operatorname{Gr}(n, n-1)$.

\begin{proposition}
Let $F$ be the standard ascending flag. Then $\Omega_F$ is open and dense in $F$.
\end{proposition}

\begin{remark}
We will soon show that $\Flag(\R^n)$ is a homogeneous space. Therefore the statement of the proposition holds for every flag, not only the standard ascending one.
\end{remark}

\begin{proof}[Proof idea]
Let $F^1 \in \Omega_F$. If another flag $F^2$ is close enough to $F^1$, then $F^1_i$ is close to $F^2_i$, and so every element of $\mathbb{R}^n = F^1_i + F_{n-i}$ is close to an element of $F^2_i + F_{n-i}$, which is not possible if $\dim(F^2_i + F_{n-i}) < n$. Thus $F^2 \in \Omega_F$, which shows that $\Omega_F$ is open.

Now let $F^1 \in \Flag(\R^n)$ be arbitrary, then for every $i$ we can modify $F^1_i$ by a small amount in a direction transverse to $F_{n-i}$ in order to obtain a subspace $F^2_i$ such that $F^2_i + F_{n-i} = \mathbb{R}^n$. It is possible to do this in a compatible way on every $i$ by an induction argument, to put these spaces together into a flag $F^2$, which is then an element of $\Omega_F$ close to $F^1$. This shows that $\Omega_F$ is dense in $\Flag(\R^n)$.
\end{proof}

A further relevant notion is the following, which may be thought of as an analogue of transversality for triples of flags:

\begin{definition}
A triple of flags $(F^1, F^2, F^3)$ is \emph{generic} if for all $a, b, c \geq 0$ such that $a + b + c = n$, it holds $F^1_a + F^2_b + F^3_c = \mathbb{R}^n$.
\end{definition}

By choosing $a, b$ or $c$ to be equal to $0$, we see that if $(F^1, F^2, F^3)$ is generic, then the three flags are pairwise transverse. The converse is not true however (contrary to what is claimed in \cite{Bonahon_2014}), as the following example shows:

\begin{example}
Let $F$ be the standard ascending, and $E$ the standard descending flag. Let $T$ be defined by $T_1 := \langle t_1 \rangle$ and $T_2 := \langle t_1, t_2 \rangle$, where
\[t_1:=
\begin{pmatrix}
y \\
z \\
1
\end{pmatrix}
\quad \text{ and } \quad t_2:=
\begin{pmatrix}
x \\
1 \\
0
\end{pmatrix}
\]
(it will be soon clear why we chose this notation). If $y \neq 0$ and $xz - y \neq 0$, then $E, F$ and $T$ are pairwise transverse. However, if $z = 0$, then the triple is not generic, since $E_1 + F_1 + T_1 = \langle e_1, e_3 \rangle$. As a concrete example, one can take $x = 1, y = 1$ and $z = 0$.
\end{example}

\subsection{The (semi)groups}

We will be working in the group $G := \GL(n,\mathbb{R})$ and the main players will be its subgroups $B$ of upper-triangular matrices, and $U$ of unipotent upper-triangular matrices. We further denote by $\tran{B}$ and $\tran{U}$ the subgroups of lower-triangular and unipotent lower-triangular matrices, and $A := B \cap \tran{B}$ the group of diagonal matrices. \\

Recall from the previous talk that a matrix $g \in G$ is called \emph{totally positive} if each of its minors is positive. We denote by $G_{> 0}$ the subset of totally positive matrices. Especially relevant for this talk is the subset $U_{> 0}$: this is defined as the set of matrices in $U$ such that every minor (that is not forced to be zero) is positive.

\begin{example}
\label{FLAGS:example:flagsinr3}

Suppose that $n = 3$. Then each $u \in U$ is of the form
\[
\begin{pmatrix}
1 & x & y \\
0 & 1 & z \\
0 & 0 & 1
\end{pmatrix}
\]
where $a, b, c \in \mathbb{R}$. Most minors are already determined: for instance the $(1, 3)$ minor is $0$, and the $(1, 1)$ minor is $1$. The condition that $u \in U_{> 0}$ then amounts to:
$$x>0, \quad y>0, \quad z>0, \quad xz - y > 0.$$
\end{example}

We remark the following fact:

\begin{lemma}
$G_{> 0}$ and $U_{> 0}$ are subsemigroups.
\end{lemma}

\begin{proof}
This follows easily from the Cauchy--Binet Formula (see Proposition \ref{prop:cauchybinet} and its proof).
\end{proof}

\subsection{The action}

There is a natural action of $G$ on $\Flag(\R^n)$, given by $g \cdot(F_1,\ldots, F_{n-1}) = (gF_1, \ldots, gF_{n-1})$. The following result gives the important properties of this action that will be used in the definition of positivity.

\begin{proposition}
\label{FLAGS:prop:action}

The action of $G$ on $\Flag(\R^n)$ has the following properties:
\begin{enumerate}
    \item \emph{(Naturality)} The action is continuous and preserves the notions of transversality and genericity.
    \item \emph{(Transitivity)} The action of $G$ on $\Flag(\R^n)$ is transitive, and the action of $G$ on pairs of transverse flags is transitive.
    \item \emph{(Stabilizers)} Let $F$ be the standard ascending flag. Then the stabilizer of $F$ is the Borel subgroup $B$, and the induced action of $U$ on $\Omega_F$ is simply transitive.
\end{enumerate}
\end{proposition}

\begin{proof}
The first item is clear from the definitions. For transitivity, let $T$ be a flag, and choose elements $t_1, \ldots, t_{n-1}$ such that $T_i = \langle t_1, \ldots, t_i \rangle$. Then the element $g \in G$ whose $i$-th column is $t_i$ sends the standard ascending flag $F$ to $T$. Now the statement on transitivity on pairs of transverse flags will follow from the third item. 

It is easy to see that $B$ is the stabilizer of $F$. Next, we prove that $U$ acts transitively on $\Omega_F$. We prove this by induction on $n$. For $n = 1$ is clear. Now let $n > 1$ and assume that the statement is true up to $(n-1)$. Let $T \in \Omega_F$: we need to show that there exists $u \in U$ such that $uE = T$, where $E$ is the standard descending flag.

Let $v \in \mathbb{R}^n$ be a vector such that $T_1 = \langle v \rangle$. Since $T \pitchfork F$, it holds $T_1 + F_{n-1} = \mathbb{R}^n$, so the last coordinate of $v$ is non-zero. Therefore, up to rescaling, we may assume that the last coordinate of $v$ is $1$. This allows to consider a matrix $u \in U$ whose last column is $v$. In other words $u E_1 = u \langle e_n \rangle = \langle v \rangle = T_1$. Since no condition has been imposed on the first $(n-1)$ columns of $u$, by induction we may choose them so that $u E_i = T_i$ for every other $i$. Thus $uE = T$.

It remains to show that the action is simply transitive. Since we already know it is transitive, it suffices to show that the stabilizer in $U$ of the standard descending flag $E$ is trivial. Now the stabilizer in $G$ of $E$ is just $\tran{B}$, so the intersection of the two stabilizers is $B \cap \tran{B}$, that is, the subgroup $A$ of diagonal matrices. Since $U$ intersects $A$ trivially, we conclude.
\end{proof}

\section{Positivity of triples of flags via total positivity of groups}

From the properties of the action in Proposition \ref{FLAGS:prop:action} we deduce that we can identify $U$ with the set of flags transverse to $F$ via the orbit of $E$:
\begin{align*}
 U & \cong \Omega_F \\
 u & \mapsto uE
\end{align*}
This allows us to give the following definition:

\begin{definition}
Let $T$ be a flag. The triple $(E,T,F)$ is called \emph{positive (relative to $E,F$)} if there exists $u \in U_{>0}$ such that $T=uE$.
\end{definition}

\begin{remark}
We follow the convention of Guichard-Wienhard \cite{GuichardWienhard18}. Some authors \cite{Bonahon_2014, Bonahon_2017, FockGoncharov06} permute the order and say $(E,F,T)$ is positive.
\end{remark}

\begin{example}
We consider flags in $\mathbb{R}^3$. Let
$$
u=\begin{pmatrix}
1 & 2 & 1 \\
0 & 1 & 1 \\
0 & 0 & 1
\end{pmatrix} , \quad \text{then} \quad
 uE=\left\{ \left\langle \begin{pmatrix}
1 \\ 1 \\ 1
\end{pmatrix} \right\rangle , \left\langle \begin{pmatrix}
1 \\ 1 \\ 1
\end{pmatrix}
,\begin{pmatrix}
2 \\ 1 \\0
\end{pmatrix}\right\rangle\right\}.
$$

Using the criterion from Example \ref{FLAGS:example:flagsinr3} it is easy to see that $u$ is totally positive, and thus $(E,uE,F)$ is a positive triple of flags. 
\end{example}

We recommend to solve the following exercise now.

\begin{exercise} \label{FLAGS:ex:1}
Consider the following flag in $\mathbb{R}^3$
$$
T = \left(\left\langle \begin{pmatrix}
1 \\ -1 \\ 1
\end{pmatrix}\right\rangle , \left\langle \begin{pmatrix}
1 \\ -1 \\ 1
\end{pmatrix},  \begin{pmatrix}
-2 \\ 1 \\ 0
\end{pmatrix}\right\rangle\right) .
$$
\begin{itemize}
    \item [(a)] Is there a $u \in U$ such that $(E,T,F) = (E,uE,F)$? 
    \item [(b)] Is there a $u \in U_{>0}$ such that $(E,T,F) = (E,uE,F)$? 
\end{itemize}
The solution to (a) is
$$
u = \begin{pmatrix}
1 & -2 & 1 \\
0 & 1 & -1 \\
0 & 0 & 1
\end{pmatrix}
$$
which has some negative entries, thus $u \notin U_{>0}$. Remember that $U$ acts on $\Omega_F$ simply transitively and thus the element $u \in U$ with $uE = T$ is unique, answering (b) negatively. \\

If $(E, T, F)$ is positive, then $F \pitchfork T$ (because the action preserves transversality: Proposition \ref{FLAGS:prop:action}) and the element $u \in U_{> 0}$ such that $uF = T$ is unique. Moreover, $E, F$ and $T$ are pairwise transverse. Something stronger is true:
\end{exercise}

\begin{lemma}
\label{FLAGS:lem:GW:generic}

Let $(E,  T, F)$ be a positive triple (relative to $E, F$). Then $(E, T, F)$ is generic.
\end{lemma}

For the proof, we introduce the following notation, which will also be used in the next section:

\begin{notation}
\label{FLAGS:not:uabc}

Let $g \in G$, and let $a, b, c \geq 0$ be such that $a + b + c = n$. Split $g$ as a block matrix as follows.
$$
g = \left(
\begin{array}{c|c|c}
(a \times a) & \, & \,  \\ \hline
\, & \, & (c \times c)  \\ \hline
\, & (b \times b) & \,
\end{array}\right)
$$
We denote by $g(a, b, c)$ the highlighted $(c \times c)$ block.
\end{notation}

While this may seem a little arbitrary, it will be clear through the course of the proof how these blocks - and their minors - appear naturally when considering triples of the form $(E, uE, F)$.

\begin{proof}
Let $u \in U_{> 0}$ be the unique element such that $T = uE$. We need to show that for every $a, b, c \geq 0$ such that $a + b + c = n$, it holds $F_a + E_b + T_c = \mathbb{R}^n$. Now we have $F_a = \langle e_1, \ldots, e_a \rangle$, $E_b := \langle e_n, \ldots, e_{n-b+1} \rangle$, and $T_c$ is the span of the last $c$ columns of the matrix $u$. Since $u$ is totally positive, $u(a, b, c)$ has positive determinant. In particular, the last $c$ columns of $u$ span $\langle e_{a+1}, \ldots, e_{n-b} \rangle$ modulo $F_a + E_b$. We conclude that $F_a + E_b + T_c = \mathbb{R}^n$.
\end{proof}

\begin{remark}
Therefore the set $\{ T \in \Flag(\R^n) : (E, T, F) \text{ is positive}\}$ is a subset of $\Omega_E \cap \Omega_F$. By Proposition \ref{FLAGS:prop:action}, it can be identified with $U_{>0}$, and in particular it is connected. In fact, Lusztig even proved that it is a connected component of $\Omega_E \cap \Omega_F$ \cite{Lusztig94}. This is not very relevant for our purposes so we would not get into detail.
\end{remark}

We can extend the definition from triples of this form to all triples by exploiting the high transitivity of the action.

\begin{definition}
A triple of flags $(F^1,F^2, F^3)$ is \emph{(GW)-positive}, if there exists $g \in G$ such that $(gF^1, gF^2, gF^3) = (E,uE,F)$ with $u \in U_{>0}$. This is the definition used in \cite{GuichardWienhard18}.
\end{definition}

One should be careful to notice that while $(E,T,F)$ is positive relative to $E,F$ implies that it is (GW)-positive, the converse does not hold, as the following exercise shows.

\begin{exercise}
Show that $(E,T,F)$ from Exercise \ref{FLAGS:ex:1} is (GW)-positive.
\end{exercise}
One can check $(E,T,F)$ is (GW)-positive using   
$
g = \begin{pmatrix}
-1 & & \\ & 1 & \\ & & -1
\end{pmatrix}.
$
In fact we have the following characterisation.

\begin{lemma}
\label{FLAGS:lem:GW:char:diagonal}

Let $u \in U$. Then $(E, uE, F)$ is (GW)-positive if and only if there exists a diagonal matrix $d$ such that $dud^{-1} \in U_{>0}$.
\end{lemma}

\begin{proof}
If the triple $(E, uE, F)$ is (GW)-positive, then there exists $d \in G$ such that $(dE, duE, dF) = (E, vE, F)$, where $v \in U_{>0}$. This implies that $d$ stabilizes both $E$ and $F$, and therefore it must be a diagonal matrix. Moreover, since $d$ stabilizes both $E$ and $F$, it is also true that $dud^{-1} E = vE$. Now it is easy to see that $dud^{-1} \in U$, and so by simple transitivity (Proposition \ref{FLAGS:prop:action}), it follows that $dud^{-1} = v \in U_{>0}$.

Conversely, if $dud^{-1} \in U_{>0}$, then $d \cdot (E, uE, F) = (dE, duE, dF) = (E, dud^{-1}E, F)$, which is positive relative to $E, F$.
\end{proof}

\begin{example} (Triples on the circle)
In $\mathbb{R}^2$, flags consist only of one subspace, and thus the flag variety can be identified with the projective line, which is homeomorphic to a circle. Note that every triple consisting of three distinct points is (GW)-positive, while only some triples are positive relative to $E,F$, namely those which are positively oriented in the usual sense.
\end{example}

\section{Positivity via triple ratios}

The notion of positivity introduced in the previous section has one disadvantage, namely that the standard ascending and descending flags seem to play a central role, even though the definition applies to all flags. In this section we introduce a set of \emph{coordinates} that parametrizes the $G$-orbits on generic triples of flags, and allows to give a definition of positivity that is more intrinsic. In this section, we follow \cite{Bonahon_2014}. \\

Let $(E, F, T)$ be a generic triple of flags. For $0 \leq i \leq n$, fix a generator $e^{(i)} \in \Lambda^i(E_i) \cong \mathbb{R}$, which we also see as an element of $\Lambda^i(\mathbb{R}^n)$. Similarly, fix generators $f^{(i)}$ and $t^{(i)}$ of $\Lambda^i(F_i)$ and $\Lambda^i(T_i)$. Now if $a + b + c = n$, then $e^{(a)} \wedge f^{(b)} \wedge t^{(c)} \in \Lambda^n(\mathbb{R}^n) \cong \mathbb{R}$. Note that since the triple is generic, this element is never $0$.

\begin{definition}
Fix an identification of $\Lambda^n(\mathbb{R}^n)$ with $\mathbb{R}$. Let $a, b, c \geq 1$ be such that $a + b + c = n$. Then we define the $(a, b, c)$-triple ratio of the triple $(E, F, T)$ as follows:
$$T_{abc}(E, F, T) := 
\frac{e^{(a+1)} \wedge f^{(b)} \wedge t^{(c-1)}}{e^{(a-1)} \wedge f^{(b)} \wedge t^{(c+1)}} \cdot
\frac{e^{(a)} \wedge f^{(b-1)} \wedge t^{(c+1)}}{e^{(a)} \wedge f^{(b+1)} \wedge t^{(c-1)}} \cdot
\frac{e^{(a-1)} \wedge f^{(b+1)} \wedge t^{(c)}}{e^{(a+1)} \wedge f^{(b-1)} \wedge t^{(c)}}.
$$
\end{definition}

Note that each $e^{(i)}$ appears exactly once on the numerator and exactly once on the denominator (if it does appear). Therefore the definition of $T_{abc}(E, F, T)$ is independent of the choice of the generators $e^{(i)}$. The same holds of course for the choices of $f^{(i)}$ and $t^{(i)}$. It is a little less clear that is independent of the choice of the identification $\Lambda^n(\mathbb{R}^n) \cong \mathbb{R}^n$:

\begin{lemma}
\label{FLAGS:lem:BD:orbits}

For all $\gamma \in G$, it holds $T_{abc}(E, F, T) = T_{abc}(\gamma E, \gamma F, \gamma T)$. In particular, $T_{abc}(E, F, T)$ is independent of the choice of the identification $\Lambda^n(\mathbb{R}^n) \cong \mathbb{R}^n$.
\end{lemma}

\begin{proof}
We choose $\gamma \cdot e^{(i)}$ as a generator of $\Lambda^i(\gamma E_i)$, which we are allowed to do by the previous remark. Then for every $a, b, c \geq 0$ such that $a + b + c = n$ it holds
\[\gamma \cdot e^{(a)} \wedge \gamma \cdot f^{(b)} \wedge \gamma \cdot t^{(c)} = \det(\gamma) \cdot e^{(a)} \wedge f^{(b)} \wedge t^{(c)}.\]
Since there are three such expressions in the numerator and three in the denominator, the $\det(\gamma)$ cancel out, and we obtain $T_{abc}(\gamma E, \gamma F, \gamma T) = T_{abc}(E, F, T)$.
\end{proof}

In fact, much more is true:

\begin{theorem}[Fock-Goncharov \cite{FockGoncharov06}]
Two triples $(E, F, T)$ and $(E', F', T')$ are in the same $G$-orbit if and only if their triple ratios coincide. Moreover, every possible set of non-zero triple ratios is attained by some triple.
\end{theorem}

The definition of positivity is then natural in terms of these coordinates:

\begin{definition}
The triple $(E, F, T)$ is called \emph{(BD)-positive} if all triple ratios are positive.
\end{definition}

In order to facilitate the computations, we prove the following:

\begin{lemma}
\label{FLAGS:lem:GW:implies:BD}

Let $a, b, c \geq 0$ be such that $a + b + c = n$. Let $F$ be the standard ascending, and $E$ the standard descending flag, and let $T := uE$ for some $u \in U$. Then, under a suitable choice of generators,
$$f^{(a)} \wedge e^{(b)} \wedge t^{(c)} = (-1)^{\lfloor b/2 \rfloor + \lfloor c/2 \rfloor + bc} \cdot \det u(a, b, c),$$
where $u(a, b, c)$ is as in Notation \ref{FLAGS:not:uabc}. Thus, for every $a, b, c \geq 1$ such that $a + b + c = n$, it holds:
$$T_{abc}(F, E, T) = 
\frac{\det u(a+1, b, c-1)}{\det u(a-1, b, c+1)} \cdot
\frac{\det u(a, b-1, c+1)}{\det u(a, b+1, c-1)} \cdot
\frac{\det u(a-1, b+1, c)}{\det u(a+1, b-1, c)}.
$$
\end{lemma}

\begin{proof}
We choose the following generators:
$$f^{(i)} := e_1 \wedge \cdots \wedge e_i; \quad e^{(i)} := e_n \wedge \cdots \wedge e_{n-i+1}; \quad t^{(i)} = u_n \wedge \cdots \wedge u_{n-i+1};$$
where $u_j$ denotes the $j$-th column of $u$. Then we use the formula for wedges in terms of determinants to obtain:
$$f^{(a)} \wedge e^{(b)} \wedge t^{(c)} = \det(e_1 \cdots e_a \mid e_n \cdots e_{n-b+1} \mid u_n \cdots u_{n - c + 1}).$$
Flipping the last $c$ columns changes this determinant by $(-1)^{\lfloor c/2 \rfloor}$, and gives
$$
\det \begin{pmatrix}
I_a & 0 & * \\
0 & 0 & u(a, b, c) \\
0 & J_b & *
\end{pmatrix}
=
\det \begin{pmatrix}
0 & u(a, b, c) \\
J_b & *
\end{pmatrix}
$$
where $J_b$ is the $(b \times b)$ matrix with $1$ on the antidiagonal and $0$ elsewhere. Next, we compute
$$
\det \begin{pmatrix}
0 & u(a, b, c) \\
J_b & *
\end{pmatrix} = 
(-1)^{b+c+1} \det \begin{pmatrix}
0 & u(a, b, c) \\
J_{b-1} & *
\end{pmatrix} =
(-1)^{b+c+1} \cdot (-1)^{b+c} \cdot \det \begin{pmatrix}
0 & u(a, b, c) \\
J_{b-2} & *
\end{pmatrix}
$$
$$
= \cdots = \left( \prod\limits_{j = c+2}^{b+c+1}(-1)^j \right) \det u(a,b,c) = (-1)^{\lfloor b/2 \rfloor} \cdot (-1)^{bc} \cdot \det u(a, b, c).
$$
The last equality follows from the fact that there are $\lfloor b/2 \rfloor$ odd numbers in $\{ c+2, \ldots, b+c+1 \}$, unless $b$ and $c$ are both odd, in which case there are $\lceil b/2 \rceil = - \lfloor b/2 \rfloor$. We conclude that
$$f^{(a)} \wedge e^{(b)} \wedge t^{(c)} = (-1)^{\lfloor b/2 \rfloor} \cdot (-1)^{\lfloor c/2 \rfloor} \cdot (-1)^{bc} \cdot \det u(a, b, c).$$

Using this formula to compute $T_{abc}(F, E, T)$, the signs $(-1)^{\lfloor b/2 \rfloor}$ and $(-1)^{\lfloor c/2 \rfloor}$ each appear exactly once in the numerator and once in the denominator, so they all cancel out. As for the terms of the form $(-1)^{bc}$, these also cancel out, since the parity of $bc$ does not change when each of $b$ and $c$ is changed by $\pm 2$.
\end{proof}

\begin{example}
\label{FLAGS:ex:T111}

Let $n = 3$, and let consider $(F, E, uE)$. There is only one triple ratio to consider, namely $T_{111}$. By the previous formula,
$$
u = 
\begin{pmatrix}
1 & x & y \\
0 & 1 & z \\
0 & 0 & 1
\end{pmatrix}
\Rightarrow T_{111}(F, E, uE) = \frac{1}{xz - y} \cdot \frac{1}{1} \cdot \frac{y}{1} = \frac{y}{xz - y}.
$$
Therefore $(F, E, uE)$ is (BD)-positive if and only if $y$ and $xz - y$ have the same sign, and are non-zero (the last condition ensures that $(F, E, uE)$ is generic).
\end{example}

\section{Equivalence of the two definitions}

The two definitions may look different at first sight, but this is not the case:

\begin{theorem}[Fock-Goncharov \cite{FockGoncharov06}]
\label{FLAGS:thm:GW:equiv:BD}

A triple $(E, T, F)$ is (GW)-positive if and only if it is (BD)-positive.
\end{theorem}

We start by remarking the following, which accounts for the fact that the two definitions use different conventions in terms of orderings of the triples (as is apparent from Lemma \ref{FLAGS:lem:GW:implies:BD}):

\begin{lemma}
The notion of (BD)-positivity is independent of the ordering of the triple.
\end{lemma}

\begin{proof}
By \cite[Lemma 5]{Bonahon_2014}, it holds:
$$T_{abc}(E, F, T) = T_{bca}(F, T, E) = T_{bac}(F, E, T)^{-1};$$
this follows from elementary computations on the wedges appearing in the definition. The first equality shows that, upon applying a cyclic permutation, the set of triple ratios itself is permuted, and in particular the positivity of each triple ratio is not affected. The same holds under a transposition flipping the first two flags: the second equality shows that then the set of triple ratios is permuted and inverted, but again the positivity of each triple ratio is not affected.
\end{proof}

Therefore it suffices to show that $(E, T, F)$ is (GW)-positive if and only if $(F, E, T)$ is (BD)-positive. 

We start with the easier implication:

\begin{proof}[Proof of (GW) $\Rightarrow$ (BD)]
Suppose that $(E, T, F)$ is (GW)-positive. Note that both notions are not affected by the action of $G$: this holds by definition for (GW)-positivity, and follows from Lemma \ref{FLAGS:lem:BD:orbits} for (BD)-positivity. Therefore we may assume that $E$ is the standard descending flag, that $F$ is the standard ascending flag, and that $T = uE$, where $u \in U_{>0}$. We now need to show that the triple $(F, E, uE)$ is (BD)-positive.

First, it is generic, by Lemma \ref{FLAGS:lem:GW:generic}. Moreover, each triple ratio $T_{abc}$ is expressed in terms of minors of $u$, by Lemma \ref{FLAGS:lem:GW:implies:BD}. But $u$ is totally positive, so each of these minors is positive. It follows that each triple ratio is positive and we conclude.
\end{proof}

The other implication is more involved, and goes beyond the scope of this exposition, so we only present the proof in case $n = 3$. Recall that in this case there is only one triple ratio involved (Example \ref{FLAGS:ex:T111}), so it is surprising that from this one can recover a totally positive matrix, something which imposes four positivity conditions (Example \ref{FLAGS:example:flagsinr3}). This apparent incompatibility of the two definitions is resolved by remembering that the definition of (GW)-positivity actually allows for more matrices than just totally positive ones: this was made explicit in Lemma \ref{FLAGS:lem:GW:char:diagonal}, and it will be crucial in the proof.

\begin{proof}[Proof of (BD) $\Rightarrow$ (GW) for $n = 3$]
Suppose that $(F, E, T)$ is (BD)-positive. Again, we note that both notions are not affected by the action of $G$, which is transitive on pair of transverse flags by Proposition \ref{FLAGS:prop:action}, so we may assume that $E$ is the standard descending flag and $F$ is the standard ascending flag. Since this triple is generic, in particular $T$ is transverse to $F$, and so by Proposition \ref{FLAGS:prop:action} there exists a unique $u \in U$ such that $T = uE$. Our goal is to show that there exists a diagonal matrix $d$ such that $dud^{-1} \in U_{> 0}$, which will conclude the proof by Lemma \ref{FLAGS:lem:GW:char:diagonal}. \\

Now the condition that $(F, E, uE)$ is (BD)-positive amounts to the positivity of the triple ratio $T_{111}(F, E, uE)$. By Example \ref{FLAGS:ex:T111}, we have:
$$
T_{111}(F, E, uE) = \frac{y}{xz - y}; \quad \text{where} \quad u = 
\begin{pmatrix}
1 & x & y \\
0 & 1 & z \\
0 & 0 & 1
\end{pmatrix}.
$$
In other words, $y$ and $(xz - y)$ are non-zero and have the same sign. We split into two cases.

Suppose that $y > 0$. Then $xz > y > 0$, so $x$ and $z$ are non-zero and have the same sign. If they are both positive, then $u$ is totally positive by Example \ref{FLAGS:example:flagsinr3} and we are done. Otherwise, if $x$ and $z$ are both negative, we let $d := \operatorname{diag}(1, -1, 1)$, and then
$$dud^{-1} = 
\begin{pmatrix}
1 & -x & y \\
0 & 1 & -z \\
0 & 0 & 1
\end{pmatrix}.
$$
This matrix is totally positive by Example \ref{FLAGS:example:flagsinr3}: indeed, all relevant entries are positive, and the only relevant other minor is $(-x)(-z) - y = xz - y > 0$.

Suppose instead that $y < 0$. Then $xz < y < 0$, so $x$ and $z$ are non-zero and have opposite signs. Let $d := (\operatorname{sgn}(x), 1, \operatorname{sgn}(z))$. Then
$$dud^{-1} = 
\begin{pmatrix}
1 & \operatorname{sgn}(x) x & \operatorname{sgn}(x) \operatorname{sgn}(z) y \\
0 & 1 & \operatorname{sgn}(z) z \\
0 & 0 & 1
\end{pmatrix} = 
\begin{pmatrix}
1 & |x| & - y \\
0 & 1 & |z| \\
0 & 0 & 1
\end{pmatrix}.
$$
This matrix is totally positive by Example \ref{FLAGS:example:flagsinr3}: indeed, all relevant entries are positive, and the only relevant other minor is $|xz| - (-y) = -(xz - y) > 0$, since we know that $xz < 0$.
\end{proof}

\section{Generalizations}

It is possible to generalize the notion of positivity of triples to quadruples and more generally to $n$-tuples. This may be done using \emph{quadruple ratios} as in \cite{Bonahon_2014} or by checking positivity of some relevant triples, as in \cite{GuichardWienhard18}. Another way to generalize positivity of triples it to consider generalized flag varieties, which we will do in this section. Yet another generalization is given by taking non-full flags. \\

So far we considered $G=\operatorname{GL}(n,\mathbb{R})$ and the Borel subgroup of upper triangular matrices $B < G$. We defined a Borel subgroup to be a connected subgroup with $\operatorname{Lie}(B) = \fraka \oplus_{\alpha \in \Sigma^+} \frakg_\alpha$. Note that 
\[
N_G (B) = \left\{ g \in G \colon gBg^{-1} \subseteq B \right\} = B. 
\]
We call all subgroups of the form $gBg^{-1}$ for $g\in G$ \emph{Borel subgroups}. Then $G$ acts by conjugation on the set of Borel subgroups $\{gBg^{-1} \}$. The group $G$ also acts transitively on the flags by multiplication, and we get a $G$-equivariant map 
\begin{align*}
    \{\text{Borel subgroups } gBg^{-1}\} &\to \Flag(\R^n) \\
    gBg^{-1} &\mapsto gF\\
    \operatorname{Stab}_G(T) & \mapsfrom T
\end{align*}
Thus we may identify
\[
\Flag(\R^n) \cong G/\operatorname{Stab}_G(F) = G/N_G(B) = G/B.
\]

The notion of flags in $\mathbb{R}^n$ is tightly intertwined with $\GL(n,\mathbb{R})$, and it is not obvious how to generalize flags. However Borel subgroups are available more generally, so we will use them to define our ``generalized flags''. Let now $G$ be a simple split real linear algebraic Lie group, for example $\SL(n,\mathbb{R})$ or $\Sp(4,\R)$. General theory \cite{borel} tells us, that there exists a Borel subgroup $B\subseteq G$, (defined as maximal among the connected solvable subgroups), which satisfies
\[
N_G (B) = \left\{ g \in G : gBg^{-1} \subseteq B \right\} = B. 
\]
We note that $G$ acts on the set of Borel subgroups by $g.B = gBg^{-1}$. It is a theorem \cite[IV.11.1]{borel} that any two Borel subgroups are conjugate, therefore the action of $G$ on the set of Borel subgroups is transitive.  
\begin{definition}
The \emph{generalized flag variety} is the set of Borel subgroups $G/B$.
\end{definition}

For positivity of triples of elements in the generalized flag variety we need an opposite Borel group $\tran{B}$ and a notion of positivity $U_{>0}$ for unipotent subgroups $U\subseteq B$. Previous talks have constructed these objects, see Chapter \ref{section:LusztigTotalPositivity}.

\begin{definition}
A triple $(F_1, F_2, F_3) \in (G/B)^3$ is \emph{positive} if there is a $g\in G$ and $u \in U_{>0}$ such that $g(F_1,F_2,F_3) = ( \tran{B} , u\tran{B}, B)$.
\end{definition}

We note that for $G=\GL(n,\mathbb{R})$, this is the definition of (GW)-positivity. 

\section{The flag variety for the symplectic group}

We work out the case $\Sp(4,\R)$ and give an interpretation to the generalized flag variety.
Recall that $\Sp(4,\R) = \{ A \in \GL(4,\R): AJ_{2,2}\tran{A} = J_{2,2} \}$ for $J_{2,2}=\begin{pmatrix}
 & I_2 \\ -I_2 &\end{pmatrix}$. A standard Borel subgroup $B$ and unipotent subgroup $U$ are given by
 \begin{align*}
     B &= \left\{ \begin{pmatrix}
     X & Y \\ 0 & \tran{X}^{-1}
     \end{pmatrix}  \in \GL(4,\R) : X = \begin{pmatrix}
     \star & \star \\ 0 & \star
     \end{pmatrix} , X\tran{Y} = \tran{Y} X
     \right\} \\
     U &= \left\{ \begin{pmatrix}
     X & Y \\ 0 & \tran{X}^{-1}
     \end{pmatrix}  \in \GL(4,\R) : X = \begin{pmatrix}
     1 & \star \\ 0 & 1
     \end{pmatrix}
     \right\} \\
      U_{>0} &= \left\{ \begin{pmatrix}
     1 & b+d & -bcd & bc \\
      0 & 1 & -(ad + ab + cd) & a+c \\
      && 1 & 0\\
      && -(b+d) & 1
     \end{pmatrix}  \in \GL(4,\R)
     \right\} \\
 \end{align*}

 We can interpret $\mathbb{R}^4 = \langle x_1, x_2, y_1, y_2\rangle$ as a symplectic vector space with the symplectic 2-form $\omega $ defined by
 \begin{align*}
     \omega(x_i, y_j) &= \delta_{ij} = - \omega (y_j, x_i), \\
     \omega(x_i, x_j) &= 0 = \omega(y_i, y_j).
 \end{align*}
 \begin{definition}
 A sub-vector space $V \subseteq \mathbb{R}^4$ is \emph{isotropic} if $\omega(V,V) = 0$.
 \end{definition}
 \begin{example}
 The sub-vector space $\langle x_1 \rangle$ is isotropic. $\langle x_1, x_2\rangle$ is maximal isotropic.
 \end{example}
 \begin{definition}
 A sub-vector space that is maximal isotropic is called a \emph{Lagrangian}.
 \end{definition}
 It can be shown that the generalized flag variety $\Sp(4,\R)/B$ can be identified with the isotropic complete flags in $\mathbb{R}^4$ with the standard symplectic form. In particular we have
 \begin{align*}
     \left\{ \text{Borel subgroups of $\Sp(4,\R)$}\right\} & \cong \left\{ \text{isotropic complete flags in $(\mathbb{R}^4,\omega)$}\right\} \\
     B &\mapsto \left\{\langle x_1\rangle, \langle x_1, x_2\rangle\right\} =:F\\
     \tran{B} &\mapsto \left\{\langle y_1\rangle, \langle y_1, y_2\rangle\right\} =:E .
 \end{align*}
 Note that $BE = E$ and $\tran{B} F=F$. This justifies the following definition.
 \begin{definition}
 A \emph{symplectic flag} is a complete isotropic flag.
 \end{definition}

\newpage

\thispagestyle{empty}

\chapter[The Maslov Index]{The Maslov Index\\ {\Large\textnormal{\textit{by Samuel Bronstein, Ilia Smilga}}}}
\addtocontents{toc}{\quad\quad\quad \textit{Samuel Bronstein, Ilia Smilga}\par}

All the definitions and theorems of this session come from Lion and Vergne's book: \cite{lio13}.

\section{Maslov index in the symplectic group}

Consider $\Sp(2n,\R)$ the subgroup of $\GL(2n,\R)$ of elements preserving the symplectic form $\omega$
defined by $w(X,Y)=\tran{X}J_{n,n}Y$ for $X,Y\in\R^{2n}$ with:
\begin{equation*}
    J_{n,n}=\left(
    \begin{array}{cc}
    0&I_n\\-I_n&0
    \end{array}\right).
\end{equation*}
A basis of $\R^{2n}$ in which $w$ translates as $J_{n,n}$ is called a \emph{symplectic basis}.
In this regard, an endomorphism of $\R^{2n}$ belongs to $\Sp(2n,\R)$ if and only if its sends a symplectic basis onto a symplectic basis.
We are interested here in the maximal isotropic subspaces, i.e.\ the Lagrangians:
\begin{definition}[Lagrangian subspace]
$L\subset\R^{2n}$ is \emph{Lagrangian} if it is $n$-dimensional and $w|_{L\times L}=0$.
The set of Lagrangian subspaces will be denoted $\mathcal{L}$.
\end{definition}
We refer to the previous sections for the fact that $\mathcal{L}$ is actually the Shilov boundary of the Siegel disk, which is the symmetric space of $\Sp(2n,\R)$. As such, there is a notion of positivity of triples on $\mathcal{L}$, of which we here bring another description.
This notion of positivity being invariant by conjugation, it is natural to study the orbits of the
$\Sp(2n,\R)$ action on triples of Lagrangians.
\begin{proposition}
\begin{itemize}
    \item
    The set $\mathcal{L}$ is homogeneous: $\mathcal{L}=\Sp(2n,\R)/T$,
    where:
    \begin{equation}
    T=\left\{\left(
    \begin{array}{cc}
    M&SM\\0&^tM^{-1}
    \end{array}\right): S\in \Sym(n,\R),M\in\GL(n,\R)\right\}
    \end{equation}
    \item
    The $\Sp(2n,\R)$ action is transitive on the set of pairs of transverse Lagrangians,
    denoted $\mathcal{L}^{2,T}$. We have the identification:
    \begin{equation}
    \mathcal{L}^{2,T}=\Sp(2n,\R)/\GL(n,\R)
    \end{equation}
\end{itemize}
\end{proposition}
\begin{proof}
Of course the action by conjugation on the considered sets is an action by isometries,
the main trouble is to convince oneself of the transitivity of those action:
For a pair of transverse Lagrangian, concatenating two basis of the Lagrangians yields a symplectic basis of $\R^{2n}$.
As we always can send a symplectic basis onto another one via an element of $\Sp(2n,\R)$,
it follows that the action is transitive.
We let the reader compute the stabilizers. The stabilizer of the standard pair of transverse Lagrangian is the subset of $T$ defined by $S=0$.
\end{proof}
More generally, one could have shown that the orbits of the conjugation action on pairs
of Lagrangians are characterized by the degree of the intersection of the Lagrangians.
Keep in mind that, given $(L_1,L_2)$ a transverse pair of Lagrangians, $w$
gives an isomorphism $L_2\approx L_1^\ast$.
That clarification done, let us give a first definition of the Maslov index, due to Kashiwara:
\begin{definition}[Maslov Index]
For $(L_1,L_2,L_3)$ a triple of Lagrangians, the \emph{Maslov index} is denoted
$\tau(L_1,L_2,L_3)$ and defined as the signature of the quadratic form $Q$
on $L_1\times L_2\times L_3$:
\begin{equation}
Q(x_1,x_2,x_3)=w(x_1,x_2)+w(x_2,x_3)+w(x_3,x_1)
\end{equation}
\end{definition}
There are several nice properties that we can easily deduce from this definition:
The Maslov index is invariant by conjugation by an element of $\Sp(2n,\R)$, it is an integer
between $-3n$ and $3n$, it is alternating.
Conveniently, we also have the following property, called the \emph{chain rule}.
\begin{proposition}[Chain rule]
For $(L_1,L_2,L_3,L_4)$ Lagrangians, we have:
\begin{equation}
\tau(L_1,L_2,L_3)=\tau(L_1,L_2,L_4)+\tau(L_2,L_3,L_4)+\tau(L_1,L_3,L_4)
\end{equation}
\end{proposition}
This last property leads to think of the Maslov index as some kind of cocycle, but it is not
the topic we are interested in here.
For transverse triples, we have another possible definition, which is easier to handle:
\begin{proposition}
For a pairwise transverse triple of Lagrangians $(L_1,L_2,L_3)$, its Maslov index is the
signature of the quadratic form $S$ on $L_2$ defined by:
\begin{equation}
S(x_2)=\omega(p_{13}x_2,p_{31}x_2)
\end{equation}
where $p_{ij}$ denotes the projection on $L_i$ with kernel $L_j$, well defined because of
transversality.
\end{proposition}
\begin{proof}
A small computation allows us to show:
\begin{flalign}
    Q(x_1,x_2,x_3)&=B(p_{13}x_2,p_{31}x_2)-B(x_1-p_{13}x_2,x_3-p_{31}x_2)\\
    &=S(y_2)-B(y_1,y_3)
\end{flalign}
with $(y_1,y_2,y_3)$ another set of coordinates. As such, it is clear that the signature
of $Q$ equals the signature of $S$.
\end{proof}
\begin{remark}
\begin{itemize}
    \item 
This proposition is actually true as soon as $L_1\cap L_3=\emptyset$, we do not need that
$L_2$ is transverse to the others. However, $L_2$ being transverse to $L_1$ and $L_3$ is
equivalent to asking that $S$ is non-degenerate.
\item
This gives a more precise bound on the Maslov index: For transverse triples, it is less or
equal than $n$.
\item
We can check that $B(p_{13}x,p_{31}y)$ is symmetric on $L_2$,
 so this is is the bilinear form associated to $S$.
\end{itemize}
\end{remark}
It remains to check that the Maslov index actually distinguishes the orbits of the $\Sp(2n,\R)$ action on the set of transverse triples.
\begin{theorem}
For $(L_1,L_2,L_3)$ transverse triple, there is a symplectic basis
$(p_1,\ldots,p_n,q_1,\ldots, q_n)$  and $0\leq k \leq n$ such that:
\begin{flalign}
L_1&=\langle p_1,\ldots,p_n \rangle\\
L_2&=\langle q_1,\ldots,q_n \rangle\\
L_3&=\langle p_1+\epsilon_1 q_1,\ldots,p_n+\epsilon_n q_n\rangle
\end{flalign}
where $\epsilon_i=1$ if $i\leq k$ and $\epsilon_i=-1$ otherwise.
Then the Maslov index is:
\begin{equation}
    \tau(L_1,L_2,L_3)=n-2k
\end{equation}
\end{theorem}
\begin{proof}
the quadratic form $S$ being non-degenerate, one can choose $q_1,\ldots,q_n$
basis of $L_2$ such that $S(q_i,q_j)=-\epsilon_j\delta_{ij}$.
But $L_1$and $L_2$ being transverse, the symplectic form $\omega$ yields an isomorphism
$L_1=L_2^\ast$. So take the dual basis of $(q_i)$ to be a basis of $L_1$, denoted $(p_i)$.

Finally, fix $z_i=p_{13}q_i$. Check that for all $j$, $B(z_i,q_j)=-\epsilon_j\delta_{ij}$
Thus (by uniqueness of the dual basis) $z_i=-\epsilon_i p_i$. This means that $(q_i+\epsilon_i p_i)$ is a basis of $L_3$, and so is $(p_i+\epsilon_i q_i)$.
\end{proof}
It is clear from this theorem that the maximality of the Maslov index corresponds to 
the maximality previously defined on $\Sp(2n,\R)$. This is the main motivation
for generalizing the Maslov index to Hermitian Lie groups of tube type.
\begin{remark}[n=1]
For $n=1$, the set of Lagrangians correspond to the circle seen as the boundary of the Poincar\'e Disk.
In this setting, the Maslov index distinguishes the two orbits, on which it assigns $1$ or $-1$.
\end{remark}

\section{Maslov index in Hermitian Lie groups of tube type}

To generalize this notion, one needs the notion of Euclidean Jordan algebra.
\begin{definition}
A \emph{Euclidean Jordan algebra} is the data of:
\begin{itemize}
    \item a real vector space $V$,
    \item a scalar product on $V$ denoted $\langle\cdot,\cdot\rangle$,
    \item a bilinear product $\circ: V\times V\rightarrow V$,
\end{itemize}
such that the following is verified:
\begin{itemize}
    \item there is $e\in V$ such that for any $x\in V$, $x\circ e=e\circ x=x$,
    \item $x\circ y=y\circ x$,
    \item $x\circ(x^2\circ y)=x^2\circ(x\circ y)$,
    \item $\langle x\circ u,v\rangle=\langle u,x\circ v\rangle$.
\end{itemize}
If $V$ is a Euclidean Jordan algebra, we also define $\Omega$ the \emph{open cone of squares} as the interior of $\{v^2: v\in V\}$.
\end{definition}
A model-example for such a structure is $V=\Sym(n,\mathbb R)$ with the natural operations and the open cone of squares coincides with the set
of positive definite matrices.

This is enough to define what is a Hermitian group of tube type:
\begin{definition}[Hermitian group of tube type]
A Hermitian Lie Group $G$ is said to be \emph{of tube type} if there is a Euclidean Jordan algebra $V$ with open cone of squares $\Omega$ such that the homogeneous space $G/K$ identifies with \emph{the tube domain}
$T_{\Omega}:=\{X+iY:X\in V,Y\in\Omega\}$.
One can rephrase that by saying that $G$ acts transitively on $T_{\Omega}$
with stabilizers conjugates of $K$.
\end{definition}
\begin{example}
For $G=\Sp(2n,\mathbb R)$, one can consider $V=\Sym(n,\mathbb R)$
with the action on $T_{\Omega}$ being:
\begin{equation}
    \left(\begin{array}{cc}
    A&B\\C&D
    \end{array}\right)
\cdot Z=(AZ+B)(CZ+D)^{-1}
\end{equation}
\end{example}
Conveniently, we have a full classification of Hermitian Lie groups of tube type into four infinite families and one exceptional Lie group:
\begin{center}
    \begin{tabular}{|c|c|}
    \hline
    G&V\\
    \hline
    $\Sp(2n,\mathbb R)$&$\Sym(n,\mathbb R)$\\
    $\SU(n,n)$&$\mathrm{Herm}(n,\C)$\\
    $\SO^\ast(4n)$&$\mathrm{Herm}(n,\mathbb H)$\\
    $E_7(-25)$&$\mathrm{Herm}(3,\mathbb O)$\\
    $\SO(2,n)$&$\mathbb R\times\mathbb R^{n-1}$ with $(a,v)\circ(b,w)=(ab+\langle v,w\rangle,aw+bv)$\\
    \hline
    \end{tabular}
\end{center}
The infinite boundary on which we want to consider the $G$-action comes from Hermitian Lie group theory, we give a vague definition here:
\begin{definition}[Shilov Boundary]
For $G$ Hermitian Lie group of tube type, the \emph{Shilov boundary} $S$
is a certain flag variety of $G$.
For $G=\Sp(2n,\mathbb R)$, $S$ corresponds to the set of Lagrangian subspaces.
\end{definition}
Actually, the action on $S$ is quite well understood:
\begin{proposition}
There is a point $\infty\in S$ such that:
\begin{itemize}
    \item $\{s\in S|s\pitchfork\infty\}\approx V\subset \overline{T}_\Omega$
    \item $\stab_G(\infty)=\{g \from T_\Omega\rightarrow T_\Omega :  g(Z)=AZ+B \text{ with } A\in \GL(V)\text{ and } B\in V\}$.
    In particular the $G$-action on the set of transverse pairs is transitive.
    \item
    $\{s\in S|s\pitchfork\infty,s\pitchfork 0\}\approx V^\ast=\{v\in V : \det v\neq 0\}$
    \item
    The stabilizer of a pair $(0,\infty)$ is $G(\Omega)=\{g\in \GL(V): g(\Omega)=\Omega\}$.
\end{itemize}
\end{proposition}
As a corollary, the $G$-orbits of transverse triples correspond to $G(\Omega)$-orbits in $V^\ast$.
The main tool that we use is the spectral decomposition for Euclidean Jordan algebras:
\begin{proposition}[Spectral decomposition]
Let $V$ be a Euclidean Jordan algebra. Then for any $v\in V$,
there is $c_1,\ldots,c_r$ a \emph{Jordan frame}, i.e.\ such that:
\begin{itemize}
    \item $\forall i$ we have $c_i^2=c_i$ and $c_i\neq c_j+c_k$,
    \item $\forall i\neq j$ we have $c_ic_j=0$,
    \item $\sum c_i=e$,
\end{itemize}
and there are unique eigenvalues $\lambda_1,\ldots,\lambda_r$ such that $v=\sum \lambda_i c_i$.
\end{proposition}
\begin{remark}
As $\det v=\prod \lambda_i$, $v$ is invertible if and only if $\lambda_i\neq 0$ for all $i\neq 0$.
Then we can define the signature of $v$ in the following way:
$\textrm{sgn}(v)=|\{i: \lambda_i>0\}|-|\{i: \lambda_i<0\}|$.
\end{remark}
It is a well known fact that the signature is invariant under the $G(\Omega)$-action on $V^\ast$, and then the Maslov index is the corresponding invariant for the action on transverse triples.

\newpage

%%%%%%%%%%%%%%%%%%%%%%%%%%%%%%%%%%%%%%%%%%%%%%%%%%%%%%%%%%%%%%%%%%%%%%%%%%%%%%%%%%%%%%%%%%%%%%%%%%%%%%%%

\part[Theta-Positivity]{Theta-Positivity
\textnormal{
\begin{minipage}[c]{15cm}
\begin{center}
    \vspace{2cm}
    {\Large Clarence Kineider}\\
    \vspace{-4mm}
    {\large \textit{ENS Rennes}}\\
    \vspace{.5cm}
    {\Large Pierre-Louis Blayac}\\
    \vspace{-4mm}
    {\large \textit{Institut des Hautes Études Scientifiques}}\\
    \vspace{.5cm}
    {\Large Balthazar Fléchelles}\\
    \vspace{-4mm}
    {\large \textit{Institut des Hautes Études Scientifiques}}\\
    \vspace{.5cm}
    {\Large Dani Kaufman}\\
    \vspace{-4mm}
    {\large \textit{Ruprecht-Karls-Universit\"at Heidelberg}}\\
    \vspace{.5cm}
     {\Large Fernando Camacho Cadena}\\
    \vspace{-4mm}
    {\large \textit{Ruprecht-Karls-Universit\"at Heidelberg}}\\
    \vspace{.5cm}
     {\Large Merik Niemeyer}\\
    \vspace{-4mm}
    {\large \textit{Ruprecht-Karls-Universit\"at Heidelberg}}
\end{center}
\end{minipage}
}}\label{chap4}

\thispagestyle{empty}

\chapter[Introduction to Theta-Positivity]{Introduction to Theta-Positivity\\ {\Large\textnormal{\textit{by Clarence Kineider}}}}
\addtocontents{toc}{\quad\quad\quad \textit{Clarence Kineider}\par}
	
	Total positivity in split real Lie groups and maximality in Hermitian Lie groups of tube type have a lot in common.
	The notion of $\Theta$-positivity will unify those two cases under the same definition. Moreover, we have a complete description of all semisimple Lie groups admitting a $\Theta$-positive structure which includes two new family of groups, namely $\SO(p,q)$, $p\neq q$ and exceptional groups with restricted root system of type $F_4$.
	
	In this chapter we will do every computation in $\SL(n,\R)$.
	
	Let $G$ a real semisimple Lie group with finite center and $\mfg$ its Lie algebra. Let $K$ be a maximal compact subgroup of $G$, $\mfk$ its Lie algebra, and $\mfk^\perp$ the orthogonal of $\mfk$ with respect to the Killing form.  Let $\mfa$ be a maximal abelian Cartan subspace of $\mfk^\perp$, and denote by $\Sigma = \Sigma(\mfg,\mfa)$ the system of restricted roots (see Part \ref{chap2} for all necessary definitions). 
	 Choose $\Delta\subset \Sigma$ a set of simple roots and let $\Sigma^+$ be the set of positive roots with respect to $\Delta$.
	 
	 \begin{example}
	 	For $G=\SL(n,\R)$, $K=\SO(n)$ and $\mfk$ is the set of skew-symmetric matrices.
	 	Then $\mfk^\perp$ is the set of traceless symmetric matrices, and $\mfa$ is the set of traceless diagonal matrices. The restricted root system is $\Sigma = \left\lbrace \alpha_{ij} := \epsilon_i-\epsilon_j : i,j \in \left\lbrace 1,\dots,n\right\rbrace, i\neq j \right\rbrace$, where $\epsilon_i :(a_{pq})_{p,q \in \left\lbrace 1,\dots,n\right\rbrace} \mapsto a_{ii}$. The standard choice of simple roots is $\Delta = (\alpha_{i, i+1})_{i\in\left\lbrace 1,\dots,n-1\right\rbrace }$, and the set of positive roots associated is $\Sigma^+ = (\alpha_{ij})_{i<j}$. The weight space $\mfg_{\alpha_{ij}}$ is the $1$-dimensional subspace of $\mathfrak{sl}(n,\R)$ spanned by $E_{ij}$.
	 \end{example}
 
 	Now let $\Theta\subset\Delta$ be any subset of simple roots ($\Theta$ can be empty). Let $\Sigma_\Theta^+=\Sigma^+\setminus\Span(\Delta\setminus\Theta)$, and let
 	\[ \mfu_\Theta = \sum_{\alpha\in\Sigma_\Theta^+} \mfg_\alpha,~~~ \mfu_\Theta^{\mathrm{opp}} = \sum_{\alpha\in\Sigma_\Theta^+}\mfg_{-\alpha} \]
 	and \[ \mfl_\Theta = \mfg_0 \oplus \sum_{\alpha\in\Sigma^+\cap\Span(\Delta\setminus\Theta)}\mfg_\alpha \oplus\mfg_{-\alpha}.\]
 	As Lie algebra, $\mfg$ splits into $\mfg = \mfu_\Theta^{\mathrm{opp}}\oplus \mfl_\Theta \oplus \mfu_\Theta$. Moreover $\mfl_\Theta$ acts on $\mfu_\Theta$ by the Lie bracket, and we would like to understand this action and its properties.
 	The group $G$ acts on $\mfg$ by the adjoint representation, denote by $P_\Theta$ the normalizer in $G$ of $\mfu_\Theta$ and $P_\Theta^{\mathrm{opp}}$ the normalizer of $\mfu_\Theta^{\mathrm{opp}}$.
 	Those are the standard parabolic subgroups associated with $\Theta$, and the corresponding Levi subgroup is $L_\Theta = P_\Theta\cap P_\Theta^{\mathrm{opp}}$. Let us complete the list with $U_\Theta = \exp(\mfu_\Theta)$ and $\mfp_\Theta = Lie(P_\Theta)$. 
 	\begin{proposition}
 		$\mfl_\Theta$ is the Lie algebra of $L_\Theta$ and $\mfp_\Theta = \mfl_\Theta\oplus\mfu_\Theta$.
 	\end{proposition}

 		\newcommand{\UP}[2]{\makebox[0pt]{\smash{\raisebox{1.5em}{$\phantom{#2}#1$}}}#2}
 		\newcommand{\LF}[1]{\makebox[0pt]{$#1$\hspace{4.5em}}}
 		\renewcommand{\arraystretch}{1.3}
 	\begin{example}
 		For $G = \SL(n,\R)$ let us choose $\Theta = \left\lbrace \alpha_{i, i+1}, \alpha_{j, j+1}\right\rbrace $ with $1\leq i<j<n$. Then the Lie algebra $\mathfrak{sl}(n,\R)$ admits the following block decomposition :\\
		\[
 			\left(
 			\begin{array}{c@{}ccc|ccc|ccc}
 				           &  *  &  \dots  &  \UP{i}{*}  &  \UP{~~i+1}{*}  &  \dots  &  \UP{j}{*}  &  \UP{~~j+1}{*}  &  \dots  &  *  \\
 					       & \vdots & \ddots & \vdots & \vdots & \ddots & \vdots & \vdots & \ddots & \vdots \\ 
 				\LF{i} & *    &  \dots &   *               & *                     &\dots     &*                  &*                      &\dots     &* \\ \hline
 				 \LF{i+1}          &  *  &  \dots  &  *  & *  &  \dots  &  *  & *  &  \dots  &  *  \\
 				           & \vdots & \ddots & \vdots & \vdots & \ddots & \vdots & \vdots & \ddots & \vdots \\
 				\LF{j} & *    &  \dots &   *               & *                     &\dots     &*                  &*                      &\dots     &* \\ \hline
 				 \LF{j+1}  &  *  &  \dots  &  *  & * &  \dots  &  *  & *  &  \dots  &  *  \\
 				           & \vdots & \ddots & \vdots & \vdots & \ddots & \vdots & \vdots & \ddots & \vdots \\
 			               & *    &  \dots &   *               & *                     &\dots     &*                  &*                      &\dots     &* \\
 			\end{array}\right), ~~\tr = 0.
 		\]
 		\renewcommand{\arraystretch}{1}
 		
 		In the following examples we refer to this block decomposition when speaking about matrices defined block-wise.
 		The Levi subalgebra $\mfl_\Theta$ is the set of traceless block-diagonal matrices, the algebra $\mfu_\Theta$ is the set of block strictly upper triangular matrices
 		and $\mfp_\Theta=\mfl_\Theta \oplus\mfu_\Theta$ is the set of traceless block-upper triangular matrices :
 		\[
 		\renewcommand{\arraystretch}{1}
 		\mfl_\Theta = \left\lbrace \left(
 		\begin{array}{c|c|c}
 			*& & \\ \hline
 			 & *& \\ \hline
 			  & & *
 		\end{array}\right), \tr = 0\right\rbrace , ~~
 	\mfu_\Theta = \left\lbrace \left(
 	\begin{array}{c|c|c}
 	& *&* \\ \hline
 	& & *\\ \hline
 	& & 
	 \end{array}\right), \tr = 0\right\rbrace ,~~
	\mfp_\Theta = \left\lbrace \left(
	\begin{array}{c|c|c}
	*& *&* \\ \hline
	& *& *\\ \hline
	& & *
	\end{array}\right), \tr = 0\right\rbrace .
 		\renewcommand{\arraystretch}{1}
 		\]
 		The parabolic subgroup $P_\Theta$ is the stabilizer of the standard (partial) flag of $\R^n$ with subspaces of dimensions $(0,i,j,n)$.
 	\end{example}
 
 Since $L_\Theta$ is a reductive group, we need to know its center, the Cartan subalgebra of its Lie algebra etc.\ in order to understand the action of $L_\Theta$ on $\mfu_\Theta$.
 	First, $\mfl_\Theta$ splits into its center $Z(\mfl_\Theta)$ and its semisimple part $\mfl'_\Theta = \left[ \mfl_\Theta,\mfl_\Theta\right]$: \[ \mfl_\Theta = Z(\mfl_\Theta) \oplus \mfl'_\Theta. \]
 	Its center $Z(\mfl_\Theta)$ once again splits in two parts: $\mfz_\Theta = Z(\mfl_\Theta)\cap \mfk^\perp = Z(\mfl_\Theta)\cap \mfa$ and $Z(\mfl_\Theta)\cap \mfk$. 

  \begin{fact}\label{fact:lTheta}
   The Cartan subalgebra of $\mfl_\Theta$ is $\mfa_\Theta = \mfa \cap \mfl'_\Theta$, its restricted root system is the restriction of $\Sigma\cap\Span(\Delta\setminus\Theta)$ to $\mfa_\Theta$, a possible choice of simple roots are the restriction of $\Delta\setminus\Theta$ to $a_\Theta$, and the corresponding Dynkin diagram is the largest subgraph of the Dynkin diagram of $G$ with vertices $\Delta\setminus\Theta$.
  \end{fact}
   Note that since $\mfa\subset\mfl_\Theta\cap \mfk^\perp$, we have \[ \mfa = \mfa_\Theta\oplus \mfz_\Theta. \] 
 	In particular, $\mfz_\Theta$ acts on $\mfu_\Theta$ by restriction of the action of $\mfa$ on $\mfg$, so $\mfu_\Theta$ decomposes into weight spaces for this action:
 	\begin{equation}\label{eq:uTheta} \mfu_\Theta = \bigoplus_{\beta\in\mfz^*_\Theta} \mfu_\beta,
    \end{equation}
 	where 
    \begin{equation}\label{eq:ubeta}
    \mfu_\beta = \left\lbrace x\in\mfu_\Theta : \forall z\in \mfz_\Theta,\, \ad(z).x = \beta(z)x \right\rbrace
    \end{equation}
    is $L_\Theta^0$-invariant.
    
    Moreover if a root $\alpha\in\mfa^*$ vanishes on $\mfz_\Theta$ (i.e. lies in $\mfa_\Theta^*$) then $\alpha\in\Span(\Delta\setminus\Theta)$.
 	
 	The inclusion $\mfz_\Theta\subset\mfa$ induces a projection $i^*:\mfa^*\to\mfz_\Theta^*$ given by the restriction of a linear form on $\mfa$ to $\mfz_\Theta$. Thus we can write $\mfu_\beta$ as a sum of weight spaces for the action of $\mfa$:
 	\[ \mfu_\beta = \sum_{\alpha\in\Sigma_\Theta^+, \alpha|_{\mfz_\Theta} = \beta} \mfg_\alpha. \]

   	Suppose that $\beta$ is the restriction of a root $\alpha\in\Theta$.
    Then $\alpha$ is the only element of $\Span(\Theta)$ which restricts to $\beta$, allowing us the following slight abuse of notation:
    \begin{equation}\label{eq:ualpha}
     \mfu_\alpha:=\mfu_\beta=\sum_{\gamma\in \Span_{\N}(\Delta\setminus\Theta)}\mfg_{\alpha+\gamma}.
    \end{equation}
 	
 	\begin{example}
 		Using the same choice of $\Theta$ as in the previous example, the objects introduced above are the following: the center $Z(\mfl_\Theta)$ of $\mfl_\Theta$ is the set of matrices of the form 
 		\[  \left(
 		\begin{array}{c|c|c}
 			\lambda I_i& & \\ \hline
 			& \mu I_{j-i}& \\ \hline
 			& & \nu I_{n-j}
 		\end{array}\right)\]
 		with $i\lambda+(j-i)\mu + (n-j)\nu = 0$. The semisimple part of $\mfl_\Theta$ is then the set of block-diagonal matrices for which every block is traceless. In this situation, $Z(\mfl_\Theta)$ is entirely contained in $\mfa$, so $\mfz_\Theta = Z(\mfl_\Theta)$. The weight spaces $\mfu_\beta$ are the block-upper triangular matrices in $\mfu_\Theta$ for which only one block is non-zero, and the $\mfu_\beta$ with $\beta\in\Theta$ are the one where the non-zero block is just above the diagonal. The third $\mfu_\beta$ (the one with the top right block) is associated to the restriction of the root $\alpha_{i,j+1} = \alpha_{i, i+1} + \dots +\alpha_{j, j+1}$ to $\mfz_\Theta$, and all the roots $\alpha_{i+1, i+2},\dots,\alpha_{j-1, j}$ vanish on $\mfz_\Theta$.
 	\end{example}

 	In \cite{GuichardWienhard18} the irreducibility of $\mfu_\beta$ is stated as a fact, but the proof is not completely straightforward and is due to Kostant.
 	
 	\begin{theorem}[{\cite[Theorem 0.1]{Kostant08}}]\label{thm:ubeta irred}
 		For all $\beta\in\mfz_\Theta^*$, $\mfu_\beta$ is an irreducible representation of $L_\Theta^0$, and 
 		\[ \left[ \mfu_\beta,\mfu_{\beta'} \right] = \mfu_{\beta+\beta'}.  \]
 	\end{theorem}
 	
 	\begin{definition}\label{def:sharp1}
 		Let $V$ be a finite dimensional real vector space. A \textit{cone} in $V$ is a subset $C$ stable by positive scalar multiplication, i.e.\ $\forall x\in C$, $\forall \lambda >0$ we have $\lambda x\in C$.
 		The cone $C$ is said \textit{sharp} if it contains no affine line.
 	\end{definition}
 
 	\begin{example}
 		The sharpness condition ensure the cone is not too "wide", for instance the half plane $\H\subset\R^2$ is a cone, but not sharp since it contains many horizontal affine lines.
 		
 		We are interested in sharp \textit{convex} cones as they will give a good analog of $\R^{>0}$ in $\R$, which is our first example of sharp convex cone. A second example of sharp convex cone is the set of positive definite symmetric matrices in the space of symmetric matrices.
 	\end{example}
 
 	Since $L_\Theta$ is not necessarily connected, denote $L_\Theta^0$ the connected component of the identity in $L_\Theta$. We now have all the tool to give the definition of $\Theta$-positivity.
 	
 	\begin{definition}
 		We say that $(G,\Theta)$ admits a \emph{$\Theta$-positive structure} if for all $\beta\in\Theta$, the action of $L_\Theta^0$ on $\mfu_\beta$ preserves a sharp convex cone.
 	\end{definition}
 	
 	We already know two notions of positivity, namely total positivity and maximality. Our goal is now to explain why those are $\Theta$-positive structure, for a suitable choice of $\Theta$.
 	
 	\begin{example}
 		Let us begin with total positivity in $G = \SL(n,\R)$ to simplify the notations.
 		For $\Theta=\Delta$ (with the usual choice of $\Delta$ for $\SL(n,\R)$), the Levi subgroup $L_\Theta$ is the subgroup of diagonal matrices of determinant 1, and all of the weight spaces associated to elements of $\Theta$ are one-dimensional: 
 		\[ \forall \beta = \alpha_{i, i+1}\in\Theta,\, \mfu_\beta = \mfg_\beta = \left\langle E_{i,i+1}\right\rangle . \]
 		In Lusztig's total positivity, see Chapter \ref{section:LusztigTotalPositivity}, the semigroup of positive elements of a split real Lie group is generated by elements of the form $u_i(t) = I_n + tE_{i,i+1} = \exp(tE_{i,i+1})$, with $t>0$. Here the sharp convex cone in $\mfu_\beta = \mfg_{\alpha_{i,i+1}}$ is the cone of positive scalar multiples of $E_{i,i+1}$.
 	\end{example}
 	
 	\begin{example}
 		Now for maximality, let us treat the case $G = \Sp(2n,\R)$.
 		For the choice $\Theta = \left\lbrace \alpha_n\right\rbrace $ where $\alpha_n$ is the root at the end of the Dynkin diagram with two arrows pointing out, there is only one space $\mfu_\beta$, namely $\mfu_{\alpha_n}$. In this case the space $\mfu_{\alpha_n}$ can be identified with the space of symmetric matrices of size $n$, and the sharp convex cone is the set of positive definite matrices.
 		More details on the computations can be found in the exercise section.
 	\end{example}
 	
 	\newpage

\thispagestyle{empty}

\chapter[The Classification Theorem]{The Classification Theorem\\ {\Large\textnormal{\textit{by Pierre-Louis Blayac, Balthazar Fléchelles}}}}
\addtocontents{toc}{\quad\quad\quad \textit{Pierre-Louis Blayac, Balthazar Fléchelles}\par}

We will try to prove the following classification theorem for $\Theta$-positive structures.
\begin{theorem}[{\cite[thm 4.3]{GuichardWienhard18}}]\label{thm:classificationThetaPositiveStructures}
A \emph{simple} Lie group $G$ admits a $\Theta$-positive structure if and only if $(G,\Theta)$ falls in one of the following cases:
\begin{enumerate}
    \item G is a split real Lie form, and $\Theta=\Delta$,
    \item $G$ is Hermitian of tube type and $\Theta = \{\alpha_r\}$,
    \item $G$ is locally isomorphic to $\SO(p,q)$ with $p< q$, and $\Theta = \{\alpha_1,\ldots,\alpha_{p-1}\}$,
    \item $G$ is a real form of $F_4$, $E_6$, $E_7$ or $E_8$ whose restricted root system is of type $F_4$, and $\Theta = \{\alpha_1,\alpha_2\}$.
\end{enumerate}
\end{theorem}

Recall that a $\Theta$-positive structure on $G$ is linked to the existence of sharp convex cones invariant under the Levi subgroup $L_\Theta^\circ$ in $\mfu_\beta$ for $\beta\in\Theta$.
Conveniently, \cite[Prop.\,4.7]{Ben2000:automorphismsCvxCones} gives a nice criterion for an irreducible representation of a semisimple Lie group to preserve a sharp convex cone, which translates into the following fact:

\begin{fact}\label{fact:PosCritDynkinDiagram}
A semi-simple Lie group $G$ admits a $\Theta$-positive structure is positive if and only if
\begin{enumerate}[label=(\roman*)]
    \item \label{cond:1dim} $\mfg_\beta$ is $1$-dimensional for every $\beta\in\Theta$,
    \item \label{cond:double arrow} The node in the Dynkin diagram of $G$ corresponding to the restricted root $\beta \in\Theta$ is either disconnected from $\Delta\setminus\Theta$, or linked to $\Delta\setminus\Theta$ via a double arrow pointing toward $\Delta\setminus\Theta$.
\end{enumerate}
\end{fact}

Once this is established, it is not hard to find the list of Theorem \ref{thm:classificationThetaPositiveStructures}.
For example, one can use the Tables 1 and 4 in \cite{OGVLieGroups} to write an explicit list of the Lie groups admitting a $\Theta$-positive structure, and then check in another reference  the list of split real Lie forms (all rows of Table 4 of \cite{OGVLieGroups} whose column ``$\dim\mfg_{\lambda_j}$'' only indicates $1$; one can also consult \cite[\S C.3-4]{knapp}) and Hermitian groups of tube type \cite[Rem.\,2.1]{BIW07}.
Another reference is \cite[Chapter X, Table VI]{Helgason}.
Beware that there are mistakes in the tables of \cite{OGVLieGroups}, but we believe that the information we need here is correctly displayed there.
The procedure is the following.
For each row of Table 4 of \cite{OGVLieGroups}:
\begin{itemize}
    \item look at the column ``$\Sigma$'' giving the restricted root system;
    \item go to the corresponding row of Table 1;
    \item since any Dynkin diagram has at most one double arrow, there are at most two subsets of vertices satisfying Condition~\ref{cond:double arrow} of Fact~\ref{fact:PosCritDynkinDiagram} (the full set of vertices and, if there is a double arrow, another subset);
    \item for each of these two subsets $i_1,\dots,i_k$, go back to Table 4 column ``$\dim \mfg_{\lambda_j}$'' to check Condition~\ref{cond:1dim} of Fact~\ref{fact:PosCritDynkinDiagram}, \emph{i.e.\ }that $\mfg_{\lambda_{i_1}},\dots,\mfg_{\lambda_{i_k}}$ all have dimension $1$.
\end{itemize}

For instance, asking $\dim(\mfg_{\lambda_j})=1$ for any $j$ is equivalent to asking the group to be real split.

\begin{remark}
    If $G_1$ and $G_2$ are two semi-simple Lie groups with respective simple roots $\Delta_1$ and $\Delta_2$, which are respectively $\Theta_1$ and $\Theta_2$-positive (and we may allow $\Theta_1$ or $\Theta_2$ to be empty, considering that any group is $\emptyset$-positive), then $G_1\times G_2$ is $\Theta_1\sqcup \Theta_2$-positive.
    This is clear for instance in view of Fact~\ref{fact:PosCritDynkinDiagram}.
\end{remark}

In order to prove Fact~\ref{fact:PosCritDynkinDiagram}, we will need a theorem by Yves Benoist concerning representations of subsemigroups of $\GL(V)$, where $V$ is a fixed real vector space of dimension $d$.

\section{A theorem by Yves Benoist}
Before quoting Benoist's theorem, we need to introduce some terminology:
\begin{definition}
We say a semi-group $G\subset \GL(V)$ is \emph{irreducible} if any invariant non-trivial subspace is full, i.e.\
\begin{equation*}
    \forall W\subset V,\quad G\cdot W= W \implies \text{$W=\{0\}$ or $W=V$}.
\end{equation*}
\end{definition}
\begin{definition}
\begin{itemize}
    \item $\gamma\in \GL(V)$ is \emph{proximal} if it has only one eigenvalue of largest modulus, and if this eigenvalue is simple (and hence real). If moreover it is positive, we say that $\gamma$ is \emph{positively proximal}.
    \item If $\gamma\in \GL(V)$ is proximal, we will denote by $x_\gamma^+\in\P(V)$ the projectivization of the eigenline corresponding to the eigenvalue of largest modulus. $x_\gamma^+$ is called the \emph{attracting fixed point} of $\gamma$ in $\P(V)$.
    We denote by $y_\gamma^-$ the $\gamma$-invariant complementary hyperplane.
    Note that $\gamma^nx\to x_\gamma^+$ when $n\to\infty$ for any $x\not\in y_\gamma^-$.
    \item A (semi-)group $\Gamma < \GL(V)$ is \emph{proximal} if it has a proximal element. If moreover every proximal element of $\Gamma$ is positively proximal, we say that $\Gamma$ is \emph{positively proximal}.
\end{itemize}
\end{definition}

These definitions are of interest for the notion of limit set:
\begin{proposition}\label{prop:limSetDef}
If $\Gamma<\GL(V)$ is an irreducible and proximal subsemigroup, then
\begin{equation*}
    \Lambda_\Gamma := \overline{\{x_\gamma^+ :\text{$\gamma\in\Gamma$ proximal}\}}
\end{equation*}
is the smallest non-empty $\Gamma$-invariant closed subset of $\P(V)$.
It is called the \emph{limit set} of $\Gamma$.
\end{proposition}
\begin{proof}
The set $\Lambda_\Gamma$ is clearly non-empty.

Let us check that it is $\Gamma$-invariant. 
It is clear that $\Lambda_\Gamma\subset\Gamma\Lambda_\Gamma$.
Consider $\gamma,\eta\in\Gamma$ with $\eta$ proximal.
If $\Gamma$ was a group we could just observe that $\gamma\eta\gamma^{-1}\in\Gamma$ is also proximal and $\gamma x_\eta^+ = x_{\gamma\eta\gamma^{-1}}^+\in\Lambda_\Gamma$.
When $\Gamma$ is not a group the argument is more delicate.
By irreducibility we can find $g\in \Gamma$ such that $g(\gamma(x_\eta^+))\not\in y_\eta^-$.
Then one may check that for $n$ large enough $\gamma\eta^ng\in \Gamma$ is proximal and that $x_{\gamma\eta^ng}^+\to\gamma x_\eta^+$.

Consider now a non-empty closed set  $F\subset\P(V)$ with $\Gamma F\subset F$, and let us check that $\Lambda_\Gamma\subset F$.
For any $x\in F$ and $\gamma\in\Gamma$ proximal, by irreducibility there exists $g\in\Gamma$ such that $g(x)\not\in y_\gamma^-$.
Then $x_{\gamma}^+=\lim \gamma^ngx\in F$.
\end{proof}

\begin{definition}
A \emph{convex cone} $C$ of $V$ is a non-trivial convex subset of $V$ invariant under multiplication by positive scalars. 
It is \emph{sharp} or \emph{properly convex} if its closure does not contain a full line.
This is consistent with Definition~\ref{def:sharp1}.
\end{definition}

\begin{theorem}[{\cite[prop. 3.11]{Ben2000:automorphismsCvxCones}}]\label{thm:SharpCvxConeIffPosProx}
An irreducible semi-group $\Gamma \subset \GL(V)$ preserves a sharp convex cone of $V$ if and only if it is positively proximal.
\end{theorem}
\begin{proof}[Proof (only an idea for the converse)]
Suppose $\Gamma$ preserves a sharp convex cone $C\subset V$. 

\medskip

First observe that for any proximal element $\gamma\in\Gamma$, there exists a point $p\in C$ which does not meet the repelling hyperplane of $\gamma$ ($C$ has non-empty interior by irreducibility of $\Gamma$), so that $\frac{\gamma^n p}{\norm{\gamma^n p}}$ converges to a point $x$ in $x_\gamma^+\cap \overline{C}$. 
Since $C$ is sharp, this is a $\gamma$-invariant half-line contained in $x_\gamma^+$, so $\gamma$ is positively proximal.
Therefore we only need to prove that $\Gamma$ is proximal.

Without loss of generality, we may assume that $\Gamma$ contains all the homotheties $\{\lambda\id\}_{\lambda>0}$. 
We will denote by $\overline{\Gamma}$ the closure of $\Gamma$ in $\End(V)$.
Consider $M := \inf_{\gamma\in\overline{\Gamma}-\{0\}} \rank \gamma >0$.

\medskip

Suppose $M = d$, i.e.\ that $\overline\Gamma\setminus\{0\}\subset\GL(V)$.
The projection in $\P(\End(V))$ of $\overline\Gamma\setminus\{0\}$ is closed and hence compact, and equals that of $\overline\Gamma\cap\SL(V)$, which is then itself compact since $\pi\colon \End(V)\setminus\{0\}\to\P(\End(V))$ is injective on $\SL(V)$.
Thus $\overline\Gamma\cap\SL(V)$ is a compact group (any relatively compact sequence of the form $(g^n)_{n\geq 0}\subset\GL(V)$ has $g^{-1}$ as an accumulation point); let $\mu$ be its Haar measure.
Then $\int_g g(x)d\mu(g)$ is non-zero and $\overline\Gamma\cap\SL(V)$-invariant for any $x\in C\setminus\{0\}$, and hence spans a $\Gamma$-invariant half-line. 
Since $\Gamma$ is irreducible, we must have $d=1$, and $\Gamma$ is proximal.

\medskip

Suppose $M<d$ and consider $\gamma \in \overline{\Gamma}\setminus\{0\}$ with rank $M$.
By irreducibility of $\Gamma$ there exists $g\in G$ such that $g\im(\gamma)\not\subseteq\ker(\gamma)$.
Let $\eta := g\gamma$, so that  $\ker(\eta)\cap\im(\eta) = \varnothing$ (otherwise $\rk(\eta^2)<\rk(\eta) = \rk(\gamma) = M$), and hence $\im(\eta)\oplus\ker(\eta) = V$.

Consider now the semigroup $\eta\overline\Gamma\eta\subset \overline\Gamma\subset\End(V)$. 
It stabilizes $\im(\gamma)$, and acts trivially on $\ker(\gamma)$, so it naturally identifies with a semi group $\Gamma'\subset\GL(\im\gamma)$ ($\gamma$ restricts to a bijection on $\im(\gamma)$), which preserves the sharp convex cone $C' = \overline{C}\cap\im(\gamma)\setminus \{0\}$. 
That $\overline{\eta\Gamma\eta}\subset \overline{\Gamma}$ implies that  $\dim(\im(\eta))\geq M':=\inf_{g\in\Gamma'\setminus\{0\}}\rk (g) \geq M = \rk(\eta)$.
Thus, the first case above yields $\rank \eta=M'=1$, so that $\eta$ and hence $\overline{\Gamma}$ is proximal.
Since being proximal is an open condition, we conclude that $\Gamma$ is proximal.

To prove the converse , one uses the irreducibility and positive proximality of $\Gamma$ to  lift in a $\Gamma$-invariant way the limit set $\Lambda_\Gamma\subset\P(V)$ to a set of half-lines in $V$ whose convex hull is a sharp convex cone.
However this requires a lot of work (several intermediate results) so we refer to Benoist's article for a complete proof.
\end{proof}

\section{Proof of the classification theorem}

We know proceed to show fact \ref{fact:PosCritDynkinDiagram} using theorem \ref{thm:SharpCvxConeIffPosProx}. For this purpose, we fix a connected reductive Lie group $G$ (it corresponds to the Levi subgroup $L_\Theta^\circ$ in theorem \ref{thm:classificationThetaPositiveStructures}), and an irreducible representation $\rho \from G \to \GL(V)$.

Recall the following facts:
\begin{itemize}
    \item the weights $\lambda\in\mfa^\ast$ of $\rho$ lie in the \emph{weight lattice}
    \begin{equation*}
        P := \{\lambda\in\mfa^\ast: \frac{2\IP{\lambda}{\alpha}}{\IP{\alpha}{\alpha}}\in\Z,\forall\alpha\in\Delta\},
    \end{equation*}
    \item two weights differ by an element of $\Span_\Z\Delta$,
    \item the set of weights of $\rho$ is (partially) ordered as follow:
    \begin{align*}
        \mu \geq \lambda &\iff \mu-\lambda\in\Span_\N\Delta\\
        &\ \, \Longrightarrow \, \mu(X_0)\geq \lambda(X_0)
    \end{align*}
    where $X_0\in\mfa$ is a (fixed) vector with $\alpha(X_0)>0$ for any $\alpha\in\Delta$.
\end{itemize}

It is a classical result of the theory of representations of reductive groups that, by irreducibility, there exists a unique highest weight $\omega\in P$.

\begin{proposition}\label{prop:critprox}
$\rho(G)$ is proximal if and only if the highest weight space $V_\omega$ of $\rho$ is of dimension $1$. Moreover, in this case, $\rho(G)$ is positively proximal if and only if $\omega\in 2P$, i.e.\
\begin{equation*}
    \forall \alpha\in\Delta,\quad \frac{\IP{\omega}{\alpha}}{\IP{\alpha}{\alpha}}\in\Z.
\end{equation*}
\end{proposition}

\begin{proof}
Suppose $V_\omega$ is of dimension $1$. 
Since $\omega$ is the highest weight, we have
\begin{equation}\label{eq:prox}
    \omega(X_0) > \lambda(X_0)
\end{equation}
for all $\lambda\neq\omega$ weight of $\rho$. 
But the eigenvalues of $\rho(e^{X_0})$ are by definition the numbers of the form $e^{\lambda(X_0)}$ where $\lambda$ is a weight of $\rho$, with corresponding eigenspace the weight space $V_\lambda$.
Thus, \eqref{eq:prox} and the hypothesis $\dim V_\omega = 1$ means that $\rho(e^{X_0})$ is proximal.

\medskip

Suppose that $\rho(G)$ is proximal, and pick $g\in G$ such that $\rho(g)$ is proximal. 
Consider the Jordan decomposition $g=kan$:
the elements $k,a,n$ pairwise commute, $k$ is conjugate to $k'\in K$, while $a$ is conjugate to $\exp(X)\in\exp(\overline{\mfa^+})$, and $n$ is conjugate to $\exp(N)\in \exp \mfu$ (where $\mfu = \sum_{\alpha\in\sigma_+} \mfg_\alpha$).

The element $\rho(k') \in \rho(K)\subset O(V,q)$ (for some scalar product $q$) only has eigenvalues of modulus $1$, and so does $\rho(k)$.
Consider a basis $e_1,\dots,e_d\in V$ with $e_i\in V_{\lambda_i}$ for some weight $\lambda_i$, and with an ordering such that $\lambda_i> \lambda_j \, \Rightarrow \, i< j$. 
Observe that $d\rho(\mfg_\alpha)\cdot V_\lambda\subset V_{\lambda+\alpha}$ for all $\alpha\in\sigma_+$ and $\lambda\in P$, so that $\rho(\exp(N))=\exp(d\rho(N))$ is upper triangular with ones on the diagonal for the basis $e_1,\dots,e_d$.
The moduli of the eigenvalues of $\rho(g)$ are hence given by those of $\rho(a)$, which is therefore proximal.

The eigenvalues of $\rho(a)$ are the numbers of the form $e^{\lambda(X)}$ where $\lambda$ is a weight of $\rho$. 
Since $\omega$ is the highest weight, $e^{\omega(X)}$ is the highest eigenvalue of $\rho(a)$.
Since $\rho(a)$ is proximal, we then have $\dim V_\omega = 1$.

\medskip

We only prove that if $G = \SL(n,\R)$ and $\rho(G)$ is positively proximal then $\omega\in 2P$ (for different $G$, the idea is similar using $\sl(2,\R)$-triplets).

Recall that the simple roots of $G$ are $\Delta = \{\alpha_1,\ldots,\alpha_{n-1}\}$ where $\alpha_i(\diag(\lambda_1,\ldots,\lambda_n)) = \lambda_{i+1}-\lambda_i$.

We want to show that for all $i\in\{1,\ldots,n-1\}$,
\begin{equation*}
    \frac{\IP{\omega}{\alpha_i}}{\IP{\alpha_i}{\alpha_i}}\in\Z.
\end{equation*}
Write (where the two $-1$'s appear at positions $i$ and $i+1$)
\begin{equation*}
    k_i = \diag(1,\ldots,1,-1,-1,1,\ldots,1) = e^{i\pi(E_{i\,i}-E_{i+1\,i+1})} = e^{i\pi H_{\alpha_i}}.
\end{equation*}
Since $\rho(k_i)$ and $\rho(e^{X_0})$ are both diagonal matrices, $\rho(k_i)$ has only $\pm 1$ as diagonal values, and $\rho(e^{X_0})$ is proximal, $\rho(k_ie^{X_0})$ is also proximal, so it is positively proximal. In particular, the eigenvalue of $\rho(k_i)$ on $V_\omega$ is $1$.

Pick $v\in V_\omega\setminus\{0\}$.
One has (after extending $\diff\rho:\mfg\to\sl_d\R$ to the complexification $\diff\rho^\C\colon\mfg^\C\to\sl(d,\C)$), 
\begin{align*}
    v = \rho(k_i)v = \rho(e^{i\pi H_{\alpha_i}})v = e^{i\pi \omega(H_{\alpha_i})}v,
\end{align*}
so that $\omega(H_{\alpha_i})$ is an even integer.

On the other hand, $H_{\alpha_i} := \frac{2t_{\alpha_i}}{\IP{\alpha_i}{\alpha_i}}$ by definition, where $t_{\alpha_i}\in\mfg$ is such that $\beta(t_{\alpha_i}) = \IP{\beta}{\alpha_i}$ for all $\beta\in\mfg^\ast$.
Hence, for every  $i\in\{1,\ldots,n-1\}$ we have
\begin{align*}
    \frac{\IP{\omega}{\alpha_i}}{\IP{\alpha_i}{\alpha_i}} &= \frac{\omega(t_{\alpha_i})}{\IP{\alpha_i}{\alpha_i}}= \frac12\omega(H_{\alpha_i}) \in \Z.\qedhere
\end{align*}
\end{proof}

\begin{proof}[Proof of Fact~\ref{fact:PosCritDynkinDiagram}]
 Note that given an irreducible representation $\rho$ of a reductive group with simple root system $\Delta$, there exists a unique lowest weight of $\rho$ relatively to the choice of $\Delta$, and it is the highest weight relatively to the opposite simple root system $-\Delta$.
 Therefore we may replace the term ``highest'' by ``lowest'' in Proposition~\ref{prop:critprox}.

 Let us fix $\beta\in\Theta$.
 Denote by $\rho \from L_\Theta^0\rightarrow\GL(\mfu_\beta)$ the irreducible representation given by \eqref{eq:ubeta}, \eqref{eq:ualpha} and Theorem~\ref{thm:ubeta irred}.
 
 By Theorem~\ref{thm:SharpCvxConeIffPosProx}, it is enough to show that $\rho(L_\Theta^0)$ is positively proximal if and only if Conditions~\ref{cond:1dim} and \ref{cond:double arrow} hold.

 By Proposition~\ref{prop:critprox} and the observation at the beginning of this proof, it is enough to check that the lowest weight space of $\rho$ has dimension $1$ if and only if Condition~\ref{cond:1dim} holds, and that the lowest weight divided by $2$ lies in the weight lattice of $L_\Theta^0$ if and only if Condition~\ref{cond:double arrow} holds.

 By Fact~\ref{fact:lTheta},
 the restricted roots of $L_\Theta^0$ have the form $\alpha_{|\mfa_\Theta}$ with $\alpha\in\Sigma\cap\Span(\Delta\setminus\Theta)$, and we can choose $\{\alpha_{|\mfa_\Theta}:a\in\Delta\setminus\Theta\}$ as simple roots.

 By \eqref{eq:ubeta},
 the weights have the form $\beta_{|a_\Theta}+\alpha_{|a_\Theta}$ where $\alpha\in\Sigma\cap\Span_\N(\Delta\setminus\Theta)$.

 Observe that $\beta_{|a_\Theta}$ is by definition the lowest weight, with weight space $\mfg_\beta$.
 Thus the lowest weight space of $\rho$ has dimension $1$ if and only if Condition~\ref{cond:1dim} holds.

 Denote by $\langle\cdot,\cdot\rangle_{\mfg}$ and $\langle\cdot,\cdot\rangle_{\mfl_\Theta}$ the Killing forms of respectively $\mfg$ and $\mfl_\Theta$.
 One may check that for all $\omega\in \mfa^*$ and $\alpha\in \Span(\Delta\setminus\Theta)$ we have
 \begin{equation}
  \frac{\langle\omega,\alpha\rangle_{\mfg}}{\langle\alpha,\alpha\rangle_{\mfg}} = \frac{\langle\omega_{|\mfa_\Theta},\alpha_{|\mfa_\Theta}\rangle_{\mfl_\Theta}}{\langle\alpha_{|\mfa_\Theta},\alpha_{|\mfa_\Theta}\rangle_{\mfl_\Theta}}
 \end{equation}
 Therefore $\beta_{|\mfa_\Theta}/2$ lies in the weight lattice of $L_\Theta^0$ if and only if for any $\alpha\in\Delta\setminus\Theta$ we have
 \begin{equation}\label{eq:PosCritDynkinDiagram}
     2\frac{\langle\beta,\alpha\rangle_{\mfg}}{\langle\alpha,\alpha\rangle_{\mfg}}\in2\Z.
 \end{equation}
 The quantity in \eqref{eq:PosCritDynkinDiagram} is worth
 \begin{itemize}
     \item $0$ if there is no arrow between $\beta$ and $\alpha$ in the Dynkin diagram of $G$,
     \item $1$ if there is a simple arrow (whatever the direction) or if there is a double or triple arrow pointing toward $\beta$,
     \item $2$ if there is a double arrow pointing toward $\alpha$,
     \item $3$ if there is a triple arrow pointing toward $\alpha$.
 \end{itemize}
 Thus $\eqref{eq:PosCritDynkinDiagram}$ holds for all $\beta\in\Theta$ and $\alpha\in\Delta\setminus\Theta$ if and only the only arrows in the Dynkin diagram of $G$ between $\Theta$ and $\Delta\setminus\Theta$ are double and point toward $\Delta\setminus\Theta$, which concludes the proof.
 \end{proof}
 
 \newpage

\thispagestyle{empty}

\chapter[Analysis of Groups With a Theta-Positive Structure]{Analysis of Groups With a Theta-Positive Structure\\ {\Large\textnormal{\textit{by Dani Kaufman}}}}
\addtocontents{toc}{\quad\quad\quad \textit{Dani Kaufman}\par}

\section{Examples of Theta-positive structures}

Other than the case where $G$ is a split real form and $\Theta = \Delta$ there are 3 new families of examples of $\Theta$-positive structures. We will call these three cases $A$, $B$ and $G$ respectively, for reasons which will be clear later. We will explore the properties of each of these examples in more detail. 

The first property to establish in each case is the structure of the positive cones in $\mathfrak{u}_\beta$. In each case, one finds that there is exactly one space, $\mathfrak{u}_\beta$, which has dimension strictly greater than 1 corresponding to the case that $\beta$ is the root linked to $\Delta\setminus \Theta$ by a double arrow. We will call this root space $\mathfrak{u}_\beta$ and all other one-dimensional root spaces $\mathfrak{u}_{\alpha_i}$. 

What is the structure of $\mathfrak{u}_\beta$ in each case?

\begin{itemize}
    \item For type $A$, $G$ is a Hermitian Lie group of tube type, meaning that its Hermitian symmetric space is of the form $V+i\Omega$. The vector space $\mathfrak{u}_\beta$ is isomorphic to $V$ and the positive cone is given by $\Omega$
    
    \item For type $B$, $G$ is locally isomorphic to $\SO(p,q)$. The space $\mathfrak{u}_\beta$ is isomorphic to $\mathbb{R}^{1,q-p+1}$ as a normed vector space and the positive cone is given by a choice of one side of the double cone consisting of vectors in $\mathfrak{u}_\beta$ with positive norm. 
    
    \item For type $G$, the space $\mathfrak{u}_\beta$ is given by the exceptional Jordan algebra of $3 \times 3$ matrices over $\mathbb{R},\mathbb{C},\mathbb{H}$ or $\mathbb{O}$ and the positive cones are given by the positive definite matrices in each algebra. 
\end{itemize}

\section{The Theta-Weyl group}

We define for each $\Theta$-positive structure a subgroup of the Weyl group of $G$ called the $\Theta$-Weyl group, denoted $W(\Theta)$. 

Let $W(\theta)$ be the subgroup generated by $w_{\alpha_i}$ and $w_\beta = w_{\alpha_{|\Theta|}}$ where $w_{\alpha_i}$ are the Weyl group generators associated to the roots $\alpha_i \in W(G)$ and $w_\beta$ is the element of $W(G)$ given by the longest element of the Weyl group of $\Delta\setminus\Theta \cup \{\beta\}$ 

The $\Theta$-Weyl group is isomorphic in each case to the Weyl group of a root system of type $A_1, B_n, G_2$ for each of the three families of non trivial $\Theta$-positive structures, justifying the use of these names for these families. We denote the Dynkin diagram for these new root systems $\Delta_\Theta$.

The Dynkin diagrams of $\Delta$ and $\Delta_\Theta$ for each family are shown in Figure \ref{fig:ThetaDynks}.

\begin{figure}
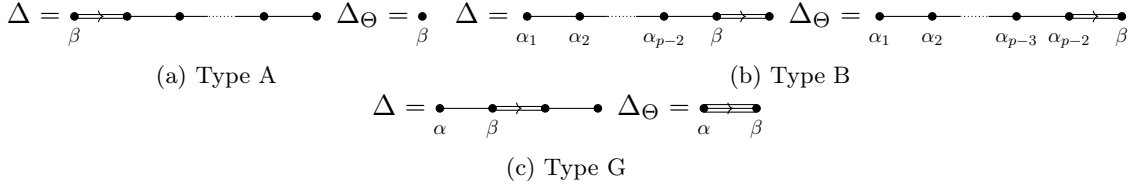

    \centering
    \def\do{}
    \begin{subfloat}[][Type A]{
        $\Delta=$\dynkin [labels={,,,,\beta},backwards]C{} 
        $\Delta_\Theta=$\dynkin [labels={\beta}] A{1}}
    \end{subfloat}
    \begin{subfloat}[][Type B]{    $\Delta=${\dynkin[labels={\alpha_1,\alpha_2,\alpha_{p-2},\beta,}]B{}} 
    $\Delta_\Theta=${\dynkin[labels={\alpha_1,\alpha_2,\alpha_{p-3},\alpha_{p-2},\beta}]B{}}}
    \end{subfloat}
    \begin{subfloat}[][Type G]{    $\Delta=$\dynkin[labels={\alpha,\beta}]F{4}
    $\Delta_\Theta=$\dynkin[labels={\alpha,\beta}]G{2}}
    \end{subfloat}

    \caption{$\Delta$ and $\Delta_\Theta$ for each family of groups with Theta-positive structures}
    \label{fig:ThetaDynks}
\end{figure}

\begin{remark}
The analysis of positivity in Hermitian Lie groups of tube type in Section 4.3 suggests that we should treat these groups as groups of $2\times2$ matrices with entries in some noncommutative ring. This agrees with the fact that $\Delta_\Theta$ is an $A_1$ root system, which is the same root system as $\SL(2,\mathbb{R})$. Thus we should think of groups with a $\Theta$-positive structure as some kind of group of type $\Delta_\Theta$ with noncommutative entries. 
\end{remark}

\section{A Theta-positive semigroup}

In analogy with Lusztig's total positivity, we can define a $\Theta$-positive semigroup, $\Utheta$. To do this, we first define
\begin{equation*}
x_{\alpha_i}\colon \mathfrak{u}_{\alpha_i} \xrightarrow{} U_{\alpha_i}, v \xrightarrow{} \exp(v)     
\end{equation*}

Recall that the $\Theta$-positive structure on $G$ gives a positive cone $c^\circ_{\alpha_i} \in \mathfrak{u}_{\alpha_i}$.
Let $\omega_\circ \in W(\Theta)$ be the longest element in the $\Theta$-Weyl group. Writing $\omega_\circ$ as a product of generators $\omega_{\alpha_{i_k}}$ we define a map
\begin{align*}
\Phi^{+}_{\omega_0}\colon \prod(c^\circ_{\alpha_{i_k}})_{\geq 0} &\to U_\Theta^{>0}\\
\Phi^+_{\omega_0}(a_1,a_2,\cdots, a_p):=& x_{\alpha_{i_1}}(a_1)x_{\alpha_{i_2}}(a_2)\cdots x_{\alpha_{i_p}}(a_p).
\end{align*}

\begin{theorem}[Theorem 4.5 in \cite{GuichardWienhard18}]
The image $U_\Theta^{>0}$ is independent of the choice of reduced expression of $\omega_\circ$.
\end{theorem}
\begin{proof}
We only give the basic ideas of a proof. Since two choices of reduced expression differ by braid relations, we only need to check that the image is preserved after each type of braid relation. Each braid relation is either of the form 
\begin{enumerate}
    \item $\omega_{\alpha_i}\omega_{\alpha_{i+1}}\omega_{\alpha_i} = \omega_{\alpha_{i+1}}\omega_{\alpha_i}\omega_{\alpha_{i+1}} $
    \item $\omega_{\alpha}\omega_{\beta}\omega_{\alpha}\omega_{\beta} = \omega_{\beta}\omega_{\alpha}\omega_{\beta}\omega_{\alpha}$
    \item $\omega_{\alpha}\omega_{\beta}\omega_{\alpha}\omega_{\beta}\omega_{\alpha}\omega_{\beta}=\omega_{\beta}\omega_{\alpha}\omega_{\beta}\omega_{\alpha}\omega_{\beta}\omega_{\alpha}$
\end{enumerate}
where for the last two relations $\alpha$ is the root in $\Theta$ adjacent to $\beta$. The first relation is between one-dimensional root spaces and is thus the same as the total positivity case. The second two relations can be established by a brute force computation in the $B_2$ and $G_2$ cases.
\end{proof}

With the notion of the semigroup $\Utheta$ we may define the theta positive semigroup $G_\Theta^{>0}$ as the semigroup generated by $U_\Theta^{opp>0},L_\Theta^\circ,\Utheta$ in a totally analogous way to the case of total positivity. 

\section{Theta-positive triples of flags}

Continuing the analogies with total positivity, we can define the notion of a positive triple of elements of $G/P_\Theta$. Let $E_\Theta, F_\Theta$ be two opposite standard flags in $G/P_\Theta$ i.e.\ we have that the stabilizer subgroups of $E$ and $F$ are $P_\Theta^{opp}$ and $P_\Theta$ respectively. 

\begin{definition}
Let $S_\Theta$ be a flag transverse to both $E_\Theta$ and $F_\Theta$. We call a triple $(E_\Theta, S_\Theta, F_\Theta)$ a \emph{$\Theta$-positive} triple of flags if $S_\Theta = uE_\Theta$ for $u \in \Utheta$
\end{definition}

\begin{theorem}[Theorem 4.7 \cite{GuichardWienhard18}]
The set $\{S_\Theta \in G/P_\Theta \mid (E_\Theta,S_\Theta,F_\Theta) \text{ is }\Theta\text{-positive}\}$  is a connected component of the collection of flags transverse to both $E_\Theta$ and $F_\Theta$.
\end{theorem}

\newpage

\thispagestyle{empty}

\chapter[Positivity in Higher Teichm\"uller Theory]{Positivity in Higher Teichm\"uller Theory\\ {\Large\textnormal{\textit{by Fernando Camacho Cadena}}}}
\addtocontents{toc}{\quad\quad\quad \textit{Fernando Camacho Cadena}\par}

Throughout this section $G$ will denote a real simple Lie group, and $\Sigma$ will be a compact closed surface of genus at least $2$ with fundamental group $\pi_1(\Sigma)$.

We recall that a \emph{higher Teichm\"uller space} is a subset of $\chi(S,G)$ which is a union of connected components consisting entirely of discrete and faithful representations. The goal of this section is to see how higher Teichm\"uller spaces arise from representations into different types of Lie groups.

The following were the known instances of higher Teichm\"uller spaces:
\begin{itemize}
    \item When $G$ is a split real simple group, as for example $\PSL(n,\R)$, the space of Hitchin representations is a higher Teichm\"uller space \cite{FockGoncharov06,Labourie:06,Guichard08}.
    \item When $G$ is a Hermitian Lie group of tube type, for example $\Sp(2n,\R)$, the space of maximal representations is a higher Teichm\"uller space \cite{BIW10}.
\end{itemize}

The main goal of this section is the following theorem, which states that there are new higher Teichm\"uller spaces.

\begin{theorem}[{\cite[Theorem A]{GLW21}}]\label{thm:there exist hTs for G with theta pos structure}
Let $G$ be a simple Lie group admitting a $\Theta$-positive structure. Then there exists a connected component of the representation variety consisting entirely of discrete and faithful representations.
\end{theorem}

From the classification of Lie groups admitting a $\Theta$-positive structure, the above theorem gives the existence of new higher Teichm\"uller spaces. Namely subsets of the representation varieties of groups isomorphic to $\SO(p,q)$ and Lie groups of exceptional type (which had not been known about before!).

To show the above theorem, Guichard, Labourie and Wienhard generalize the following common property shared by Hitchin and maximal representations.

\begin{theorem}[{\cite{FockGoncharov06,Labourie:06,Guichard08,BIW10}}]
Let $G$ be either a simple Lie group which is either real split, or Hermitian of tube type. Let $P_\Theta$ be the parabolic subgroup of $G$ corresponding to the $\Theta$-positive structure (when $G$ is split real, $\Theta$ is the set of simple roots, and when $G$ is Hermitian of tube type, $\Theta=\{\alpha_r\}$, see \cite{GuichardWienhard18}). Then a representation $\rho\colon \pi_1(\Sigma)\to G$ is Hitchin, or respectively maximal, if an only if there exists a $\rho$-equivariant continuous map
\[
    \xi\colon\del \pi_1(\Sigma)\to G/P_\Theta,
\]
sending positive triples in $\del \pi_1(\Sigma)$ to positive triples in $G/P_\Theta$.
\end{theorem}

As we saw in previous sections, there is also a notion of positive of tuples in (generalized) flag varieties $G/P_\Theta$. The above result then motivates the following definition.

\begin{definition}[{\cite[Definition 5.1]{GLW21}}]
Let $G$ be a Lie group admitting a $\Theta$-positive structure. A representation $\rho\colon\pi_1(\Sigma)\to G$ is said to be \textbf{$\Theta$-positive} if there exists a $\rho$-equivariant continuous map
\[
    \xi\colon \del \pi_1(\Sigma)\to G/P_\Theta
\]
that sends positive triples in $\del \pi_1(\Sigma)$ to positive triples in $G/P_\Theta$.
\end{definition}

In fact, the higher Teichm\"uller spaces from theorem \ref{thm:there exist hTs for G with theta pos structure} consist entirely of positive representations. In this section we do not present the full argument, but focus on the following

\begin{theorem}[{\cite[Theorem B]{GLW21}}]\label{thm: positive reps are Anosov}
Let $G$ be a simple Lie group admitting a $\Theta$-positive structure, and $\rho\colon\pi_1(\Sigma)\to G$ a $\Theta$-positive representation. Then $\rho$ is also a $\Theta$-Anosov representation.
\end{theorem}

We will give a definition of a $\Theta$-Anosov representation in section \ref{section:anosov representations definition for GLW21}. We nevertheless state here two essential properties of Anosov representations, due to Labourie, and Guichard-Wienhard.

\begin{theorem}[{\cite{Labourie:06,GW12:AnosovDiscontinuity}}]\label{thm: Anosov representations are open and discrete and faithful}
The space of $\Theta$-Anosov representations is open in $\Hom(\pi_1(\Sigma),G)$, and every $\Theta$-Anosov representation is discrete and faithful.
\end{theorem}

One can then deduce the following result.

\begin{corollary}[{\cite{GLW21}}]\label{cor: theta positive representations are open}
The set of $\Theta$-positive representations is an open subset of $\Hom(\pi_1(\Sigma),G)$.
\end{corollary}

\begin{remark}
In the case when $G$ is locally isomorphic to $\SO(p,q)$, Beyrer and Pozzetti showed in \cite{BeyrerPozzetti21} that the subset of $\Theta$-positive representations is closed. Together with the fact from \cite{GLW21} that the set of $\Theta$-positive representations is open and that $\Theta$-Anosov representations are discrete and faithful, this implies that the subset of $\Theta$-positive representations is itself a higher Teichm\"uller space.
\end{remark}

The method in \cite{GLW21} to prove theorem \ref{thm:there exist hTs for G with theta pos structure} is to find a union of connected components in the representation variety consisting entirely of $\Theta$-positive representations. Then by corollary \ref{cor: theta positive representations are open}, all these representations are discrete and faithful. Finding a union of connected components involves methods from Higgs bundle theory as well as previous work by \cite{BradlowCollierGarciaPradaGothen}. Here we will only present the ideas for the proof of theorem \ref{thm: positive reps are Anosov}.

We begin with some preliminaries on diamonds, which are structures on the flag varieties associated to the $\Theta$-positive structure. We then move on to some definitions and facts about Anosov representations and finish with the idea of the proof of theorem \ref{thm: positive reps are Anosov}.

\section{Diamonds}
For the rest of this section, we will assume that $G$ is simple and admits a $\Theta$-positive structure, denote by $U^{>0}$ the $\Theta$-positive subsemigroup of $G$, and by $F_\Theta$ the flag variety $G/P_\Theta$. We recall here some instances of this.

\begin{examples}
\begin{itemize}
    \item When $G=\SL(2,\R)$,
    \[
        U^{>0}=\left\{\begin{pmatrix}1&t\\ &1\end{pmatrix}:t>0\right\}.
    \]
    \item When $G=\SL(3,\R)$,
    \[
        U^{>0}=\left\{\begin{pmatrix}1 & a+c & ab\\ & 1 & b\\ & & 1\end{pmatrix}:a,b,c>0\right\}.
    \]
    \item When $G=\Sp(2n,\R)$,
    \[
    U^{>0}=\left\{\begin{pmatrix}I_n&N\\ &I_n\end{pmatrix}:N \textnormal{ is positive definite and symmetric}\right\}
    \]
\end{itemize}   
\end{examples}

To define positivity in the flag variety, we fix two standard flags. Namely, we let $F^o\in F_\Theta$ be the flag stabilized by $P_\Theta$, and $E^o$ be the flag stabilized by $P_\Theta^{opp}$. Recall now that if $T\in F_\Theta$ is a third flag, the triple $(F^o,T,E^o)$ is \emph{positive} if there exists an element $u_T\in U^{>0}$ such that $T=u_TE$. We say that a general triple of flags $(F,T,E)$ is positive if there exists $g\in G$ such that 
\[
    gF^o=F,\quad gE^o=E, \quad \textnormal{and}\quad T=guE^o,
\]
for some $u\in U^{>0}$. In particular this implies that for two flags $F$ and $E$ which are transverse, letting $g\in G$ such that $gF^o=F$ and $gE^o=E$, the subset
\[
    \{(F,gug^{-1}E,E):u\in U^{>0}\}
\]
consists entirely of positive triples. Here is an example of what such a set looks like.

\begin{example}
In the case of $G=\SL(3,\R)$, $\Theta$ is the set of all positive roots, and we can choose
\[
    P_\Theta = \left\{\begin{pmatrix}* & * & *\\ 0 & * & *\\ 0 & 0 & *\end{pmatrix}\right\},
\]
making $F_\Theta$ the space $\Flag(\R^3)$ of full flags in $\R^3$. The standard flag $F^o$ corresponds to the flag $\{\Span(e_1),\Span(e_1,e_2)\}$, and the opposite flag is $\{\Span(e_3),\Span(e_3,e_2)\}$. where $e_1,e_2,e_3$ are the standard basis vectors for $\R^3$. We can visualize a full flag in $\mathbb{RP}^2$ as a point together with a (projective) line through that point. To picture the set $\{(F^o,uE^o,E^o):u\in U^{>0}\}$, we look at a coordinate patch in $\mathbb{RP}^2$:
\begin{align*}
    \{[x:y:z]\in\mathbb{RP}^2 : x\neq 0\}&\rightarrow \R^2\\
    [x:y:z]&\mapsto \left[1:\frac{y}{x}:\frac{z}{x}\right].
\end{align*}
Using the explicit form of $U^{>0}$ for $\SL(3,\R)$, one can see that the orbit of $\Span(e_1)$ is
\[
    \left\{\left[1:\frac{1}{a}:\frac{1}{ab}\right] : a,b>0\right\},
\]
and the orbit of $\Span(e_2)$ is
\[
\left\{\left[1:\frac{1}{a+c}:\frac{1}{0}\right] : a,c>0\right\},
\]
In coordinates, the orbit of $\Span(e_1)$ is the (open) upper right quadrant. One can do a similar computation for the orbit of the plane $\Span(e_1,e_2)$ by identifying planes in $\R^3$ to orthogonal lines (equivalently by identifying $(\mathbb{RP}^2)^*$ with $\mathbb{RP}^2$), and then see that the orbit (in coordinates) will also be a quadrant. Now imagine $\infty$ in the coordinates patch in for $\mathbb{RP}^2$ as a single point at infinity, thereby "closing up" the upper right quadrant into a "diamond" with the standard flags as extremities.
\end{example}
From this vague picture, we make the following definition.

\begin{definition}
Given two transverse flags $F,E\in F_\Theta$, a \textbf{diamond} is a subset of $F_\Theta$ of the form
\[
    \left\{guE^o:u\in U^{>0}\right\}
\]
for some $g\in G$. The \textbf{extremities} of the diamond are the flags $gF^o$ and $gE^o$.
\end{definition}
As noted above, for any $T$ in a diamond with extremities $F$ and $E$, we have that $(F,T,E)$ is a positive triple.

\begin{remark}
The terminology of diamond was actually coined by Labourie and Toulisse in \cite{LabourieToulisse20} for the case when $G=\SO(2,n)$ with $n\geq 3$. There are charts on the flag variety associated to $G$ which map to Minkowski space. In appropriate charts, a diamond is the intersection of the future and past light cone of the extremities.
\end{remark}

\begin{example}
When $G=\SL(2,\R)$, we have the following identification
\[
    \SL(2,\R)/P\cong S^1,
\]
where $P$ is the group of upper triangular matrices. Here, diamonds are line segments on the circle as in figure \ref{fig:prop positivity for diamonds}. We note here that the flag variety of $\SL(2,\R)$ and the one for $\PSL(2,\R)$ can be further identified with the visual boundary of $\mathbb{H}^2$, its symmetric space.
\end{example}
There is another characterization of diamonds in terms of transverse flags. To state the result, let $\Omega_F\subset F_\Theta$ be the subset of the flag variety consisting of flags transverse to $F$.

\begin{proposition}[{\cite[Proposition 2.5]{GLW21}}]\label{prop: diamonds are intersections of simultaneously transverse flags}
A diamond with extremities $E,F\in F_\Theta$ (recall here that $E$ and $F$ must be transverse) is a connected component of $\Omega_F\cap \Omega_E$.
\end{proposition}
In particular, we can specify a diamond by giving two transverse flags $x$ and $z$, and a point $y$ in the interior of $\Omega_F\cap \Omega_E$. We will denote such a diamond by $V_y(x,z)$.

We will now see how diamonds are related to positive representations
\begin{proposition}[{\cite[Proposition 2.7]{GLW21}}]\label{prop: relationship between diamonds and positivity}
Assume $\xi\colon S^1\to F_\Theta$ sends positive triples in $S^1$ to positive triples in $F_\Theta$. Then for any oriented positive tuple $(x_1,\dots,x_n)$ in $S^1$, there exist diamonds $V_{ij}$ such that
\begin{enumerate}
    \item $\xi(x_j)\in V_{ik}$ if $(x_i,x_j,x_k)$ is positive, and
    \item\label{condition: diamonds and positivity condition two} $V_{ij}\subset V_{km}$ if $(x_k,x_i,x_j,x_m)$ is positive.
\end{enumerate}
\end{proposition}

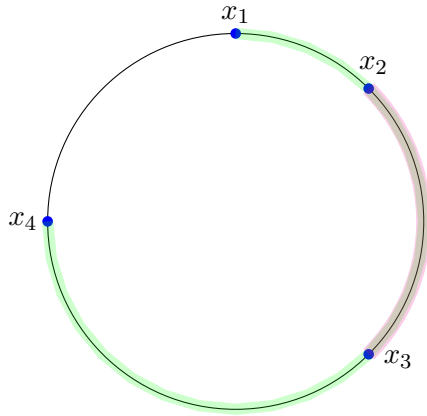
\begin{figure}[H]
    \centering
    \begin{tikzpicture}
    \coordinate (center) at (0,0);
  \def\radius{2.5cm}
  \draw (center) circle[radius=\radius];
  \fill[blue] (center) ++(90:\radius) circle[radius=2pt];
  \node[above] at (0,2.5) {$x_1$};
    \fill[blue] (center) ++(45:\radius) circle[radius=2pt];
  \node[above] at (1.3*1.41,1.3*1.41) {$x_2$};
    \fill[blue] (center) ++(180:\radius) circle[radius=2pt];
  \node[left] at (-2.5,0) {$x_4$};
    \fill[blue] (center) ++(315:\radius) circle[radius=2pt];
  \node[right] at (1.3*1.41,-1.3*1.41) {$x_3$};
  \draw[line width = 1.5mm, opacity = 0.2, green,domain=90:-180] plot ({2.5*cos(\x)}, {2.5*sin(\x)});
  \draw[line width = 2mm, opacity = 0.2, magenta,domain=45:-45] plot ({2.5*cos(\x)}, {2.5*sin(\x)});
        \end{tikzpicture}
        \caption{The circle with the clockwise orientation for positivity, and diamonds defined by $x_1,x_4$ in green, and $x_2,x_3$ in magenta.}
        \label{fig:prop positivity for diamonds}
\end{figure}
In figure \ref{fig:prop positivity for diamonds}, we think of assigning to the green and magenta arcs diamonds in the flag variety. Assertion \ref{condition: diamonds and positivity condition two} in proposition \ref{prop: relationship between diamonds and positivity} says that the containment of diamonds in the flag variety is exactly the containment of the arcs (diamonds) on the circle.

As we will see later on, we will need metrics on (subsets) of the flag varieties. For the rest of the section, we describe how to put metrics on diamonds, and state certain contraction properties.
\section{Metrics on diamonds}
We start with a particular class of positive triples.

\begin{definition}
Let $H<G$ be a subgroup isomorphic to $\PSL(2,\R)$. A \emph{positive circle} is a closed $H$-orbit in $F_\Theta$ parameterized by a $\PSL(2,\R)$-equivariant map
\[
    \mathbb{RP}^1\to F_\Theta.
\]
\end{definition}
Here we think of $\mathbb{RP}^1\cong S^1$ as the flag variety of $\PSL(2,\R)$ (which can also be identified with $\del_\infty\mathbb{H}^2$).
\begin{example}
Let $G=\PSL(d,\R)$, then the flag variety $F_\Theta$ can be identified with $G/P$, where $P$ is the group of upper triangular matrices. Thus the flag variety consists of honest full flags. We recall that a representation $\rho\colon\pi_1(\Sigma)\to G$ is Fuchsian (and in particular Hitchin) if it factors through the irreducible representation $\iota\colon\PSL(2,\R)\to\PSL(d,\R)$. We then obtain a $\rho$-equivariant continuous boundary map $\xi\colon \del \pi_1(\Sigma)\to F_\Theta$. Then setting $H=\iota(\PSL(2,\R))$, we have that the image of the boundary map $\xi$ is a positive circle.
\end{example}

Now we define a \emph{tripod} to be a triple of pairwise distinct points lying on a positive circle. Even though the definition of a positive circle seems quite restrictive, and there is no mention of positivity, we have the following.

\begin{proposition}[{\cite[Proposition 2.9, Proposition 2.10]{GLW21}}]\label{prop: two points are in a positive circle and positive circles are positive}
\begin{enumerate}
    \item Given two transverse points $a,b\in F_\Theta$, there exists a positive circle passing through $a$ and $b$.
    \item The arc on a positive circle from $a$ to $b$ is completely contained in a diamond with extremities $a$ and $b$.
    \item The equivariant map associated to the positive circle is positive, i.e.\ it sends positive triples to positive triples.
\end{enumerate} 
\end{proposition}
The first part of the proposition implies that given any diamond $V_y(x,z)\subset F_\Theta$, there is a positive circle going through $x$ and $z$. However, it is important to note that there might not be a positive circle passing through $x,y,$ and $z$. Moreover, the third part gives that a tripod is positive.

To put a metric on a diamond, we look for a way to parameterize them. Recall that from \cite[Theorem 4.5]{GuichardWienhard18}, the positive subsemigroup $U^{>0}$ can be parameterized by an open cone $C\subset \mfg^N$
(where $N$ depends on $\Theta$). That is, there is a diffeomorphism
\[
    F\colon C\to U^{>0}.
\]
Let $(x,y,z)$ be a tripod and choose some positive circle passing through these points as well as an isomorphism $\sigma\colon \to U_z$, where $U_z$ is the unipotent radical of the stabilizer of $z$ in $G$. Then the map
\begin{align*}
    C&\to F_\Theta\\
    u&\mapsto \sigma\circ F(u).x
\end{align*}
parametrizes the diamond $V_y(x,z)$. There is a special way of choosing the isomorphism $\sigma$ such that $h\mapsto y$ under the above parametrization, and where $h\in C$ is a preferred unipotent associated to the positive circle (see section 4.1 of \cite{GLW21} for details).

Pulling back the Euclidean metric on the cone $C$, we get a complete Riemannian metric $g_{(x,y,z)}$ on the diamond $V_y(x,z)$.

\begin{warning}
Assume we have two tripods $(x,y,z)$ and $(x,y',z)$ so that the diamonds $V_y(x,z)$ and $V_{y'}(x,z)$ agree and $y$ and $y'$ lie on the same positive circle through $x$ and $z$. If $y\neq y'$, then the metrics $g_{(x,y,z)}$ and $g_{(x,y',z)}$ need not be the same.
\end{warning}
Given an arbitrary positive triple $(x,y,z)$ (not necessarily a tripod), it is still possible to put a Riemannian metric $g_{(x,y,z)}$ on the diamond $V_{y}(x,z)$. The metric is defined by taking the average of the metrics coming from tripods that are "as close as possible" to the triple $(x,y,z)$. We stress again that the metric depends on the triple $(x,y,z)$ as opposed to only on the diamond defined by the triple.\\

The next proposition is about how metrics for diamonds contained in each other are related.
\begin{proposition}[{\cite[Proposition 4.11]{GLW21}}]\label{prop: containment in diamonds means metrics expand}
Let $(x_m,y_m,z_m)$ be a sequence of positive triples such that
\begin{itemize}
    \item $V_{y_m}(x_m,z_m)\subset V_{y_{m+1}(x_{m+1},z_{m+1}})$, and 
    \item $\bigcap_{m\in\N}V_{y_m}(x_m,z_m)=\{\bullet\}$.
\end{itemize}
Then 
\[
    g_{(x_0,y_0,z_0)}\leq k_m g_{(x_m,y_m,z_m)},
\]
with $k_m\to 0$ as $m\to\infty$.
\end{proposition}
We do not reproduce the proof here, but remark that the vague intuition behind this statement is as follows. The metric on diamonds is defined by the pullback to the cone $C$. Thus, if a diamond $V$ is contained in another diamond $W$, the metric in $W$ has to be larger.

\section{Basics of Anosov representations}\label{section:anosov representations definition for GLW21}
In this section, we give a definition of Anosov representations which will be useful for us, but do not give the full details.

Begin by choosing an auxiliary hyperbolic metric on $\Sigma$ and with this there is an identification
\[
    T^1\mathbb{H}^2/\pi_1(\Sigma)\cong T^1\Sigma,
\]
where $T^1\mathbb{H}^2$ is the unit tangent bundle of the hyperbolic plane, and $T^1\Sigma$ is the unit tangent bundle to $\Sigma$. Then we identify $T^1\mathbb{H}^2$ with the set of positive triples in $\del \mathbb{H}^2$ as follows. Positivity in $\del \mathbb{H}^2$ is inherited from the identification with the circle $S^1$. Let $(x,y,z)$ be a positive triple and consider the oriented geodesic $c$ from $x$ to $z$ in $\mathbb{H}^2$. Then there is a unique geodesic with $y$ as an endpoint intersecting $c$ orthogonally. Thus $y$ determines a point on the geodesic, and with that, a vector in the unit tangent bundle. Note that the triple $(z,y,x)$ is also positive and gives the negative of the vector determined by $(x,y,z)$. Moreover, the geodesic flow will be denoted by $(x,y,z)\mapsto (x,y_s,z)$, where $s$ is the time parameter. See figure \ref{fig:parametrization of UH with positive triples}.

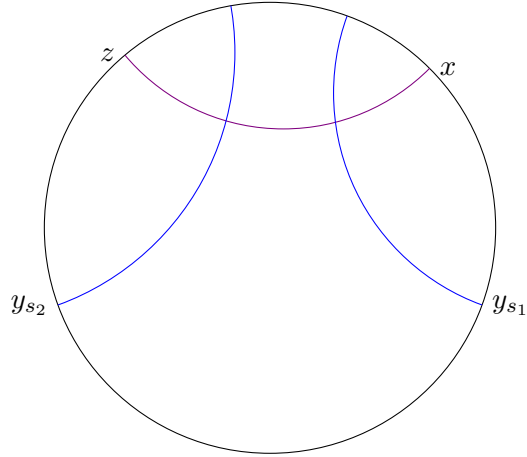
\begin{figure}[H]
    \centering
\begin{tikzpicture}[scale=3]
\node[left] at ({cos(130)},{sin(130)}) {$z$};
\node[right] at ({cos(45)},{sin(45)}) {$x$};
\node[right] at ({cos(-20)},{sin(-20)}) {$y_{s_1}$};
\node[left] at ({cos(200)},{sin(200)}) {$y_{s_2}$};
\draw (0,0) circle (1);
\clip (0,0) circle (1);
\hgline{130}{45}{violet};
\hgline{70}{-20}{blue};
\hgline{100}{200}{blue};
\end{tikzpicture}
\caption{Parametrization of unit tangent bundle with positive triples and the clockwise orientation on the circle. Here the violet geodesic with $x$ as the forward endpoint, and $z$ as the backward endpoint. The blue geodesics meet it orthogonally to determine a vector in $T^1\mathbb{H}^2$. In the picture, $s_2<s_1$.}
\label{fig:parametrization of UH with positive triples}
\end{figure}

The following is an equivalent definition of Anosov representations given by Guichard and Wienhard in \cite[Proposition 2.7]{GW12:AnosovDiscontinuity}.

\begin{definition}
A representation $\rho\colon \pi_1(\Sigma)\to G$ is $\Theta$-Anosov if the following conditions hold.
\begin{enumerate}
    \item \label{def:condition 1 of Anosov representations, equivariant continuous map}There exists a continuous $\rho$-equivariant map $\xi\colon \del\pi_1(T^1\Sigma)\to F_\Theta $ sending distinct points $x,y\in\del\pi_1(T^1\Sigma)$ to transverse points in $F_\Theta$.
    \item\label{def: condition 2 of Anosov representation, metrics contracting} There exists a continuous $\pi_1(\Sigma)$-equivariant family of norms on 
    \[
        \left\{(T_{\xi(z)}F_\Theta, (x,y,z)):(x,y,z)\in T^1\mathbb{H}^2\right\}
    \]
    satisfying the following. There exist positive constants $A$ and $a$ such that for all $t>0$, and $(x,y,z)\in T^1\mathbb{H}^2$ 
    \[
        \norm{\cdot}_{(x,y,z)}\leq Ae^{-at}\norm{\cdot}_{(x,y_{-s},z)},\quad \textnormal{for all } s\geq 0.
    \]
\end{enumerate}
\end{definition}
To be precise, the definition should include a second boundary map to the flag variety $F/P_\Theta^{opp}$, but we only work with one boundary map. Moreover, the continuity for the family of norms in condition \ref{def: condition 2 of Anosov representation, metrics contracting} is with respect to taking the trivial bundle $T^1\mathbb{H}^2\times F_\Theta$.

\section{Proof sketch of Theorem \ref{thm: positive reps are Anosov}}
Part \ref{def:condition 1 of Anosov representations, equivariant continuous map} of the definition of Anosov representations is immediate from the definition of positive representations. For part \ref{def: condition 2 of Anosov representation, metrics contracting}, let $\xi\colon\del\mathbb{H}^2\to F_\Theta$ be the boundary map associated to the positive representation. The goal is to define diamonds with $\xi(z)$ in its interior, so that we can assign the norm on $T_{\xi(z)}F_\Theta$ coming from the Riemannian metric on the diamond.

We begin by assigning to $(x,y,z)\in T^1\mathbb{H}^2$ a fourth element $w(x,y,z)\in\del\mathbb{H}^2$ satisfying
\[
    \mathrm{cr}_{\mathbb{RP}^1}(x,y,z,w(x,y,z))=-1,
\]
where $\mathrm{cr}_{\mathbb{RP}^1}$ is a cross ratio on $\mathbb{RP}^1\cong\del\mathbb{H}^2$. In the literature, such a quadruple is said to be harmonic. One can easily check that the quadruple $(x,y,z,w(x,y,z))$ is a positive quadruple, see figure \ref{fig: harmonic quadruples on the circle}. For example, taking the usual cross ratio on $\mathbb{RP}^1$, the quadruple $(0,1,\infty,-1)$ (in coordinates defined by normalizing the first coordinate of the line) is harmonic. We then assign to the triple $(x,y,z)$ the diamond
\[
    \mathcal{Y}_{(x,y,z)}\coloneqq V_{\xi(z)}(\xi(y),\xi(w(x,y,z)))\subset F_\Theta.
\]
The point $\xi(z)$ lies in the interior of the above diamond. This means that we can denote by $\norm{\cdot}_{(x,y,z)}$ the norm on $T_{\xi(z)}F_\Theta$ coming from the Riemannian metric $g_{(\xi(y),\xi(z),\xi(w(x,y,z))}$. To see that the metrics are expanding along the geodesic flow, we use proposition \ref{prop: containment in diamonds means metrics expand}. The conditions that need to be satisfied are
\begin{enumerate}
    \item \label{condition to satisfy containment} $\mathcal{Y}_{(x,y_{-s},z)}\subset \mathcal{Y}_{(x,y,z)}$ for all $s\geq 0$, and
    \item \label{condition to satisfy limit to a point} $\mathcal{Y}_{(x,y_{-s},z)}\rightarrow \{\xi(z)\}$ as $s\to +\infty$.
\end{enumerate}
For condition \ref{condition to satisfy containment}, one can check with a small computation that if $(x,y_1,y_0,z)$ is a positive quadruple, then $(x,y_1,y_0,z,w(x,y_0,z),w(x,y_1,z))$ is a positive tuple as well, see figure \ref{fig: harmonic quadruples on the circle}. Since the boundary map is positive, it follows from proposition \ref{prop: relationship between diamonds and positivity} that indeed $\mathcal{Y}_{(x,y_{-s},z)}\subset \mathcal{Y}_{(x,y,z)}$ for all $s\geq 0$.

\begin{figure}[H]
    \centering
    \begin{tikzpicture}
    \coordinate (center) at (0,0);
  \def\radius{2.5cm}
  \draw (center) circle[radius=\radius];
  \fill[blue] (center) ++(90:\radius) circle[radius=2pt];
  \node[above] at (0,2.5) {$x$};
    \fill[blue] (center) ++(45:\radius) circle[radius=2pt];
  \node[above] at (1.3*1.41,1.3*1.41) {$y_1$};
    \fill[blue] (center) ++(180:\radius) circle[radius=2pt];
  \node[left] at (-2.5,0) {$z$};
    \fill[blue] (center) ++(315:\radius) circle[radius=2pt];
  \node[right] at (1.3*1.41,-1.3*1.41) {$y_0$};
      \fill[blue] (center) ++(150:\radius) circle[radius=2pt];
    \node[left] at ({2.5*cos(150)},{2.5*sin(150)} ) {$w(x,y_0,z)$};
    \fill[blue] (center) ++(120:\radius) circle[radius=2pt];
    \node[left] at ({2.5*cos(120)},{2.5*sin(120)} ) {$w(x,y_1,z)$};
        \end{tikzpicture}
        \caption{Harmonic quadruples on $S^1$.}
        \label{fig: harmonic quadruples on the circle}
\end{figure}
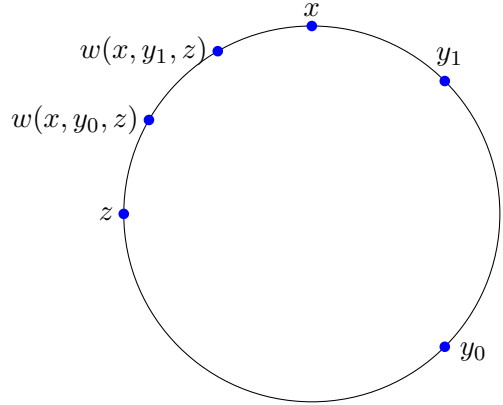

For condition \ref{condition to satisfy limit to a point}, observe (from another calculation) that \[
    \lim_{y\to z} w(x,y,z) = z.
\]
Then using the continuity of the boundary  map, and that $y_{-s}\to z$ as $s\to\infty$, we get condition \ref{condition to satisfy limit to a point}. Now we can apply proposition \ref{prop: containment in diamonds means metrics expand} to get that 
\[
    \norm{\cdot}_{(x,y,z)}\leq k_s\norm{\cdot}_{(x,y_{-s},z)}
\]
with $k_s\to 0$ as $s\to\infty$. One can then finish the proof by using a compactness argument to get the constants $A$ and $a$ in the definition of Anosov representations.

\newpage

\thispagestyle{empty}

\chapter[Exercises]{Exercises\\ {\Large\textnormal{\textit{by Merik Niemeyer}}}}
\addtocontents{toc}{\quad\quad\quad \textit{Merik Niemeyer}\par}

Our goal in this exercise is to find $\Theta$-positive structures on both $\SO(2,3)$ and $\Sp(4,\mathbb{R})$.

\section{The indefinite special orthogonal group}

Let us first introduce the setting. A lot of this you have seen already in the talk on $\SO(p,q)$.

Define a quadratic form on $\mathbb{R}^5$ by $x\mapsto \tran{x}Qx$ with
\begin{equation*}
	Q=
	\left(\begin{smallmatrix}
		0 & 0 & 0 & 0 & -1\\
		0 & 0 & 0 & 1 & 0\\
		0 & 0 & -1 & 0 & 0\\
		0 & 1 & 0 & 0 & 0\\
		-1 & 0 & 0 & 0 & 0
	\end{smallmatrix}\right).
\end{equation*}
Then $G:=\SO(Q)=\{X\in\mathrm{SL}(5,\mathbb{R}): \tran{X}QX=Q\}\cong \SO(2,3)$.
Recall that matrices in the Lie algebra $\mathfrak{g}=\mathfrak{so}(2,3)$ are of the form
\begin{equation*}
	\left(\begin{array}{cc|c|cc}
		a_{11} & a_{12} & v & b & 0\\
		a_{21} & a_{22} & w & 0 & b\\
		\hline
		r & s & 0 & w & -v\\
		\hline
		c & 0 & s & -a_{22} & a_{12}\\
		0 & c & -r & a_{21} & -a_{11}
	\end{array}\right).
\end{equation*}
with real entries.
You have also seen already that a Cartan subspace is given by
\begin{equation*}
	\mathfrak{h}=\{\mathrm{diag}(\lambda,\mu,0,-\mu,-\lambda) : \lambda,\mu\in\mathbb{R}\}=\mathbb{R}\underbrace{\mathrm{diag}(1,0,0,0,-1)}_{=:h_1}\oplus\;\mathbb{R}\underbrace{\mathrm{diag}(0,1,0,-1,0)}_{:=h_2},
\end{equation*}
i.e.\ the diagonal matrices in $\mathfrak{g}$. Define the matrices
\begin{align*}
	e_1&:=\left(\begin{array}{cc|c|cc}
		0 & 1 & 0 & 0 & 0\\
		0 & 0 & 0 & 0 & 0\\
		\hline
		0 & 0 & 0 & 0 & 0\\
		\hline
		0 & 0 & 0 & 0 & 1\\
		0 & 0 & 0 & 0 & 0\\
	\end{array}\right), \hspace{5.5mm}
	e_2:=\left(\begin{array}{cc|c|cc}
		0 & 0 & 0 & 0 & 0\\
		0 & 0 & 1 & 0 & 0\\
		\hline
		0 & 0 & 0 & 1 & 0\\
		\hline
		0 & 0 & 0 & 0 & 0\\
		0 & 0 & 0 & 0 & 0\\
	\end{array}\right),\\
	e_3&:=\left(\begin{array}{cc|c|cc}
		0 & 0 & 1 & 0 & 0\\
		0 & 0 & 0 & 0 & 0\\
		\hline
		0 & 0 & 0 & 0 & -1\\
		\hline
		0 & 0 & 0 & 0 & 0\\
		0 & 0 & 0 & 0 & 0\\
	\end{array}\right), \hspace{3mm}
	e_4:=\left(\begin{array}{cc|c|cc}
		0 & 0 & 0 & 1 & 0\\
		0 & 0 & 0 & 0 & 1\\
		\hline
		0 & 0 & 0 & 0 & 0\\
		\hline
		0 & 0 & 0 & 0 & 0\\
		0 & 0 & 0 & 0 & 0\\
	\end{array}\right),
\end{align*}
and $f_i:=\tran{e_i}$. Together with $h_1$ and $h_2$ these span the Lie algebra.
\begin{remark}
	Note that $\mathfrak{g}$ is generated by $h_1,h_2,e_1,e_2,f_1,f_2$.
\end{remark}
Define two functions $\alpha_1,\alpha_2\in\mathfrak{h^*}$ by
\begin{align*}
	&\alpha_1(\mathrm{diag}(\lambda,\mu,0,-\mu,-\lambda)):=\lambda-\mu,\\
	&\alpha_2(\mathrm{diag}(\lambda,\mu,0,-\mu,-\lambda)):=\mu.
\end{align*}
\begin{exercise}
	Compute the root space decomposition of $\mathfrak{g}$ and prove that $\mathfrak{g}$ has a $B_2$ root system.
\end{exercise}
In case you do not want to spend time on this repetition, here are the positive root spaces:
\begin{equation*}
	\mathfrak{g}_{\alpha_1}=\mathbb{R}e_1,\hspace{2mm}
	\mathfrak{g}_{\alpha_2}=\mathbb{R}e_2,\hspace{2mm}
	\mathfrak{g}_{\alpha_1+\alpha_2}=\mathbb{R}e_3,\hspace{2mm}
	\mathfrak{g}_{\alpha_1+2\alpha_2}=\mathbb{R}e_4.
\end{equation*}
The corresponding negative root spaces are spanned by the $f_i$ and $\mathfrak{g}_0=\mathfrak{h}$ since $G$ is split. Also, we see that $\Delta=\{\alpha_1,\alpha_2\}$ is a set of simple roots.

\subsection{The totally positive structure}
The first $\Theta$-positive structure we want to investigate is the totally positive structure, i.e.\ pick $\Theta=\Delta$.
However, we will look at this from the viewpoint of $\Theta$-positivity.
\begin{exercise}
	For this choice of $\Theta$ compute $\mathfrak{u}_\Theta,\mathfrak{u}_\Theta^{\mathrm{opp}}$ and $\mathfrak{l}_\Theta$. Find the center $\mathfrak{z}_\Theta$ of $\mathfrak{l}_\Theta$ and the weight spaces $\mathfrak{u}_\alpha$ for the (adjoint) action of this on $\mathfrak{u}_\Theta$.
\end{exercise}

\begin{remark}
	Note that the weight spaces coincide with the root spaces and that all of these are $1$-dimensional.
\end{remark}
In a second step, we wish to prove that $G$ actually admits a $\Theta$-positive structure for this choice of $\Theta$, i.e.\ we need to find invariant cones in the weight spaces:
\begin{exercise}
	By exponentiating, find $L_\Theta^\circ$ (the identity component of the Levi subgroup) and compute the action on $\mathfrak{u}_\alpha$ for $\alpha\in\Theta$. Prove that there exists a sharp convex cone in each weight space, which is invariant under this action. 
\end{exercise}
With the $\Theta$-positive structure in place, we will compute the $\Theta$-positive semigroup.
\begin{exercise}
	Observe that the Weyl group $W(\Theta)$ is just the full Weyl group $W$.
	Use the longest word to give a parametrization $U_\Theta^{>0}$.
\end{exercise}
\begin{hint}
	The longest word in the Weyl group of a $B_2$ root system is $\sigma_1\sigma_2\sigma_1\sigma_2=\sigma_2\sigma_1\sigma_2\sigma_1$. It might not be worth it to explicitly compute a matrix form of an element in $U_\Theta^{>0}$.
\end{hint}

Now you could compute $U_\Theta^{\mathrm{opp},>0}$ in much the same way and finally get $G_\Theta^{>0}$.

\begin{remark}
	Note that this $\Theta$-positive semigroup coincides with Lusztig's totally positive semigroup.
\end{remark}

\begin{exercise}
As a non-example, show that $\SO(2,q)$ for $q>3$ does not admit a $\Theta$-positive structure if we chose $\Theta=\Delta$, the set of simple (restricted) roots.	
\end{exercise}

\subsubsection*{Another $\mathbf{\Theta}$-positive structure}
Now we will define a second $\Theta$-positive structure on $G$ by choosing $\Theta=\{\alpha_1\}$.
\begin{remark}
	Recall the classification of Lie groups admitting a $\Theta$-positive structure. As $G$ is both split real and (locally isomorphic to) $\SO(p,q)$, we find two different $\Theta$-positive structures on $G$.
\end{remark}
We can proceed in much the same way as before.
\begin{exercise}
	Prove that $G$ admits a $\Theta$-positive structure by explicitly calculating the $L_\Theta^\circ$-invariant cone in $\mathfrak{u}_{\alpha_1}$.
\end{exercise}
\begin{hint}
	You should find that elements of $L_\Theta^\circ$ are of the form
	\begin{equation*}
	\begin{pmatrix}
		a & 0 & 0\\
		0 & M & 0\\
		0 & 0 & a^{-1}
	\end{pmatrix}
\end{equation*}
where $a\in\mathbb{R}_{>0}$ and $M\in\SO(J)^\circ$ (here we denote by $J$ the center $3\times 3$-block in $Q$) and that elements of $\mathfrak{u}_{\alpha_1}=\mathfrak{u}_\Theta$ are of the form
\begin{equation*}
	\begin{pmatrix}
		0 &\tran{v} & 0\\
		0 & 0 & Jv\\
		0 & 0 & 0
	\end{pmatrix}
\end{equation*}
for some $v\in\mathbb{R}^3$.
The sharp convex cone will be defined by the conditions $\tran{v}Jv\geq 0$ and $v_1\geq 0$.
\end{hint}
We can also compute the Weyl group $W(\Theta)$ in order to define $U_\Theta^{>0}$ (and from this $G_\Theta^{>0}$ like before):
\begin{exercise}
	Compute the Weyl group $W(\Theta)$ and show that this is isomorphic to the Weyl group of an $A_1$ root system and give a parametrization of $U_\Theta^{>0}$.
\end{exercise}
\begin{remark}
	We have also learned that $G$ is of Hermitian type of tube type. The $\Theta$-positive structure it admits as such a group is the same one we just described.
\end{remark}
	
\subsection{Another $\mathbf{\Theta}$-positive structure?}
\begin{exercise}
	Show that for the choice $\Theta=\{\alpha_2\}$, G does \textbf{not} admit a $\Theta$-positive structure.
\end{exercise}

\section{The symplectic group}

Define $J_{2,2}$ as the block matrix
\begin{equation*}
	J_{2,2}=
	\begin{pmatrix}
		0 & I_2\\
		-I_2 & 0
	\end{pmatrix}
\end{equation*}
and recall that the symplectic group is $G':=\Sp(4,\mathbb{R})=\{M\in \GL(4,\R): \tran{M}J_{2,2}M=J_{2,2}\}$.
Its Lie algebra is $\mathfrak{g}'=\mathfrak{sp}(4,\mathbb{R})=\{X \mathfrak{gl}(4,\R): \tran{X}J_{2,2}+J_{2,2}X=0\}$, so an element in $\mathfrak{g}'$ is of the form
\begin{equation*}
	X=
	\begin{pmatrix}
		A & B\\
		C & D\\
	\end{pmatrix},
\end{equation*}
where $C,B$ are symmetric ($2\times 2$)-matrices and $D=-\tran{A}$.
A Cartan subspace is given by
\begin{equation*}
	\mathfrak{h}'=\{\mathrm{diag}(\lambda,\mu,-\lambda,-\mu): \lambda,\mu\in\mathbb{R}\}
\end{equation*}
\begin{exercise}
	Compute the root decomposition of $\mathfrak{g}'$ and prove that $\mathfrak{g}'$ has a $B_2$ root system.
\end{exercise}

\subsection{A local isomorphism with $\SO(2,3)$}
We observed that the root systems of $G$ and $G'$ are both of type $B_2$ and thus isomorphic.
It is not hard to write this isomorphism down explicitly and starting from this we can construct an isomorphism of the Lie algebras $\mathfrak{g}$ and $\mathfrak{g}'$.
The resulting isomorphism $\psi \from \mathfrak{g}\to\mathfrak{g}'$ is defined on the generators of $\mathfrak{g}$ by:
\begin{align*}
	\psi(h_1)&=\mathrm{diag}(1/2,1/2,-1/2,-1/2)=:h_1',\\
	\psi(h_2)&=\mathrm{diag}(1/2,-1/2,1/2,-1/2)=:h_2',\\
	\psi(e_1)&=
		\left(\begin{array}{cc|cc}
		0 & 0 & 0 & 0\\
		0 & 0 & 0 & 1\\
		\hline
		0 & 0 & 0 & 0\\
		0 & 0 & 0 & 0\\
		\end{array}\right)=:e_1',\\
	\psi(e_2)&=
		\left(\begin{array}{cc|cc}
		0 & 1 & 0 & 0\\
		0 & 0 & 0 & 0\\
		\hline
		0 & 0 & 0 & 0\\
		0 & 0 & -1 & 0\\
		\end{array}\right)=:e_2',\\
	\psi(f_1)&=(e_1')^T=:f_1',\\
	\psi(f_2)&=(e_2')^T/2=:f_2'.
\end{align*}
Under this isomorphism the simple roots $\alpha_1$ and $\alpha_2$ correspond to $\alpha_1',\alpha_2'\in\mathfrak{h}'$, respectively, which are defined by
\begin{align*}
	&\alpha_1'(\mathrm{diag}(\lambda,\mu,-\lambda,-\mu))=2\mu,\\
	&\alpha_2'(\mathrm{diag}(\lambda,\mu,-\lambda,-\mu))=\lambda-\mu.
\end{align*}
\begin{remark}
	This isomorphism of Lie algebras induces a local isomorphism between the Lie groups $G$ and $G'$.
\end{remark}

\subsection{$\Theta$-positive structures on $\Sp(4,\mathbb{R})$}
Here we only consider the $\Theta$-positive structure, which $G'$ carries as a group locally isomorphic to $\SO(2,3)$ (see above), which is actually the same as the one which it carries as a Hermitian group of tube type.
We can easily carry over our previous calculations using the above isomorphism.
\begin{exercise}
	Use the isomorphism $\psi \from \mathfrak{g}\to\mathfrak{g}'$ to prove that $G'$ admits a $\Theta$-positive structure for the choice $\Theta=\{\alpha_1'\}$.
\end{exercise}
\begin{hint}
	You should find that the weight space $\mathfrak{u}_{\alpha_1'}$ is isomorphic to the space of symmetric matrices and that the invariant cone is given by the positive semi-definite matrices.
\end{hint}

Some notes (almost solutions) on the exercises.

\section{Partial solutions}
\subsection{$\SO(2,3)$}
\subsubsection{The totally positive structure}
Define a quadratic form on $\mathbb{R}^5$ by $x\mapsto \tran{x}Qx$ with
\begin{equation*}
	Q=
	\left(\begin{smallmatrix}
		0 & 0 & 0 & 0 & -1\\
		0 & 0 & 0 & 1 & 0\\
		0 & 0 & -1 & 0 & 0\\
		0 & 1 & 0 & 0 & 0\\
		-1 & 0 & 0 & 0 & 0
	\end{smallmatrix}\right).
\end{equation*}
Then $G:=\SO(Q)=\{M\in\mathrm{SL}(5,\mathbb{R}): \tran{M}QM=Q\}\cong \SO(2,3)$ and matrices in the Lie algebra $\mathfrak{g}=\mathfrak{so}(2,3)$ are of the form
\begin{equation*}
	\left(\begin{array}{cc|c|cc}
		a_{11} & a_{12} & v & b & 0\\
		a_{21} & a_{22} & w & 0 & b\\
		\hline
		r & s & 0 & w & -v\\
		\hline
		c & 0 & s & -a_{22} & a_{12}\\
		0 & c & -r & a_{21} & -a_{11}
	\end{array}\right).
\end{equation*}

To work out the root decomposition we pick the maximal toral subalgebra
\begin{equation*}
	\mathfrak{h}=\{\mathrm{diag}(\lambda,\mu,0,-\mu,-\lambda):\lambda,\mu\in\mathbb{R}\}=\mathbb{R}\underbrace{\mathrm{diag}(1,0,0,0,-1)}_{=:h_1}\oplus\;\mathbb{R}\underbrace{\mathrm{diag}(0,1,0,-1,0)}_{:=h_2},
\end{equation*}
i.e.\ the diagonal matrices in $\mathfrak{g}$.

Now we will investigate the root spaces: For this, define the following matrices:
\begin{align*}
	e_1&:=\left(\begin{array}{cc|c|cc}
		0 & 1 & 0 & 0 & 0\\
		0 & 0 & 0 & 0 & 0\\
		\hline
		0 & 0 & 0 & 0 & 0\\
		\hline
		0 & 0 & 0 & 0 & 1\\
		0 & 0 & 0 & 0 & 0\\
	\end{array}\right), \hspace{5.5mm}
	e_2:=\left(\begin{array}{cc|c|cc}
		0 & 0 & 0 & 0 & 0\\
		0 & 0 & 1 & 0 & 0\\
		\hline
		0 & 0 & 0 & 1 & 0\\
		\hline
		0 & 0 & 0 & 0 & 0\\
		0 & 0 & 0 & 0 & 0\\
	\end{array}\right),\\
	e_3&:=\left(\begin{array}{cc|c|cc}
		0 & 0 & 1 & 0 & 0\\
		0 & 0 & 0 & 0 & 0\\
		\hline
		0 & 0 & 0 & 0 & -1\\
		\hline
		0 & 0 & 0 & 0 & 0\\
		0 & 0 & 0 & 0 & 0\\
	\end{array}\right), \hspace{3mm}
	e_4:=\left(\begin{array}{cc|c|cc}
		0 & 0 & 0 & 1 & 0\\
		0 & 0 & 0 & 0 & 1\\
		\hline
		0 & 0 & 0 & 0 & 0\\
		\hline
		0 & 0 & 0 & 0 & 0\\
		0 & 0 & 0 & 0 & 0\\
	\end{array}\right).
\end{align*}
These form a linear basis for the subspace of strictly upper triangular matrices in $\mathfrak{g}$. In fact this subspace is a sub\textit{algebra} and as such is generated by $e_1$ and $e_2$.
 
Moreover, each $e_i$ spans a positive root spaces: Define the functions $\alpha_1,\alpha_2\in\mathfrak{h^*}$ by
\begin{align*}
	&\alpha_1(\mathrm{diag}(\lambda,\mu,0,-\mu,-\lambda)):=\lambda-\mu,\\
	&\alpha_2(\mathrm{diag}(\lambda,\mu,0,-\mu,-\lambda)):=\mu,
\end{align*}
and hence the positive root spaces are
\begin{equation*}
	\mathfrak{g}_{\alpha_1}=\mathbb{R}e_1,\hspace{2mm}
	\mathfrak{g}_{\alpha_2}=\mathbb{R}e_2,\hspace{2mm}
	\mathfrak{g}_{\alpha_1+\alpha_2}=\mathbb{R}e_3,\hspace{2mm}
	\mathfrak{g}_{\alpha_1+2\alpha_2}=\mathbb{R}e_4.
\end{equation*}
The corresponding negative root spaces are spanned by the matrices $f_i:=\tran{e_i}$ for $i\in\{1,2,3,4\}$. We thus find that this is a $B_2$ root system with $\alpha_1$ and $\alpha_2$ a pair of simple roots.

We first pick $\Theta=\Delta$, so that
\begin{align*}
	\mathfrak{u}_\Theta &=\sum_{\alpha\in\Sigma^+}\mathfrak{g}_\alpha,\\
	\mathfrak{u}_\Theta^\mathrm{opp}&=\sum_{\alpha\in\Sigma^+}\mathfrak{g}_{-\alpha},
\end{align*}
and clearly in this case $\mathfrak{l}_\Theta=\mathfrak{z}_\Theta=\mathfrak{g}_0=\mathfrak{h}$. Consequently, the weight spaces for the action $\mathfrak{z}_\Theta\curvearrowright\mathfrak{u}_\Theta$ coincide with the root spaces and we find $\mathfrak{u}_{\alpha_i}=\mathfrak{g}_{\alpha_i}$. By exponentiating elements in $\mathfrak{l}_\Theta$ we find that the identity component of the Levi subgroup is
\begin{equation*}
	L_\Theta^\circ=\{\mathrm{diag}(a,b,1,b^{-1},a^{-1}): a,b\in\mathbb{R}_+\}.
\end{equation*}
If we let this act on $\mathfrak{u}_{\alpha_1}$ by the adjoint action, we find that $X=\mathrm{diag}(a,b,1,b^{-1},a^{-1})$ maps $e_1$ to $ab^{-1}e_1$. Thus, if we identify the one-dimensional space $\mathfrak{u}_{\alpha_1}$ with $\mathbb{R}$, we find that the sharp convex cone $\mathbb{R}_+$ is preserved by this action ($a,b>0$). Analogously for $\mathfrak{u}_{\alpha_2}$, and we proved that $G$ admits a $\Theta$-positive structure.

Exponentiating the root spaces for roots in $\Theta$ yields two maps
\begin{align*}
	x_i:\mathbb{R}&\to G\\
	t&\mapsto\exp(te_i),
\end{align*}
for $i\in\{1,2\}$. More explicitly, these map to the subgroup $U_\Theta$ of $G$ of unipotent upper triangular matrices. As we choose $\Theta=\Delta$, we find that $W(\Theta)=W$, i.e.\ the Weyl group of the root system of $G$, which is of type $B_2$. The longest word in such a root system can be written as $\sigma_1\sigma_2\sigma_1\sigma_2$, and therefore we obtain a parametrization of the totally positive subsemigroup $U_\Theta^{>0}$ of $U_\Theta$, namely
\begin{equation*}
	U_\Theta^{>0}:=\left\{x_1(x)x_2(v)x_1(y)x_2(w)=:F_{1212}(x,v,y,w): x,v,y,w\in\mathbb{R}^{>0}\right\}.
\end{equation*}
An element in this can be explicitly computed to be
\begin{equation*}
	F_{1212}(x,v,y,w)=
	\begin{pmatrix}
		1 & x+y & xv+(x+y)w & x(v+w)^2/2+yw^2/2 & xyv^2/2\\
		0 & 1 & v+w & (v+w)^2/2 & yv^2/2\\
		0 & 0 & 1 & v+w & yv\\
		0 & 0 & 0 & 1 & x+y\\
		0 & 0 & 0 & 0 & 1
	\end{pmatrix}.
\end{equation*}

Analogously one can compute the positive part $O^{>0}$ of the subgroup $O$ of unipotent lower triangular matrices. The totally positive subsemigroup $G^{>0}$ is then
\begin{equation*}
	G^{>0}=O^{>0}A^\circ U^{>0},
\end{equation*}
where $A^\circ$ is the connected component of the identity in the subgroup $A<G$ of diagonal matrices in $G$.

It is not hard to see that the positive semigroup defined this way is just Lusztig's totally positive subgroup and much of the discussion actually mirrors Lusztig's description. However, this is only possible because we are dealing with a split real group.

If we naively tried to extend this notion to non-split groups we run into problems as the following example shows: Consider $G_2=\SO(2,q)$ with $q>3$. The Lie algebra then has the form
\begin{equation*}
	\mathfrak{so}(2,q)=\left\{
	\left(\begin{array}{cc|c|cc}
		a_{11} & a_{12} & \tran{v} & b & 0\\
		a_{21} & a_{22} & \tran{w} & 0 & b\\
		\hline
		r & s & 0 & w & -v\\
		\hline
		c & 0 & \tran{s} & -a_{22} & a_{12}\\
		0 & c & -\tran{r} & a_{21} & -a_{11}
	\end{array}\right): a_{ij},b,c\in\mathbb{R},v,w,r,s\in\mathbb{R}^{q-2}
	\right\},
\end{equation*}
and we recall that the restricted root system of this group is also of type $B_2$ and indeed that calculation is very similar to the above but we find that the (restricted) root space $g_{\alpha_2}$ is ($q-2$)-dimensional in this case and contains all elements of the Lie algebra of the above form with all but the $w$-entries 0. Now, if we take $\Theta=\Delta$, $\mathfrak{u}_\Theta$ and $\mathfrak{u}_\Theta^\mathrm{opp}$ are again just the sum of all positive and negative root spaces, respectively, and we also find $\mathfrak{l}_\Theta=\mathfrak{g}_0$ again. However, this takes a slightly different form since $G_2$ is not split.
\begin{equation*}
	\mathfrak{l}_\Theta=\mathfrak{g}_0=\mathfrak{a}\oplus C_{\mathfrak{k}}(\mathfrak{a})=\left\{\left(
	\begin{smallmatrix}
		\lambda &&&&\\
		& \mu &&&\\
		&& M &&\\
		&&& -\mu &\\
		&&&& -\lambda
	\end{smallmatrix}\right)
	: M=-\tran{M}\right\},
\end{equation*}
where $\mathfrak{a}$ is a Cartan subspace and $\mathfrak{g}=\mathfrak{k}\oplus\mathfrak{p}$ the Cartan decomposition. The center clearly is $\mathfrak{z}_\Theta=\mathfrak{a}$ and thus the weight spaces for the action $\mathfrak{z}_\Theta\curvearrowright\mathfrak{u}_\Theta$ are just the root spaces again. Exponentiating the above expression, we find that a general element in $L_\Theta^\circ$ has the form
\begin{equation*}
	Y=
	\left(\begin{smallmatrix}
		a &&&&\\
		& b &&&\\
		&& S &&\\
		&&& b^{-1} &\\
		&&&& a^{-1}\\
	\end{smallmatrix}\right)
\end{equation*}
with $S\in\SO(q-2)$.
If we consider the adjoint action on the weight space $\mathfrak{u}_{\alpha_2}=\mathfrak{g}_{\alpha_2}\cong\mathbb{R}^{q-2}$, we find $Y.v=bSv$, so that we are effectively dealing with an $\SO(q-2)$-action on $\mathbb{R}^{q-2}$, which clearly has \underline{no} invariant sharp convex cone, so that $G_2$ does not admit a $\Theta$-positive structure for this particular choice of $\Theta$.

\subsubsection{Another $\mathbf{\Theta}$-positive structure}
The group $G=\SO(2,3)$ can also be endowed with another $\Theta$-positive structure. For this, we mostly use the same setup as before and have $\Delta=\{\alpha_1,\alpha_2\}$. Now we pick $\Theta=\{\alpha_1\}$ and thus have (with $\Sigma_\Theta^+=\Sigma^+\setminus(\mathrm{Span}(\Delta-\Theta))$, i.e.\ all positive roots that contain some contribution from $\Theta$)
\begin{align*}
	\mathfrak{u}_\Theta &:=\sum_{\alpha\in\Sigma_\Theta^+}\mathfrak{g}_\alpha=\mathfrak{g}_{\alpha_1}\oplus\mathfrak{g}_{\alpha_1+\alpha_2}\oplus\mathfrak{g}_{\alpha_1+2\alpha_2}=\mathbb{R}e_1\oplus\mathbb{R}e_3\oplus\mathbb{R}e_4,\\
	\mathfrak{u}_\Theta^{\mathrm{opp}}&:=\sum_{\alpha\in\Sigma_\Theta^+}\mathfrak{g}_{-\alpha}=\mathfrak{g}_{-\alpha_1}\oplus\mathfrak{g}_{-\alpha_1-\alpha_2}\oplus\mathfrak{g}_{-\alpha_1-2\alpha_2}=\mathbb{R}f_1\oplus\mathbb{R}f_3\oplus\mathbb{R}f_4,\\
	\mathfrak{l}_\Theta &:=\mathfrak{g}_0\oplus\sum_{\alpha\in\mathrm{Span}(\Delta-\Theta)\cap\Sigma^+}(\mathfrak{g}_\alpha\oplus\mathfrak{g}_{-\alpha})=\mathfrak{h}\oplus\mathfrak{g}_{\alpha_2}\oplus\mathfrak{g}_{-\alpha_2}.
\end{align*}
In order to determine the $\Theta$-positive subsemigroup, we need to consider the adjoint action of the center $\mathfrak{z}_\Theta$ of $\mathfrak{l}_\Theta$ on $\mathfrak{u}_\Theta$.
In this simple setup, we clearly have
\begin{equation*}
	\mathfrak{z}_\Theta=\{\mathrm{diag}(\lambda,0,0,0,-\lambda): \lambda\in\mathbb{R}\},
\end{equation*}
and thus $\mathfrak{u}_\Theta=\mathfrak{u}_{\alpha_1}$.

By exponentiating an element in $\mathfrak{l}_\Theta$ we see that a general element in the identity component of the Levi subgroup $L_\Theta$ is of the form
\begin{equation*}
	\begin{pmatrix}
		a & 0 & 0\\
		0 & M & 0\\
		0 & 0 & a^{-1}
	\end{pmatrix}
\end{equation*}
where $a\in\mathbb{R}^{>0}$ and $M\in\SO(J)^\circ$. Here we denote by $J$ the center ($3\times 3$)-block in $Q$, i.e.\
\begin{equation*}
	J=
	\begin{pmatrix}
		0 & 0 & 1\\
		0 & -1 & 0\\
		1 & 0 & 0^1
	\end{pmatrix}.
\end{equation*}
Now, elements of $\mathfrak{u}_{\alpha_1}=\mathfrak{u}_\Theta$ are of the form
\begin{equation*}
	\begin{pmatrix}
		0 & \tran{v} & 0\\
		0 & 0 & Jv\\
		0 & 0 & 0
	\end{pmatrix}
\end{equation*}
for some $v\in\mathbb{R}^3$.
Therefore, we identify this space with $\mathbb{R}^3$ and letting $L_\Theta^\circ$ act on $\mathfrak{u}_{\alpha_1}$ by conjugation, we find that the sharp convex cone $c$ defined by the conditions $\tran{v}Jv\geq 0$ and $v_1\geq 0$ is invariant under this action.
So $G$ admits a $\Theta$-positive structure for the above choice of $\Theta$.

Our goal is to find the $\Theta$-positive subsemigroup $G_\Theta^{>0}$ and for this we need to consider the Weyl group $W(\Theta)$. In our case this is very simple as we have $\Theta=\{\alpha_1\}$, which only contains a single element. Thus $\beta_\Theta=\alpha_1$ and $W(\Theta)$ is the subgroup of the full Weyl group $W$ that is generated by the longest element $s$ of the Weyl group $W_{\{\beta_\Theta\}\cup(\Delta-\Theta)}=W$. This is isomorphic to the Weyl group of an $A_1$ root system and thus the longest element in $W(\Theta)$ is simply $s$.

Accordingly the $\Theta$-positive subsemigroup of $U_\Theta=\exp(\mathfrak{u}_\Theta)$ is the image of the map
\begin{align*}
	F:c^\circ &\to U_\Theta\\
	v&\mapsto\exp(v),
\end{align*}
i.e. $U_\Theta^{>0}=F(c^\circ)$.

Similarly, we can compute $U_\Theta^{\mathrm{opp},>0}$ and thus fin $G_\Theta^{>0}$, which is generated by $U_\Theta^{>0},U_\Theta^{\mathrm{opp},>0}$ and $L_\Theta^\circ$.

\subsubsection{Another $\mathbf{\Theta}$-positive structure?}
Finally, we want to see whether we can find another $\Theta$-positive structure on $G$ by taking $\Theta=\{\alpha_2\}$ (keep the notation from before). The calculation works much the same as in the previous case but we find $\mathfrak{u}_\Theta=\mathfrak{u}_{\alpha_2}\oplus\mathfrak{u}_{2\alpha_2}$, where $\mathfrak{u}_{\alpha_2}=\mathbb{R}e_2\oplus\mathbb{R}e_3\cong\mathbb{R}^2$. A calculation shows that the action $L_\Theta^\circ\curvearrowright\mathfrak{u}_{\alpha_2}$ corresponds to an $\mathrm{SL}(2,\mathbb{R})$-action on $\mathbb{R}^2$, which has no sharp invariant cone and so $G$ does not admit a $\Theta$-positive structure for this choice of $\Theta$.

\subsection{$\mathbf{Sp(4,\mathbb{R})}$}
\subsubsection{A local isomorphism with $\mathbf{G}$}
Define $J_{2,2}$ as the block matrix
\begin{equation*}
	J_{2,2}=
	\begin{pmatrix}
		0 & I_2\\
		-I_2 & 0
	\end{pmatrix}
\end{equation*}
and recall that the symplectic group is $G':=\Sp(4,\mathbb{R})=\{M\in \GL(4,\R):\tran{M}J_{2,2} M=J_{2,2}\}$. Its Lie algebra is $\mathfrak{g}'=\mathfrak{sp}(4,\mathbb{R})=\{X\in \mathfrak{gl}(4,\R) : \tran{X}Q+QX=0\}$, so an element in $\mathfrak{g}'$ is of the form
\begin{equation*}
	X=
	\begin{pmatrix}
		A & B\\
		C & D\\
	\end{pmatrix},
\end{equation*}
where $C,B$ are symmetric ($2\times 2$)-matrices and $D=-\tran{A}$.
A maximal toral subalgebra is given by
\begin{equation*}
	\mathfrak{h}'=\{\mathrm{diag}(\lambda,\mu,-\lambda,-\mu):\lambda,\mu\in\mathbb{R}\}=\mathbb{R}\underbrace{\mathrm{diag}(1,0,-1,0)}_{=:\tilde{h}_1}\oplus\,\mathbb{R}\underbrace{\mathrm{diag}(0,1,0,-1)}_{=:\tilde{h}_2}.
\end{equation*}
As a vector space $\mathfrak{g}'$ is spanned by the matrices
\begin{align*}
	\tilde{e}_1&:=\left(\begin{array}{cc|cc}
		0 & 1 & 0 & 0\\
		0 & 0 & 0 & 0\\
		\hline
		0 & 0 & 0 & 0\\
		0 & 0 & -1 & 0\\
	\end{array}\right), \hspace{3mm}
	\tilde{e}_2:=\left(\begin{array}{cc|cc}
		0 & 0 & 1 & 0\\
		0 & 0 & 0 & 0\\
		\hline
		0 & 0 & 0 & 0\\
		0 & 0 & 0 & 0\\
	\end{array}\right),\\
	\tilde{e}_3&:=\left(\begin{array}{cc|cc}
		0 & 0 & 0 & 0\\
		0 & 0 & 0 & 1\\
		\hline
		0 & 0 & 0 & 0\\
		0 & 0 & 0 & 0\\
	\end{array}\right), \hspace{5.5mm}
	\tilde{e}_4:=\left(\begin{array}{cc|cc}
		0 & 0 & 0 & 1\\
		0 & 0 & 1 & 0\\
		\hline
		0 & 0 & 0 & 0\\
		0 & 0 & 0 & 0\\
	\end{array}\right),
\end{align*}
and their transposed matrices $\tilde{f}_i=\tilde{e}_i^T$.
If we set $X=\lambda\tilde{h}_1+\mu\tilde{h}_2=\mathrm{diag}(\lambda,\mu,-\lambda,-\mu)$ and compute the bracket of $X$ with the $\tilde{e}_i$, we find
\begin{align*}
	[X,\tilde{e}_1]&=(\lambda-\mu)\tilde{e}_1,\\
	[X,\tilde{e}_2]&=2\lambda\tilde{e}_2,\\
	[X,\tilde{e}_3]&=2\mu\tilde{e}_3,\\
	[X,\tilde{e}_4]&=(\lambda+\mu)\tilde{e}_4,
\end{align*}
and thus if we set
\begin{align*}
	\tilde{\alpha}_1(X)&=\lambda-\mu,\\
	\tilde{\alpha}_2(X)&=2\mu,
\end{align*}
we have the following root spaces
\begin{equation*}
	\mathfrak{g}_{\tilde{\alpha}_1}=\mathbb{R}\tilde{e}_1,\hspace{2mm}
	\mathfrak{g}_{\tilde{\alpha}_2}=\mathbb{R}\tilde{e}_3,\hspace{2mm}
	\mathfrak{g}_{\tilde{\alpha}_1+\tilde{\alpha}_2}=\mathbb{R}\tilde{e}_4,\hspace{2mm}
	\mathfrak{g}_{2\tilde{\alpha}_1+\tilde{\alpha}_2}=\mathbb{R}\tilde{e}_2.
\end{equation*}
The corresponding negative root spaces are spanned by the $\tilde{f}_i$ and we are clearly dealing with another root system of type $B_2$.

Therefore the root systems of $G$ and $G'$ are isomorphic, which means that the Lie algebras $\mathfrak{g}$ and $\mathfrak{g}'$ are isomorphic. The isomorphism of the root systems is given in terms of the simple roots as
\begin{align*}
	\alpha_1\mapsto\tilde{\alpha}_2=:\alpha_1',\\
	\alpha_2\mapsto\tilde{\alpha}_1=:\alpha_2',
\end{align*}
which can be seen by comparing the expressions for the positive roots in terms of the simple roots in both cases. In order to construct a Lie algebra isomorphism $\psi:\mathfrak{g}\to\mathfrak{g}'$, recall that $\mathfrak{g}$ is generated by $h_1,h_2,e_1,e_2,f_1,f_2$. In order to define $\psi$ it suffices to map these generators in such a way that the bracket is preserved (this gives as a Lie algebra homomorphism).

We begin by mapping $e_1,e_2$: These span the root spaces $\mathfrak{g}_{\alpha_1}$ and $\mathfrak{g}_{\alpha_2}$, respectively, and thus we need to map these to elements in $\mathfrak{g}'_{\tilde{\alpha}_2}$ and $\mathfrak{g}'_{\tilde{\alpha}_1}$, respectively. Thus we define
\begin{align*}
	\psi(e_1)=\tilde{e}_3=:e_1',\\
	\psi(e_2)=\tilde{e}_1=:e_2'.
\end{align*}
Next, let us choose images for $h_1,h_2$. We want $[h_i,e_j]=[\psi(h_i),\psi(e_j)]$ and by a quick calculation, we find
\begin{align*}
	\psi(h_1)&=\mathrm{diag}(1/2,1/2,-1/2,-1/2)=:h_1',\\
	\psi(h_2)&=\mathrm{diag}(1/2,-1/2,1/2,-1/2)=:h_2'.
\end{align*}
Finally, we find images for $f_1,f_2$. These again need to span the corresponding root spaces and need to be chosen in such a way that we get $[e_i,f_j]=[\psi(e_i),\psi(f_j)]$. Another round of calculations leads us to
\begin{align*}
	\psi(f_1)&=\tilde{f}_3=:f_1',\\
	\psi(f_2)&=\frac{\tilde{f}_1}{2}=:f_2'.
\end{align*}
If we furthermore observe that $h_1',h_2',e_1',e_2',f_1',f_2'$ span the Lie algebra $\mathfrak{g}'$, we conclude that $\psi$ is an isomorphism of Lie algebras as desired. As such it induces through exponentiating a local isomorphism of the Lie groups $G=\SO(2,3)$ and $G'=\Sp(4,\mathbb{R})$.

\subsubsection{$\mathbf{\Theta}$-positive structures on $\mathbf{G}'$}
Finally, let us consider $\Theta$-positive structures on $G'$. Due to the local isomorphism with $G$, which we just described, we expect to find two such structures. Indeed, $G'$ is both split real and of Hermitian type of tube type and by the above also locally isomorphic to $\SO(2,3)$, which are three cases for which a $\Theta$-positive structure can be found but it is not hard to see that the last two give the same structure in this case.

We will not investigate the totally positive structure on $G'$, instead let us focus on the other $\Theta$-positive structure on $G'$, which is obtained by picking $\Theta'=\{\alpha_1'\}$. Using the Lie algebra isomorphism $\psi:\mathfrak{g}\to\mathfrak{g}'$, we can basically copy the above formulas (for $G$, in our notation we just need to add a lot of 's) and find
\begin{align*}
	\mathfrak{u'}_{\Theta'} &:=\sum_{\alpha\in\Sigma_\Theta'^+}\mathfrak{g'}_\alpha=\mathfrak{g'}_{\alpha_1'}\oplus\mathfrak{g'}_{\alpha_1'+\alpha_2'}\oplus\mathfrak{g'}_{\alpha_1'+2\alpha_2'}=\mathbb{R}e_1'\oplus\mathbb{R}e_3'\oplus\mathbb{R}e_4',\\
	\mathfrak{u'}_{\Theta'}^{\mathrm{opp}}&:=\sum_{\alpha\in\Sigma_\Theta'^+}\mathfrak{g'}_{-\alpha}=\mathfrak{g'}_{-\alpha_1'}\oplus\mathfrak{g'}_{-\alpha_1'-\alpha_2'}\oplus\mathfrak{g'}_{-\alpha_1'-2\alpha_2'}=\mathbb{R}f_1'\oplus\mathbb{R}f_3'\oplus\mathbb{R}f_4',\\
	\mathfrak{l'}_{\Theta'} &:=\mathfrak{g'}_0\oplus\sum_{\alpha\in\mathrm{Span}(\Delta'-\Theta')\cap\Sigma'^+}(\mathfrak{g'}_\alpha\oplus\mathfrak{g'}_{-\alpha})=\mathfrak{h'}\oplus\mathfrak{g'}_{\alpha_2'}\oplus\mathfrak{g'}_{-\alpha_2'}.
\end{align*}
The center $\mathfrak{z'}_{\Theta'}$ of $\mathfrak{l'}_{\Theta'}$ can simply be read of or obtained using $\psi$ and we find
\begin{equation*}
	\mathfrak{z'}_{\Theta'}=\{\mathrm{diag}(\lambda,\lambda,-\lambda,-\lambda): \lambda\in\mathbb{R}\}
\end{equation*}
and thus $\mathfrak{u'}_{\Theta'}=\mathfrak{u'}_{\alpha_1'}$ again only consists of a single weight space.

By exponentiating an element in $\mathfrak{l'}_{\Theta'}$, we see that a general element $X\in L_\Theta'^\circ$ has the form
\begin{equation*}
	X=
	\begin{pmatrix}
		A & 0\\
		0 & \tran{A}^{-1}
	\end{pmatrix}
\end{equation*}
for $A\in\mathrm{GL}(2,\mathbb{R})^\circ$. Now $Y\in\mathfrak{u'}_{\alpha_1'}$ is of the form
\begin{equation*}
	Y=
	\begin{pmatrix}
		0 & M\\
		0 & 0
	\end{pmatrix},
\end{equation*}
for some symmetric matrix $M$. Therefore we identify this weight space with the space of symmetric matrices.

As we want to consider the action of $L_\Theta'^\circ$ on $\mathfrak{u'}_{\alpha_1'}$ by conjugation, we calculate
\begin{equation*}
	XYX^{-1}=
	\begin{pmatrix}
		0 & AM\tran{A}\\
		0 & 0
	\end{pmatrix}.
\end{equation*}
Thus we see that the cone $c'$ of positive (semi-)definite matrices in $\mathfrak{u'}_{\alpha_1'}$ is preserved by the action of $L_\Theta'^\circ$.

From here the computation of the $\Theta$-positive semigroup is analogous to the one for $G$.

\newpage

\bibliographystyle{amsalpha}
\bibliography{bibliography.bib}

\newpage 

\end{document}